\documentclass[a4paper]{amsart}

\usepackage{amsmath,amssymb,mathtools}
\usepackage{amsthm,amscd}
\usepackage{mathrsfs}
\usepackage[foot]{amsaddr}
\usepackage{aliascnt}
\usepackage{graphicx}
\usepackage[all]{xy}
\usepackage{pgf,tikz,pgfplots}
\usepackage{tikz-cd}
\usepackage{blkarray}
\usepackage{enumitem}
\usepackage{environ,etoolbox}
\usepackage{caption}
\usepackage{wrapfig}
\usepackage{tcolorbox}
\usepackage{microtype}
\usepackage{bbm}
\usepackage{hyperref}
\pgfplotsset{compat=1.18}

\allowdisplaybreaks[4] 
\setlist[enumerate,1]{leftmargin=2.3em}
\setlist[enumerate,2]{leftmargin=1.8em}

\setlist[itemize,1]{leftmargin=2.3em}
\setlist[itemize,2]{leftmargin=1.8em}

\usetikzlibrary{graphs,arrows.meta,calc,decorations.pathmorphing}
\tikzcdset{
	scale cd/.style={every label/.append style={scale=#1},cells={nodes={scale=#1}}},
	cells={font=\everymath\expandafter{\the\everymath\displaystyle}}
}

\newcommand*{\newaliastheorem}[3]{%
	\newaliascnt{#1}{#2}%
	\newtheorem{#1}[#1]{#3}%
	\aliascntresetthe{#1}%
	\expandafter\newcommand\csname #1autorefname\endcsname{#3}%
}

\theoremstyle{plain}

\newaliastheorem{thm}{theorem}{Theorem}
\newaliastheorem{lem}{theorem}{Lemma}
\newaliastheorem{prop}{theorem}{Proposition}
\newaliastheorem{cor}{theorem}{Corollary}

\newaliastheorem{mainthm}{maintheorem}{Theorem}
\newaliastheorem{maincor}{maintheorem}{Corollary}
\newaliastheorem{mainprop}{maintheorem}{Proposition}

\theoremstyle{definition}
\newaliastheorem{defn}{theorem}{Definition}
\newaliastheorem{fact}{theorem}{Fact}
\newtheorem*{acknowledge}{Acknowledgement}
\newtheorem*{ai}{AI declaration}

\theoremstyle{remark}
\newaliastheorem{rem}{theorem}{Remark}
\newaliastheorem{eg}{theorem}{Example}
\newaliastheorem{nota}{theorem}{Notation}
\newaliastheorem{conv}{theorem}{Convention}

\newcommand{\bfj}{\mathbf{j}}

\newcommand{\sfb}{\mathsf{b}}

\newcommand{\calF}{\mathcal{F}}

\newcommand{\calL}{\mathcal{L}}

\newcommand{\calN}{\mathcal{N}}
\newcommand{\calO}{\mathcal{O}}

\newcommand{\frakA}{\mathfrak{A}}
\newcommand{\frakB}{\mathfrak{B}}

\newcommand{\frakD}{\mathfrak{D}}

\newcommand{\rmc}{\mathrm{c}}

\newcommand{\Z}{\mathbb{Z}}

\newcommand{\R}{\mathbb{R}}
\newcommand{\C}{\mathbb{C}}

\newcommand{\M}{\mathbb{M}}

\newcommand{\open}{\mathsf{o}}
\newcommand{\closed}{\mathsf{c}}

\newcommand{\Open}{\mathcal{O}}
\newcommand{\Closed}{\mathcal{C}}
\newcommand{\LC}{\mathcal{LC}}
\newcommand{\LCc}{\mathcal{LC}_{\rmc}}

\DeclareMathOperator{\ev}{ev}
\DeclareMathOperator{\id}{id}

\DeclareMathOperator{\Ext}{Ext}

\DeclareMathOperator{\Prim}{Prim}

\DeclareMathOperator{\Add}{Add}
\DeclareMathOperator{\Mod}{\mathfrak{Mod}}

\newcommand{\conti}{\mathrm{C}}
\newcommand{\loc}{\textup{loc}}
\newcommand{\ex}{\textup{ex}}
\newcommand{\cont}{\textup{cont}}

\newcommand{\LCcat}{\mathfrak{LC}}
\newcommand{\Cstar}{\mathfrak{C^{*}}}
\newcommand{\SCstar}{\mathfrak{C^{*\!}sep}}
\newcommand{\E}{\mathrm{E}}
\newcommand{\KK}{\mathrm{KK}}
\newcommand{\Ecat}{\mathfrak{E}}
\newcommand{\KKcat}{\mathfrak{KK}}
\newcommand{\Boot}{\mathcal{B}}
\newcommand{\Ab}{\mathfrak{Ab}}
\newcommand{\Abc}{\mathfrak{Ab}_{\leq \aleph_0}}
\newcommand{\NT}{\mathcal{NT}}

\newcommand{\FK}{\mathrm{FK}}

\newcommand{\Ktheory}{\mathrm{K}}

\newcommand{\blank}{\,\text{\textvisiblespace}\,}

\newcommand{\Msquare}[8]{%
	\tikz[
	baseline=(current bounding box.center),
	x=1.2cm, y=1cm,
	auto,
	every node/.style={font=\scriptsize, inner sep=1pt},
	line width=0.35pt
	]{
		\node (00) at (0,1) {#1};
		\node (01) at (1,1) {#2};
		\node (10) at (0,0) {#3};
		\node (11) at (1,0) {#4};
		\draw[-{Implies}, double] (00) to node {#5} (01);
		\draw[-{Implies}, double] (00) to node[swap] {#6} (10);
		\draw[-{Implies}, double] (01) to node {#7} (11);
		\draw[-{Implies}, double] (10) to node[swap] {#8} (11);
	}
}

\usepackage[noabbrev,compress]{cleveref}

\crefname{defn}{Definition}{Definitions}
\crefname{thm}{Theorem}{Theorems}
\crefname{prop}{Proposition}{Propositions}
\crefname{lem}{Lemma}{Lemmas}
\crefname{cor}{Corollary}{Corollaries}
\crefname{rem}{Remark}{Remarks}
\crefname{nota}{Notation}{Notations}
\crefname{conv}{Convention}{Conventions}
\crefname{eg}{Example}{Examples}
\crefname{mainthm}{Theorem}{Theorems}
\crefname{maincor}{Corollary}{Corollaries}
\crefname{mainprop}{Proposition}{Propositions}
\crefname{section}{Section}{Sections}
\crefname{subsection}{Subsection}{Subsections}
\crefname{appendix}{Appendix}{Appendices}
\crefname{figure}{Figure}{Figures}

\title[Equivalence of categories of bivariant K-theory]{Equivalence of categories of bivariant K-theory for C*-algebras over topological spaces via reflection functors}
\author[Nanami Hashimoto]{Nanami Hashimoto}
\address{Department of Mathematics, Faculty of Science and Technology, Keio University, 3-14-1 Hiyoshi, Kohoku-ku, Yokohama, 223-8522, Japan}
\email{nanami\_hashimoto@keio.jp}

\begin{document}
	
	\begin{abstract}
		
		In this paper, for various pairs of topological spaces $X$ and $Y$, we introduce an analogue of the Bern\v{s}te\u{\i}n-Gel'fand-Ponomarev reflection functors between the categories of C*-algebras over $X$ and $Y$. We prove that these functors induce equivalences $\mathfrak{KK}(X)_{\mathrm{loc}} \simeq \mathfrak{KK}(Y)_{\mathrm{loc}}$ between Meyer and Nest's localizing subcategories in Kirchberg's ideal-related KK-theory, and $\mathfrak{E}(X) \simeq \mathfrak{E}(Y)$ between the entire categories in Dadarlat and Meyer's ideal-related E-theory. Both equivalences restrict to the respective bootstrap categories: $\mathcal{B}(X) \simeq \mathcal{B}(Y)$ and $\mathcal{B}_{\mathrm{E}}(X) \simeq \mathcal{B}_{\mathrm{E}}(Y)$. Combining these equivalences with a combinatorial argument in quiver representation theory, we obtain the same equivalences for any two finite $T_0$-spaces whose Hasse diagrams are orientations of the same tree. Moreover, we prove a rearrangement property of reflection functors, which yields an analogue of Coxeter functors. Finally, we apply reflection functors to give a conceptual alternative proof of Bentmann and Köhler's universal coefficient theorem for filtrated K-theory over finite $T_0$-spaces whose Hasse diagrams are orientations of Dynkin diagrams of type A.
		
	\end{abstract}
	
	\maketitle
	
	\tableofcontents
	
	\section{Introduction}
	
	One of the central achievements in the classification theory of C*-algebras follows from combining the universal coefficient theorem (UCT) of Rosenberg and Schochet \cite{RS_1987} with the classification theorem proved independently by Kirchberg \cite{Kirchberg_1994} and Phillips \cite{Phillips_2000}. The UCT implies that, for separable C*-algebras $A$ and $B$ in the bootstrap class $\Boot$, isomorphisms $\Ktheory_0(A) \simeq \Ktheory_0(B)$ and $\Ktheory_1(A) \simeq \Ktheory_1(B)$ in K-theory lift to a KK-equivalence between $A$ and $B$. The latter theorem, in turn, shows that a KK-equivalence between separable, stable, nuclear, purely infinite, simple C*-algebras lifts to a $\ast$-isomorphism. Consequently, such algebras in the bootstrap class $\Boot$ are completely classified by their K-theory. A natural problem is to extend this classification result to non-simple C*-algebras.
	
	A convenient way to encode the ideal structure is to consider C*-algebras over a topological space. A C*-algebra over a topological space $X$ is a C*-algebra $A$ equipped with a suitable family of ideals $A(U) \subset A$ indexed by the open subsets $U \subset X$. Kirchberg \cite{Kirchberg_2000} developed an $X$-equivariant version of Kasparov's KK-theory for C*-algebras over $X$. A deep classification theorem due to Kirchberg asserts that a $\KK(X)$-equivalence between separable, stable, nuclear, $\calO_{\infty}$-stable, tight C*-algebras $A$ and $B$ over $X$ lifts to an $X$-equivariant $\ast$-isomorphism. A complete proof of this theorem was later given by Gabe \cite{Gabe_2024}. 
	
	Our long-term goal is to construct a computable K-theoretic invariant that satisfies an $X$-equivariant version of the UCT and whose isomorphisms therefore lift to $\KK(X)$-equivalences between C*-algebras in the $X$-equivariant bootstrap class $\Boot(X)$. In the non-equivariant case, that is, when $X$ is a one-point space, the UCT of Rosenberg and Schochet \cite{RS_1987} shows that ordinary K-theory provides such an invariant.
	Building on earlier work by R{\o}rdam \cite{Rordam_1997},
	Bonkat \cite{Bonkat_2002}, and Restorff \cite{Restorff_2006,Restorff_2008}, Meyer and Nest \cite{MN_2012} constructed filtrated K-theory. This invariant records the K-theory of the subquotients associated with locally closed subsets of $X$, together with all natural transformations between the corresponding K-theory functors. They proved that filtrated K-theory satisfies the UCT for C*-algebras over finite totally ordered spaces. Recall here that every finite $T_0$-space can be identified with a partially ordered set and visualized in a Hasse diagram; see \cref{section:topologies_orders_quivers}. Meyer and Nest also showed that filtrated K-theory fails to satisfy the UCT for a four-point space whose Hasse diagram is an orientation of the Dynkin diagram $D_4$. 
	
	Bentmann and Köhler subsequently showed that, among connected finite $T_0$-spaces, filtrated K-theory satisfies the UCT precisely for accordion spaces; see \cite[Theorem~1.3]{BK_2011}. An accordion space is a finite $T_0$-space whose Hasse diagram is an orientation of a Dynkin diagram of type A. They proved the UCT for accordion spaces by reducing to the totally ordered case treated by Meyer and Nest.
	
	Meyer and Nest \cite{MN_2012} also showed that, for the four-point space in their counterexample, adjoining one further K-theory functor to filtrated K-theory yields an invariant satisfying the UCT. Bentmann \cite{Bentmann_2010} studied refinements of filtrated K-theory for other non-accordion four-point spaces. 
	Related classification results for C*-algebras over four-point spaces, under additional hypotheses including real rank zero, were obtained by Arklint, Restorff, and Ruiz \cite{ARR_2012,ARR_2016} and Arklint, Bentmann, and Katsura \cite{ABK_2014}. 
	
	Motivated by the preceding results, we aim to relate classification problems for C*-algebras over different topological spaces through equivalences of the corresponding bootstrap categories. To this end, we construct, for various pairs of topological spaces $X$ and $Y$, \textit{reflection functors} between the category of C*-algebras over $X$ and that of C*-algebras over $Y$. Our construction is inspired by the reflection functors that Bern\v{s}te\u{\i}n, Gel'fand, and Ponomarev \cite{BGP_1973} introduced in the representation theory of quivers to give a new proof of Gabriel's theorem \cite{Gabriel_1972}. We prove that our reflection functors induce an equivalence between the bootstrap categories $\Boot(X)$ and $\Boot(Y)$. This shows that the classification problem in $\Boot(X)$ up to $\KK(X)$-equivalence and the classification problem in $\Boot(Y)$ up to $\KK(Y)$-equivalence are essentially the same: a solution to either problem can be transferred to the other through the reflection functors.
	
	To describe these equivalences, we briefly recall the relevant categorical framework, restricting for simplicity to finite $T_0$-spaces. Let $X$ be such a space. Let $\Cstar(X)$ denote the category of C*-algebras over $X$, and let $\SCstar(X)$ denote the full subcategory of $\Cstar(X)$ consisting of separable C*-algebras over $X$. 
	Let $\KKcat(X)$ denote the category of Kirchberg's ideal-related KK-theory for separable C*-algebras over $X$ as in \cite[Definition~3.3]{MN_2009}; its isomorphisms are precisely the $\KK(X)$-equivalences. Then $\KKcat(X)$ is a triangulated category with countable coproducts; see \cite[Propositions~3.4 and~3.11]{MN_2009}. The $X$-equivariant bootstrap category $\Boot(X)$ is defined to be the localizing subcategory of $\KKcat(X)$ generated by all the C*-algebras over $X$ with underlying C*-algebra $\C$; see \cite[Definition~4.11]{MN_2009}. 
	
	In \cref{section:reflection_functors}, we first construct finite $T_0$-spaces $W^{\closed}$ and $W^{\open}$, inspired by Ladkani's work \cite{Ladkani_2007} on universal derived equivalences of partially ordered sets. Starting from a finite set $J$, finite partially ordered sets $K$ and $L_j$ for $j \in J$, and order-preserving maps $\theta_j \colon K \to L_j$ for $j \in J$, we define two finite $T_0$-spaces $W^{\closed}$ and $W^{\open}$ with underlying set
	\[
	W:=K \amalg \coprod_{j \in J} L_j. 
	\]
	The superscripts $\closed$ and $\open$ indicate that the subset $K \subset W$ is closed in $W^{\closed}$ and open in $W^{\open}$. For a concrete example of this construction and the Hasse diagrams of the resulting spaces, see \cref{eg:diamond,eg:Wc_Wo}. 
	
	We then construct reflection functors
	\[
	\begin{tikzcd}
		\Cstar(W^{\closed}) \ar[bend left=10]{r}{S_{\closed}^{\open}} &
		\ar[bend left=10]{l}{S_{\open}^{\closed}} \Cstar(W^{\open})\mathrlap{,}
	\end{tikzcd}
	\]
	which restrict to functors
	\[
	\begin{tikzcd}
		\SCstar(W^{\closed}) \ar[bend left=10]{r}{S_{\closed}^{\open}} &
		\ar[bend left=10]{l}{S_{\open}^{\closed}} \SCstar(W^{\open})\mathrlap{.}
	\end{tikzcd}
	\]
	The reflection functor $S_{\open}^{\closed}$ is defined as the mapping cone of a morphism from one object to a direct sum, whereas $S_{\closed}^{\open}$ is defined as the mapping cone of a morphism from a direct sum to one object; see \cref{def:reflection_functor}. 
	This construction is reminiscent of the classical reflection functors introduced by Bern\v{s}te\u{\i}n, Gel'fand, and Ponomarev \cite{BGP_1973}. 
	
	By the universal property of ideal-related KK-theory, the reflection functors between $\SCstar(W^{\closed})$ and $\SCstar(W^{\open})$ induce functors
	\[
	\begin{tikzcd}
		\KKcat(W^{\closed}) \ar[bend left=10]{r}{S_{\closed}^{\open}} & \KKcat(W^{\open})\mathrlap{.} \ar[bend left=10]{l}{S_{\open}^{\closed}}
	\end{tikzcd}
	\]
	In general, these functors do not seem to yield an equivalence between $\KKcat(W^{\closed})$ and $\KKcat(W^{\open})$, because ideal-related KK-theory does not satisfy excision for arbitrary short exact sequences of C*-algebras. 
	We prove that the reflection functors induce equivalences after restriction to the localizing subcategories described below. Our first main result is the following theorem.
	
	\begin{mainthm}[\cref{thm:equivalence_KKloc_Wc_Wo}]\label{mainthm:equivalence_KKloc_Wc_Wo}
		The reflection functors 
		\[
		\begin{tikzcd}
			\KKcat(W^{\closed})_{\loc} \ar[bend left=10]{r}{S_{\closed}^{\open}} &
			\ar[bend left=10]{l}{S_{\open}^{\closed}} \KKcat(W^{\open})_{\loc}
		\end{tikzcd}
		\] 
		together with the natural transformations 
		\[\eta \colon \id_{W^{\closed}} \Rightarrow S_{\open}^{\closed}S_{\closed}^{\open}, \quad \varepsilon \colon S_{\closed}^{\open}S_{\open}^{\closed} \Rightarrow \id_{W^{\open}}\]
		defined in \textup{\cref{section:reflection_functors}} form an adjoint equivalence of triangulated categories with countable coproducts, which restricts to an adjoint equivalence between $\Boot(W^{\closed})$ and $\Boot(W^{\open})$. 
	\end{mainthm}
	
	For a finite $T_0$-space $X$, Meyer and Nest introduced the localizing subcategory $\KKcat(X)_{\loc}$ of $\KKcat(X)$; see \cite[Definition~4.8]{MN_2009}. It is worth noting that every separable nuclear C*-algebra over $X$ belongs to $\KKcat(X)_{\loc}$; see \cite[Proposition~4.10]{MN_2009}. Moreover, $\Boot(X)$ is the full subcategory of $\KKcat(X)_{\loc}$ consisting of separable C*-algebras $A$ over $X$ such that $A(U)$ belongs to the bootstrap category $\Boot$ for every open subset $U \subset X$; see \cite[Proposition 4.12]{MN_2009}. In view of these observations, \cref{mainthm:equivalence_KKloc_Wc_Wo} is already sufficient for applications to the classification theory of C*-algebras over finite $T_0$-spaces.
	
	In \cref{subsection:higher-dimensional_mapping_cones}, we introduce a notion of \textit{higher-dimensional mapping cones} for cubical diagrams of C*-algebras; the usual mapping cone of a $\ast$-homomorphism is precisely the $1$-dimensional case. 
	In \cref{subsection:composition_reflection_functor}, we describe the composites
	$S_{\open}^{\closed}S_{\closed}^{\open}$ and
	$S_{\closed}^{\open}S_{\open}^{\closed}$ as $2$-dimensional mapping cones; see \cref{prop:composition_reflection_functors}. 
	After restricting these composites to the localizing subcategories $\KKcat(W^{\closed})_{\loc}$ and $\KKcat(W^{\open})_{\loc}$, we can combine \cref{thm:key_diagram} with Bott periodicity to obtain the natural isomorphisms $\eta$ and $\varepsilon$ above; see \cref{subsection:application_to_bivariant_K-theory_1}.
	
	Recall that any equivalence of categories can be promoted to an adjoint equivalence by choosing suitable natural isomorphisms. The point here is that the particular natural isomorphisms $\eta$ and $\varepsilon$ constructed above already satisfy the zigzag identities and hence serve as the unit and counit without modification. This shows that their construction is especially natural. We establish one of these zigzag identities by describing the relevant functors and natural transformations as suitable $3$-dimensional mapping cones and then applying a homotopy between morphisms of $2$-dimensional mapping cones; see \cref{thm:zigzag_identity,lem:homotopy_2-dimensional_mapping_cone}. Since $\eta$ and $\varepsilon$ are isomorphisms, this zigzag identity implies the other by an elementary categorical fact (see \cref{lem:zigzag_identities_equivalence}).
	
	The corresponding statement becomes cleaner in ideal-related E-theory. Let $\Ecat(X)$ denote the category of Dadarlat and Meyer's ideal-related E-theory for separable C*-algebras over a finite $T_0$-space $X$; see \cite{DM_2012}. Note that Kirchberg's classification theorem also applies to ideal-related E-theory; see \cite[Theorem~7.1]{Gabe_2022}. The $X$-equivariant E-theoretic bootstrap category $\Boot_{\E}(X)$ is the full subcategory of $\Ecat(X)$ consisting of separable C*-algebras $A$ over $X$ such that $A(U)$ belongs to the non-equivariant E-theoretic bootstrap category $\Boot_{\E}$ for all open subsets $U \subset X$; see \cite[Definition~4.1]{DM_2012}. Unlike KK-theory, E-theory satisfies excision for all short exact sequences of C*-algebras. This stronger excision property allows the reflection functors to yield an equivalence between the entire categories $\Ecat(W^{\closed})$ and $\Ecat(W^{\open})$, rather than only between suitable localizing subcategories as in KK-theory. We thus obtain the following result.
	
	\begin{mainthm}[\cref{thm:equivalence_E_Wc_Wo}]\label{mainthm:equivalence_E_Wc_Wo}
		The statement of \textup{\cref{mainthm:equivalence_KKloc_Wc_Wo}} remains valid after replacing $\KKcat(W^{\closed})_{\loc}$, $\KKcat(W^{\open})_{\loc}$, $\Boot(W^{\closed})$, and $\Boot(W^{\open})$ with $\Ecat(W^{\closed})$, $\Ecat(W^{\open})$, $\Boot_{\E}(W^{\closed})$, and $\Boot_{\E}(W^{\open})$, respectively.
	\end{mainthm}
	
	In the special case where $K$ consists of a single point $x$, the construction of $W^{\closed}$ and $W^{\open}$ includes the classical reflection of an oriented tree at a sink or a source considered in \cite{BGP_1973}. Indeed, under the convention fixed in \cref{section:topologies_orders_quivers}, the point $x$ is a sink in the Hasse diagram of $W^{\closed}$ and a source in that of $W^{\open}$, and the latter is obtained from the former by reversing all arrows entering $x$. Any two orientations of the same tree are related by an admissible sequence of such reflections; see \cite[Theorem~1.2(1)]{BGP_1973}. Therefore, applying \cref{mainthm:equivalence_KKloc_Wc_Wo,mainthm:equivalence_E_Wc_Wo} successively yields the following corollary. 
	
	\begin{maincor}[\cref{cor:KKloc_tree,cor:E_tree}] \label{maincor:KKloc_E_tree}
		Let $X$ and $Y$ be finite $T_0$-spaces whose Hasse diagrams are orientations of the same tree. Then there exist equivalences
		\[
		\begin{aligned}
			\KKcat(X)_{\loc} &\simeq \KKcat(Y)_{\loc},
			&\quad \Boot(X) &\simeq \Boot(Y),\\
			\Ecat(X) &\simeq \Ecat(Y),
			&\quad \Boot_{\E}(X) &\simeq \Boot_{\E}(Y)
		\end{aligned}
		\]
		as triangulated categories with countable coproducts.
	\end{maincor}
	
	This corollary shows that, for a finite $T_0$-space $X$ whose Hasse diagram is an orientation of a tree, the categories $\KKcat(X)_{\loc}$, $\Boot(X)$, $\Ecat(X)$, and $\Boot_{\E}(X)$ depend, up to equivalence, only on the underlying tree and not on its orientation.
	Thus, we may freely choose an orientation of a fixed tree to define $X$ when studying classification in $\Boot(X)$ up to $\KK(X)$-equivalence or in $\Boot_{\E}(X)$ up to $\E(X)$-equivalence.
	
	It is worth emphasizing that \cref{mainthm:equivalence_KKloc_Wc_Wo,mainthm:equivalence_E_Wc_Wo} also yield equivalences between categories associated with spaces whose Hasse diagrams have different underlying undirected graphs, including pairs for which one graph is a tree and the other contains a cycle; see, for example, \cref{eg:diamond}. For such pairs, these equivalences allow the corresponding classification problems to be transferred from the space whose Hasse diagram contains an undirected cycle to the space whose Hasse diagram is an orientation of a tree.
	
	The reflection functors introduced in \cref{section:reflection_functors} change whether a single subset $K$ is open or closed.
	In \cref{section:higher-dimensional_reflection_functors}, we generalize this construction to allow such changes at several subsets $K_i$ simultaneously.
	
	Let $I$ and $J$ be finite sets, and choose subsets $J_i \subset J$ for $i \in I$ such that the bipartite graph with vertex set $I \amalg J$ and edges $\{i, j\}$ for $i \in I$ and $j \in J_i$ is a tree. Given finite partially ordered sets $K_i$ for $i \in I$ and $L_j$ for $j \in J$, together with order-preserving maps $\theta_{i, j} \colon K_i \to L_j$ for $i \in I$ and $j \in J_i$, we construct, for each $\alpha \in \{\closed, \open\}^I$, a finite $T_0$-space $W^{\alpha}$ with underlying set
	\[W:=\coprod_{i \in I} K_i \amalg \coprod_{j \in J} L_j. \]
	For each $i \in I$, the subset $K_i \subset W$ is closed in $W^\alpha$ if $\alpha(i)=\closed$ and open in $W^\alpha$ if $\alpha(i)=\open$. This construction recovers the earlier construction of $W^{\closed}$ and $W^{\open}$ when $|I|=1$. See \cref{eg:Walpha_I_two_point_set} for the case $|I|=2$.
	
	For $\alpha, \beta \in \{\closed, \open\}^I$, we then construct a \textit{higher-dimensional reflection functor}
	\[S_{\alpha}^{\beta} \colon \Cstar(W^{\alpha}) \to \Cstar(W^{\beta}) \]
	as an $I^{\alpha, \beta}$-dimensional mapping cone, where 
	\[I^{\alpha, \beta}:=\{i \in I \mid \alpha(i) \neq \beta(i)\}. \]
	If $\alpha=\beta$, or equivalently, $I^{\alpha, \beta}=\emptyset$, then $S_{\alpha}^{\beta}$ is the identity functor on $\Cstar(W^{\alpha})$. 
	If $I^{\alpha, \beta}=\{i_0\}$ for some $i_0 \in I$, then $S_{\alpha}^{\beta}$ coincides with the reflection functor that changes $K_{i_0}$ from open to closed or from closed to open; see \cref{prop:higher-dimensional_reflection_functor_singleton}.
	More generally, $S_{\alpha}^{\beta}$ is obtained by combining the constructions of the reflection functors associated with the indices in $I^{\alpha, \beta}$ into a single $I^{\alpha, \beta}$-dimensional mapping cone; see \cref{eg:3-dimensional_reflection_functor} for a concrete example in the case $|I^{\alpha, \beta}|=3$.
	
	Our next main result shows that higher-dimensional reflection functors satisfy the following rearrangement property. 
	We prove this theorem in
	\cref{subsection:rearrangements_of_higher-dimensional_reflection_functors}. 
	
	\begin{mainthm}[\cref{thm:rearrangement_reflection_functors}]\label{mainthm:rearrangement_reflection_functors}
		For $\alpha, \beta, \gamma \in \{\closed, \open\}^I$ satisfying $I^{\alpha, \beta} \cap I^{\beta, \gamma}=\emptyset$, equivalently, $I^{\alpha, \beta} \cup I^{\beta, \gamma}=I^{\alpha, \gamma}$, we have
		\[S_{\beta}^{\gamma}S_{\alpha}^{\beta}=S_{\alpha}^{\gamma}\]
		as functors $\Cstar(W^{\alpha}) \to \Cstar(W^{\gamma})$. 
	\end{mainthm}
	
	Applying this theorem repeatedly, we obtain the following corollary.
	
	\begin{maincor}[\cref{cor:Salphabeta_composition_reflection_functors}]\label{maincor:Salphabeta_composition_reflection_functors}
		Let $\alpha, \beta \in \{\closed, \open\}^I$ with $n:=|I^{\alpha, \beta}| \geq 1$. 
		Let $f \colon \{1, 2, \dots, n\} \to I^{\alpha, \beta}$ be any bijection. Let 
		\[\alpha=\alpha_0, \alpha_1, \dots, \alpha_n=\beta\]
		be the unique sequence in $\{\closed, \open\}^I$ such that $I^{\alpha_{k-1}, \alpha_k}=\{f(k)\}$ for each $k=1, 2, \dots, n$. 
		Then 
		\[S_{\alpha_{n-1}}^{\alpha_n}S_{\alpha_{n-2}}^{\alpha_{n-1}} \cdots S_{\alpha_0}^{\alpha_1}=S_{\alpha}^{\beta} \]
		in $[\Cstar(W^{\alpha}), \Cstar(W^{\beta})]$. 
	\end{maincor}
	
	For each $k=1, 2, \dots, n$, the functor $S_{\alpha_{k-1}}^{\alpha_k}$ is the reflection functor that changes $K_{f(k)}$ from open to closed or from closed to open, since $I^{\alpha_{k-1}, \alpha_k}=\{f(k)\}$.
	This corollary shows that, for any ordering of the indices in $I^{\alpha, \beta}$, the composite of the corresponding reflection functors equals the higher-dimensional reflection functor $S_{\alpha}^{\beta}$. 
	
	Since a higher-dimensional reflection functor can be described as a composite of reflection functors by \cref{maincor:Salphabeta_composition_reflection_functors} and each reflection functor induces equivalences between the relevant categories by \cref{mainthm:equivalence_KKloc_Wc_Wo,mainthm:equivalence_E_Wc_Wo}, we obtain the following result.

	\begin{maincor}[\cref{cor:equivalence_KKloc_Walpha_Wbeta,cor:equivalence_E_Walpha_Wbeta}]
		For $\alpha, \beta \in \{\closed, \open\}^I$, the higher-dimensional reflection functors $S_{\alpha}^{\beta}$ and $S_{\beta}^{\alpha}$ induce equivalences
		\[
		\begin{aligned}
			\KKcat(W^{\alpha})_{\loc} &\simeq \KKcat(W^{\beta})_{\loc},
			&\quad \Boot(W^{\alpha}) &\simeq \Boot(W^{\beta}),\\
			\Ecat(W^{\alpha}) &\simeq \Ecat(W^{\beta}),
			&\quad \Boot_{\E}(W^{\alpha}) &\simeq \Boot_{\E}(W^{\beta})
		\end{aligned}
		\]
		as triangulated categories with countable coproducts. 
	\end{maincor}

	In the representation theory of quivers, the Coxeter functors associated with an oriented tree are obtained by composing the reflection functors along an admissible numbering of its vertices; see \cite{BGP_1973}. As a consequence of \cref{maincor:Salphabeta_composition_reflection_functors}, \cref{cor:Coxeter_functor_independence} shows that the analogous construction for our reflection functors is unchanged when the vertices are reordered in another admissible way, paralleling the corresponding property of Coxeter functors. 
	
	In \cref{section:application_of_reflection_functors_to_filtrated_K-theory}, we apply reflection functors to give a concise, conceptual alternative proof of Bentmann and Köhler's UCT for C*-algebras over accordion spaces. Rather than computing generators and relations case by case as in \cite[Theorem~5.12 and Section~6]{BK_2011}, we show that, when $K=\{x\}$ and the source $x$ in the Hasse diagram of $W^{\open}$ has degree at most $2$, the categories of natural transformations of filtrated K-theory for $W^{\open}$ and $W^{\closed}$ are isomorphic. This follows by combining the equivalence of the corresponding categories of bivariant K-theory induced by the reflection functor $S_{\open}^{\closed}$ with a computation of the composites of $S_{\open}^{\closed}$ with the K-theory functors associated with connected locally closed subsets of $W^{\closed}$; see \cref{thm:isomorphism_NT}. The resulting isomorphism induces an isomorphism between the target abelian categories of filtrated K-theory that preserves projective resolutions of length $1$; see \cref{cor:equivalence_UCT}. Applying this result to two accordion spaces related by a reflection at a source and iterating along an admissible sequence of sources, we deduce the UCT for any accordion space from Meyer and Nest's UCT for totally ordered spaces \cite{MN_2012}; see \cref{thm:UCT_accordion}.
	
	Motivated by the results of Meyer and Nest \cite{MN_2012} and Bentmann \cite{Bentmann_2010} on refinement of filtrated K-theory for four-point spaces, together with the results for the type A case above, we hope to develop an analogue in C*-algebra classification theory of Gabriel's ADE classification of quivers of finite representation type \cite{Gabriel_1972}. Inspired by the role of reflection functors and Coxeter functors in Bern\v{s}te\u{\i}n, Gel'fand, and Ponomarev's proof \cite{BGP_1973} of Gabriel's theorem, we aim to use the reflection functors introduced in this paper to construct a computable K-theoretic invariant satisfying the UCT for finite $T_0$-spaces whose Hasse diagrams are orientations of Dynkin diagrams of type ADE. We leave this problem for future work.
	
	\begin{acknowledge}
		The author would like to thank his supervisor Takeshi Katsura for his support, encouragement, and many helpful suggestions throughout this work, especially on homotopies involving $2$-dimensional mapping cones in \cref{lem:homotopy_2-dimensional_mapping_cone,lem:antihomotopy_2-dimensional_mapping_cone} and on the topological setting in which higher-dimensional reflection functors are applied. The author is also grateful to Yosuke Kubota, Kan Kitamura, and Taro Sogabe for fruitful discussions. This work was supported by JST SPRING Grant Number JPMJSP2123, The Keio University Doctorate Student Grant-in-Aid Program from Ushioda Memorial Fund, and the KLL Research Grant from the Keio Leading-edge Laboratory of Science and Technology.
	\end{acknowledge}

	\begin{ai}
		The author used ChatGPT (OpenAI) including GPT-6 Astra and GPT-6.1 Sol for language and LaTeX editing and as an additional check on mathematical arguments. All mathematical ideas and results were developed and independently verified by the author, who takes full responsibility for the manuscript.
	\end{ai}

	\section{Alexandrov topologies and orientations of trees}\label{section:topologies_orders_quivers}
	
	In this section, we fix the notation and conventions concerning Alexandrov spaces, preordered sets, Hasse diagrams, and orientations of trees that will be used throughout the paper.
	
	\subsection{Alexandrov topologies and preorders}
	
	We first recall the correspondence between Alexandrov spaces and preordered sets. 
	For the Alexandrov topology and the specialization preorder, we follow the conventions used in \cite{MN_2009}. 
	
	A \textit{preordered set} is a pair $(V,R)$ consisting of a set $V$ and a reflexive and transitive relation $R \subset V \times V$.
	For $x,y \in V$, we often write $x \leq y$ if $(x,y) \in R$.
	A preordered set $(V,R)$ is \textit{partially ordered} if $R$ is antisymmetric.
	
	Let $(V,R)$ be a preordered set.
	A subset $U \subset V$ is said to be \textit{Alexandrov open} if
	\[
	R \cap \bigl(U \times (V \setminus U)\bigr)=\emptyset.
	\]
	The collection $\Open_R$ of all Alexandrov open subsets of $V$ forms a topology on $V$, called the \textit{Alexandrov topology}. 
	
	For a topological space $(V,\Open)$, the \textit{specialization preorder} $R_{\Open}$ on $V$ is defined as the set of all $(x, y) \in V \times V$ such that the closure of $\{x\}$ is contained in the closure of $\{y\}$. 
	A topological space $(V,\Open)$ is called an \textit{Alexandrov space} if
	\[
	\Open_{R_{\Open}}=\Open.
	\]
	For every preordered set $(V,R)$, we have
	\[
	R_{\Open_R}=R.
	\]
	Therefore, we may identify Alexandrov spaces with preordered sets. 
	Moreover, Alexandrov $T_0$-spaces correspond to partially ordered sets. Every Alexandrov $T_1$-space is discrete. Every finite topological space is an Alexandrov space.
	
	Let $(V, R)$ be a preordered set. 
	A subset $Y \subset V$ is closed in the Alexandrov space $(V, \Open_R)$ if and only if
	\[
	R \cap \bigl((V \setminus Y) \times Y\bigr)=\emptyset,
	\]
	and $Y$ is locally closed in $(V, \Open_R)$ if and only if
	\[
	(R \times Y) \cap (Y \times R)
	\subset
	Y \times Y \times Y.
	\]
	
	The following lemma follows directly from the definitions and will be used to prove
	\cref{prop:connectedness_Wc_Wo,prop:connectedness_Walpha,lem:connected_locally_closed_subsets_Wc_Wo}. 
	Throughout this paper, we adopt the convention that the empty set is not connected.
	
	\begin{lem}\label{lem:connected_Alexandrov_space}
		Let $(V,R)$ be a preordered set.
		Then the Alexandrov space $(V,\Open_R)$ is connected if and only if $V \neq \emptyset$ and the equivalence relation on $V$ generated by $R$ is $V \times V$.
	\end{lem}
	
	Let $(V_1,R_1)$ and $(V_2,R_2)$ be preordered sets.
	A map $\tau \colon V_1 \to V_2$ is a continuous map between the Alexandrov spaces $(V_1,\Open_{R_1})$ and $(V_2,\Open_{R_2})$ if and only if it is an \textit{order-preserving map} between the preordered sets $(V_1,R_1)$ and $(V_2,R_2)$, that is,
	\[
	(\tau(x),\tau(x')) \in R_2,
	\]
	for every $(x,x') \in R_1$.
	
	The following elementary lemma can be reformulated in terms of Alexandrov spaces and continuous maps. This reformulation is useful for verifying that an Alexandrov space is $T_0$ and will be used to prove \cref{prop:Walpha_T0}.
	
	\begin{lem}\label{lem:T0condition}
		Let $(V_1,R_1)$ be a preordered set, let $(V_2,R_2)$ be a partially ordered set, and let $\tau \colon (V_1,R_1) \to (V_2,R_2)$ be an order-preserving map. Then $(V_1,R_1)$ is partially ordered if and only if, for every $y \in V_2$, the preorder on $\tau^{-1}(\{y\})$ induced by $R_1$ is a partial order.
	\end{lem}
	
	\begin{proof}
		The ``only if'' part is clear. Let $x, x' \in V_1$ with $(x, x'), (x', x) \in R_1$. Since $\tau$ is an order-preserving map, we have $(\tau(x), \tau(x')), (\tau(x'), \tau(x)) \in R_2$. Since $R_2$ is a partial order on $V_2$, we have $\tau(x)=\tau(x')$. Put $y:=\tau(x)=\tau(x') \in V_2$. Then $x, x' \in \tau^{-1}(\{y\})$. Since $\tau^{-1}(\{y\})$ is partially ordered, we obtain $x=x'$. This shows the ``if'' part.  
	\end{proof}
	
	\subsection{Partial orders and Hasse diagrams}\label{subsection:partial_orders_and_Hasse_diagrams}
	
	We next recall the correspondence between finite partially ordered sets and Hasse diagrams. 
	
	A \textit{quiver} (or \textit{directed graph}) is a pair $(V,E)$ consisting of a set $V$ and a subset $E \subset \{(x,y) \in V \times V \mid x \neq y\}$. 
	For $(x, y) \in E$, we draw an arrow from $y$ to $x$; see \cref{eg:Hasse_diagram} below.
	A \textit{sink} (respectively, \textit{source}) is an element $x \in V$ such that
	$(y, x) \notin E$ (respectively, $(x, y) \notin E$) for every $y \in V$.
	
	A \textit{path} in a quiver $(V,E)$ is a finite sequence $(x_i)_{i=0}^n$ in $V$ such that $(x_{i-1},x_i) \in E$ for every $i=1,2,\dots,n$. 
	It is called a \textit{path from $x_n$ to $x_0$}
	and its \textit{length} is $n$. A path of length $0$ consists of a single vertex. 
	A \textit{cycle} is a path $(x_i)_{i=0}^n$ of length $n \geq 2$ such that $x_0=x_n$.
	
	A \textit{Hasse diagram} is a quiver $(V, E)$ such that $V$ is finite and, whenever there exists a path from $y$ to $x$ of length at least $2$, we have $(x, y) \notin E$. Note that a Hasse diagram does not contain a cycle. 
	
	For a finite partially ordered set $(V,R)$, define $E_R$ to be the set of all elements $(x,y) \in R$ such that $x \neq y$ and there is no element $z \in V \setminus \{x,y\}$ such that $(x,z),(z,y) \in R$. 
	Then $(V,E_R)$ is a Hasse diagram.
	Conversely, for a Hasse diagram $(V,E)$, define $R_E$ to be the set of all $(x,y) \in V \times V$ such that there exists a path from $y$ to $x$ in $(V,E)$.
	Then $(V,R_E)$ is a partially ordered set.
	These constructions are mutually inverse:
	\[
	R_{E_R}=R, \quad 
	E_{R_E}=E.
	\]
	
	Consequently, finite $T_0$-spaces, finite partially ordered sets, and Hasse diagrams may be identified with one another.
	
	\begin{eg}\label{eg:Hasse_diagram}
		Let $V:=\{1,2,3,4\}$ and $E:=\{(1,2),(2,3),(2,4)\}$. 
		Then $(V,E)$ is the Hasse diagram
		
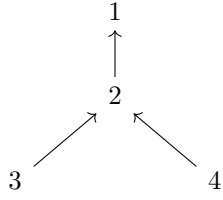
\[
		\begin{tikzcd}
			& 1 & \\
			& 2 \ar[u] & \\
			3 \ar[ur] & & 4 \ar[ul]
		\end{tikzcd}
		\]
		and
		\begin{align*}
			R_E
			&=
			\{(1,2),(2,3),(2,4),(1,3),(1,4)\}
			\cup
			\{(x,x) \mid x \in V\},
			\\
			\Open_{R_E}
			&=
			\{\emptyset,\{3\},\{4\},\{3,4\},\{2,3,4\},V\}.
		\end{align*}
		The closed subsets of the finite $T_0$-space $(V, \Open_{R_E})$ are $\emptyset$, $\{1\}$, $\{1, 2\}$, $\{1, 2, 3\}$, $\{1, 2, 4\}$, and $V$.
		The locally closed subsets of $(V, \Open_{R_E})$ are $\emptyset$, $\{1\}$, $\{2\}$, $\{3\}$, $\{4\}$, $\{1, 2\}$, $\{2, 3\}$, $\{2, 4\}$, $\{3, 4\}$, $\{1, 2, 3\}$, $\{1, 2, 4\}$, $\{2,3,4\}$, and $V$. 
	\end{eg}
	
	Note that our convention for the direction of arrows in Hasse diagrams agrees with
	that used in \cite{MN_2009} and is opposite to that used in \cite{Ladkani_2007}.

	\subsection{Connected subsets of undirected graphs}
	We recall the notions of undirected graphs needed later. 
	
	An \textit{undirected graph} is a pair $G=(V, E)$ of a set $V$ and a subset $E \subset \{e \subset V \mid |e|=2\}$. For a subset $Y \subset V$, the \textit{subgraph of $G$ induced by $Y$} is the undirected graph $(Y, \{e \in E \mid e \subset Y\})$. 
	
	Let $G=(V, E)$ be an undirected graph. 
	A \textit{path} in $G$ is a finite sequence $(x_i)_{i=0}^n$ in $V$ such that $\{x_{i-1}, x_i\} \in E$ for $i=1, 2, \dots, n$. It is called a \textit{path between $x_0$ and $x_n$}. A path $(x_i)_{i=0}^n$ in $G$ is said to be \textit{simple} if $x_i \neq x_j$ for $i, j=0, 1, \dots, n$ with $i \neq j$. 
	
	An undirected graph $G=(V, E)$ is said to be \textit{connected} if $V \neq \emptyset$ and for $x, y \in V$, there exists a path between $x$ and $y$ in $G$. 
	A subset $Y \subset V$ is said to be \textit{connected} in an undirected graph $G=(V, E)$ if the subgraph of $G$ induced by $Y$ is connected. A \textit{connected component} of an undirected graph $G=(V, E)$ is a maximal connected subset of $G$. 
	
	A \textit{tree} is an undirected graph $G=(V, E)$ such that $V$ is non-empty and finite and, for every $x, y \in V$, there exists a unique simple path between $x$ and $y$ in $G$. For connected subsets $Y_1, Y_2 \subset V$ in a tree $G=(V, E)$, the intersection $Y_1 \cap Y_2$ is also connected in $G$ if $Y_1 \cap Y_2 \neq \emptyset$. 
	
	An \textit{orientation} of an undirected graph $G=(V, E)$ is a quiver $X=(V, E_X)$ such that, for all distinct $x, y \in V$, exactly one of $(x, y)$ and $(y, x)$ belongs to $E_X$ if $\{x, y\} \in E$, and neither otherwise. 
	In this case, $G$ is called the \textit{underlying} undirected graph of $X$. 
	Any orientation of a tree is a Hasse diagram and hence may be identified with a finite $T_0$-space. 
	
	The following notion will be used in \cref{section:application_of_reflection_functors_to_filtrated_K-theory}. 
	
	\begin{defn}		
		An \textit{accordion space} is a finite $T_0$-space whose Hasse diagram is an orientation of a Dynkin diagram of type A. 
	\end{defn}
	
	The following lemma will be used to show \cref{prop:Lj_locally_closed}. 
	
	\begin{lem}\label{lem:connected_locally_closed_tree}
		Let $G$ be a tree with vertex set $V$. Let $X$ be a finite $T_0$-space whose Hasse diagram is an orientation of $G$. If a subset $Y \subset V$ is connected in $G$, then $Y$ is locally closed in $X$. 
	\end{lem}
	
	\begin{proof}
		Let $R$ be the specialization order on $X$. 
		Let $(x, y, z) \in (R \times Y) \cap (Y \times R)$. 
		Since $Y$ is connected in the tree $G$ and $x, z \in Y$, the subgraph induced by $Y$ contains the unique simple path between $x$ and $z$. 
		Since $(x, y), (y, z) \in R$, this path passes through $y$. 
		Hence, $y \in Y$. 
		This shows that $Y$ is locally closed in $X$. 
	\end{proof}

	\subsection{Reflections of orientations of a tree}\label{subsection:reflection_of_orientations_in_quivers}
	
	We recall the reflection of an orientation at a sink or source and the argument showing that any two orientations of the same tree are related by an admissible sequence of such reflections; see \cite{BGP_1973} and \cite[VII.5]{ASS_2006}. 
	
	Let $G$ be a tree with vertex set $V$, and let $X=(V, E_X)$ be an orientation of $G$. For $x \in V$, define an orientation $\sigma_x X=(V, E_{\sigma_x X})$ of $G$ by reversing all arrows incident to $x$, that is,
	\[
	E_{\sigma_x X}:=
	\{(y, x) \mid (x, y) \in E_X\}
	\cup
	\{(x, y) \mid (y, x) \in E_X\}
	\cup
	\{(y, z) \in E_X \mid y \neq x,\ z \neq x\}.
	\]
	Observe that $\sigma_x\sigma_x X=X$ for $x \in V$, and that $\sigma_y\sigma_x X=\sigma_x\sigma_y X$ for $x, y \in V$. 
	
	If $x$ is a sink in $X$, then 
	\[
	E_{\sigma_x X}=\{(y,x)\mid (x,y)\in E_X\} \cup \{(y,z)\in E_X \mid y\neq x,\ z\neq x\},
	\]
	and $x$ is a source in $\sigma_x X$. 
	If $x$ is a source in $X$, then
	\[
	E_{\sigma_x X}=\{(x,y)\mid (y,x)\in E_X\} \cup \{(y,z)\in E_X \mid y\neq x,\ z\neq x\},
	\]
	and $x$ is a sink in $\sigma_x X$. 
	
	An \textit{admissible sequence of sinks} (respectively, \textit{admissible sequence of sources}) in $X$ is a finite sequence $(x_1,x_2,\dots,x_n)$ in $V$ such that $x_i$ is a sink (respectively, source) in $\sigma_{x_{i-1}}\sigma_{x_{i-2}}\cdots\sigma_{x_1}X$ for every $i=1,2,\dots,n$ with the convention $\sigma_{x_{i-1}}\sigma_{x_{i-2}}\cdots\sigma_{x_1}X:=X$ if $i=1$.
	Note that, for a finite sequence $(x_i)_{i=1}^n$ in $V$, the sequence $(x_1,x_2,\dots,x_n)$ is an admissible sequence of sinks in $X$ if and only if $(x_n,x_{n-1},\dots,x_1)$ is an admissible sequence of sources in $\sigma_{x_n}\sigma_{x_{n-1}}\cdots\sigma_{x_1}X$. 
	Since $G$ is a tree, $X$ admits an admissible sequence $(x_1,x_2,\dots,x_n)$ of sinks such that $x_1,x_2,\dots,x_n$ are all the elements of $V$.
	In this case,
	\[
	\sigma_{x_n}\sigma_{x_{n-1}}\cdots\sigma_{x_1}X=X.
	\]
	
	Every orientation of a tree is a Hasse diagram and hence may be identified with a finite $T_0$-space.
	If $X$ is a finite $T_0$-space whose Hasse diagram is an orientation of a tree and $x \in X$, we denote by $\sigma_xX$ the finite $T_0$-space whose Hasse diagram is obtained by applying $\sigma_x$ to the Hasse diagram of $X$. 
	A sink (respectively, source) in a Hasse diagram is a closed point (respectively, an open point) in the corresponding finite $T_0$-space.
	We call an admissible sequence of sinks (respectively, sources) in the Hasse diagram of $X$ an \textit{admissible sequence of closed points} (respectively, \textit{admissible sequence of open points}) in $X$.
	
	The following lemma is a reformulation of \cite[Theorem~1.2(1)]{BGP_1973} (see also \cite[Lemma~5.2]{ASS_2006}) in terms of finite $T_0$-spaces.
	It will be used to prove \cref{cor:KKloc_tree,cor:E_tree,thm:UCT_accordion}. 
	
	\begin{lem}\label{lem:tree_orientation}
		Let $X$ and $Y$ be finite $T_0$-spaces whose Hasse diagrams are orientations of the same tree.
		Then there exists an admissible sequence $(x_1,x_2,\dots,x_n)$ of open points in $X$ such that
		\[
		\sigma_{x_n}\sigma_{x_{n-1}}\cdots\sigma_{x_1}X=Y.
		\]
	\end{lem}

	\section{C*-algebras over topological spaces}\label{section:Cstar-algebras_over_topological_spaces}
	
	In this section, we develop the basic theory of C*-algebras over topological spaces using functor categories, rather than the formulation of Meyer and Nest based on primitive ideal spaces in \cite{MN_2009}. After fixing notation, we give a functorial description of C*-algebras over a topological space and develop the permanence properties of inductive limits, tensor actions, direct sums, functors associated with continuous maps and locally closed subsets, and higher-dimensional mapping cones. We finally recall Bott periodicity for homotopy-invariant, stable, half-exact functors on a category of C*-algebras over a topological space. 
	
	\subsection{Notation for C*-algebras and functor categories}\label{subsection:notation_for_Cstar-algebras_and_functor_categories}
	
	We begin by fixing the notation and conventions for C*-algebras and functor categories used throughout the paper. 
	
	For a locally compact Hausdorff space $\Omega$, let $\conti_0(\Omega)$ denote the C*-algebra of all continuous functions $\Omega \to \C$ vanishing at infinity. 
	If $\Omega$ is compact, then write $\conti(\Omega):=\conti_0(\Omega)$. 
	
	\begin{defn}
		We define C*-algebras
		\[
		S:=\conti_0((0, 1)), \quad C:=\conti_0([0, 1))
		\]
		called the \textit{suspension} and the \textit{cone}, respectively.
	\end{defn}
	
	In this paper, $\otimes$ denotes the \textit{maximal tensor product} of C*-algebras. 
	Let $\Omega$ be a locally compact Hausdorff space. For a C*-algebra $A$, the C*-algebra 
	\[\conti_0(\Omega, A):=\conti_0(\Omega) \otimes A\] 
	consists of all continuous functions $\Omega \to A$ vanishing at infinity. For a $\ast$-homomorphism $\varphi \colon A \to B$, the $\ast$-homomorphism 
	\[\varphi_*:=\id_{\conti_0(\Omega)} \otimes \varphi \colon \conti_0(\Omega, A) \to \conti_0(\Omega, B)\] 
	is given by $\varphi_*(f):=\varphi \circ f$ for $f \in \conti_0(\Omega, A)$. Recall that $\conti_0(\Omega', \conti_0(\Omega))=\conti_0(\Omega' \times \Omega)$ for locally compact Hausdorff spaces $\Omega$ and $\Omega'$.
	
	For a C*-algebra $D$ and $i=0, 1, \ldots$, define a C*-algebra $D^i$ recursively by
	\[
	D^i:=
	\begin{cases}
		\C & \text{if $i=0$},\\
		D \otimes D^{i-1} & \text{if $i \geq 1$}.
	\end{cases}
	\]
	
	\begin{defn}
		Let $\Cstar$ denote the category whose objects are C*-algebras and whose morphisms are $\ast$-homomorphisms. Let $\SCstar$ denote the full subcategory of $\Cstar$ consisting of separable C*-algebras. 
	\end{defn}

	For categories $\frakA$ and $\frakB$, let $[\frakA, \frakB]$ denote the functor category from $\frakA$ to $\frakB$. The objects are functors $F \colon \frakA \to \frakB$. The morphisms from a functor $F \colon \frakA \to \frakB$ to a functor $G \colon \frakA \to \frakB$ are natural transformations $\xi \colon F \Rightarrow G$. The composition of morphisms $\xi \colon F \Rightarrow G$ and $\zeta \colon G \Rightarrow H$ is the vertical composition 
	\[\zeta \circ \xi \colon F \Rightarrow H. \]
	
	Let $\frakA$, $\frakB$, and $\frakD$ be categories. For functors
	$F_1,F_2 \colon \frakA \to \frakB$ and
	$G_1,G_2 \colon \frakB \to \frakD$, and natural transformations
	$\xi \colon F_1 \Rightarrow F_2$ and
	$\zeta \colon G_1 \Rightarrow G_2$, let
	\[
	\zeta\xi \colon G_1F_1 \Rightarrow G_2F_2
	\]
	denote their horizontal composite.
	For a functor $G \colon \frakB \to \frakD$, we write $G\xi \colon GF_1 \Rightarrow GF_2$ for the horizontal composite of the identity natural transformation on $G$ with $\xi$.
	Similarly, for a functor $F \colon \frakA \to \frakB$, we write $\zeta F \colon G_1F \Rightarrow G_2F$ for the horizontal composite of $\zeta$ with the identity natural transformation on $F$.
	
	Let $\frakA$, $\frakB$, and $\frakD$ be categories. For a functor $F \colon \frakB \to \frakD$, composition with $F$ on the left and on the right defines functors
	\[
	F_* \colon [\frakA, \frakB] \to [\frakA, \frakD], \quad 
	F^* \colon [\frakD, \frakA] \to [\frakB, \frakA],
	\]
	respectively.
	
	If $A$ is an object of a category $\frakA$, then we write $A \in \frakA$. 
	
	Let $\frakA$ be a category. Let $0$ denote the functor $\frakA \to \Cstar$ that sends every object $Y \in \frakA$ to the zero C*-algebra $0$. The functor $0 \colon \frakA \to \Cstar$ is a zero object in $[\frakA, \Cstar]$. For $A, B \in [\frakA, \Cstar]$, let $0 \colon A \Rightarrow B$ denote the vertical composite $A \Rightarrow 0 \Rightarrow B$ in $[\frakA, \Cstar]$. A sequence
	\[
	\begin{tikzcd}
		0 \ar[Rightarrow]{r} & A \ar[Rightarrow]{r}{\iota} & E \ar[Rightarrow]{r}{\pi} & B \ar[Rightarrow]{r} & 0
	\end{tikzcd}
	\]
	of objects and morphisms in $[\frakA, \Cstar]$ is called a \textit{short exact sequence} if 
	\[
	\begin{tikzcd}
		0 \ar{r} & A(Y) \ar{r}{\iota_Y} & E(Y) \ar{r}{\pi_Y} & B(Y) \ar{r} & 0
	\end{tikzcd}
	\]
	is a short exact sequence in $\Cstar$ for each $Y \in \frakA$. 
	
	We regard a preordered set $\Lambda=(\Lambda, R_\Lambda)$ as a category whose object set is $\Lambda$ and in which, for $\lambda, \mu \in \Lambda$, there exists a unique morphism from $\lambda$ to $\mu$ if $(\lambda, \mu) \in R_{\Lambda}$ and no morphism otherwise. Note that, under the convention for Hasse diagrams fixed in \cref{subsection:partial_orders_and_Hasse_diagrams}, the direction of morphisms in this category is opposite to that of the arrows in the Hasse diagram of $\Lambda$.
	
	Let $\Lambda$ be an upward-directed set. Since $\Cstar$ admits inductive limits, the diagonal functor $\Cstar \to [\Lambda, \Cstar]$ has the left adjoint $\varinjlim_{\Lambda} \colon [\Lambda, \Cstar] \to \Cstar$. 
	Composition with this left adjoint on the left defines a functor
	\[[\Lambda, [\frakA, \Cstar]]=[\frakA, [\Lambda, \Cstar]] \to [\frakA, \Cstar], \] 
	which we also denote by $\varinjlim_{\Lambda}$.
	The resulting functor $\varinjlim_{\Lambda}$ is left adjoint to the diagonal functor $[\frakA, \Cstar] \to [\Lambda, [\frakA, \Cstar]]$. Thus, $[\frakA, \Cstar]$ admits inductive limits, and they are computed pointwise: for $A \in [\Lambda, [\frakA, \Cstar]]$ and $Y \in \frakA$,
	\[
	\Bigl(\varinjlim_{\lambda \in \Lambda} A_\lambda\Bigr)(Y)=\varinjlim_{\lambda \in \Lambda} (A_\lambda(Y)).
	\]
	
	We use $\blank$ as a placeholder for the argument of a functor. 
	The category $\Cstar$ is a symmetric monoidal category in which the tensor product is the maximal tensor product functor $\blank \otimes \blank \colon \Cstar \times \Cstar \to \Cstar$ and the tensor unit is $\C$; see \cite[Lemma~11]{Meyer_2008}. For a category $\frakA$, the functor category $[\frakA, \Cstar]$ is a $\Cstar$-module category in the sense of \cite[Chapter 7]{EGNO_2015}. Its tensor action is the composite
	\[\Cstar \times [\frakA, \Cstar] \to [\frakA, \Cstar] \times [\frakA, \Cstar]=[\frakA, \Cstar \times \Cstar] \to [\frakA, \Cstar], \]
	which we also denote by $\blank \otimes \blank$. 
	Here, the first functor is the product of the diagonal functor $\Cstar \to [\frakA, \Cstar]$ and the identity functor on $[\frakA, \Cstar]$, while the last functor $[\frakA, \Cstar \times \Cstar] \to [\frakA, \Cstar]$ is given by composing $\blank \otimes \blank \colon \Cstar \times \Cstar \to \Cstar$ on the left. 
	This tensor action is computed pointwise: for $(D, A) \in \Cstar \times [\frakA, \Cstar]$ and $Y \in \frakA$,
	\[(D \otimes A)(Y)=D \otimes (A(Y)). \]
	The associativity and unit constraints in $[\frakA, \Cstar]$ are inherited pointwise from those of $\Cstar$.
	
	Let $\frakA$ and $\frakB$ be categories. A functor $F \colon [\frakA, \Cstar] \to [\frakB, \Cstar]$ is said to be
	\begin{itemize}
		\item \textit{exact} if $F$ preserves short exact sequences; 
		\item \textit{continuous} if $F$ commutes with inductive limits. 
	\end{itemize}
	A functor $F \colon [\frakA, \Cstar] \to [\frakB, \Cstar]$ is said to be a \textit{$\Cstar$-module functor} if it commutes with tensor actions. More explicitly, a $\Cstar$-module functor from $[\frakA, \Cstar]$ to $[\frakB, \Cstar]$ is a pair $(F, \varphi)$ consisting of a functor $F \colon [\frakA, \Cstar] \to [\frakB, \Cstar]$ and a natural isomorphism $\varphi \colon F(\blank \otimes \blank) \Rightarrow \blank \otimes F(\blank)$ between the functors $\Cstar \times [\frakA, \Cstar] \to [\frakB, \Cstar]$ satisfying the associativity and unit constraints; see \cite[Definition~7.2.1]{EGNO_2015}. For example, for a category $\frakA$ and $D \in \Cstar$, the functor $D \otimes \blank \colon [\frakA, \Cstar] \to [\frakA, \Cstar]$ together with the tensor flip isomorphism $D \otimes (\blank \otimes \blank) \Rightarrow \blank \otimes (D \otimes \blank)$ between the functors $\Cstar \times [\frakA, \Cstar] \to [\frakA, \Cstar]$ is a $\Cstar$-module functor. 
	
	Let $\frakA$ and $\frakB$ be categories. For $\Cstar$-module functors $(F, \varphi)$ and $(G, \psi)$ from $[\frakA, \Cstar]$ to $[\frakB, \Cstar]$, a \textit{morphism} of $\Cstar$-module functors from $(F, \varphi)$ to $(G, \psi)$ is a natural transformation $\xi \colon F \Rightarrow G$ making the diagram
	\[
	\begin{tikzcd}
		F(D \otimes A) \ar[Rightarrow]{r}{\varphi_{D, A}} \ar[Rightarrow, swap]{d}{\xi_{D \otimes A}} & D \otimes F(A) \ar[Rightarrow]{d}{\id_D \otimes \xi_A} \\
		G(D \otimes A) \ar[Rightarrow, swap]{r}{\psi_{D, A}} & D \otimes G(A)
	\end{tikzcd}
	\]
	commute in $[\frakB, \Cstar]$ for $D \in \Cstar$ and $A \in [\frakA, \Cstar]$; see \cite[Definition~7.2.2]{EGNO_2015}. 
	In this paper, we refer to such a natural transformation as a \textit{$\Cstar$-module natural transformation}.

	All module functors $(F,\varphi)$ considered in this paper arise from constructions such as tensor products and mapping cones, where the natural isomorphism $\varphi$ is canonically induced, often by a universal property. We shall therefore, by abuse of notation, suppress $\varphi$ and simply refer to the functor $F$ as a module functor. Moreover, throughout the paper, we write canonical isomorphisms, such as those induced by universal properties, as equalities unless the isomorphisms themselves are under discussion.
	
	For each $t \in [0, 1]$, let $\ev_t \colon \conti([0, 1]) \twoheadrightarrow \C$ denote the $\ast$-homomorphism defined by $\ev_t(f):=f(t)$ for $f \in \conti([0, 1])$. 
	For a category $\frakA$ and $B \in [\frakA, \Cstar]$, the natural transformation $\ev_t \otimes \id_B \colon \conti([0, 1]) \otimes B \Rightarrow \C \otimes B=B$ is simply denoted by $\ev_t$. 
	
	For a category $\frakA$ and $A, B \in [\frakA, \Cstar]$, two morphisms $\varphi_0, \varphi_1 \colon A \Rightarrow B$ in $[\frakA, \Cstar]$ are said to be \textit{homotopic} if there exists a morphism $\varphi \colon A \Rightarrow \conti([0, 1]) \otimes B$ in $[\frakA, \Cstar]$ satisfying $\varphi_0=\ev_0 \circ \varphi$ and $\varphi_1=\ev_1 \circ \varphi$. An object $A \in [\frakA, \Cstar]$ is said to be \textit{contractible} if the identity morphism $\id \colon A \Rightarrow A$ and the zero morphism $0 \colon A \Rightarrow A$ are homotopic. Recall that the cone $C$ is contractible in $\Cstar$.
	
	All the notation and conventions introduced in this subsection also apply with $\Cstar$ replaced by $\SCstar$, except that inductive limits are restricted to those indexed by countable upward-directed sets.

	\subsection{Categories of C*-algebras over a topological space}
	
	In this subsection, we recall the category $\Cstar(X)$ of C*-algebras over a topological space $X$ and describe it in terms of suitable functors. For this purpose, we define a category $\LCcat(X)$ whose objects are the locally closed subsets of $X$ and show that C*-algebras over $X$ may be identified with suitable functors from $\LCcat(X)$ to $\Cstar$. This viewpoint allows us to regard the C*-algebras associated with locally closed subsets of $X$, together with the canonical maps between them, as parts of a single functor.
	
	Throughout this subsection, we fix a topological space $X$. 
	
	\begin{defn}
		Let $\Open(X)$, $\Closed(X)$, and $\LC(X)$ denote the sets of all open, closed, and locally closed subsets of $X$, respectively. 
	\end{defn}
	
	Recall that $\Open(X)$ and $\Closed(X)$ are subsets of $\LC(X)$. For a subset $Y \subset X$, we endow $Y$ with the relative topology inherited from $X$. For $Z \in \LC(X)$ and $Y \in \LC(Z)$, we have $Y \in \LC(X)$. 
	
	\begin{defn}\label{def:C-algebras_over_X}
		A \textit{C*-algebra over $X$} is a C*-algebra $A$ together with
		ideals $A(U) \subset A$ for $U \in \Open(X)$ satisfying the following properties:
		\begin{enumerate}[label=\textnormal{(\roman*)}]
			\item $A(\emptyset)=0$, $A(X)=A$;
			\item for every $U_1, U_2 \in \Open(X)$, we have
			\[A(U_1 \cap U_2)=A(U_1) \cap A(U_2), \quad A(U_1 \cup U_2)=A(U_1)+A(U_2); \]
			\item 
			for every $U \in \Open(X)$, every upward-directed set $\Lambda$, and every increasing net $(U_{\lambda})_{\lambda \in \Lambda} \subset \Open(U)$ with $U=\bigcup_{\lambda \in \Lambda} U_{\lambda}$, we have
			\[A(U)=\overline{\bigcup_{\lambda \in \Lambda}A(U_{\lambda})}.\]
		\end{enumerate}
	\end{defn}

	\begin{rem}
		In \cite[Definition~2.3]{MN_2009}, a C*-algebra over $X$ is defined to be a C*-algebra $A$ together with a continuous map $\Prim A \to X$, where $\Prim A$ denotes the primitive ideal space of $A$ equipped with the Jacobson topology. If $X$ is a sober space, then the two definitions of a C*-algebra over $X$ are essentially the same by \cite[Lemma 2.25]{MN_2009}. Even if $X$ is not sober, we may assume that $X$ is sober without loss of generality by replacing $X$ with the soberification of $X$. We may also assume that $X$ is $T_0$ without loss of generality by replacing $X$ with the $T_0$-quotient of $X$. Recall that every finite $T_0$-space is sober. 
	\end{rem}
	
	\begin{defn}
		A $\ast$-homomorphism $\varphi \colon A \to B$ between C*-algebras over $X$ is said to be \textit{$X$-equivariant} if $\varphi(A(U)) \subset B(U)$ for all $U \in \Open(X)$. 
	\end{defn}

	\begin{defn}
		Let $\Cstar(X)$ denote the category whose objects are C*-algebras over $X$ and whose morphisms are $X$-equivariant $\ast$-homomorphisms. Let $\SCstar(X)$ denote the full subcategory of $\Cstar(X)$ consisting of all objects $A \in \Cstar(X)$ with $A(X) \in \SCstar$. 
	\end{defn}

	\begin{defn}
		We define a category $\LCcat(X)$ as follows. Its object set is $\LC(X)$. For $Y_1, Y_2 \in \LC(X)$, the morphism set from $Y_1$ to $Y_2$ is the set of all $Z \in \Closed(Y_1) \cap \Open(Y_2)$. For $Y_1, Y_2, Y_3 \in \LC(X)$, $Z_1 \in \Closed(Y_1) \cap \Open(Y_2)$, and $Z_2 \in \Closed(Y_2) \cap \Open(Y_3)$, the composition of the morphisms $Z_1$ and $Z_2$ is defined as the intersection $Z_1 \cap Z_2 \in \Closed(Y_1) \cap \Open(Y_3)$. This composition is associative, and $\id_Y:=Y$ is the identity morphism on $Y$ for $Y \in \LC(X)$. 
	\end{defn}
	
	For $Z \in \LC(X)$ and $Y \in \Open(Z)$, we write $\iota_Y^Z:=Y$ for the morphism from $Y$ to $Z$. 
	For $Z \in \LC(X)$ and $Y \in \Closed(Z)$, define $\pi_Z^Y:=Y$ for the morphism from $Z$ to $Y$. 
	Then, for $Y_1, Y_2 \in \LC(X)$, every morphism $Z$ from $Y_1$ to $Y_2$ can be written as the composite $\iota_Z^{Y_2} \pi_{Y_1}^Z$. 
	In the category $\LCcat(X)$, we have the following relations:
	\begin{enumerate}[label=\textnormal{(\roman*)}]
		\item $\iota_Y^Y=\pi_Y^Y=\id_Y$ for $Y \in \LC(X)$; 
		\item $\iota_{Y_2}^{Y_3} \iota_{Y_1}^{Y_2}=\iota_{Y_1}^{Y_3}$ for $Y_3 \in \LC(X)$, $Y_2 \in \Open(Y_3)$, and $Y_1 \in \Open(Y_2)$; 
		\item $\pi_{Y_2}^{Y_3} \pi_{Y_1}^{Y_2}=\pi_{Y_1}^{Y_3}$ for $Y_1 \in \LC(X)$, $Y_2 \in \Closed(Y_1)$, and $Y_3 \in \Closed(Y_2)$; 
		\item $\pi_{Y_2}^{Y_3} \iota_{Y_1}^{Y_2}=\iota_{Y_1 \cap Y_3}^{Y_3} \pi_{Y_1}^{Y_1 \cap Y_3}$ for $Y_2 \in \LC(X)$, $Y_1 \in \Open(Y_2)$, and $Y_3 \in \Closed(Y_2)$. 
	\end{enumerate}
	
	\begin{defn}
		Let $[\LCcat(X), \Cstar]_{\substack{\ex \\ \cont}}$ denote the full subcategory of $[\LCcat(X), \Cstar]$ consisting of all objects $A \in [\LCcat(X), \Cstar]$ satisfying the following properties:
		\begin{enumerate}[label=\textnormal{(\roman*)}]
			\item for every $Y_2 \in \LC(X)$ and $Y_1 \in \Open(Y_2)$ with $Y_3:=Y_2 \setminus Y_1$, the sequence 
			\[
			\begin{tikzcd}
				0 \ar{r} & A(Y_1) \ar{r}{A(\iota_{Y_1}^{Y_2})} & A(Y_2) \ar{r}{A(\pi_{Y_2}^{Y_3})} & A(Y_3) \ar{r} & 0
			\end{tikzcd}
			\]
			is exact; 
			\item for every $Y \in \LC(X)$, every upward-directed set $\Lambda$, and every increasing net $(Y_{\lambda})_{\lambda \in \Lambda} \subset \Open(Y)$ with $Y=\bigcup_{\lambda \in \Lambda} Y_{\lambda}$, we have
			\[A(Y)=\overline{\bigcup_{\lambda \in \Lambda} A(\iota_{Y_{\lambda}}^Y)(A(Y_{\lambda}))}. \]
		\end{enumerate}
	\end{defn}
	
	The set $\Open(X)$ is partially ordered by inclusion and hence may be regarded as a category. The category $\Open(X)$ may be regarded as a subcategory of $\LCcat(X)$. We may and do identify $\Cstar(X)$ with a full subcategory of $[\Open(X), \Cstar]$. 
	
	\begin{prop}\label{prop:CstarX_functor}
		The functor $[\LCcat(X), \Cstar] \to [\Open(X), \Cstar]$ given by composition with the inclusion functor $\Open(X) \hookrightarrow \LCcat(X)$ on the right restricts to an equivalence
		\[[\LCcat(X), \Cstar]_{\substack{\ex \\ \cont}} \to \Cstar(X). \]
	\end{prop}
	
	The proof of \cref{prop:CstarX_functor} is rather tedious and is given in \cref{appendix:proof_CstarX_functor}. 
	Using this equivalence, we shall identify $\Cstar(X)$ with the category $[\LCcat(X), \Cstar]_{\substack{\ex \\ \cont}}$. 
	Then a morphism $\varphi \colon A \to B$ in $\Cstar(X)$ is regarded as a natural transformation. We use the same symbol $\varphi$ for its component $\varphi_X \colon A(X) \to B(X)$ at $X$.
	The category $\SCstar(X)$ is the full subcategory of $\Cstar(X)$ consisting of all objects belonging to $[\LCcat(X), \SCstar]$. 
	
	\begin{defn}
		For each $Y \in \LC(X)$, let
		\[\ev_X^Y \colon \Cstar(X) \to \Cstar\]
		denote the functor obtained by restricting the evaluation functor $[\LCcat(X), \Cstar] \to \Cstar$ at $Y$. 
		This functor restricts to a functor $\ev_X^Y \colon \SCstar(X) \to \SCstar$. We simply write $\ev_X:=\ev_X^X$. 
	\end{defn}
	
	The functors $\ev_X \colon \Cstar(X) \to \Cstar$ and $\ev_X \colon \SCstar(X) \to \SCstar$ may be regarded as the forgetful functors. 		
	If $X$ is a one-point space, then they are isomorphisms of categories. In this case, we may and do identify $\Cstar(X)$ and $\SCstar(X)$ with $\Cstar$ and $\SCstar$ via $\ev_X$, respectively. 
	
	\subsection{Inductive limits and tensor actions}\label{subsection:tensor_actions_inductive_limits_and_direct_sums}
	
	By the description of $\Cstar(X)$ in \cref{prop:CstarX_functor}, exact and continuous functors on diagrams in $\Cstar$ can be applied pointwise to diagrams in $\Cstar(X)$, as recorded in the following lemma. The proof follows directly from the definitions.
	
	\begin{lem}\label{lem:restriction_exact_continuous}
		Let $X$ be a topological space. Let $\frakA$ be a category. If a functor $F \colon [\frakA, \Cstar] \to \Cstar$ is exact and continuous, then the functor 
		\[[\frakA, [\LCcat(X), \Cstar]]=[\LCcat(X), [\frakA, \Cstar]] \xrightarrow{F_*} [\LCcat(X), \Cstar]\]
		given by composition with $F$ on the left restricts to a functor
		\[[\frakA, \Cstar(X)] \to \Cstar(X). \]
	\end{lem}
	
	We first apply \cref{lem:restriction_exact_continuous} to inductive limits.
	
	\begin{prop}\label{prop:Cstar(X)_inductive_limit}
		Let $X$ be a topological space. Then $\Cstar(X)$ is closed under inductive limits of $[\LCcat(X), \Cstar]$, and $\SCstar(X)$ is closed under countable inductive limits of $[\LCcat(X), \SCstar]$. 
	\end{prop}

	\begin{proof}
		For an upward-directed set $\Lambda$, the inductive limit functor $\varinjlim_{\Lambda} \colon [\Lambda, \Cstar] \to \Cstar$ is exact and continuous. Hence, the inductive limit functor $\varinjlim_{\Lambda} \colon [\Lambda, [\LCcat(X), \Cstar]] \to [\LCcat(X), \Cstar]$ restricts to a functor $\varinjlim_{\Lambda} \colon [\Lambda, \Cstar(X)] \to \Cstar(X)$ by \cref{lem:restriction_exact_continuous}. This shows the assertion for $\Cstar(X)$. The same argument for a countable upward-directed set $\Lambda$ proves the assertion for $\SCstar(X)$.
	\end{proof}
	
	We next apply \cref{lem:restriction_exact_continuous} to tensor actions.
	
	\begin{prop}\label{prop:Cstar(X)_module_subcategory}
		Let $X$ be a topological space. Then $\Cstar(X)$ is a $\Cstar$-module subcategory of $[\LCcat(X), \Cstar]$, and $\SCstar(X)$ is a $\SCstar$-module subcategory of $[\LCcat(X), \SCstar]$. 
	\end{prop}
	
	\begin{proof}
		For $D \in \Cstar$, the functor $D \otimes \blank \colon \Cstar \to \Cstar$ is exact and continuous (see \cite{BO_2008}). Hence, the functor $D \otimes \blank \colon [\LCcat(X), \Cstar] \to [\LCcat(X), \Cstar]$ restricts to a functor $D \otimes \blank \colon \Cstar(X) \to \Cstar(X)$ by \cref{lem:restriction_exact_continuous}. 
		This proves the assertion for $\Cstar(X)$.
		The same argument, with $D \in \SCstar$, proves the assertion for $\SCstar(X)$.
	\end{proof}

	We will also apply \cref{lem:restriction_exact_continuous} to C*-direct sums and mapping cones in \cref{subsection:natural_transformations_associated_with_direct_sums,subsection:higher-dimensional_mapping_cones}. 
	
	\begin{rem}\label{rem:full_subcategory_functor_category}
		Let $X$ be a topological space. 
		Since $\Cstar(X)$ is a full subcategory of $[\LCcat(X), \Cstar]$, for any category $\frakA$, we may regard $[\frakA, \Cstar(X)]$ as a full subcategory of
		\[
		[\frakA, [\LCcat(X), \Cstar]]=[\frakA \times \LCcat(X), \Cstar].
		\]
		Hence, exactness of sequences in $[\frakA, \Cstar(X)]$ is defined by exactness in $[\frakA \times \LCcat(X), \Cstar]$.
		By \cref{prop:Cstar(X)_inductive_limit,prop:Cstar(X)_module_subcategory}, $[\frakA, \Cstar(X)]$ is closed under inductive limits and tensor actions in $[\frakA \times \LCcat(X), \Cstar]$. Moreover, a homotopy between two morphisms in $[\frakA, \Cstar(X)]$ is defined to be a homotopy in $[\frakA \times \LCcat(X), \Cstar]$.
		
		Accordingly, for topological spaces $X$ and $Y$, categories $\frakA$ and $\frakB$, and a functor
		\[
		F \colon [\frakA, \Cstar(X)] \to [\frakB, \Cstar(Y)],
		\]
		we use the notions of exactness, continuity, and a $\Cstar$-module functor in the same sense as in \cref{subsection:notation_for_Cstar-algebras_and_functor_categories}.
		
		The same convention applies with $\Cstar$ replaced by $\SCstar$. Thus, for a functor
		\[
		F \colon [\frakA, \SCstar(X)] \to [\frakB, \SCstar(Y)],
		\]
		we use the notions of exactness and of a $\SCstar$-module functor in the same way. We explicitly say that $F$ commutes with countable inductive limits rather than call it continuous.
	\end{rem}
	
	\begin{rem}\label{rem:induced_functor_exact_continuous_module}
		In the setting of \cref{lem:restriction_exact_continuous}, the functor $[\frakA, \Cstar(X)] \to \Cstar(X)$ induced by $F$ is exact and continuous. If $F$ is a $\Cstar$-module functor, then so is the induced functor $[\frakA, \Cstar(X)] \to \Cstar(X)$. 
	\end{rem}

	\begin{nota}\label{nota:tensor_product}
		Let $X$ be a topological space. For a C*-algebra $D$, we also use $D$ to denote the functor
		\[
		D \otimes \blank \colon \Cstar(X) \to \Cstar(X).
		\]
		For a $\ast$-homomorphism $\rho \colon D_1 \to D_2$, we also use $\rho$ to denote the natural transformation
		\[
		\rho \otimes \blank \colon D_1 \otimes \blank \Rightarrow D_2 \otimes \blank.
		\]
		By the convention above, $S$ and $C$ also denote the functors
		\[
		S,C \colon \Cstar(X) \to \Cstar(X).
		\]
		Since $S$ and $C$ are separable, these functors restrict to functors
		\[
		S,C \colon \SCstar(X) \to \SCstar(X).
		\]
	\end{nota}
	
	The construction in \cref{lem:restriction_exact_continuous} can also be applied pointwise to functors from $\Cstar(X)$ to $\Cstar(Y)$ as in the following lemma. 
	
	\begin{lem}\label{lem:preservation_exact_continuous_module}
		Let $X$ and $Y$ be topological spaces. Let $\frakA$ be a category, and let $G \colon [\frakA, \Cstar] \to \Cstar$ be a functor. Assume that $G$ is an exact, continuous, $\Cstar$-module functor. Let $G_Y \colon [\frakA, \Cstar(Y)] \to \Cstar(Y)$ be the functor induced by $G$ as in \textup{\cref{lem:restriction_exact_continuous}}. Let 
		\[G_{X, Y} \colon [\frakA, [\Cstar(X), \Cstar(Y)]] = [\Cstar(X), [\frakA, \Cstar(Y)]] \xrightarrow{G_{Y*}} [\Cstar(X), \Cstar(Y)]\] 
		be the functor given by composition with $G_Y$ on the left. 
		\begin{enumerate}[label=\textnormal{(\arabic*)}]
			\item \label{lem:preservation_exact_continuous_module_1}
			If a functor $F \colon \Cstar(X) \to [\frakA, \Cstar(Y)]$ 
			is exact, continuous, or a $\Cstar$-module functor, then 
			$G_{X, Y}(F) \colon \Cstar(X) \to \Cstar(Y)$ has the same property.
			\item \label{lem:preservation_exact_continuous_module_2}
			Let $F_1, F_2 \colon \Cstar(X) \to [\frakA, \Cstar(Y)]$ be $\Cstar$-module functors. If $\xi \colon F_1 \Rightarrow F_2$ is a $\Cstar$-module natural transformation, then so is $G_{X, Y}(\xi) \colon G_{X, Y}(F_1) \Rightarrow G_{X, Y}(F_2)$. 
		\end{enumerate}
	\end{lem}
	
	\begin{proof}
		We know that $G_Y$ is an exact, continuous, $\Cstar$-module functor; see \cref{rem:induced_functor_exact_continuous_module}. 
		\begin{enumerate}[label=\textnormal{(\arabic*)}]
			\item 
			Since the classes of exact functors, continuous functors, and $\Cstar$-module functors are each closed under composition, $G_{X, Y}(F)=G_YF$ has the same property as $F$. 
			\item 
			Since the horizontal composition of $\Cstar$-module natural transformations is again a $\Cstar$-module natural transformation, $G_{X, Y}(\xi)=G_Y \xi$ is a $\Cstar$-module natural transformation. 
			\qedhere
		\end{enumerate}
	\end{proof}
	
	In concrete applications of
	\cref{lem:restriction_exact_continuous,lem:preservation_exact_continuous_module},
	we use the same notation for the induced functors as for the original functor.
	
	\cref{lem:preservation_exact_continuous_module} will be used repeatedly in \cref{section:reflection_functors,section:higher-dimensional_reflection_functors}. The principal functors constructed there, including reflection functors, are defined using C*-direct sums and mapping cones. By applying this lemma to these constructions, we will deduce exactness, continuity, and the module property of these functors from the corresponding properties of their constituent functors and natural transformations; see \cref{prop:reflection_functor_exact_continuous_module,prop:Tc_To_exact_continuous_module,prop:eta_epsilon_module,prop:Salphabeta_exact_continuous_module}.
	
	\begin{rem}
		Let $X$ and $Y$ be topological spaces, and let $\frakA$ be a category. Under the identification
		\[
		[\frakA, [\Cstar(X), \Cstar(Y)]]=[\frakA \times \Cstar(X), \Cstar(Y)],
		\]
		exact sequences, inductive limits, and tensor actions are defined in $[\frakA, [\Cstar(X), \Cstar(Y)]]$ by \cref{rem:full_subcategory_functor_category}.                     
		In the setting of \cref{lem:preservation_exact_continuous_module}, the functor
		\[
		G_{X, Y} \colon [\frakA, [\Cstar(X), \Cstar(Y)]] \to [\Cstar(X), \Cstar(Y)]
		\]
		is an exact, continuous, $\Cstar$-module functor. \
	\end{rem}

	\begin{rem}\label{rem:exact_continuous_module_iff}
		Let $X$ and $Y$ be topological spaces. Let $\frakA$ be a category. Let $F \colon \Cstar(X) \to [\frakA, \Cstar(Y)]$ be a functor. Under the identification $[\Cstar(X), [\frakA, \Cstar(Y)]]=[\frakA, [\Cstar(X), \Cstar(Y)]]$, we may regard $F$ as a functor $\frakA \to [\Cstar(X), \Cstar(Y)]$. 
		\begin{enumerate}[label=\textnormal{(\arabic*)}]
			\item The functor $F \colon \Cstar(X) \to [\frakA, \Cstar(Y)]$ is exact (respectively, continuous) if and only if $F(A) \colon \Cstar(X) \to \Cstar(Y)$ is exact (respectively, continuous) for every $A \in \frakA$. 
			\item The functor $F \colon \Cstar(X) \to [\frakA, \Cstar(Y)]$ is a $\Cstar$-module functor if and only if $F(A) \colon \Cstar(X) \to \Cstar(Y)$ is a $\Cstar$-module functor for every $A \in \frakA$ and $F(\varphi) \colon F(A) \Rightarrow F(B)$ is a $\Cstar$-module natural transformation for every $A, B \in \frakA$ and $\varphi \in \frakA(A, B)$. 
		\end{enumerate}
	\end{rem}

	\subsection{Natural transformations associated with direct sums}\label{subsection:natural_transformations_associated_with_direct_sums}
	
	The construction of reflection functors will require natural transformations whose source or target is a direct sum of functors. In this subsection, we record the pointwise direct sum construction in functor categories and introduce notation for assembling natural transformations to and from specified summands.
	
	Let $J$ be a set, regarded as a discrete category. 
	The C*-direct sum functor 
	\[\bigoplus_J \colon [J, \Cstar] \to \Cstar\] 
	is an exact, continuous, $\Cstar$-module functor. Hence, for a topological space $X$, it induces a functor
	\[\bigoplus_J \colon [J, \Cstar(X)] \to \Cstar(X) \]
	as in \cref{lem:restriction_exact_continuous}. 
	Moreover, for topological spaces $X$ and $Y$, we obtain the functor
	\[\bigoplus_J \colon [J, [\Cstar(X), \Cstar(Y)]] \to [\Cstar(X), \Cstar(Y)] \] as in \cref{lem:preservation_exact_continuous_module}. 
	If $J$ is countable, replacing $\Cstar(\blank)$ with $\SCstar(\blank)$ in the functors above yields well-defined functors as well.
	
	For each $j \in J$, let $F_j \colon \Cstar(X) \to \Cstar(Y)$ be a functor.
	If the functors $F_j$ for $j \in J$ are all exact, all continuous, or all $\Cstar$-module functors, then the functor
	\[
	\bigoplus_{j \in J} F_j \colon \Cstar(X) \to \Cstar(Y)
	\]
	has the corresponding property by \cref{lem:preservation_exact_continuous_module} \labelcref{lem:preservation_exact_continuous_module_1}. 
	
	\begin{defn}\label{def:natural_transformation_direct_sum}
		Let $J$ be a set and $J_0 \subset J$. 
		Let $F$ and $G_j$ for $j \in J$ be functors $\Cstar(X) \to \Cstar(Y)$. 
		\begin{enumerate}[label=\textnormal{(\arabic*)}]
			\item \label{def:natural_transformation_direct_sum_1}
			For each $j \in J_0$, let $\varphi^{(j)} \colon F \Rightarrow G_j$ be a natural transformation. Note that, for each $j \in J_0$ and $A \in \Cstar(X)$, the component $\varphi^{(j)}_A \colon FA \to G_jA$ at $A$ is a $Y$-equivariant $\ast$-homomorphism. If $J_0$ is \textit{finite}, then, for each $A \in \Cstar(X)$, we can define a $Y$-equivariant $\ast$-homomorphism
			\[(\varphi^{(j)}_A)_{j \in J_0} \colon FA \to \bigoplus_{j \in J} G_jA, \quad a \mapsto (a_j)_{j \in J}, \]
			by
			\[a_j:=
			\begin{dcases}
				\varphi^{(j)}_A(a) & \text{if $j \in J_0$}, \\
				0 & \text{if $j \in J \setminus J_0$}. 
			\end{dcases}
			\]
			These $Y$-equivariant $\ast$-homomorphisms $(\varphi^{(j)}_A)_{j \in J_0}$ for $A \in \Cstar(X)$ define a natural transformation
			\[(\varphi^{(j)})_{j \in J_0} \colon F \Rightarrow \bigoplus_{j \in J} G_j. \]
			
			\item \label{def:natural_transformation_direct_sum_2}
			For each $j \in J_0$, let $\psi^{(j)} \colon G_j \Rightarrow F$ be a natural transformation. Note that, for each $j \in J_0$ and $A \in \Cstar(X)$, the component $\psi^{(j)}_A \colon G_jA \to FA$ at $A$ is a $Y$-equivariant $\ast$-homomorphism. 
			The natural transformations $\psi^{(j)} \colon G_j \Rightarrow F$ for $j \in J_0$ are said to be \textit{mutually orthogonal} if, for each $A \in \Cstar(X)$, the $\ast$-homomorphisms
			\[
			\psi^{(j)}_A \colon G_jA(Y) \to FA(Y), \quad j \in J_0,
			\]
			are mutually orthogonal. In this case, we can define a $Y$-equivariant $\ast$-homomorphism
			\[\sum_{j \in J_0} \psi^{(j)}_A \colon \bigoplus_{j \in J} G_j A \to FA\]
			by
			\[\Bigl(\sum_{j \in J_0} \psi^{(j)}_A\Bigr)((a_j)_{j \in J}):=\sum_{j \in J_0} \psi^{(j)}_A(a_{j})\] 
			for $(a_j)_{j \in J} \in \bigoplus_{j \in J} G_j A$. 
			These $Y$-equivariant $\ast$-homomorphisms $\sum_{j \in J_0} \psi^{(j)}_A$ for $A \in \Cstar(X)$ define a natural transformation
			\[\sum_{j \in J_0} \psi^{(j)} \colon \bigoplus_{j \in J} G_j \Rightarrow F. \]
		\end{enumerate}
	\end{defn}
	
	In the setting of \cref{def:natural_transformation_direct_sum}, suppose that $F$ and $G_j$ for $j \in J$ are $\Cstar$-module functors. Then $\bigoplus_{j \in J} G_j \colon \Cstar(X) \to \Cstar(Y)$ is also a $\Cstar$-module functor. If $\varphi^{(j)} \colon F \Rightarrow G_j$ is a $\Cstar$-module natural transformation for each $j \in J_0$ and if $J_0$ is finite, then 
	\[(\varphi^{(j)})_{j \in J_0} \colon F \Rightarrow \bigoplus_{j \in J} G_j \] 
	is also a $\Cstar$-module natural transformation. 
	If the mutually orthogonal natural transformations $\psi^{(j)}$ for $j \in J_0$ are all $\Cstar$-module natural transformations, then 
	\[\sum_{j \in J_0} \psi^{(j)} \colon \bigoplus_{j \in J} G_j \Rightarrow F \]
	is also a $\Cstar$-module natural transformation. 
	
	We shall use the operations in \cref{def:natural_transformation_direct_sum} when defining higher-dimensional reflection functors; see \cref{def:phialphabeta}.
	When defining reflection functors in \cref{def:reflection_functor}, we shall use these operations only in the special case where $J_0=J$. 
	The finiteness of $J_0$ required in \cref{def:natural_transformation_direct_sum} \labelcref{def:natural_transformation_direct_sum_1} will affect the setup; see \cref{rem:J_finite}.
	The mutual orthogonality required in \cref{def:natural_transformation_direct_sum} \labelcref{def:natural_transformation_direct_sum_2} will also affect the setup in \cref{section:reflection_functors,section:higher-dimensional_reflection_functors}.

	\subsection{Functors associated with continuous maps and locally closed subsets}
	
	We introduce the functors associated with continuous maps and locally closed subsets, and record the relations and natural transformations among them that will be used later.
	
	For a topological space $X$, the identity functor $\Cstar(X) \to \Cstar(X)$ is denoted by $\id_X$. 
	
	\begin{defn}\label{def:functor_induced_by_continuous_map}
		Let $X$ and $Y$ be topological spaces, and let $\theta \colon X \to Y$ be a continuous map. The map $\LC(Y) \ni Z \mapsto \theta^{-1}(Z) \in \LC(X)$ yields a functor $\LCcat(Y) \to \LCcat(X)$. 
		Composition with this functor on the right defines a functor $[\LCcat(X), \Cstar] \to [\LCcat(Y), \Cstar]$, which we denote by $\theta_*$. This functor restricts to functors
		\[\theta_* \colon \Cstar(X) \to \Cstar(Y), \quad \theta_* \colon \SCstar(X) \to \SCstar(Y). \]
		If $X$ is a subspace of $Y$ and $\theta \colon X \to Y$ is the inclusion map, we write $i_X^Y:=\theta_*$ for both functors.
	\end{defn}
	
	By definition, for $A \in \Cstar(X)$ and $Z \in \LC(Y)$, we have $\theta_* A(Z)=A(\theta^{-1}(Z))$. In particular, $\theta_* A(Y)=A(X)$. 
	For a morphism $\varphi \colon A \to B$ in $\Cstar(X)$ and $Z \in \LC(Y)$, we have $(\theta_*\varphi)_Z=\varphi_{\theta^{-1}(Z)} \colon A(\theta^{-1}(Z)) \to B(\theta^{-1}(Z))$. This shows the following lemma.
	
	\begin{lem}\label{lem:evaluation_extension}
		For a continuous map $\theta \colon X \to Y$ between topological spaces and for $Z \in \LC(Y)$, we have
		\[\ev_Y^Z \theta_*=\ev_X^{\theta^{-1}(Z)}. \]
	\end{lem}
	
	The following proposition follows immediately from the definitions. 
	
	\begin{prop}\label{prop:induced_functor_exact_continuous_module}
		For a continuous map $\theta \colon X \to Y$ between topological spaces, the functor $\theta_* \colon \Cstar(X) \to \Cstar(Y)$ is an exact, continuous, $\Cstar$-module functor. 
	\end{prop}

	\begin{defn}\label{def:restriction_functor}
		Let $X$ be a topological space and $Y \in \LC(X)$.  
		Composition with the inclusion functor $\LCcat(Y) \hookrightarrow \LCcat(X)$ on the right defines a functor $[\LCcat(X), \Cstar] \to [\LCcat(Y), \Cstar]$, which we denote by $r_X^Y$. 
		This functor restricts to functors
		\[r_X^Y \colon \Cstar(X) \to \Cstar(Y), \quad r_X^Y \colon \SCstar(X) \to \SCstar(Y).\]
	\end{defn}
	
	By definition, for $A \in \Cstar(X)$ and $Z \in \LC(Y)$, we have $r_X^Y A(Z)=A(Z)$. For a morphism $\varphi \colon A \to B$ in $\Cstar(X)$ and $Z \in \LC(Y)$, we have $(r_X^Y \varphi)_Z=\varphi_Z \colon A(Z) \to B(Z)$. This shows the following lemma. 
	
	\begin{lem}\label{lem:evaluation_restriction}
		Let $X$ be a topological space, $Y \in \LC(X)$, and $Z \in \LC(Y)$. Then
		\[\ev_Y^Z r_X^Y=\ev_X^Z. \]
	\end{lem}
	
	The following proposition follows immediately from the definitions. 
	
	\begin{prop}\label{prop:restriction_functor_exact_continuous_module}
		Let $X$ be a topological space and $Y \in \LC(X)$. The functor $r_X^Y \colon \Cstar(X) \to \Cstar(Y)$ is an exact, continuous, $\Cstar$-module functor. 
	\end{prop}

	The following lemma is straightforward to verify and will be used repeatedly later.
	
	\begin{lem}\label{lem:relations_extension_restriction_functors}
		\
		\begin{enumerate}[label=\textnormal{(\arabic*)}]
			\item\label{lem:relations_extension_restriction_functors_1} For a topological space $X$, we have 
			\[i_{\emptyset}^X=0, \quad r_X^{\emptyset}=0.\]
			
			\item\label{lem:relations_extension_restriction_functors_2} For a topological space $X$, we have 
			\[i_X^X=r_X^X=\id_X.\]
			
			\item\label{lem:relations_extension_restriction_functors_3} For continuous maps $\theta \colon X \to Y$ and $\tau \colon Y \to Z$ between topological spaces, we have 
			\[\tau_* \theta_*=(\tau \circ \theta)_*. \] 
			In particular, for a topological space $Z$ and for subspaces $X \subset Y \subset Z$, we have 
			\[i_Y^Z i_X^Y=i_X^Z. \]
			
			\item\label{lem:relations_extension_restriction_functors_4} Let $X$ be a topological space, $Y \in \LC(X)$, and $Z \in \LC(Y)$. Then
			\[r_Y^Z r_X^Y=r_X^Z. \] 
			
			\item\label{lem:relations_extension_restriction_functors_5} For a continuous map $\theta \colon X \to Y$ between topological spaces and for $Z \in \LC(Y)$, we have 
			\[r_Y^Z \theta_*=\theta^Z_* r_X^{\theta^{-1}(Z)}, \] 
			where $\theta^Z \colon \theta^{-1}(Z) \to Z$ is the restriction of $\theta$. 
			In particular, for a topological space $X$ and for $Y, Z \in \LC(X)$, we have
			\[r_X^Z i_Y^X=i_{Y \cap Z}^Z r_Y^{Y \cap Z}. \]
		\end{enumerate}
	\end{lem}

	Let $X$ and $Y$ be topological spaces. Let $Z_1, Z_2 \in \LC(X)$ be such that $Z_1 \subset Z_2$. Let $\theta_1 \colon Z_1 \to Y$ and $\theta_2 \colon Z_2 \to Y$ be continuous maps. Consider the following two assumptions: 
	\begin{itemize} 
		\item[(O)]
		$Z_1 \in \Open(Z_2)$, and $\theta_1^{-1}(U) \subset \theta_2^{-1}(U)$ for every $U \in \Open(Y)$;
		\item[(C)]
		$Z_1 \in \Closed(Z_2)$, and $\theta_1^{-1}(F) \subset \theta_2^{-1}(F)$ for every $F \in \Closed(Y)$. 
	\end{itemize}
	
	\begin{rem}\label{rem:assumptions_O_C}
		\
		\begin{enumerate}[label=\textnormal{(\arabic*)}]
			\item \label{rem:assumptions_O_C_1}
			The second condition in the assumption \textup{(O)} is equivalent to $\theta_2^{-1}(F) \cap Z_1 \subset \theta_1^{-1}(F)$ for every $F \in \Closed(Y)$. 
			\item \label{rem:assumptions_O_C_2}
			The second condition in the assumption \textup{(C)} is equivalent to $\theta_2^{-1}(U) \cap Z_1 \subset \theta_1^{-1}(U)$ for every $U \in \Open(Y)$. 
		\end{enumerate}
	\end{rem}
	
	\begin{lem}\label{lem:iota_pi} \
		\begin{enumerate}[label=\textnormal{(\arabic*)}]
			\item \label{lem:iota_pi_1}
			Under the assumption \textup{(O)}, the inclusion maps $A(\iota_{Z_1}^{Z_2}) \colon A(Z_1) \to A(Z_2)$ for $A \in \Cstar(X)$ define a $\Cstar$-module natural transformation
			\[\theta_{1*} r_X^{Z_1} \Rightarrow \theta_{2*} r_X^{Z_2}. \]
			
			\item \label{lem:iota_pi_2}
			Under the assumption \textup{(C)}, the quotient maps $A(\pi_{Z_2}^{Z_1}) \colon A(Z_2) \to A(Z_1)$ for $A \in \Cstar(X)$ define a $\Cstar$-module natural transformation
			\[\theta_{2*} r_X^{Z_2} \Rightarrow \theta_{1*} r_X^{Z_1}. \]
		\end{enumerate}
	\end{lem}
	
	\begin{proof}
		\
		\begin{enumerate}[label=\textnormal{(\arabic*)}]
			\item 
			Let $A \in \Cstar(X)$.
			For $U \in \Open(Y)$, the assumption \textup{(O)} gives
			\[A(\iota_{Z_1}^{Z_2})(\theta_{1*} r_X^{Z_1}A(U))=A(\theta_1^{-1}(U))\subset A(\theta_2^{-1}(U))=\theta_{2*} r_X^{Z_2}A(U). \]
			This shows that $A(\iota_{Z_1}^{Z_2}) \colon A(Z_1) \to A(Z_2)$ is a $Y$-equivariant $\ast$-homomorphism $\theta_{1*} r_X^{Z_1}A \to \theta_{2*} r_X^{Z_2}A$. 
			The remaining verification is straightforward.

			\item 
			Let $A \in \Cstar(X)$. For $U \in \Open(Y)$, the assumption \textup{(C)} (see also \cref{rem:assumptions_O_C} \labelcref{rem:assumptions_O_C_2}) gives
			\[A(\pi_{Z_2}^{Z_1})(\theta_{2*} r_X^{Z_2}A(U))=A(\theta_2^{-1}(U) \cap Z_1) \subset A(\theta_1^{-1}(U))=\theta_{1*} r_X^{Z_1}A(U). \]
			This shows that $A(\pi_{Z_2}^{Z_1}) \colon A(Z_2) \to A(Z_1)$ is a $Y$-equivariant $\ast$-homomorphism $\theta_{2*} r_X^{Z_2}A \to \theta_{1*} r_X^{Z_1}A$. 
			The remaining verification is straightforward.
			\qedhere
		\end{enumerate}
	\end{proof}
	
	\begin{defn}\label{def:iota_pi} \
		\begin{enumerate}[label=\textnormal{(\arabic*)}]
			\item \label{def:iota_pi_1}
			Under the assumption \textup{(O)}, we also denote by $\iota_{Z_1}^{Z_2}$ the natural transformation $\theta_{1*} r_X^{Z_1} \Rightarrow \theta_{2*} r_X^{Z_2}$ specified in \cref{lem:iota_pi} \labelcref{lem:iota_pi_1}.
			\item \label{def:iota_pi_2}
			Under the assumption \textup{(C)}, we also denote by $\pi_{Z_2}^{Z_1}$ the  natural transformation $\theta_{2*} r_X^{Z_2} \Rightarrow \theta_{1*} r_X^{Z_1}$ specified in \cref{lem:iota_pi} \labelcref{lem:iota_pi_2}. 
		\end{enumerate}
	\end{defn}
	
	\begin{eg}\label{eg:iota_pi_evaluation}
		Consider the case where $Y$ is a one-point space. Under the identification $\Cstar(Y)=\Cstar$, we have $\theta_{j*} r_X^{Z_j}=\ev_X^{Z_j}$ in $[\Cstar(X), \Cstar]$ for $j=1, 2$. 
		\begin{enumerate}[label=\textnormal{(\arabic*)}]
			\item \label{eg:iota_pi_evaluation_1}
			If $Z_1 \in \Open(Z_2)$, then the assumption \textup{(O)} holds and hence we obtain the $\Cstar$-module natural transformation $\iota_{Z_1}^{Z_2} \colon \ev_X^{Z_1} \Rightarrow \ev_X^{Z_2}$. 
			\item \label{eg:iota_pi_evaluation_2}
			If $Z_1 \in \Closed(Z_2)$, then the assumption \textup{(C)} holds and hence we obtain the $\Cstar$-module natural transformation $\pi_{Z_2}^{Z_1} \colon \ev_X^{Z_2} \Rightarrow \ev_X^{Z_1}$. 
		\end{enumerate}
	\end{eg}
	
	\begin{eg}\label{eg:iota_pi_ir}
		Consider the case where $X=Y$ and where $\theta_1$ and $\theta_2$ are the inclusion maps. Then $\theta_{j*}=i_{Z_j}^X$ in $[\Cstar(Z_j), \Cstar(X)]$ for $j=1, 2$. 
		\begin{enumerate}[label=\textnormal{(\arabic*)}]
			\item 
			If $Z_1 \in \Open(Z_2)$, then the assumption \textup{(O)} holds and hence we obtain the $\Cstar$-module natural transformation $\iota_{Z_1}^{Z_2} \colon i_{Z_1}^X r_X^{Z_1} \Rightarrow i_{Z_2}^X r_X^{Z_2}$. 
			\item 
			If $Z_1 \in \Closed(Z_2)$, then the assumption \textup{(C)} holds and hence we obtain the $\Cstar$-module natural transformation $\pi_{Z_2}^{Z_1} \colon i_{Z_2}^X r_X^{Z_2} \Rightarrow i_{Z_1}^X r_X^{Z_1}$. 
		\end{enumerate}
	\end{eg}

	The following lemma shows that the relevant horizontal compositions of the natural transformations introduced above are determined by the corresponding compositions of functors. The inverse images and restricted maps appearing in the statement arise from \cref{lem:relations_extension_restriction_functors} \labelcref{lem:relations_extension_restriction_functors_5}, and the remaining verification is straightforward. It will be used repeatedly later.
	
	\begin{lem}\label{lem:composition_theta_r_iota_pi}
		\
		\begin{enumerate}[label=\textnormal{(\arabic*)}]
			\item \label{lem:composition_theta_r_iota_pi_1}
			Let $X'$ be a topological space. Let $Z \in \LC(X')$ and $\tau \colon Z \to X$ be a continuous map. For $j=1, 2$, let $\tau^{Z_j} \colon \tau^{-1}(Z_j) \to Z_j$ denote the restriction of $\tau$. 
			
			\begin{enumerate}[label=\textnormal{(\roman*)}]
				\item \label{lem:composition_theta_r_iota_pi_1_1}
				Under the assumption \textup{(O)}, we have $\tau^{-1}(Z_1) \in \Open(\tau^{-1}(Z_2))$ and 
				\[(\theta_1 \circ \tau^{Z_1})^{-1}(U) \subset (\theta_2 \circ \tau^{Z_2})^{-1}(U)\] 
				for every $U \in \Open(Y)$. Moreover, the horizontal composite 
				\[\iota_{Z_1}^{Z_2} \tau_* r_{X'}^Z \colon \theta_{1*} r_X^{Z_1} \tau_* r_{X'}^Z \Rightarrow \theta_{2*} r_X^{Z_2} \tau_* r_{X'}^Z\] 
				is equal to
				\[\iota_{\tau^{-1}(Z_1)}^{\tau^{-1}(Z_2)} \colon (\theta_1 \circ \tau^{Z_1})_* r_{X'}^{\tau^{-1}(Z_1)} \Rightarrow (\theta_2 \circ \tau^{Z_2})_* r_{X'}^{\tau^{-1}(Z_2)}\]
				in $[\Cstar(X'), \Cstar(Y)]$. 
				
				\item \label{lem:composition_theta_r_iota_pi_1_2}
				Under the assumption \textup{(C)}, we have $\tau^{-1}(Z_1) \in \Closed(\tau^{-1}(Z_2))$ and 
				\[(\theta_1 \circ \tau^{Z_1})^{-1}(F) \subset (\theta_2 \circ \tau^{Z_2})^{-1}(F)\] 
				for every $F \in \Closed(Y)$. Moreover, the horizontal composite 
				\[\pi_{Z_2}^{Z_1} \tau_* r_{X'}^Z \colon \theta_{2*} r_X^{Z_2} \tau_* r_{X'}^Z \Rightarrow \theta_{1*} r_X^{Z_1} \tau_* r_{X'}^Z\]
				is equal to
				\[\pi_{\tau^{-1}(Z_2)}^{\tau^{-1}(Z_1)} \colon (\theta_2 \circ \tau^{Z_2})_* r_{X'}^{\tau^{-1}(Z_2)} \Rightarrow (\theta_1 \circ \tau^{Z_1})_* r_{X'}^{\tau^{-1}(Z_1)}\]
				in $[\Cstar(X'), \Cstar(Y)]$. 
			\end{enumerate}
			
			\item \label{lem:composition_theta_r_iota_pi_2}
			Let $Y'$ be a topological space. Let $Z' \in \LC(Y)$ be such that $\theta_1^{-1}(Z') \subset \theta_2^{-1}(Z')$, and let $\tau' \colon Z' \to Y'$ be a continuous map. For $j=1, 2$, let $\theta_j^{Z'} \colon \theta_j^{-1}(Z') \to Z'$ denote the restriction of $\theta_j$. 
			
			\begin{enumerate}[label=\textnormal{(\roman*)}]
				\item \label{lem:composition_theta_r_iota_pi_2_1}
				Under the assumption \textup{(O)}, we have $\theta_1^{-1}(Z') \in \Open(\theta_2^{-1}(Z'))$ and 
				\[(\tau' \circ \theta_1^{Z'})^{-1}(U) \subset (\tau' \circ \theta_2^{Z'})^{-1}(U)\] 
				for every $U \in \Open(Y')$.
				Moreover, the horizontal composite 
				\[\tau'_* r_Y^{Z'} \iota_{Z_1}^{Z_2} \colon \tau'_* r_Y^{Z'} \theta_{1*} r_X^{Z_1} \Rightarrow \tau'_* r_Y^{Z'} \theta_{2*} r_X^{Z_2}\]
				is equal to 
				\[\iota_{\theta_1^{-1}(Z')}^{\theta_2^{-1}(Z')} \colon (\tau' \circ \theta_1^{Z'})_* r_X^{\theta_1^{-1}(Z')} \Rightarrow (\tau' \circ \theta_2^{Z'})_* r_X^{\theta_2^{-1}(Z')}\]
				in $[\Cstar(X), \Cstar(Y')]$. 
				
				\item \label{lem:composition_theta_r_iota_pi_2_2} 
				Under the assumption \textup{(C)}, we have $\theta_1^{-1}(Z') \in \Closed(\theta_2^{-1}(Z'))$ and 
				\[(\tau' \circ \theta_1^{Z'})^{-1}(F) \subset (\tau' \circ \theta_2^{Z'})^{-1}(F)\] 
				for every $F \in \Closed(Y')$.
				Moreover, the horizontal composite 
				\[\tau'_* r_Y^{Z'} \pi_{Z_2}^{Z_1} \colon \tau'_* r_Y^{Z'} \theta_{2*} r_X^{Z_2} \Rightarrow \tau'_* r_Y^{Z'} \theta_{1*} r_X^{Z_1}\]
				is equal to 
				\[\pi_{\theta_2^{-1}(Z')}^{\theta_1^{-1}(Z')} \colon (\tau' \circ \theta_2^{Z'})_* r_X^{\theta_2^{-1}(Z')} \Rightarrow (\tau' \circ \theta_1^{Z'})_* r_X^{\theta_1^{-1}(Z')} \]
				in $[\Cstar(X), \Cstar(Y')]$. 
			\end{enumerate}
		\end{enumerate}
	\end{lem}
	
	\begin{nota}
		Let $X$ and $Y$ be topological spaces, let $Z \in \LC(X)$, and let $\theta_1, \theta_2 \colon Z \to Y$ be continuous maps. 
		Suppose that $\theta_1^{-1}(U) \subset \theta_2^{-1}(U)$ for every $U \in \Open(Y)$. 
		Then both natural transformations
		\[
		\iota_Z^Z, \pi_Z^Z \colon \theta_{1*} r_X^Z \Rightarrow \theta_{2*} r_X^Z
		\]
		are defined and we have $\iota_Z^Z=\pi_Z^Z$. 
		We denote this common natural transformation by $\id_Z$.
	\end{nota}
	
	Using the natural transformations in \cref{eg:iota_pi_ir}, we obtain the following two lemmas. 
	
	\begin{lem}\label{lem:ses_ir}
		Let $X$ be a topological space. For $Y_2 \in \LC(X)$ and $Y_1 \in \Open(Y_2)$ with $Y_3:=Y_2 \setminus Y_1$, the sequence
		\[
		\begin{tikzcd}
			0 \ar[Rightarrow]{r} & i_{Y_1}^X r_X^{Y_1} \ar[Rightarrow]{r}{\iota_{Y_1}^{Y_2}} & i_{Y_2}^X r_X^{Y_2} \ar[Rightarrow]{r}{\pi_{Y_2}^{Y_3}} & i_{Y_3}^X r_X^{Y_3} \ar[Rightarrow]{r} & 0
		\end{tikzcd}
		\] 
		is a short exact sequence in $[\Cstar(X), \Cstar(X)]$. 
	\end{lem}
	
	\begin{proof}
		This follows from the short exact sequence
		\[
		\begin{tikzcd}
			0 \ar{r} & A(Y_1 \cap Z) \ar{r}{A(\iota_{Y_1 \cap Z}^{Y_2 \cap Z})} & A(Y_2 \cap Z) \ar{r}{A(\pi_{Y_2 \cap Z}^{Y_3 \cap Z})} & A(Y_3 \cap Z) \ar{r} & 0
		\end{tikzcd}
		\]
		in $\Cstar$ for $A \in \Cstar(X)$ and $Z \in \LC(X)$. 
	\end{proof}
	
	The following lemma will be used in \cref{thm:key_diagram}. 
	
	\begin{lem}\label{lem:disjoint_ir}
		Let $X$ be a topological space. Let $Y \in \LC(X)$ and $(Y_j)_{j \in J}$ be a finite family of pairwise disjoint clopen subsets of $Y$ such that $\bigcup_{j \in J} Y_j=Y$. Then we have
		\[i_Y^X r_X^Y=\bigoplus_{j \in J} i_{Y_j}^X r_X^{Y_j} \]
		in $[\Cstar(X), \Cstar(X)]$ via the morphisms $\iota_{Y_j}^Y  \colon i_{Y_j}^X r_X^{Y_j} \Rightarrow i_Y^X r_X^Y$ and $\pi_Y^{Y_j}  \colon i_Y^X r_X^Y \Rightarrow i_{Y_j}^X r_X^{Y_j}$ for $j \in J$. 
	\end{lem}

	\begin{proof}
		For each $j \in J$, we have $\pi_Y^{Y_j} \circ \iota_{Y_j}^Y=\id_{Y_j}$. 
		The natural transformations $\iota_{Y_j}^Y \circ \pi_Y^{Y_j}  \colon i_Y^X r_X^Y \Rightarrow i_Y^X r_X^Y$ for $j \in J$ are mutually orthogonal, and 
		\[\sum_{j \in J} \iota_{Y_j}^Y \circ \pi_Y^{Y_j}=\id_Y. \]
		This shows the assertion.  
	\end{proof}

	\subsection{Higher-dimensional mapping cones}\label{subsection:higher-dimensional_mapping_cones}

	We shall introduce higher-dimensional mapping cones, which play a crucial role in the study of reflection functors. 
	
	Let $\{0, 1\}$ denote the partially ordered set consisting of the two numbers $0$ and $1$ with $0<1$. Then $\{0, 1\}$ is regarded as the category whose objects are $0$ and $1$, and whose morphisms are $0 \to 0$, $1 \to 1$, and $0 \to 1$. 
	
	\begin{defn}
		Let $I$ be a finite set. Define a partial order $R_{\{0, 1\}^I}$ on $\{0, 1\}^I$ by
		\[R_{\{0, 1\}^I}:=\{(\lambda, \mu) \in \{0, 1\}^I \times \{0, 1\}^I \mid \text{$\lambda(i) \leq \mu(i)$ for all $i \in I$}\}. \] 
		We regard the partially ordered set $\bigl(\{0, 1\}^I, R_{\{0, 1\}^I}\bigr)$ as a category and denote it again by $\{0, 1\}^I$.
	\end{defn}
	
	For a category $\frakA$, an object $(A; \varphi) \in [\{0, 1\}^I, \frakA]$ consists of objects $A_{\lambda} \in \frakA$ for $\lambda \in \{0, 1\}^I$ and morphisms $\varphi_{\lambda}^{\mu} \colon A_{\lambda} \to A_{\mu}$ in $\frakA$ for $(\lambda, \mu) \in R_{\{0, 1\}^I}$ such that $\varphi_{\lambda}^{\lambda}=\id_{A_{\lambda}}$ for all $\lambda \in \{0, 1\}^I$ and $\varphi_{\mu}^{\nu} \varphi_{\lambda}^{\mu}=\varphi_{\lambda}^{\nu}$ for all $\lambda, \mu, \nu \in \{0, 1\}^I$ with $(\lambda, \mu), (\mu, \nu) \in R_{\{0, 1\}^I}$. Roughly speaking, $(A; \varphi)$ is a commutative diagram in $\frakA$ shaped like an $I$-dimensional hypercube. 
	If $I=\emptyset$, then $\{0, 1\}^I$ is the category with one object and one morphism, and hence we may regard $[\{0, 1\}^I, \frakA]=\frakA$. 
	If $|I|=1$, then we may regard $\{0, 1\}^I=\{0, 1\}$ and the category $[\{0, 1\}, \frakA]$ is the category of morphisms in $\frakA$.

	\begin{defn}
		Let $I$ be a finite set. 
		For $\lambda \in \{0, 1\}^I$, define a locally compact Hausdorff space
		\[\Omega_{\lambda}:=\prod_{i \in I; \lambda(i)=1} [0, 1). \]
		For $(\lambda, \mu) \in R_{\{0, 1\}^I}$, define a continuous map $\omega_{\lambda}^{\mu} \colon \Omega_{\lambda} \to \Omega_{\mu}$ by
		\[
		\omega_{\lambda}^{\mu}(t)_i:=
		\begin{dcases}
			t_i & \text{if $\lambda(i)=1$},\\
			0 & \text{if $\lambda(i)=0$}
		\end{dcases}
		\]
		for $t \in \Omega_{\lambda}$ and $i \in I$ with $\mu(i)=1$. 
	\end{defn}

	Note that $\Omega_{\lambda}$ is a one-point space for $\lambda \in \{0, 1\}^I$ such that $\lambda(i)=0$ for all $i \in I$.

	\begin{eg}\label{eg:1-dimensional_Omega}
		If $|I|=1$, then $\{0, 1\}^I=\{0, 1\}$. Observe that $\Omega_0$ is a one-point space and $\Omega_1=[0, 1)$. The continuous map $\omega_0^1 \colon \Omega_0 \hookrightarrow \Omega_1$ sends the unique point of $\Omega_0$ to $0$ in $\Omega_1$. 
	\end{eg}
	
	\begin{eg}\label{eg:2-dimensional_Omega}
		If $|I|=2$, then we may regard
		\[\{0, 1\}^I=\{0, 1\}^2=\{00, 01, 10, 11\}. \]
		Observe that $\Omega_{00}$ is a one-point space, $\Omega_{01}=\Omega_{10}=[0, 1)$, and $\Omega_{11}=[0, 1) \times [0, 1)$. The continuous maps $\omega_{01}^{11} \colon \Omega_{01} \hookrightarrow \Omega_{11}$ and $\omega_{10}^{11} \colon \Omega_{10} \hookrightarrow \Omega_{11}$ are given by
		\[\omega_{01}^{11}(t):=(0, t), \quad \omega_{10}^{11}(t):=(t, 0)\]
		for $t \in [0, 1)$. 
	\end{eg}
	
	\begin{defn}
		Let $I$ be a finite set.
		For $(A; \varphi) \in [\{0, 1\}^I, \Cstar]$, we define a C*-algebra
		\[
		\begin{aligned}
			M^I(A; \varphi):=\Bigg\{ {}&(f_{\lambda})_{\lambda} \in
			\bigoplus_{\lambda \in \{0, 1\}^I} \conti_0(\Omega_{\lambda}, A_{\lambda}) \ \Biggm| \\
			&\varphi_{\lambda}^{\mu} \circ f_{\lambda}=f_{\mu} \circ \omega_{\lambda}^{\mu}
			\text{ for all }(\lambda, \mu) \in R_{\{0, 1\}^I}\Bigg\}.
		\end{aligned}
		\]
		We call $M^I(A;\varphi)$ the \textit{$I$-dimensional mapping cone} of $(A;\varphi)$.
		For a morphism $\xi \colon (A; \varphi) \Rightarrow (B; \psi)$ in $[\{0, 1\}^I, \Cstar]$, the $\ast$-homomorphism
		\[
		\bigoplus_{\lambda \in \{0, 1\}^I}\xi_{\lambda*}
		\colon
		\bigoplus_{\lambda \in \{0, 1\}^I}\conti_0(\Omega_{\lambda}, A_{\lambda})
		\to
		\bigoplus_{\lambda \in \{0, 1\}^I}\conti_0(\Omega_{\lambda}, B_{\lambda})
		\]
		restricts to a $\ast$-homomorphism
		\[
		M^I\xi \colon M^I(A;\varphi) \to M^I(B;\psi)
		\]
		by the naturality of $\xi$.
		This defines a functor
		\[
		M^I \colon [\{0, 1\}^I, \Cstar] \to \Cstar,
		\]
		which we call the \textit{$I$-dimensional mapping cone functor}.
		Moreover, this functor restricts to a functor
		\[
		M^I \colon [\{0, 1\}^I, \SCstar] \to \SCstar.
		\]
	\end{defn}

	\begin{eg}\label{eg:0-dimensional_mapping_cone}
		If $I=\emptyset$, then $M^I$ is the identity functor on $\Cstar$ under the identification $[\{0, 1\}^I, \Cstar]=\Cstar$. 
	\end{eg}
	
	\begin{eg}\label{eg:1-dimensional_mapping_cone}
		Consider the case where $|I|=1$; see \cref{eg:1-dimensional_Omega}. 
		In this case, we simply write $M:=M^I$. 
		An object $(A; \varphi) \in [\{0, 1\}, \Cstar]$ is a $\ast$-homomorphism $\varphi \colon A_0 \to A_1$. 
		The $1$-dimensional mapping cone $M(A; \varphi)$ is usually denoted by $C_{\varphi}$ and is called the \textit{mapping cone of $\varphi$}. 
		We have
		\[C_{\varphi}=\{(a, f) \in A_0 \oplus \conti_0([0, 1), A_1) \mid \varphi(a)=f(0)\}. \]
		Observe that
		\begin{itemize}
			\item $C_{\varphi}=A_0$ if $A_1=0$;
			\item $C_{\varphi}=SA_1$ if $A_0=0$;
			\item $C_{\varphi}=CA_1$ if $\varphi \colon A_0 \to A_1$ is a $\ast$-isomorphism.
		\end{itemize} 
	\end{eg}

	\begin{eg}
		Consider the case where $|I|=2$; see \cref{eg:2-dimensional_Omega}.
		For an object $(A; \varphi) \in [\{0, 1\}^2, \Cstar]$, the $2$-dimensional mapping cone $M^2(A; \varphi)$ is the C*-algebra consisting of all elements
		\[(a, f, g, h) \in A_{00} \oplus \conti_0([0, 1), A_{01}) \oplus \conti_0([0, 1), A_{10}) \oplus \conti_0([0, 1) \times [0, 1), A_{11}) \]
		satisfying $\varphi_{00}^{01}(a)=f(0)$, $\varphi_{00}^{10}(a)=g(0)$, $\varphi_{01}^{11} \circ f=h \circ \omega_{01}^{11}$, and $\varphi_{10}^{11} \circ g=h \circ \omega_{10}^{11}$.
		Observe that
		\begin{itemize}
			\item $M^2(A; \varphi)=C_{\varphi_{00}^{01}}$ if $A_{10}=A_{11}=0$;
			\item $M^2(A; \varphi)=C_{\varphi_{00}^{10}}$ if $A_{01}=A_{11}=0$;
			\item $M^2(A; \varphi)=SC_{\varphi_{01}^{11}}$ if $A_{00}=A_{10}=0$; 
			\item $M^2(A; \varphi)=SC_{\varphi_{10}^{11}}$ if $A_{00}=A_{01}=0$.
		\end{itemize} 
	\end{eg}

	Let $\frakA$ be a category. 
	For a finite set $I$, we use the same symbol $M^I$ for the functor
	\[[\{0, 1\}^I, [\frakA, \Cstar]] = [\frakA, [\{0, 1\}^I, \Cstar]] \xrightarrow{M^I_*} [\frakA, \Cstar]. \]
	If $|I|=1$, then we simply write $M:=M^I$. 
	For $(A; \varphi) \in [\{0, 1\}, [\frakA, \Cstar]]$, we write $C_{\varphi}:=M(A; \varphi) \in [\frakA, \Cstar]$. Observe that $C_{\varphi}$ is the pullback along the morphisms $\varphi \colon A_0 \Rightarrow A_1$ and $\ev_0 \colon CA_1 \Rightarrow A_1$ in $[\frakA, \Cstar]$; see \cref{eg:mapping_cone_pullback}. 
	
	The following notation for $2$-dimensional mapping cones will be used frequently throughout the paper. 
	
	\begin{nota}\label{nota:2-dimensional_mapping_cone}
		Suppose that $|I|=2$. 
		For an object $(A; \varphi) \in [\{0, 1\}^2, [\frakA, \Cstar]]$, we write the $2$-dimensional mapping cone $M^2(A; \varphi) \in [\frakA, \Cstar]$ as
		\[M^2 \mathopen{} \left(
		\begin{tikzcd}
			A_{00} \ar{r}{\varphi_{00}^{01}} \ar[swap]{d}{\varphi_{00}^{10}} & A_{01} \ar{d}{\varphi_{01}^{11}} \\
			A_{10} \ar[swap]{r}{\varphi_{10}^{11}} & A_{11}
		\end{tikzcd}\right). 
		\]
		For a morphism $\xi \colon (A; \varphi) \Rightarrow (B; \psi)$ in $[\{0, 1\}^2, [\frakA, \Cstar]]$, we write the morphism $M^2\xi \colon M^2(A; \varphi) \to M^2(B; \psi)$ in $[\frakA, \Cstar]$ as 
		\[M^2
		\!\begin{pmatrix}
			\xi_{00} & \xi_{01} \\
			\xi_{10} & \xi_{11}
		\end{pmatrix} \colon M^2\mathopen{} \left(
		\begin{tikzcd}
			A_{00} \ar{r}{\varphi_{00}^{01}} \ar[swap]{d}{\varphi_{00}^{10}} & A_{01} \ar{d}{\varphi_{01}^{11}} \\
			A_{10} \ar[swap]{r}{\varphi_{10}^{11}} & A_{11}
		\end{tikzcd}\right) \to M^2\mathopen{} \left(
		\begin{tikzcd}
			B_{00} \ar{r}{\psi_{00}^{01}} \ar[swap]{d}{\psi_{00}^{10}} & B_{01} \ar{d}{\psi_{01}^{11}} \\
			B_{10} \ar[swap]{r}{\psi_{10}^{11}} & B_{11}
		\end{tikzcd}\right).\] 
	\end{nota}
	
	The following notation for $3$-dimensional mapping cones will be used in the proof of \cref{thm:zigzag_identity}.
	
	\begin{nota}\label{nota:3-dimensional_mapping_cone}
		If $|I|=3$, then we may regard
		\[\{0, 1\}^I=\{0, 1\}^3=\{000, 001, 010, 100, 011, 101, 110, 111\}. \]
		An object $(A; \varphi) \in [\{0, 1\}^3, [\frakA, \Cstar]]$ may be regarded as the following commutative diagram in $[\frakA, \Cstar]$: 
		\[
		\begin{tikzcd}[ampersand replacement=\&, every label/.append style={font=\small}, cells={nodes={font=\small}}]
			\& {} \&
			A_{000}
			\ar[swap]{dl}{\varphi_{000}^{010}}
			\ar[near start]{dd}{\varphi_{000}^{100}}
			\ar{rr}{\varphi_{000}^{001}}
			\& \& 
			A_{001}
			\ar{dl}{\varphi_{001}^{011}}
			\ar{dd}{\varphi_{001}^{101}}
			\\
			\& 
			A_{010}
			\ar[swap]{dd}{\varphi_{010}^{110}}
			\& \& 
			A_{011}
			\ar[from=ll, crossing over, swap, near start, "\varphi_{010}^{011}"]
			\\
			\& \&
			A_{100}
			\ar{dl}{\varphi_{100}^{110}}
			\ar[near start, swap]{rr}{\varphi_{100}^{101}}
			\& \& 
			A_{101}
			\ar{dl}{\varphi_{101}^{111}}
			\\
			\&
			A_{110}
			\ar[swap]{rr}{\varphi_{110}^{111}}
			\& \& 
			A_{111}\mathrlap{.}
			\ar[from=uu, crossing over, near start, "\varphi_{011}^{111}"]
		\end{tikzcd}
		\]
		For a morphism $\xi \colon (A; \varphi) \Rightarrow (B; \psi)$ in $[\{0, 1\}^3, [\frakA, \Cstar]]$, we write the morphism $M^3\xi \colon M^3(A; \varphi) \to M^3(B; \psi)$ in $[\frakA, \Cstar]$ as 
		\[
		M^3\!\begin{pmatrix}
			& \xi_{000} & & \xi_{001} \\
			\xi_{010} & &  \xi_{011} & \\
			& \xi_{100} & & \xi_{101} \\
			\xi_{110} & &  \xi_{111} & \\
		\end{pmatrix}.
		\]
	\end{nota}

	The following lemma is straightforward to verify.
	
	\begin{lem}\label{lem:product_index_sets}
		Let $I_1$ and $I_2$ be finite sets and $I:=I_1 \amalg I_2$. 
		Then
		\[
		\{0, 1\}^{I_1} \times \{0, 1\}^{I_2}=\{0, 1\}^I, 
		\quad
		R_{\{0, 1\}^{I_1}} \times R_{\{0, 1\}^{I_2}}=R_{\{0, 1\}^I}
		\]
		as sets. 
		In other words, $\{0, 1\}^{I_1} \times \{0, 1\}^{I_2}=\{0, 1\}^I$ as categories. 
	\end{lem}	
	
	In \cref{appendix:proof_of_mapping_cones_of_mapping_cones}, we establish the universal property of $M^I$ as a limit in \cref{prop:mapping_cone_limit} and apply it to prove the following lemma. 
	
	\begin{lem}\label{lem:mapping_cone_of_mapping_cone}
		Let $I_1$ and $I_2$ be finite sets and $I:=I_1 \amalg I_2$. Under the identification
		\[[\{0, 1\}^{I_2}, [\{0, 1\}^{I_1}, \Cstar]]=[\{0, 1\}^{I_1} \times \{0, 1\}^{I_2}, \Cstar]=[\{0, 1\}^I, \Cstar], \]
		the composite
		\[[\{0, 1\}^{I_2}, [\{0, 1\}^{I_1}, \Cstar]] \xrightarrow{M^{I_1}_*} [\{0, 1\}^{I_2}, \Cstar] \xrightarrow{M^{I_2}} \Cstar \]
		is equal to the functor 
		\[M^I \colon [\{0, 1\}^I, \Cstar] \to \Cstar. \]
	\end{lem}
	
	For any finite set $I$, applying the preceding lemma inductively on $|I|$ shows that the $I$-dimensional mapping cone of a cube in $\Cstar$ is obtained by taking ordinary mapping cones successively in each coordinate direction, in any order. We can therefore prove properties of $I$-dimensional mapping cones by applying the corresponding properties of ordinary mapping cones at each step, as in the proofs of
	\cref{prop:mapping_cone_exact_continuous_module,lem:rearrangement_mapping_cones,lem:admissible_mapping_cones}.
	
	The following folklore fact will be used repeatedly. Since the author has been unable to find a reference in the literature, we include a proof.
	
	\begin{lem}\label{lem:exact_module_functor_commute_mapping_cone}
		Let $\frakA$ and $\frakB$ be categories and let $F \colon [\frakA, \Cstar] \to [\frakB, \Cstar]$ be an exact $\Cstar$-module functor. Then $F$ commutes with mapping cones. 
	\end{lem}
	
	\begin{proof}
		Let $\varphi \colon A_0 \Rightarrow A_1$ be a morphism in $[\frakA, \Cstar]$.
		Since $F$ is an exact $\Cstar$-module functor, we obtain the following commutative diagram with exact rows in $[\frakB, \Cstar]$, in which the dashed arrow is to be constructed:
		\[
		\begin{tikzcd}[sep=small]
			{} & 0
			\ar[Rightarrow]{rr}
			& {} &
			[-1em] FSA_1
			\ar[Rightarrow, swap]{dl}{\simeq}
			\ar[equal]{dd}
			\ar[Rightarrow]{rr}
			& & 
			[-2em] FC_{\varphi}
			\ar[Rightarrow, dashed]{dl}
			\ar[Rightarrow]{dd}
			\ar[Rightarrow]{rr}
			& & 
			FA_0
			\ar[equal]{dl}
			\ar[Rightarrow, near start]{dd}{F\varphi}
			\ar[Rightarrow]{rr}
			& &
			0
			\\
			0
			\ar[Rightarrow]{rr}
			& &
			SFA_1
			\ar[Rightarrow]{rr}
			& & 
			C_{F\varphi}
			\ar[Rightarrow, from=ll, crossing over]
			\ar[Rightarrow]{rr}
			& & 
			FA_0
			\ar[Rightarrow, from=ll, crossing over]
			& &
			0
			\ar[Rightarrow, from=ll, crossing over]
			\\
			& 0
			\ar[Rightarrow]{rr}
			& &
			FSA_1
			\ar[Rightarrow]{dl}{\simeq}
			\ar[Rightarrow]{rr}
			& & 
			FCA_1
			\ar[Rightarrow]{dl}{\simeq}
			\ar[Rightarrow]{rr}
			& & 
			FA_1
			\ar[equal]{dl}
			\ar[Rightarrow]{rr}
			& &
			0
			\\
			0
			\ar[Rightarrow]{rr}
			& &
			SFA_1
			\ar[Rightarrow, from=uu, equal, crossing over]
			\ar[Rightarrow]{rr}
			& & 
			CFA_1
			\ar[Rightarrow, from=uu, crossing over]
			\ar[Rightarrow]{rr}
			& & 
			FA_1
			\ar[Rightarrow, from=uu, crossing over, near start, "F\varphi"]
			\ar[Rightarrow]{rr}
			& &
			0\mathrlap{.}
		\end{tikzcd}
		\]
		By the universal property of $C_{F\varphi}$ as a pullback, there exists a unique morphism $FC_{\varphi} \Rightarrow C_{F\varphi}$, represented by the dashed arrow, that makes the entire diagram commute. By the five lemma, this canonical morphism is an isomorphism. 
		Moreover, this isomorphism is natural in $(A; \varphi) \in [\{0, 1\}, [\frakA, \Cstar]]$. This proves the assertion. 
	\end{proof}

	The case $|I|=1$ of the following proposition is also a folklore fact. We include a brief proof. 
	
	\begin{prop}\label{prop:mapping_cone_exact_continuous_module}
		For a finite set $I$, the functor $M^I \colon [\{0, 1\}^I, \Cstar] \to \Cstar$ is an exact, continuous, $\Cstar$-module functor. 
	\end{prop}
	
	\begin{proof}
		We prove the case $|I|=1$. Exactness of $M$ follows from exactness of the suspension functor $S \colon \Cstar \to \Cstar$ and the nine lemma (see also \cite[Theorem~9.1]{Pedersen_1999}). 
		For an upward-directed set $\Lambda$, the inductive limit functor $\varinjlim_{\Lambda} \colon [\Lambda, \Cstar] \to \Cstar$ is an exact $\Cstar$-module functor, and hence commutes with mapping cones by \cref{lem:exact_module_functor_commute_mapping_cone}. This shows that $M$ is continuous. 
		For a C*-algebra $D$, the functor $D \otimes \blank \colon \Cstar \to \Cstar$ is an exact $\Cstar$-module functor, and hence commutes with mapping cones by \cref{lem:exact_module_functor_commute_mapping_cone}. Since $S \colon \Cstar \to \Cstar$ satisfies the associativity and unit constraints, so does $M$. This shows that $M$ is a $\Cstar$-module functor. 
		This proves the case $|I|=1$. 
		The general case follows from \cref{lem:mapping_cone_of_mapping_cone}. 
	\end{proof}
	
	Let $I$ be a finite set. For a topological space $X$, we obtain the functor
	\[M^I \colon [\{0, 1\}^I, \Cstar(X)] \to \Cstar(X)\]
	as in \cref{lem:restriction_exact_continuous}. 
	This functor restricts to a functor
	\[M^I \colon [\{0, 1\}^I, \SCstar(X)] \to \SCstar(X). \]
	
	For a morphism $\varphi \colon A_0 \to A_1$ in $\SCstar(X)$, the mapping cone short exact sequence of $\varphi$ is the short exact sequence
	\[
	\begin{tikzcd}
		0 \ar{r} & SA_1 \ar{r} & C_{\varphi} \ar{r} & A_0 \ar{r} & 0
	\end{tikzcd}
	\]
	in $\SCstar(X)$. It is obtained by taking the mapping cones of the vertical morphisms in the following commutative diagram with exact rows in $\SCstar(X)$:
	\[
	\begin{tikzcd}
		0 \ar{r} & 0 \ar{r} \ar{d} & A_0 \ar{r} \ar{d}{\varphi} & A_0 \ar{r} \ar{d}& 0 \\
		0 \ar{r} & A_1 \ar{r}  & A_1 \ar{r} & 0 \ar{r} & 0\mathrlap{.}
	\end{tikzcd}
	\]
	Every mapping cone short exact sequence is semi-split in the sense of \cite[Definition~3.5]{MN_2009}. 
	
	The following lemma will be used to show \cref{lem:key_diagram_semi-split}. 
	
	\begin{lem}\label{lem:semi-split_2-dimensional_mapping_cone}
		For $(A; \varphi) \in [\{0, 1\}^2, \SCstar(X)]$, the sequence 
		\[
		\begin{tikzcd}[ampersand replacement=\&, every label/.append style={font=\tiny}, cells={nodes={font=\tiny}}]
			0 \ar{r}
			\& [-1.5em]
			{M^2 \! \left(
				\begin{tikzpicture}[auto]
					\node (00) at (-0.6, 0.5) {$0$}; \node (01) at (0.6, 0.5) {$0$};
					\node (10) at (-0.6, -0.5) {$A_{10}$};  \node (11) at (0.6, -0.5) {$A_{11}$};
					\draw [->] (00) to node {} (01);
					\draw [->] (00) to node[swap] {} (10);
					\draw [->] (01) to node {} (11);
					\draw [->] (10) to node[swap] {$\varphi_{10}^{11}$} (11);
				\end{tikzpicture} \right)} 
			\ar{r}[xshift=0.5em, yshift=0.2em]{
				M^2 \! \begin{pmatrix}
					0 & 0 \\ \id & \id
				\end{pmatrix}
			}
			\&
			{M^2 \! \left(
				\begin{tikzpicture}[auto]
					\node (00) at (-0.6, 0.5) {$A_{00}$}; \node (01) at (0.6, 0.5) {$A_{01}$};
					\node (10) at (-0.6, -0.5) {$A_{10}$};  \node (11) at (0.6, -0.5) {$A_{11}$};
					\draw [->] (00) to node {$\varphi_{00}^{01}$} (01);
					\draw [->] (00) to node[swap] {$\varphi_{00}^{10}$} (10);
					\draw [->] (01) to node {$\varphi_{01}^{11}$} (11);
					\draw [->] (10) to node[swap] {$\varphi_{10}^{11}$} (11);
				\end{tikzpicture} \right)}
			\ar{r}[xshift=0.5em, yshift=0.2em]{
				M^2 \! \begin{pmatrix}
					\id & \id \\ 0 & 0
				\end{pmatrix}
			}
			\&
			{M^2 \! \left(
				\begin{tikzpicture}[auto]
					\node (00) at (-0.6, 0.5) {$A_{00}$}; \node (01) at (0.6, 0.5) {$A_{01}$};
					\node (10) at (-0.6, -0.5) {$0$};  \node (11) at (0.6, -0.5) {$0$};
					\draw [->] (00) to node {$\varphi_{00}^{01}$} (01);
					\draw [->] (00) to node[swap] {} (10);
					\draw [->] (01) to node {} (11);
					\draw [->] (10) to node[swap] {} (11);
				\end{tikzpicture} \right)}
			\ar{r}
			\& [-1.5em] 0
		\end{tikzcd}
		\]
		is a semi-split short exact sequence in $\SCstar(X)$. 
	\end{lem}
	
	\begin{proof}
		This sequence is the mapping cone short exact sequence of the $X$-equivariant $\ast$-homomorphism $C_{\varphi_{00}^{01}} \to C_{\varphi_{10}^{11}}$ induced by the commutative square defining $(A;\varphi)$.
		Hence, it is semi-split. 
	\end{proof}

	Let $I$ be a finite set. For topological spaces $X$ and $Y$, we obtain the functor
	\[M^I \colon [\{0, 1\}^I, [\Cstar(X), \Cstar(Y)]] \to [\Cstar(X), \Cstar(Y)]\]
	as in \cref{lem:preservation_exact_continuous_module}. This functor restricts to a functor 
	\[M^I \colon [\{0, 1\}^I, [\SCstar(X), \SCstar(Y)]] \to [\SCstar(X), \SCstar(Y)]. \]

	To compute compositions of functors constructed from higher-dimensional mapping cones in \cref{section:reflection_functors,section:higher-dimensional_reflection_functors}, we combine two diagrams indexed by $\{0, 1\}^{I_1}$ and $\{0, 1\}^{I_2}$ into a single diagram indexed by $\{0, 1\}^{I_1 \amalg I_2}$. We introduce the following notation for this construction.
	
	\begin{defn}
		Let $I_1$ and $I_2$ be finite sets and $I:=I_1 \amalg I_2$. Let $X$, $Y$, and $Z$ be topological spaces. 
		For 
		\[
		\begin{aligned}
			(F; \varphi)&\in [\{0, 1\}^{I_1}, [\Cstar(X), \Cstar(Y)]],\\
			(G; \psi)&\in [\{0, 1\}^{I_2}, [\Cstar(Y), \Cstar(Z)]], 
		\end{aligned}
		\]
		we define 
		\[(GF; \psi\varphi) \in [\{0, 1\}^I,  [\Cstar(X), \Cstar(Z)]] \] 
		as follows. For $(\lambda, \lambda') \in \{0, 1\}^{I_1} \times \{0, 1\}^{I_2}=\{0, 1\}^I$, we define 
		\[(GF)_{(\lambda, \lambda')} \in [\Cstar(X), \Cstar(Z)]\] 
		by
		\[(GF)_{(\lambda, \lambda')}:=G_{\lambda'}F_{\lambda}. \]
		For $((\lambda, \mu), (\lambda', \mu')) \in R_{\{0, 1\}^{I_1}} \times R_{\{0, 1\}^{I_2}}=R_{\{0, 1\}^I}$, we define a morphism 
		\[(\psi\varphi)_{(\lambda, \lambda')}^{(\mu, \mu')} \colon (GF)_{(\lambda, \lambda')} \Rightarrow (GF)_{(\mu, \mu')}\]
		in $[\Cstar(X), \Cstar(Z)]$ to be the horizontal composite
		\[(\psi\varphi)_{(\lambda, \lambda')}^{(\mu, \mu')}:= \psi_{\lambda'}^{\mu'} \varphi_{\lambda}^{\mu} \colon G_{\lambda'}F_{\lambda} \Rightarrow G_{\mu'}F_{\mu}. \]
	\end{defn}
	
	\begin{eg}\label{eg:composition_1+1=2}
		Consider the case where $|I_1|=|I_2|=1$. 
		For 
		\[
		\begin{aligned}
			(F; \varphi) &\in [\{0, 1\}, [\Cstar(X), \Cstar(Y)]],\\
			(G; \psi) &\in [\{0, 1\}, [\Cstar(Y), \Cstar(Z)]], 
		\end{aligned}
		\]
		the object
		\[(GF; \psi\varphi) \in [\{0, 1\}^2, [\Cstar(X), \Cstar(Z)]] \]
		is the following commutative diagram in $[\Cstar(X), \Cstar(Z)]$:
		\[
		\begin{tikzcd}
			G_0F_0 \ar[Rightarrow]{r}{G_0\varphi} \ar[Rightarrow, swap]{d}{\psi F_0} & G_0F_1 \ar[Rightarrow]{d}{\psi F_1} \\
			G_1F_0 \ar[Rightarrow, swap]{r}{G_1 \varphi} & G_1F_1\mathrlap{.}
		\end{tikzcd}
		\]
	\end{eg}
	
	\begin{eg}\label{eg:composition_1+2=3}
		Consider the case where $|I_1|=1$ and $|I_2|=2$. 
		For 
		\[
		\begin{aligned}
			(F; \varphi)&\in [\{0, 1\}, [\Cstar(X), \Cstar(Y)]],\\
			(G; \psi)&\in [\{0, 1\}^2, [\Cstar(Y), \Cstar(Z)]],
		\end{aligned}
		\]
		the object
		\[(GF; \psi\varphi) \in [\{0, 1\}^3,  [\Cstar(X), \Cstar(Z)]] \] 
		is the following commutative diagram in $[\Cstar(X), \Cstar(Z)]$:
		\[
		\begin{tikzcd}[ampersand replacement=\&, every label/.append style={font=\small}, cells={nodes={font=\small}}]
			\& {} \&
			G_{00}F_0
			\ar[Rightarrow, swap]{dl}{\psi_{00}^{01}F_0}
			\ar[Rightarrow, near start]{dd}{\psi_{00}^{10}F_0}
			\ar[Rightarrow]{rr}{G_{00} \varphi}
			\& \& 
			G_{00}F_1
			\ar[Rightarrow]{dl}{\psi_{00}^{01} F_1}
			\ar[Rightarrow]{dd}{\psi_{00}^{10} F_1}
			\\
			\& 
			G_{01}F_0
			\ar[Rightarrow, swap]{dd}{\psi_{01}^{11}F_0}
			\& \& 
			G_{01}F_1
			\ar[Rightarrow, from=ll, crossing over, swap, near start, "G_{01} \varphi"]
			\\
			\& \&
			G_{10}F_0
			\ar[Rightarrow]{dl}{\psi_{10}^{11}F_0}
			\ar[Rightarrow, near start, swap]{rr}{G_{10} \varphi}
			\& \& 
			G_{10}F_1
			\ar[Rightarrow]{dl}{\psi_{10}^{11}F_1}
			\\
			\&
			G_{11}F_0
			\ar[Rightarrow, swap]{rr}{G_{11}\varphi}
			\& \& 
			G_{11}F_1\mathrlap{.}
			\ar[Rightarrow, from=uu, crossing over, near start, "\psi_{01}^{11}F_1"]
		\end{tikzcd}
		\]
	\end{eg}
	
	The following lemma shows that the above composition of diagrams is compatible with higher-dimensional mapping cones. It will be used repeatedly later.
	
	\begin{lem}\label{lem:rearrangement_mapping_cones}
		Let $I_1$ and $I_2$ be finite sets and $I:=I_1 \amalg I_2$. 
		Let $X$, $Y$, and $Z$ be topological spaces. 
		Let
		\[
		\begin{aligned}
			(F; \varphi)&\in [\{0, 1\}^{I_1}, [\Cstar(X), \Cstar(Y)]],\\
			(G; \psi)&\in [\{0, 1\}^{I_2}, [\Cstar(Y), \Cstar(Z)]].
		\end{aligned}
		\]
		Under the identification 
		\[[\{0, 1\}^{I_2}, [\Cstar(Y), \Cstar(Z)]]=[\Cstar(Y), [\{0, 1\}^{I_2}, \Cstar(Z)]], \]
		suppose that $(G; \psi) \colon \Cstar(Y) \to [\{0, 1\}^{I_2}, \Cstar(Z)]$ is an exact $\Cstar$-module functor. 
		Then
		\[M^{I_2}(G; \psi) M^{I_1}(F; \varphi)=M^I(GF; \psi\varphi)\]
		in $[\Cstar(X), \Cstar(Z)]$. 
	\end{lem}
	
	\begin{proof}
		The assertion follows from the following calculation:
		\begin{align*}
			M^{I_2}(G; \psi) M^{I_1}(F; \varphi)&=M^{I_2}(GM^{I_1}(F; \varphi); \psi M^{I_1}(F; \varphi)) \\
			&=M^{I_2}(M^{I_1}_*(GF; \psi\varphi))=M^I(GF; \psi\varphi). 
		\end{align*}
		The first equality follows by definition, and the third follows from \cref{lem:mapping_cone_of_mapping_cone}. When $|I_1|=1$, the second equality follows by the same argument as \cref{lem:exact_module_functor_commute_mapping_cone} applied to the exact $\Cstar$-module functor $(G; \psi)$; see also \cref{rem:exact_continuous_module_iff}. The general case of the second equality then follows from \cref{lem:mapping_cone_of_mapping_cone}. This completes the proof. 
	\end{proof}

	\subsection{Homotopy-invariance, stability, and half-exactness}
	
	Let $X$ be a topological space and let $\frakA$ be an additive category.
	The notions of homotopy-invariance, stability, half-exactness, and split-exactness of functors from $\Cstar(X)$ to $\frakA$ are defined in the same way as the corresponding notions for functors from $\Cstar$ to $\frakA$; see \cite[Sections~21 and~22]{Blackadar_1998}.
	The following lemma is proved by the same argument as its non-equivariant counterpart for homotopy-invariant, stable, half-exact functors from $\Cstar$ to $\frakA$; see also \cite[Section~4]{Cuntz_1984}. 
	
	\begin{lem}\label{lem:Bott}
		Let $F \colon \Cstar(X) \to \frakA$ be a homotopy-invariant, stable, half-exact functor. For a short exact sequence
		\[
		\begin{tikzcd}
			0 \ar{r} & A \ar{r}{\iota} & E \ar{r}{\pi} & B \ar{r} & 0
		\end{tikzcd}
		\] 
		in $\Cstar(X)$, the following statements hold. 
		\begin{enumerate}[label=\textnormal{(\arabic*)}]
			\item \label{prop:Bott_1}
			If $FA=FSA=0$, then $F\pi \colon FE \to FB$ is an isomorphism in $\frakA$. 
			In particular, if $A$ is contractible, then $F\pi \colon FE \to FB$ is an isomorphism in $\frakA$. 
			\item \label{prop:Bott_2}
			If $FB=FSB=0$, then $F\iota \colon FA \to FE$ is an isomorphism in $\frakA$.
			In particular, if $B$ is contractible, then $F\iota \colon FA \to FE$ is an isomorphism in $\frakA$.
		\end{enumerate}
		Furthermore, $F$ satisfies Bott periodicity, that is, $FS^2 \simeq F$ in $[\Cstar(X), \frakA]$. 
	\end{lem}

	The following lemma is easily checked with the help of the unit condition of a $\Cstar$-module functor. 
	
	\begin{lem}\label{lem:homotopy-invariant_stable_module}
		Let $\frakA$ be a category. 
		If $F \colon \Cstar(X) \to \Cstar(Y)$ is a $\Cstar$-module functor and if $G \colon \Cstar(Y) \to \frakA$ is a homotopy-invariant and stable functor, then $GF \colon \Cstar(X) \to \frakA$ is homotopy-invariant and stable. 
	\end{lem}
	
	All the notions and statements introduced in this subsection also apply with $\Cstar(\blank)$ replaced by $\SCstar(\blank)$, or by the corresponding categories of nuclear C*-algebras or separable nuclear C*-algebras.

	\section{Bivariant K-theory for C*-algebras over topological spaces}\label{section:ideal-related_KK-theory_and_E-theory}
	In this section, we recall basic facts about ideal-related KK-theory and E-theory for C*-algebras over topological spaces. 
	
	\subsection{Axiomatic properties}\label{subsection:Axiomatic_properties}
	In this paper, we will not need any of the concrete constructions of ideal-related KK-theory and E-theory. It suffices to recall their axiomatic properties. 
	
	Let $\KKcat$ denote the additive category of Kasparov's KK-theory \cite{Kasparov_1980}, and let $\Ecat$ denote the additive category of Connes and Higson's E-theory \cite{CH_1990}. We refer the reader to \cite{Blackadar_1998} for further details. The objects of both $\KKcat$ and $\Ecat$ are separable C*-algebras.
	
	Let $X$ be a finite topological space.  
	Let $\KKcat(X)$ denote the additive category of Kirchberg's ideal-related KK-theory \cite{Kirchberg_2000}; see \cite[Definition~3.3]{MN_2009}. Let $\Ecat(X)$ denote the additive category of Dadarlat and Meyer's ideal-related E-theory; see \cite[Theorem~2.25]{DM_2012}.
	The objects of both $\KKcat(X)$ and $\Ecat(X)$ are separable C*-algebras over $X$.
	
	Recall that there exists an identity-on-objects, homotopy-invariant, stable, split-exact functor $\KK_X \colon \SCstar(X) \to \KKcat(X)$ such that, for an additive category $\frakA$ and a homotopy-invariant, stable, split-exact functor $F \colon \SCstar(X) \to \frakA$, there exists a unique additive functor $F' \colon \KKcat(X) \to \frakA$ such that $F'\KK_X=F$; see \cite[Theorem~3.7]{MN_2009}. There also exists an identity-on-objects, homotopy-invariant, stable, half-exact functor $\E_X \colon \SCstar(X) \to \Ecat(X)$ satisfying the analogous universal property, with half-exactness in place of split-exactness; see \cite[Theorem 2.25]{DM_2012}.

	We recall the following folklore fact and include a brief proof for the reader's convenience. It will reduce the argument in the proof of \cref{lem:induced_natural_transformation_KK_E}.

	\begin{lem}\label{lem:morphisms_in_KK_E}
		For every morphism $f \colon A \to B$ in $\KKcat(X)$, there exists a finite sequence  
		\[
		\begin{tikzcd}
			A=A_0 \ar{r}{f_0} & A_1 \ar{r}{f_1} & \cdots \ar{r}{f_{n-1}} & A_n=B
		\end{tikzcd}
		\]
		of objects and morphisms in $\KKcat(X)$ such that $f=f_{n-1} f_{n-2} \cdots f_0$ and, for each $j=0, 1, \ldots, n-1$, either $f_j=\KK_X(\varphi_j)$ for some morphism $\varphi_j \colon A_j \to A_{j+1}$ in $\SCstar(X)$ or $f_j=\KK_X(\psi_j)^{-1}$ for some morphism $\psi_j \colon A_{j+1} \to A_j$ in $\SCstar(X)$.
		The analogous statement holds for morphisms in $\Ecat(X)$.
	\end{lem}
	
	\begin{proof}
		By \cite[Lemma~2.26]{DM_2012} (see also \cite[Proposition~25.6.2]{Blackadar_1998}), the morphism in $\Ecat(X)$ associated with any approximately $X$-equivariant asymptotic morphism can be written as $\E_X(\rho)\E_X(\pi)^{-1}$, where $\rho$ and $\pi$ are $X$-equivariant $\ast$-homomorphisms and $\E_X(\pi)$ is an isomorphism.
		Reviewing Cuntz's proof of Bott periodicity in \cite[Section~4]{Cuntz_1984} (see also \cite[Exercise~9.4.2]{Blackadar_1998}), we see that the Bott periodicity isomorphisms $S^2A \simeq A$ and $S^2B \simeq B$ in $\Ecat(X)$ can also be written as finite composites of morphisms of the form $\E_X(\varphi)$ or $\E_X(\psi)^{-1}$, where $\varphi$ and $\psi$ are $X$-equivariant $\ast$-homomorphisms and $\E_X(\psi)$ is an isomorphism.
		This shows the assertion for $\Ecat(X)$.
		By using the description of morphism sets in $\KKcat(X)$ given by \cite[Theorem 5.2]{DM_2012}, we can apply the same argument to show the assertion for $\KKcat(X)$.
	\end{proof}
	
	The forgetful functor $\ev_X \colon \SCstar(X) \to \SCstar$ induces functors
	\[\ev_X \colon \KKcat(X) \to \KKcat, \quad
	\ev_X \colon \Ecat(X) \to \Ecat.\]
	If $X$ is a one-point space, then they are isomorphisms of categories. In this case, we may identify $\KKcat(X)$ and $\Ecat(X)$ with $\KKcat$ and $\Ecat$ via $\ev_X$, respectively. 
	
	\begin{rem}
		The category $\KKcat(X)$ can be defined for any topological space $X$; see \cite{Gabe_2024}. However, the definition of $\KKcat(X)_{\loc}$ and $\Boot(X)$ for an infinite topological space $X$ is unclear; see \cite[Sections~4.3 and 4.4]{MN_2009}. Therefore, for simplicity and clarity, we restrict our attention in this section and \cref{subsection:application_to_bivariant_K-theory_1,subsection:application_to_bivariant_K-theory_2} to the case of finite topological spaces. 
	\end{rem}
	
	\begin{rem}
		The categories $\Ecat(X)$ and $\Boot_{\E}(X)$ can be defined for any second countable space $X$; see \cite{DM_2012}. 
		The results for ideal-related E-theory also remain valid for suitable infinite topological spaces; see \cref{rem:equivalence_E_Wc_Wo_infinite,rem:equivalence_E_Walpha_Wbeta_infinite}. 
	\end{rem}
	
	The category $\KKcat(X)$ is triangulated with respect to the following structure.
	We refer the reader to \cite[Appendix~A]{MN_2004} and \cite[Section~3.3]{MN_2009} for further details. 
	The translation functor is the suspension functor 
	\[S=\conti_0((0, 1)) \otimes \blank \colon \KKcat(X) \to \KKcat(X). \]
	A triangle is distinguished if it is isomorphic to a mapping cone triangle 
	\[
	\begin{tikzcd}
		SA_1 \ar{r} & C_{\varphi} \ar{r} & A_0 \ar{r}{\KK_X(\varphi)} & A_1
	\end{tikzcd}
	\]
	of a morphism $\varphi \colon A_0 \to A_1$ in $\SCstar(X)$. 
	Moreover, for a countable family $\calF$ of separable C*-algebras over $X$, the C*-direct sum $\bigoplus_{A \in \calF} A$ is a coproduct in $\KKcat(X)$. 
	Analogously, $\Ecat(X)$ is a triangulated category with countable coproducts; see \cite[Theorem~2.27]{DM_2012}. 
	
	We recall excision properties of ideal-related KK-theory and E-theory, which will affect the main results in \cref{subsection:application_to_bivariant_K-theory_1,subsection:application_to_bivariant_K-theory_2}. 
	For a short exact sequence
	\[
	\begin{tikzcd}
		0 \ar{r} & A \ar{r}{\iota} & E \ar{r}{\pi} & B \ar{r} & 0
	\end{tikzcd}
	\] 
	in $\SCstar(X)$, we have the following commutative diagram with exact rows in $\SCstar(X)$:
	\[
	\begin{tikzcd}
		0 \ar{r} & A \ar{r}{\iota} \ar{d} & E \ar{r}{\pi} \ar[swap]{d}{\pi} & B \ar{r} \ar{d}{\id} & 0 \\
		0 \ar{r} & 0 \ar{r} & B \ar[swap]{r}{\id} & B \ar{r} & 0\mathrlap{.}
	\end{tikzcd}
	\] 
	By taking the mapping cones of the vertical morphisms, we obtain a short exact sequence 
	\[
	\begin{tikzcd}
		0 \ar{r} & A \ar{r} & C_{\pi} \ar{r} & CB \ar{r} & 0
	\end{tikzcd}
	\] 
	in $\SCstar(X)$. 
	We write $\iota'$ for the resulting morphism $A \hookrightarrow C_{\pi}$ in $\SCstar(X)$. 
	Since $CB$ is contractible in $\SCstar(X)$, the morphism $\E_X(\iota') \colon A \to C_{\pi}$ is an isomorphism in $\Ecat(X)$ by \cref{lem:Bott}. This implies that any short exact sequence in $\SCstar(X)$ gives rise to a distinguished triangle in $\Ecat(X)$. However, $\KK_X(\iota') \colon A \to C_{\pi}$ is not an isomorphism in $\KKcat(X)$ in general. This distinction leads to the following notion.
	
	\begin{defn}[{\cite[Definition~3.10]{MN_2009}}]
		A short exact sequence 
		\[
		\begin{tikzcd}
			0 \ar{r} & A \ar{r}{\iota} & E \ar{r}{\pi} & B \ar{r} & 0
		\end{tikzcd}
		\] 
		in $\SCstar(X)$ is said to be \textit{admissible} if $\KK_X(\iota') \colon A \to C_{\pi}$ is an isomorphism in $\KKcat(X)$. 
	\end{defn}
	
	It follows that an admissible short exact sequence in $\SCstar(X)$ gives rise to a distinguished triangle in $\KKcat(X)$. By \cite[Theorem~3.6]{MN_2009}, every semi-split short exact sequence is admissible. Since every mapping cone short exact sequence is semi-split, it is admissible. 
	
	By \cite[Corollary~1.2.6]{Neeman_2001} (see also \cref{lem:Bott}), we have the following lemma, which will play a crucial role in \cref{subsection:application_to_bivariant_K-theory_1,subsection:reflection_of_connected_locally_closed_subsets}.
	
	\begin{lem}\label{lem:contractible_isomorphism_KK_E}
		Let 
		\[
		\begin{tikzcd}
			0 \ar{r} & A \ar{r}{\iota} & E \ar{r}{\pi} & B \ar{r} & 0
		\end{tikzcd}
		\] 
		be a short exact sequence in $\SCstar(X)$. 
		If $A$ is contractible, then $\E_X(\pi)$ is an isomorphism in $\Ecat(X)$, and if $B$ is contractible, then $\E_X(\iota)$ is an isomorphism in $\Ecat(X)$.
		The analogous statement holds in ideal-related KK-theory provided that the given short exact sequence is admissible.
	\end{lem}
	
	\begin{rem}
		The maximal tensor product functor $\otimes \colon \KKcat \times \KKcat \to \KKcat$
		equips $\KKcat$ with the structure of a tensor triangulated category; see \cite{Dell'Ambrogio_2010}. Similarly, the maximal tensor product functor $\otimes \colon \KKcat \times \KKcat(X) \to \KKcat(X)$ defines a tensor action of $\KKcat$ on $\KKcat(X)$ that is compatible with the triangulated category structure and countable coproducts in the sense of \cite[Definition~3.2]{Stevenson_2013}.
		Note that, in particular, this action makes $\KKcat(X)$ a $\KKcat$-module category in the sense of \cite[Definition~7.1.1]{EGNO_2015}. Moreover, $\KKcat(X)_{\loc}$ is a $\KKcat$-module subcategory of $\KKcat(X)$.
		Similarly, $\Ecat$ is a tensor triangulated category (see \cite{Thom_2003}), and $\Ecat(X)$ admits a tensor action of $\Ecat$. The equivalences established in this paper for $\KKcat(\blank)_{\loc}$ and $\Ecat(\blank)$ are in fact equivalences of triangulated categories compatible with the respective tensor actions of $\KKcat$ and $\Ecat$.
		For simplicity, we will state the equivalences only as equivalences of triangulated categories with countable coproducts in \cref{subsection:application_to_bivariant_K-theory_1,subsection:application_to_bivariant_K-theory_2}.
	\end{rem}
	
	\subsection{Functors between categories of bivariant K-theory}
	
	The functors and natural transformations introduced in \cref{section:reflection_functors,section:higher-dimensional_reflection_functors} are first defined at the level of C*-algebras over topological spaces. This subsection records the properties needed to pass them to the corresponding categories of bivariant K-theory. 
	
	Throughout this subsection, let $X$ and $Y$ be finite topological spaces. The following consequence of the universal properties recalled above is briefly indicated in \cite{MN_2004,MN_2009}.
	
	\begin{lem}\label{lem:induced_functor_KK_E}
		For an exact $\SCstar$-module functor $F \colon \SCstar(X) \to \SCstar(Y)$, there exists a unique functor $\widetilde{F} \colon \KKcat(X) \to \KKcat(Y)$ making the following diagram commute:
		\[
		\begin{tikzcd}
			\SCstar(X) \ar[swap]{d}{\KK_X} \ar{r}{F} & \SCstar(Y) \ar{d}{\KK_Y} \\
			\KKcat(X) \ar[dashed, swap]{r}{\widetilde{F}} & \KKcat(Y)\mathrlap{.}
		\end{tikzcd}
		\]
		The induced functor $\widetilde{F} \colon \KKcat(X) \to \KKcat(Y)$ is a triangulated functor. 
		Similarly, $F$ induces a triangulated functor $\widetilde{F} \colon \Ecat(X) \to \Ecat(Y)$. Moreover, if $F$ commutes with countable direct sums, then the induced functors $\widetilde{F} \colon \KKcat(X) \to \KKcat(Y)$ and $\widetilde{F} \colon \Ecat(X) \to \Ecat(Y)$ commute with countable coproducts. 
	\end{lem}
	
	\begin{proof}
		Since $F$ is a $\SCstar$-module functor and $\KK_Y$ is homotopy-invariant and stable, \cref{lem:homotopy-invariant_stable_module} shows that $\KK_Y F$ is homotopy-invariant and stable. Since $F$ is exact and $\KK_Y$ is split-exact, $\KK_Y F$ is split-exact. By the universal property of $\KKcat(X)$, there exists a unique functor $\widetilde{F} \colon \KKcat(X) \to \KKcat(Y)$ making the above diagram commute. Since $F$ is an exact $\SCstar$-module functor, $F$ commutes with mapping cones; see \cref{lem:exact_module_functor_commute_mapping_cone}. Hence, $\widetilde{F}$ is a triangulated functor. This shows the assertion for ideal-related KK-theory. Since $\E_Y F \colon \SCstar(X) \to \Ecat(Y)$ is half-exact, the same argument applies to show the assertion for ideal-related E-theory. The final assertion follows directly from the definitions. 
	\end{proof}
	
	In concrete applications of \cref{lem:induced_functor_KK_E},
	we use the same notation for the induced functors as for the original functor.
	
	\begin{eg}
		Let $\theta \colon X \to Y$ be a continuous map. 
		The functor $\theta_* \colon \SCstar(X) \to \SCstar(Y)$ is an exact $\SCstar$-module functor that commutes with countable direct sums.
		Hence, it induces a triangulated functor 
		\[\theta_* \colon \KKcat(X) \to \KKcat(Y)\] 
		that commutes with countable coproducts. 
		Similarly, it induces a triangulated functor 
		\[\theta_* \colon \Ecat(X) \to \Ecat(Y)\]
		that commutes with countable coproducts. 
	\end{eg}
	
	\begin{eg}
		Suppose that $Y \in \LC(X)$. The functor $r_X^Y \colon \SCstar(X) \to \SCstar(Y)$ is an exact $\SCstar$-module functor that commutes with countable direct sums. Hence, it induces a triangulated functor 
		\[r_X^Y \colon \KKcat(X) \to \KKcat(Y)\]
		that commutes with countable coproducts. 
		Similarly, it induces a triangulated functor 
		\[r_X^Y \colon \Ecat(X) \to \Ecat(Y)\] 
		that commutes with countable coproducts. 
	\end{eg}
	
	By a \textit{morphism of triangulated functors}, we mean a natural transformation compatible with the translation isomorphisms; see \cite[Definition~10.1.9(ii)]{KP_2006}. 
	
	\begin{lem}\label{lem:induced_natural_transformation_KK_E}
		Let $F$ and $G$ be exact $\SCstar$-module functors $\SCstar(X) \to \SCstar(Y)$. 
		Let $\widetilde{F}$ and $\widetilde{G}$ denote the induced functors $\KKcat(X) \to \KKcat(Y)$ as in \textup{\cref{lem:induced_functor_KK_E}}. For a natural transformation $\eta \colon F \Rightarrow G$, the morphisms $\KK_Y(\eta_A) \colon \widetilde{F}(A) \to \widetilde{G}(A)$ in $\KKcat(Y)$ for $A \in \KKcat(X)$ define a natural transformation $\widetilde{\eta} \colon \widetilde{F} \Rightarrow \widetilde{G}$. 
		Similarly, $\eta$ induces a natural transformation $\widetilde{\eta}$ between the induced functors $\Ecat(X) \to \Ecat(Y)$. 
		If $\eta$ is a $\SCstar$-module natural transformation, then the induced natural transformations $\widetilde{\eta}$ in ideal-related KK-theory and E-theory are morphisms of triangulated functors.
	\end{lem}

	\begin{proof}
		We first verify the naturality of $\widetilde{\eta}$ with respect to all morphisms in $\KKcat(X)$. By \cref{lem:morphisms_in_KK_E}, it suffices to verify the naturality with respect to morphisms of the form $f=\KK_X(\varphi)$, where $\varphi \colon A \to B$ is a morphism in $\SCstar(X)$. The naturality of $\eta$ gives
		\begin{align*}
			\widetilde{\eta}_B \circ \widetilde{F}(f)&=\KK_Y(\eta_B) \circ \KK_Y(F(\varphi))=\KK_Y(\eta_B \circ F(\varphi)) \\
			&=\KK_Y(G(\varphi) \circ \eta_A)=\KK_Y(G(\varphi)) \circ \KK_Y(\eta_A)=\widetilde{G}(f) \circ \widetilde{\eta}_A. 
		\end{align*}
		This shows the assertion for ideal-related KK-theory. The same argument applies to ideal-related E-theory. The final assertion follows immediately from the definitions. 
	\end{proof}
	
	The following lemma is an immediate consequence of \cref{lem:induced_functor_KK_E} and the compatibility of exact $\SCstar$-module functors with mapping cones established in \cref{lem:exact_module_functor_commute_mapping_cone}.
	
	\begin{lem}\label{lem:admissible_exact_module}
		Let $F \colon \SCstar(X) \to \SCstar(Y)$ be an exact $\SCstar$-module functor. For an admissible short exact sequence
		\[
		\begin{tikzcd}
			0 \ar{r} & A \ar{r}{\iota} & E \ar{r}{\pi} & B \ar{r} & 0
		\end{tikzcd}
		\] 
		in $\SCstar(X)$, the short exact sequence
		\[
		\begin{tikzcd}
			0 \ar{r} & FA \ar{r}{F\iota} & FE \ar{r}{F\pi} & FB \ar{r} & 0
		\end{tikzcd}
		\] 
		in $\SCstar(Y)$ is admissible. 
	\end{lem}
	
	The case $|I|=1$ of the following lemma is briefly indicated in the proof of \cite[Proposition~4.7]{MN_2009}.
	Since the proof essentially requires 2-dimensional mapping cones in \cref{subsection:higher-dimensional_mapping_cones}, we give a complete proof.
	
	\begin{lem}\label{lem:admissible_mapping_cones}
		Let $I$ be a finite set. Let
		\[
		\begin{tikzcd}
			0 \ar[Rightarrow]{r} & (A; \varphi) \ar[Rightarrow]{r}{\iota} & (E; \omega) \ar[Rightarrow]{r}{\pi} & (B; \psi) \ar[Rightarrow]{r} & 0
		\end{tikzcd}
		\] 
		be a short exact sequence in $[\{0, 1\}^I, \SCstar(X)]$. Suppose that the short exact sequence
		\[\begin{tikzcd}
			0 \ar{r} & A_{\lambda} \ar{r}{\iota_{\lambda}} & E_{\lambda} \ar{r}{\pi_{\lambda}} & B_{\lambda} \ar{r} & 0
		\end{tikzcd}\] 
		in $\SCstar(X)$ is admissible for every $\lambda \in \{0, 1\}^I$. Then the short exact sequence
		\[\begin{tikzcd}
			0 \ar{r} & M^I(A; \varphi) \ar{r}{M^I\iota} & M^I(E; \omega) \ar{r}{M^I\pi} & M^I(B; \psi) \ar{r} & 0
		\end{tikzcd}\] 
		in $\SCstar(X)$ is admissible.
	\end{lem}

	\begin{proof}
		We prove the case $|I|=1$. 
		Using the functoriality and exactness of 
		\[M^2 \colon [\{0, 1\}^2, \SCstar(X)] \to \SCstar(X),\] 
		we obtain the following commutative diagram 
		with exact rows in $\SCstar(X)$:
		\[
		\begin{tikzcd}[ampersand replacement=\&, every label/.append style={font=\tiny}, cells={nodes={font=\tiny}}]
			0 \ar{r}  
			\& [-1.5em]
			{M^2 \! \left(
				\begin{tikzpicture}[auto]
					\node (00) at (-0.6, 0.5) {$0$}; \node (01) at (0.6, 0.5) {$0$};
					\node (10) at (-0.6, -0.5) {$A_1$};  \node (11) at (0.6, -0.5) {$0$};
					\draw [->] (00) to node {} (01);
					\draw [->] (00) to node[swap] {} (10);
					\draw [->] (01) to node {} (11);
					\draw [->] (10) to node[swap] {} (11);
				\end{tikzpicture} \right)} 
			\ar{r}[xshift=0.2em]{
				M^2 \! \begin{pmatrix}
					0 & 0 \\ \id & 0
				\end{pmatrix}
			}
			\ar[hookrightarrow]{d}[swap]{
				M^2 \! \begin{pmatrix}
					0 & 0 \\ \iota_1 & 0
				\end{pmatrix}
			}
			\&
			{M^2 \! \left(
				\begin{tikzpicture}[auto]
					\node (00) at (-0.6, 0.5) {$A_0$}; \node (01) at (0.6, 0.5) {$0$};
					\node (10) at (-0.6, -0.5) {$A_1$};  \node (11) at (0.6, -0.5) {$0$};
					\draw [->] (00) to node {} (01);
					\draw [->] (00) to node[swap] {$\varphi_0^1$} (10);
					\draw [->] (01) to node {} (11);
					\draw [->] (10) to node[swap] {} (11);
				\end{tikzpicture} \right)}
			\ar{r}[xshift=0.2em]{
				M^2 \! \begin{pmatrix}
					\id & 0 \\ 0 & 0
				\end{pmatrix}
			}
			\ar[hookrightarrow]{d}[swap]{
				M^2 \! \begin{pmatrix}
					\iota_0 & 0 \\ \iota_1 & 0
				\end{pmatrix}
			}
			\&
			{M^2 \! \left(
				\begin{tikzpicture}[auto]
					\node (00) at (-0.6, 0.5) {$A_0$}; \node (01) at (0.6, 0.5) {$0$};
					\node (10) at (-0.6, -0.5) {$0$};  \node (11) at (0.6, -0.5) {$0$};
					\draw [->] (00) to node {} (01);
					\draw [->] (00) to node[swap] {} (10);
					\draw [->] (01) to node {} (11);
					\draw [->] (10) to node[swap] {} (11);
				\end{tikzpicture} \right)} 
			\ar{r} 
			\ar[hookrightarrow]{d}{
				M^2 \! \begin{pmatrix}
					\iota_0 & 0 \\ 0 & 0
				\end{pmatrix}
			}
			\& 
			[-1.5em] 0 
			\\
			0 \ar{r}  
			\& [-1.5em]
			{M^2 \! \left(
				\begin{tikzpicture}[auto]
					\node (00) at (-0.6, 0.5) {$0$}; \node (01) at (0.6, 0.5) {$0$};
					\node (10) at (-0.6, -0.5) {$E_1$};  \node (11) at (0.6, -0.5) {$B_1$};
					\draw [->] (00) to node {} (01);
					\draw [->] (00) to node[swap] {} (10);
					\draw [->] (01) to node {} (11);
					\draw [->] (10) to node[swap] {$\pi_1$} (11);
				\end{tikzpicture} \right)} 
			\ar{r}[swap, xshift=0.5em, yshift=-0.2em]{
				M^2 \! \begin{pmatrix}
					0 & 0 \\ \id & \id
				\end{pmatrix}
			}
			\&
			{M^2 \! \left(
				\begin{tikzpicture}[auto]
					\node (00) at (-0.6, 0.5) {$E_0$}; \node (01) at (0.6, 0.5) {$B_0$};
					\node (10) at (-0.6, -0.5) {$E_1$};  \node (11) at (0.6, -0.5) {$B_1$};
					\draw [->] (00) to node {$\pi_0$} (01);
					\draw [->] (00) to node[swap] {$\omega_0^1$} (10);
					\draw [->] (01) to node {$\psi_0^1$} (11);
					\draw [->] (10) to node[swap] {$\pi_1$} (11);
				\end{tikzpicture} \right)}
			\ar{r}[swap, xshift=0.5em, yshift=-0.2em]{
				M^2 \! \begin{pmatrix}
					\id & \id \\ 0 & 0
				\end{pmatrix}
			}
			\&
			{M^2 \! \left(
				\begin{tikzpicture}[auto]
					\node (00) at (-0.6, 0.5) {$E_0$}; \node (01) at (0.6, 0.5) {$B_0$};
					\node (10) at (-0.6, -0.5) {$0$};  \node (11) at (0.6, -0.5) {$0$};
					\draw [->] (00) to node {$\pi_0$} (01);
					\draw [->] (00) to node[swap] {} (10);
					\draw [->] (01) to node {} (11);
					\draw [->] (10) to node[swap] {} (11);
				\end{tikzpicture} \right)} \ar{r} 
			\& 
			[-1.5em] 0\mathrlap{.}
		\end{tikzcd}
		\]
		This diagram is identical to the following one:
		\[
		\begin{tikzcd}
			0 \ar{r} & SA_1 \ar{r} \ar[swap, hookrightarrow]{d}{S\iota_1'} & M(A; \varphi) \ar{r} \ar[swap, hookrightarrow]{d}{M(\iota)'} & A_0 \ar{r} \ar[hookrightarrow]{d}{\iota_0'} & 0 \\
			0 \ar{r} & SC_{\pi_1} \ar{r} & C_{M(\pi)} \ar{r} & C_{\pi_0} \ar{r} & 0\mathrlap{.}
		\end{tikzcd}
		\] 
		The resulting two short exact sequences are mapping cone short exact sequences, and hence give distinguished triangles in $\KKcat(X)$. Since $\KK_X(S\iota_1') \colon SA_1 \to SC_{\pi_1}$ and $\KK_X(\iota_0') \colon A_0 \to C_{\pi_0}$ are isomorphisms in $\KKcat(X)$, the five lemma for morphisms of distinguished triangles (see \cite[Proposition~1.1.20]{Neeman_2001}) shows that $\KK_X(M(\iota)') \colon M(A; \varphi) \to C_{M(\pi)}$ is also an isomorphism in $\KKcat(X)$. This proves the case $|I|=1$. The general case follows from \cref{lem:mapping_cone_of_mapping_cone}.
		\qedhere
	\end{proof}
	
	\subsection{Bootstrap categories}
	
	Let $X$ be a finite topological space. 
	Let $\KKcat(X)_{\loc}$ denote the localizing subcategory of $\KKcat(X)$ defined in \cite[Definition~4.8]{MN_2009}. One of the equivalent definitions of $\KKcat(X)_{\loc}$ is the full subcategory of $\KKcat(X)$ consisting of all separable C*-algebras $A$ over $X$ such that the short exact sequence 
	\[
	\begin{tikzcd}
		0 \ar{r} & i_{Y_1}^X r_X^{Y_1}A \ar{r} & i_{Y_2}^X r_X^{Y_2}A \ar{r} & i_{Y_3}^X r_X^{Y_3}A \ar{r} & 0
	\end{tikzcd}
	\] 
	in $\SCstar(X)$ given by \cref{lem:ses_ir} is admissible for $Y_2 \in \LC(X)$ and $Y_1 \in \Open(Y_2)$ with $Y_3:=Y_2 \setminus Y_1$; see \cite[Proposition~4.7]{MN_2009}. 
	
	Let $\Boot(X)$ denote the localizing subcategory of $\KKcat(X)_{\loc}$ defined in \cite[Definition~4.11]{MN_2009}, and let $\Boot_{\E}(X)$ denote the localizing subcategory of $\Ecat(X)$ defined in \cite[Definition~4.1]{DM_2012}.
	
	With the help of \cref{lem:relations_extension_restriction_functors,lem:admissible_exact_module}, we can easily verify the following two propositions. 
	
	\begin{prop}\label{prop:extension_functor_KKloc_bootstrap}
		For a continuous map $\theta \colon X \to Y$ between finite topological spaces, the functor $\theta_* \colon \KKcat(X) \to \KKcat(Y)$ restricts to functors 
		\[\theta_* \colon \KKcat(X)_{\loc} \to \KKcat(Y)_{\loc}, \quad \theta_* \colon \Boot(X) \to \Boot(Y). \]
		Likewise, the functor $\theta_* \colon \Ecat(X) \to \Ecat(Y)$ restricts to a functor 
		\[\theta_* \colon \Boot_{\E}(X) \to \Boot_{\E}(Y).\]
	\end{prop}

	\begin{prop}\label{prop:restriction_functor_KKloc_bootstrap}
		For a finite topological space $X$ and $Y \in \LC(X)$, the functor $r_X^Y \colon \KKcat(X) \to \KKcat(Y)$ restricts to functors 
		\[r_X^Y \colon \KKcat(X)_{\loc} \to \KKcat(Y)_{\loc}, \quad r_X^Y \colon \Boot(X) \to \Boot(Y). \]
		Likewise, the functor $r_X^Y \colon \Ecat(X) \to \Ecat(Y)$ restricts to a functor \[r_X^Y \colon \Boot_{\E}(X) \to \Boot_{\E}(Y). \]
	\end{prop}

	\section{Reflection functors}\label{section:reflection_functors}
	
	In this section, we introduce reflection functors. Throughout this section, we fix
	\begin{itemize}
		\item a finite set $J$;
		\item a preordered set $K=(K, R_K)$; 
		\item preordered sets $L_j=(L_j, R_{L_j})$ for $j \in J$; 
		\item order-preserving maps $\theta_j \colon K \to L_j$ for $j \in J$; 
		\item a C*-algebra $D$; 
		\item C*-algebras $D_j$ for $j \in J$; 
		\item mutually orthogonal $\ast$-homomorphisms $\rho_j \colon D_j \to D$ for $j \in J$.
	\end{itemize}	
	We define a set 
	\[W:=K \amalg \coprod_{j \in J} L_j. \]
	
	In \cref{subsection:setting_of_topological_spaces}, we define two preorders $R_{W^{\closed}}$ and $R_{W^{\open}}$ on the set $W$ using $R_K$, $(R_{L_j})_{j \in J}$, and $(\theta_j)_{j \in J}$; see \cref{def:RWc_RWo}. The resulting preordered sets are regarded as Alexandrov spaces and are denoted by $W^{\closed}$ and $W^{\open}$. Here, the labels $\closed$ and $\open$ indicate that the subset $K \subset W$ is closed in $W^{\closed}$ and open in $W^{\open}$; see \cref{prop:K_closed_open}. 
	
	In \cref{subsection:definition_of_reflection_functors}, we define reflection functors
	\[
	\begin{tikzcd}[row sep=large]
		\Cstar(W^{\closed}) \ar[bend left=10]{r}{S_{\closed}^{\open}} &
		\ar[bend left=10]{l}{S_{\open}^{\closed}} \Cstar(W^{\open})\mathrlap{.}
	\end{tikzcd}
	\]
	The reflection functor $S_{\open}^{\closed}$ will be defined independently of $D$, $(D_j)_{j \in J}$, and $(\rho_j)_{j \in J}$, while $S_{\closed}^{\open}$ will depend on them. Roughly speaking, $S_{\open}^{\closed}$ is defined as the mapping cone of a morphism from one object to a direct sum, whereas $S_{\closed}^{\open}$ is defined as the mapping cone of a morphism from a direct sum to one object. A key point is that a sum of $\ast$-homomorphisms is not, in general, again a $\ast$-homomorphism, and that only mutually orthogonal $\ast$-homomorphisms can be added; see \cref{subsection:natural_transformations_associated_with_direct_sums}. Hence, to define $S_{\closed}^{\open}$, we need to use $D$, $(D_j)_{j \in J}$, and $(\rho_j)_{j \in J}$. 
	
	In \cref{subsection:composition_reflection_functor}, we study the composites 
	\[S_{\open}^{\closed}S_{\closed}^{\open} \colon \Cstar(W^{\closed}) \to \Cstar(W^{\closed}), \quad S_{\closed}^{\open}S_{\open}^{\closed} \colon \Cstar(W^{\open}) \to \Cstar(W^{\open}). \]
	In \cref{subsection:zigzag_identity}, we examine the zigzag identities.
	The main results in \cref{subsection:composition_reflection_functor,subsection:zigzag_identity} are \cref{thm:key_diagram,thm:zigzag_identity}. 
	In \cref{subsection:application_to_Bott_functors_1}, we apply \cref{thm:key_diagram} to homotopy-invariant, stable, half-exact functors. 
	In \cref{subsection:application_to_bivariant_K-theory_1}, we apply \cref{thm:key_diagram,thm:zigzag_identity} to ideal-related KK-theory and E-theory. The main results in \cref{subsection:application_to_Bott_functors_1,subsection:application_to_bivariant_K-theory_1} are \cref{thm:reflection_Bott,thm:equivalence_KKloc_Wc_Wo,thm:equivalence_E_Wc_Wo,cor:KKloc_tree,cor:E_tree}.
	
	It is worth emphasizing that \cref{thm:key_diagram,thm:zigzag_identity} hold for arbitrary $D$, $(D_j)_{j \in J}$, and $(\rho_j)_{j \in J}$. When we apply reflection functors to ideal-related KK-theory and E-theory in \cref{subsection:application_to_bivariant_K-theory_1}, however, we use the specific choice of $D$, $(D_j)_{j \in J}$, and $(\rho_j)_{j \in J}$ specified in \cref{eg:reflection_suspension} (see also \cref{eg:reflection_matrix}). 
	
	\subsection{Setting of topological spaces}\label{subsection:setting_of_topological_spaces}
	
	\begin{defn}\label{def:RWc_RWo}
		We define two relations $R_{W^{\closed}}$ and $R_{W^{\open}}$ on $W$ by
		\begin{align*}
			R_{W^{\closed}} &:=R_K \amalg \coprod_{j \in J} R_{L_j} \amalg \coprod_{j \in J} \{(x, y) \in K \times L_j \mid (\theta_j(x), y) \in R_{L_j}\}, \\
			R_{W^{\open}} &:=R_K \amalg \coprod_{j \in J} R_{L_j} \amalg \coprod_{j \in J} \{(y, x) \in L_j \times K \mid (y, \theta_j(x)) \in R_{L_j}\}.
		\end{align*}
	\end{defn}
	
	In what follows, we use the index set $\{\closed, \open\}$.
	
	\begin{lem}\label{lem:RWc_RWo_preorder}
		For $\alpha \in \{\closed, \open\}$, the relation $R_{W^{\alpha}}$ is a preorder on $W$.  
	\end{lem}
	
	\begin{proof}
		The reflexivity of $R_{W^{\alpha}}$ follows from the reflexivity of $R_K$ and $R_{L_j}$ for $j \in J$. The transitivity of $R_{W^{\alpha}}$ is a routine verification using the transitivity of $R_K$ and $R_{L_j}$ for $j \in J$ together with the assumption that each $\theta_j \colon K \to L_j$ is order-preserving. More explicitly, it follows from the same case-by-case argument as in the proof of \cref{lem:RWalpha_preorder}: for $\alpha=\closed$, one uses Cases 1, 2, 3, and 5, while for $\alpha=\open$, one uses Cases 1, 2, 4, and 6.
	\end{proof}
	
	\begin{defn}\label{def:Wc_Wo}
		In view of \cref{lem:RWc_RWo_preorder}, for $\alpha \in \{\closed, \open\}$, we define $W^{\alpha}$ to be the Alexandrov space associated with the preordered set $(W,R_{W^{\alpha}})$.
	\end{defn}

	\begin{prop}
		The two relative topologies on $K$ induced from $W^{\closed}$ and $W^{\open}$ coincide with the Alexandrov topology on $K$ given by the preorder $R_K$. 
	\end{prop}
	
	\begin{proof}
		This follows from $R_{W^{\alpha}} \cap (K \times K)=R_K$ for $\alpha \in \{\closed, \open\}$. 
	\end{proof}

	\begin{prop}
		The two relative topologies on $\coprod_{j \in J} L_j$ induced from $W^{\closed}$ and $W^{\open}$ coincide with the Alexandrov topology on $\coprod_{j \in J} L_j$ given by the preorder $\coprod_{j \in J} R_{L_j}$. In particular, for each $j \in J$, the two relative topologies on $L_j$ induced from $W^{\closed}$ and $W^{\open}$ coincide with the Alexandrov topology on $L_j$ given by the preorder $R_{L_j}$. 
	\end{prop}
	
	\begin{proof}
		The first assertion follows from 
		\[R_{W^{\alpha}} \cap \Bigl(\coprod_{j \in J} L_j \times \coprod_{j \in J} L_j\Bigr)=\coprod_{j \in J} R_{L_j}\]
		for $\alpha \in \{\closed, \open\}$. 
		The second assertion follows from the first. 
	\end{proof}

	\begin{prop}\label{prop:K_closed_open}
		We have $K \in \Closed(W^{\closed})$ and $K \in \Open(W^{\open})$.
	\end{prop}
	
	\begin{proof}
		This follows from $R_{W^{\closed}} \cap ((W \setminus K) \times K)=\emptyset$ and $R_{W^{\open}} \cap (K \times (W \setminus K))=\emptyset$. 
	\end{proof}

	\begin{prop}\label{prop:Lj_open_closed}
		For each $j \in J$, we have $L_j \in \Open(W^{\closed})$ and $L_j \in \Closed(W^{\open})$.
	\end{prop}
	
	\begin{proof}
		This follows from $R_{W^{\closed}} \cap (L_j \times (W \setminus L_j))=\emptyset$ and $R_{W^{\open}} \cap ((W \setminus L_j) \times L_j)=\emptyset$. 
	\end{proof}
	
	\begin{prop}
		The following conditions are equivalent: 
		\begin{enumerate}[label=\textnormal{(\roman*)}]
			\item $W^{\closed}$ is $T_0$; 
			\item $W^{\open}$ is $T_0$;
			\item $K$ is $T_0$, and $L_j$ is $T_0$ for every $j \in J$.
		\end{enumerate}
	\end{prop}
	
	\begin{proof}
		The implications $\text{(i)} \Rightarrow \text{(iii)}$ and $\text{(ii)} \Rightarrow \text{(iii)}$ are evident. Let $\alpha \in \{\closed, \open\}$. For $x, y \in W$ with $(x, y), (y, x) \in R_{W^{\alpha}}$, either $(x, y), (y, x) \in R_K$ or $(x, y), (y, x) \in R_{L_j}$ for some $j \in J$. If $(K, R_K)$ and $(L_j, R_{L_j})$ are partially ordered sets, we get $x=y$. Hence, $R_{W^{\alpha}}$ is a partial order on $W$. This shows $\text{(iii)} \Rightarrow \text{(i)}$ and $\text{(iii)} \Rightarrow \text{(ii)}$. 
	\end{proof}

	\begin{eg}\label{eg:diamond}
		Consider the case where $J=\{j_0\}$.
		Let $K:=\{1,2,3\}$ and $L_{j_0}:=\{4\}$ be the partially ordered sets whose Hasse diagrams are depicted below. 
		\[\setlength{\arraycolsep}{1.5em}
		\begin{array}{cc}
			\begin{tikzcd}
				& 1 &\\
				2 \ar{ur} & & 3 \ar{ul}
			\end{tikzcd}
			&
			\begin{tikzcd}
				4
			\end{tikzcd}
			\\
			K & L_{j_0}
		\end{array}
		\]
		Let $\theta_{j_0} \colon K \to L_{j_0}$ be the order-preserving map defined by
		\[\theta_{j_0}(1)=\theta_{j_0}(2)=\theta_{j_0}(3):=4.\]
		Let $W^{\closed}$ and $W^{\open}$ be the finite $T_0$-spaces given by $K$, $L_{j_0}$, and $\theta_{j_0}$.
		Then the Hasse diagrams of $W^{\closed}$ and $W^{\open}$ are depicted below.
		\[\setlength{\arraycolsep}{1.5em}
		\begin{array}{cc}
			\begin{tikzcd}
				& 1 &\\
				2 \ar{ur} &  & 3 \ar{ul} \\
				& 4 \ar{ul} \ar{ur} &
			\end{tikzcd}
			&
			\begin{tikzcd}
				& 4 &\\
				& 1 \ar{u} & \\
				2 \ar{ur} & & 3 \ar{ul}
			\end{tikzcd}
			\\
			W^{\closed} & W^{\open}
		\end{array}
		\]
	\end{eg}
	
	\begin{eg}\label{eg:Wc_Wo}
		Consider the case where $J=\{j_1,j_2,j_3\}$.
		Let $K:=\{1,2,3,4\}$, $L_{j_1}:=\{5,6\}$, $L_{j_2}:=\{7\}$, and $L_{j_3}:=\{8,9,10,11\}$ be the partially ordered sets whose Hasse diagrams are depicted below. 
		\[\setlength{\arraycolsep}{1.5em}
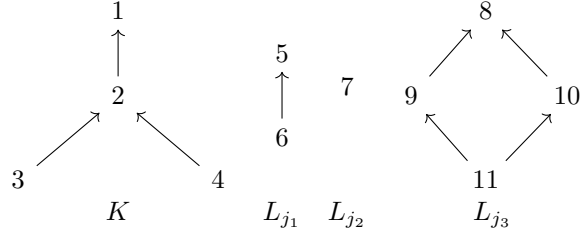
		\begin{array}{cccc}
			\begin{tikzcd}
				& 1 &\\
				& 2 \ar{u} & \\
				3 \ar{ur} & & \ar{ul} 4
			\end{tikzcd}
			&
			\begin{tikzcd}
				5 \\
				6 \ar{u}
			\end{tikzcd}
			&
			7
			&
			\begin{tikzcd}[column sep=small]
				& 8 &\\
				9 \ar{ur} &  & \ar{ul} 10 \\
				& \ar{ul} 11 \ar{ur} &
			\end{tikzcd}
			\\[2em]
			K & L_{j_1} & L_{j_2} & L_{j_3}
		\end{array}
		\]
		Let $\theta_{j_1} \colon K \to L_{j_1}$, $\theta_{j_2} \colon K \to L_{j_2}$, and $\theta_{j_3} \colon K \to L_{j_3}$ be the order-preserving maps defined by
		\begin{gather*}
			\theta_{j_1}(1):=5, \quad
			\theta_{j_1}(2)=\theta_{j_1}(3)=\theta_{j_1}(4):=6, \\
			\theta_{j_2}(1)=\theta_{j_2}(2)=\theta_{j_2}(3)=\theta_{j_2}(4):=7, \\
			\theta_{j_3}(1)=\theta_{j_3}(2):=8, \quad
			\theta_{j_3}(3):=9, \quad
			\theta_{j_3}(4):=10.
		\end{gather*}
		Let $W^{\closed}$ and $W^{\open}$ be the finite $T_0$-spaces given by $K$, $(L_j)_{j \in J}$, and $(\theta_j)_{j \in J}$.
		Then the Hasse diagrams of $W^{\closed}$ and $W^{\open}$ are depicted below.
		\[
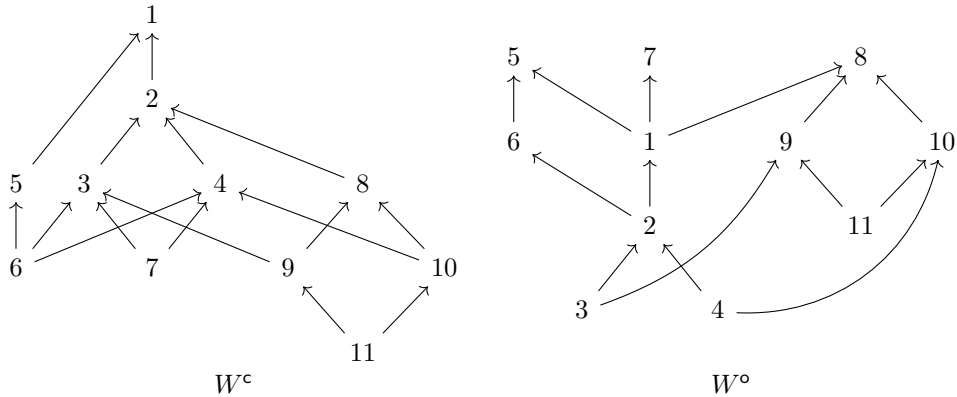
		\begin{array}{cc}
			\begin{tikzcd}[column sep=small]
				& & 1 & & &  & \\
				& & 2 \ar{u} & & &  & \\
				5 \ar{uurr} & 3 \ar{ur} & & \ar{ul} 4 & & \ar{ulll} 8 & \\
				6 \ar{u} \ar{ur} \ar{urrr} & & 7 \ar{ul} \ar{ur} & & \ar{ulll} 9 \ar{ur} & & \ar{ulll} \ar{ul} 10 \\
				& & & & & \ar{ul} 11 \ar{ur} &
			\end{tikzcd}
			&
			\begin{tikzcd}[column sep=small]
				5 & & 7 & & & 8 & \\
				6 \ar{u} & & \ar{ull} 1 \ar{u} \ar{urrr} & & 9 \ar{ur} & & \ar{ul} 10 \\
				& & \ar{ull} 2 \ar{u} &  & & \ar{ul} 11 \ar{ur} & \\
				& 3 \ar{ur} \ar[bend right=20]{uurrr} & & \ar{ul} 4 \ar[bend right=40]{uurrr} & & &
			\end{tikzcd}
			\\[2em]
			W^{\closed} & W^{\open}
		\end{array}
		\]
	\end{eg}
	
	We record a criterion for the connectedness of $W^{\alpha}$ for $\alpha \in \{\closed, \open\}$ in the following proposition.
	
	\begin{prop}\label{prop:connectedness_Wc_Wo}
		The following conditions are equivalent:
		\begin{enumerate}[label=\textnormal{(\roman*)}]
			\item $W^{\closed}$ is connected;
			\item $W^{\open}$ is connected;
			\item $W \neq \emptyset$ and the equivalence relation on $W$ generated by
			\[
			R_K \amalg \coprod_{j \in J}R_{L_j}
			\amalg \coprod_{j \in J}\{(x,\theta_j(x)) \mid x \in K\}
			\]
			is equal to $W \times W$.
		\end{enumerate}
		Furthermore, if $J \neq \emptyset$, $K \neq \emptyset$, and $L_j$ is connected for every $j \in J$, then these equivalent conditions hold.
	\end{prop}
	
	\begin{proof}
		Let $R$ denote the equivalence relation on $W$ defined in (iii).
		For $\alpha \in \{\closed, \open\}$, it is straightforward to verify from the definitions that the equivalence relation on $W$ generated by $R_{W^{\alpha}}$ is equal to $R$.
		Hence, by \cref{lem:connected_Alexandrov_space}, the conditions (i), (ii), and (iii) are equivalent.
		
		Assume that $J \neq \emptyset$, $K \neq \emptyset$, and $L_j$ is connected for every $j \in J$.
		Choose $j_0 \in J$.
		Since $L_{j_0}$ is connected, it is non-empty by convention, and hence we may choose $y_0 \in L_{j_0}$.
		In particular, $W \neq \emptyset$.
		We show that $R=W \times W$.
		Since $R$ is symmetric and transitive, it suffices to show that $(w,y_0) \in R$ for every $w \in W$.
		We distinguish three cases.
		\begin{description}[font=\normalfont\scshape]
			\item[Case 1] $w \in L_{j_0}$.
			Since $L_{j_0}$ is connected, \cref{lem:connected_Alexandrov_space} gives $L_{j_0} \times L_{j_0} \subset R$.
			Hence, $(w, y_0) \in R$.
			
			\item[Case 2] $w \in K$.
			Then $(w,\theta_{j_0}(w)) \in R$.
			Since $L_{j_0}$ is connected, $(\theta_{j_0}(w), y_0) \in R$.
			By transitivity of $R$, we obtain $(w, y_0) \in R$.
			
			\item[Case 3] $w \in L_j$ for some $j \in J \setminus \{j_0\}$.
			Since $K \neq \emptyset$, choose $x \in K$.
			Then $(x,\theta_{j_0}(x)) \in R$ and $(x, \theta_j(x)) \in R$.
			Since $L_{j_0}$ is connected, $(y_0, \theta_{j_0}(x)) \in R$. 
			Since $L_j$ is connected, $(w, \theta_j(x)) \in R$. 
			Thus, by symmetry and transitivity of $R$, we obtain $(w, y_0) \in R$.
		\end{description}
		This completes the proof.
	\end{proof}

	\begin{rem}
		The construction in \cref{def:Wc_Wo} is inspired by Ladkani's work on universal derived equivalences of partially ordered sets in \cite{Ladkani_2007}. Indeed, the construction of $W^{\closed}$ and $W^{\open}$ can be regarded as a special case of Ladkani's construction when $K$ and all the $L_j$ for $j \in J$ are finite partially ordered sets. 
	\end{rem}

	\begin{rem}\label{rem:BGP-reflection}
		We explain how the construction in \cref{def:Wc_Wo} recovers the reflection of orientations in Hasse diagrams introduced in \cref{subsection:reflection_of_orientations_in_quivers}.
		
		Let $G=(V, E)$ be a tree and fix $x \in V$. Set 
		\[J_x:=\{y \in V \mid \{x, y\} \in E\}. \]
		For each $y \in J_x$, let $L_{x, y}$ denote the connected component containing $y$ in the subgraph induced by $V \setminus \{x\}$.
		Since $G$ is a tree, we have $L_{x, y} \cap L_{x, y'}=\emptyset$ for $y, y' \in J_x$ with $y \neq y'$.
		Hence,
		\[
		V=\{x\} \amalg \coprod_{y \in J_x} L_{x, y}
		\]
		as sets. 
		For each $y \in J_x$, define a map $\theta_{x, y} \colon \{x\} \to L_{x, y}$ by $\theta_{x, y}(x):=y$. 
		
		Given an orientation $X$ of $G$, let $L_{x, y}^X$ denote the partially ordered set $L_{x, y}$ equipped with the order induced by $X$. 
		Regarding $\{x\}$ as a totally ordered set, the map $\theta_{x, y} \colon \{x\} \to L_{x, y}^X$ is order-preserving.
		Let $V^{\closed}$ and $V^{\open}$ be the finite $T_0$-spaces obtained from the data $\{x\}$, $(L_{x, y}^X)_{y \in J_x}$, and $(\theta_{x, y})_{y \in J_x}$ as in \cref{def:Wc_Wo}. 
		If $x$ is a sink (closed point) in $X$, then 
		\[X=V^{\closed}, \quad \sigma_x X=V^{\open}. \]
		If $x$ is a source (open point) in $X$, then 
		\[X=V^{\open}, \quad \sigma_x X=V^{\closed}. \]
	\end{rem}

	\subsection{Definition of reflection functors}\label{subsection:definition_of_reflection_functors}
	
	\begin{defn}
		For each $j \in J$, we define a subset 
		\[W_j:=K \amalg L_j \subset W. \]
	\end{defn}
	
	\begin{prop}\label{prop:Wj_open_closed}
		For each $j \in J$, we have $W_j \in \Closed(W^{\closed})$ and $W_j \in \Open(W^{\open})$.
	\end{prop}
	
	\begin{proof}
		This follows from $R_{W^{\closed}} \cap ((W \setminus W_j) \times W_j)=\emptyset$ and $R_{W^{\open}} \cap (W_j \times (W \setminus W_j))=\emptyset$. 
	\end{proof}

	\begin{defn}\label{def:thetatildej}
		For each $j \in J$, we define a map 
		\[\widetilde{\theta}_j \colon W_j \to L_j\] 
		by
		\[
		\widetilde{\theta}_j(x):=
		\begin{dcases}
			x & \text{if $x \in L_j$},\\
			\theta_j(x) & \text{if $x \in K$}.
		\end{dcases}
		\]
	\end{defn}

	\begin{defn}
		For $j \in J$ and $\alpha \in \{\closed, \open\}$, let $W_j^{\alpha}$ denote the subspace of $W^{\alpha}$ with underlying set $W_j$. 
	\end{defn}
	
	\begin{prop}\label{prop:thetatildej_continuous}
		For $j \in J$ and $\alpha \in \{\closed, \open\}$, the map $\widetilde{\theta}_j \colon W_j \to L_j$ is a continuous map $W_j^{\alpha} \to L_j$. 
	\end{prop}
	
	\begin{proof}
		Observe that 
		\[
		R_{W^{\alpha}} \cap (W_j \times W_j)=
		\begin{dcases}
			R_K \amalg R_{L_j} \amalg \{(x,y) \in K \times L_j \mid (\theta_j(x),y) \in R_{L_j}\}
			& \text{if $\alpha=\closed$},\\
			R_K \amalg R_{L_j} \amalg \{(y,x) \in L_j \times K \mid (y,\theta_j(x)) \in R_{L_j}\}
			& \text{if $\alpha=\open$}.
		\end{dcases}
		\]
		Since $\theta_j \colon K \to L_j$ is an order-preserving map between the preordered sets $(K, R_K)$ and $(L_j, R_{L_j})$, it is easy to verify from the above description that $\widetilde{\theta}_j \colon W_j \to L_j$ is an order-preserving map between the preordered sets $(W_j, R_{W^{\alpha}} \cap (W_j \times W_j))$ and $(L_j, R_{L_j})$. This shows the assertion. 
	\end{proof}
	
	\begin{defn}\label{def:p_Wj_Lj}
		For $j \in J$ and $\alpha \in \{\closed, \open\}$, let 
		\[p_{W_j^{\alpha}}^{L_j} \colon \Cstar(W_j^{\alpha}) \to \Cstar(L_j)\]
		denote the functor induced by the continuous map $\widetilde{\theta}_j \colon W_j^{\alpha} \to L_j$ as in \cref{def:functor_induced_by_continuous_map}.
	\end{defn}
	
	The following elementary lemma plays a crucial role. 
	
	\begin{lem}\label{lem:preimage_thetaj}
		Fix $j \in J$. 
		\begin{enumerate}[label=\textnormal{(\arabic*)}]
			\item \label{lem:preimage_thetaj_1}
			For $U \in \Open(W^{\closed})$, we have $U \cap K \subset \theta_j^{-1}(U \cap L_j)$.
			
			\item \label{lem:preimage_thetaj_2}
			For $F \in \Closed(W^{\open})$, we have $F \cap K \subset \theta_j^{-1}(F \cap L_j)$.
		\end{enumerate}
	\end{lem}
	
	\begin{proof}
		By symmetry, it suffices to show \labelcref{lem:preimage_thetaj_1}. 
		Take $x \in U \cap K$. Since $(x, \theta_j(x)) \in R_{W^{\closed}}$ and $U \in \Open(W^{\closed})$, we obtain $\theta_j(x) \in U$. This shows that $U \cap K \subset \theta_j^{-1}(U \cap L_j)$. 
	\end{proof}

	\begin{defn}\label{def:iota_pi_Wj_K}
		Fix $j \in J$.
		\begin{enumerate}[label=\textnormal{(\arabic*)}]
			\item
			By \cref{lem:preimage_thetaj} \labelcref{lem:preimage_thetaj_1}, for every $U \in \Open(W^{\closed})$, we have
			\[
			U \cap K
			\subset
			\theta_j^{-1}(U \cap L_j)
			\subset
			\widetilde{\theta}_j^{-1}(U \cap L_j).
			\]
			Since $K \in \Open(W_j^{\open})$, \cref{def:iota_pi} \labelcref{def:iota_pi_1} gives a morphism
			\[
			\iota_K^{W_j^{\open}}
			\colon
			i_K^{W^{\closed}} r_{W^{\open}}^K
			\Rightarrow
			i_{L_j}^{W^{\closed}} p_{W_j^{\open}}^{L_j} r_{W^{\open}}^{W_j^{\open}}
			\]
			in $[\Cstar(W^{\open}), \Cstar(W^{\closed})]$. 
			
			\item
			By \cref{lem:preimage_thetaj} \labelcref{lem:preimage_thetaj_2}, for every $F \in \Closed(W^{\open})$, we have
			\[
			F \cap K
			\subset
			\theta_j^{-1}(F \cap L_j)
			\subset
			\widetilde{\theta}_j^{-1}(F \cap L_j).
			\]
			Since $K \in \Closed(W_j^{\closed})$, \cref{def:iota_pi} \labelcref{def:iota_pi_2} gives a morphism
			\[
			\pi_{W_j^{\closed}}^K
			\colon
			i_{L_j}^{W^{\open}} p_{W_j^{\closed}}^{L_j} r_{W^{\closed}}^{W_j^{\closed}}
			\Rightarrow
			i_K^{W^{\open}} r_{W^{\closed}}^K
			\]
			in $[\Cstar(W^{\closed}), \Cstar(W^{\open})]$. 
		\end{enumerate}
	\end{defn}
	
	The following definition uses the operations on natural transformations associated with direct sums introduced in \cref{subsection:natural_transformations_associated_with_direct_sums}, as well as mapping cones in functor categories developed in \cref{subsection:higher-dimensional_mapping_cones}. Following \cref{nota:tensor_product}, we regard $D$ and $D_j$ as functors, and $\rho_j$ as a natural transformation for each $j \in J$.
	
	\begin{defn}\label{def:reflection_functor}
		The \textit{reflection functors}
		\[
		\begin{tikzcd}[row sep=large]
			\Cstar(W^{\closed}) \ar[bend left=10]{r}{S_{\closed}^{\open}} &
			\ar[bend left=10]{l}{S_{\open}^{\closed}} \Cstar(W^{\open}) 
		\end{tikzcd}
		\]
		are defined as follows. 
		\begin{enumerate}[label=\textnormal{(\arabic*)}]
			\item \label{def:reflection_functor_1}
			Since $J$ is finite, we can define a morphism
			\[(\iota_K^{W_j^{\open}})_{j \in J} \colon i_K^{W^{\closed}} r_{W^{\open}}^K  \Rightarrow \bigoplus_{j \in J} i_{L_j}^{W^{\closed}} p_{W_j^{\open}}^{L_j} r_{W^{\open}}^{W_j^{\open}} \]
			in $[\Cstar(W^{\open}), \Cstar(W^{\closed})]$. 
			We define a functor 
			\[S_{\open}^{\closed} \colon \Cstar(W^{\open}) \to \Cstar(W^{\closed})\] 
			to be the mapping cone of this morphism.
			
			\item \label{def:reflection_functor_2}
			Since the $\ast$-homomorphisms $\rho_j \colon D_j \to D$ for $j \in J$ are mutually orthogonal, the corresponding morphisms $\rho_j \colon D_j \Rightarrow D$ in $[\Cstar(W^{\closed}), \Cstar(W^{\closed})]$ for $j \in J$ are also mutually orthogonal. Hence, the horizontal composites $\pi_{W_j^{\closed}}^K \rho_j \colon i_{L_j}^{W^{\open}} p_{W_j^{\closed}}^{L_j} r_{W^{\closed}}^{W_j^{\closed}} D_j \Rightarrow i_K^{W^{\open}} r_{W^{\closed}}^K D$ in $[\Cstar(W^{\closed}), \Cstar(W^{\open})]$ for $j \in J$ are mutually orthogonal. We can therefore define a morphism 
			\[\sum_{j \in J} \pi_{W_j^{\closed}}^K \rho_j \colon \bigoplus_{j \in J} i_{L_j}^{W^{\open}} p_{W_j^{\closed}}^{L_j} r_{W^{\closed}}^{W_j^{\closed}} D_j \Rightarrow i_K^{W^{\open}} r_{W^{\closed}}^K D \]
			in $[\Cstar(W^{\closed}), \Cstar(W^{\open})]$. 
			We define a functor 
			\[S_{\closed}^{\open} \colon \Cstar(W^{\closed}) \to \Cstar(W^{\open})\] 
			to be the mapping cone of this morphism.
		\end{enumerate}
	\end{defn}

	\begin{rem}\label{rem:reflection_functor_locally_closed_subset}
		For clarity, we record the evaluations of the reflection functors at some key locally closed subsets.
		\begin{enumerate}[label=\textnormal{(\arabic*)}]
			\item \label{rem:reflection_functor_locally_closed_subset_1}
			Let $B \in \Cstar(W^{\open})$.
			\begin{itemize}
				\item The C*-algebra $S_{\open}^{\closed}B(W^{\closed})$ is the mapping cone of the $\ast$-homomorphism 
				\[B(K) \to \bigoplus_{j \in J} B(W_j^{\open}), \quad b \mapsto \Bigl(B\bigl(\iota_K^{W_j^{\open}}\bigr)(b)\Bigr)_{j \in J}. \]
				\item \label{rem:reflection_functor_locally_closed_subset_2}
				For each $j \in J$, the C*-algebra $S_{\open}^{\closed}B(W_j^{\closed})$ is the mapping cone of the inclusion map $B(\iota_K^{W_j^{\open}}) \colon B(K) \to B(W_j^{\open})$. 
				\item For each $j \in J$, we have $S_{\open}^{\closed}B(L_j)=SB(W_j^{\open})$. 
				\item We have $S_{\open}^{\closed}B(K)=B(K)$. 
			\end{itemize}
			
			\item
			Let $A \in \Cstar(W^{\closed})$.
			\begin{itemize}
				\item The C*-algebra $S_{\closed}^{\open}A(W^{\open})$ is the mapping cone of the $\ast$-homomorphism
				\[\bigoplus_{j \in J} D_j A(W_j^{\closed}) \to DA(K), \quad (a_j)_{j \in J} \mapsto \sum_{j \in J} \rho_j A(\pi_{W_j^{\closed}}^K)(a_j). \]
				\item For each $j \in J$, the C*-algebra $S_{\closed}^{\open}A(W_j^{\open})$ is the mapping cone of the $\ast$-homomorphism
				$\rho_j A(\pi_{W_j^{\closed}}^K) \colon D_j A(W_j^{\closed}) \to DA(K)$. 
				\item For each $j \in J$, we have $S_{\closed}^{\open}A(L_j)=D_j A(W_j^{\closed})$. 
				\item We have $S_{\closed}^{\open}A(K)=SDA(K)$. 
			\end{itemize}
		\end{enumerate}
	\end{rem}

	\begin{prop}\label{prop:reflection_functor_exact_continuous_module}
		The reflection functors 
		\[
		\begin{tikzcd}
			\Cstar(W^{\closed}) \ar[bend left=10]{r}{S_{\closed}^{\open}} &
			\ar[bend left=10]{l}{S_{\open}^{\closed}} \Cstar(W^{\open})
		\end{tikzcd}
		\]
		are exact, continuous, $\Cstar$-module functors.  
	\end{prop}
	
	\begin{proof}
		We prove the assertion only for
		$S_{\closed}^{\open} \colon \Cstar(W^{\closed}) \to \Cstar(W^{\open})$,
		since the proof for $S_{\open}^{\closed}$ is analogous.
		We know that $\bigoplus_j i_{L_j}^{W^{\open}} p_{W_j^{\closed}}^{L_j} r_{W^{\closed}}^{W_j^{\closed}} D_j$ and $i_K^{W^{\open}} r_{W^{\closed}}^K D$
		are exact, continuous, $\Cstar$-module functors $\Cstar(W^{\closed}) \to \Cstar(W^{\open})$; see \cref{prop:induced_functor_exact_continuous_module,prop:restriction_functor_exact_continuous_module}. 
		Since $\pi_{W_j^{\closed}}^K \rho_j \colon i_{L_j}^{W^{\open}} p_{W_j^{\closed}}^{L_j} r_{W^{\closed}}^{W_j^{\closed}} D_j \Rightarrow i_K^{W^{\open}} r_{W^{\closed}}^K D$ is a $\Cstar$-module natural transformation for every $j \in J$, so is the natural transformation $\sum_j \pi_{W_j^{\closed}}^K \rho_j \colon \bigoplus_j i_{L_j}^{W^{\open}} p_{W_j^{\closed}}^{L_j} r_{W^{\closed}}^{W_j^{\closed}} D_j \Rightarrow i_K^{W^{\open}} r_{W^{\closed}}^K D$. 
		It follows that the functor $\Cstar(W^{\closed}) \to [\{0, 1\}, \Cstar(W^{\open})]$ corresponding to $\sum_j \pi_{W_j^{\closed}}^K \rho_j$ is an exact, continuous, $\Cstar$-module functor; see \cref{rem:exact_continuous_module_iff}. 
		By \cref{lem:preservation_exact_continuous_module} \labelcref{lem:preservation_exact_continuous_module_1} applied to the mapping cone functor 
		\[M \colon [\{0, 1\}, [\Cstar(W^{\closed}), \Cstar(W^{\open})]] \to [\Cstar(W^{\closed}), \Cstar(W^{\open})], \]
		and \cref{prop:mapping_cone_exact_continuous_module}, $S_{\closed}^{\open}$ is an exact, continuous, $\Cstar$-module functor. 
	\end{proof}

	\begin{rem}\label{rem:J_finite}
		Even if $J$ is infinite, we can still define Alexandrov spaces $W^{\closed}$ and $W^{\open}$ from a preordered set $K$, preordered sets $L_j$ for $j \in J$, and order-preserving maps $\theta_j \colon K \to L_j$ for $j \in J$ as in \cref{def:Wc_Wo}. Suppose we are given a C*-algebra $D$, C*-algebras $D_j$ for $j \in J$, and mutually orthogonal $\ast$-homomorphisms $\rho_j \colon D_j \to D$ for $j \in J$. We can still define a functor $S_{\closed}^{\open} \colon \Cstar(W^{\closed}) \to \Cstar(W^{\open})$ as in \cref{def:reflection_functor} \labelcref{def:reflection_functor_2}. Then $S_{\closed}^{\open} \colon \Cstar(W^{\closed}) \to \Cstar(W^{\open})$ is an exact, continuous, $\Cstar$-module functor. 
		However, the natural transformation
		\[(\iota_K^{W_j^{\open}})_{j \in J} \colon i_K^{W^{\closed}} r_{W^{\open}}^K  \Rightarrow \bigoplus_{j \in J} i_{L_j}^{W^{\closed}} p_{W_j^{\open}}^{L_j} r_{W^{\open}}^{W_j^{\open}} \]
		is not well-defined in general. Hence, an analogous definition of $S_{\open}^{\closed} \colon \Cstar(W^{\open}) \to \Cstar(W^{\closed})$ is not available. 
	\end{rem}

	\subsection{Compositions of reflection functors}\label{subsection:composition_reflection_functor}
	
	In this subsection, we first describe the composites $S_{\open}^{\closed}S_{\closed}^{\open}$ and $S_{\closed}^{\open}S_{\open}^{\closed}$ as $2$-dimensional mapping cones in \cref{prop:composition_reflection_functors}. For this purpose, the following functor is needed. 
	
	\begin{defn}
		For each $j \in J$, let 
		\[p_K^{L_j} \colon \Cstar(K) \to \Cstar(L_j)\]
		denote the functor induced by the continuous map $\theta_j \colon K \to L_j$ as in \cref{def:functor_induced_by_continuous_map}.
	\end{defn}
	
	The following two lemmas compute the relevant compositions of the functors appearing in the construction of the reflection functors.
	
	\begin{lem}\label{lem:relations_ri_vanish}
		Let $\alpha \in \{\closed, \open\}$. 
		\begin{enumerate}[label=\textnormal{(\arabic*)}]
			\item \label{lem:relations_ri_vanish_1}
			For each $j \in J$, we have $r_{W^{\alpha}}^K i_{L_j}^{W^{\alpha}}=0$. 
			
			\item \label{lem:relations_ri_vanish_2}
			For each $j, j' \in J$ with $j \neq j'$, we have $r_{W^{\alpha}}^{W_{j'}^{\alpha}} i_{L_j}^{W^{\alpha}}=0$. 
		\end{enumerate}
	\end{lem}

	\begin{proof}
		\
		\begin{enumerate}
			\item This follows from $K \cap L_j=\emptyset$ and \cref{lem:relations_extension_restriction_functors} \labelcref{lem:relations_extension_restriction_functors_1,lem:relations_extension_restriction_functors_5}. 
			
			\item Since $j \neq j'$, we have $L_j \cap W_{j'}=\emptyset$. 
			The assertion now follows from \cref{lem:relations_extension_restriction_functors} \labelcref{lem:relations_extension_restriction_functors_1,lem:relations_extension_restriction_functors_5}. 
			\qedhere
		\end{enumerate}
	\end{proof}
	
	\begin{lem}\label{lem:relations_pri}
		Let $\alpha \in \{\closed, \open\}$. 
		\begin{enumerate}[label=\textnormal{(\arabic*)}]
			\item \label{lem:relations_pri_1}
			We have $r_{W^{\alpha}}^K i_K^{W^{\alpha}}=\id_K$. 
			
			\item \label{lem:relations_pri_2}
			For each $j \in J$, we have $p_{W_j^{\alpha}}^{L_j} r_{W^{\alpha}}^{W_j^{\alpha}} i_K^{W^{\alpha}}=p_K^{L_j}$. 
			
			\item \label{lem:relations_pri_3}
			For each $j \in J$, we have $p_{W_j^{\alpha}}^{L_j} r_{W^{\alpha}}^{W_j^{\alpha}} i_{L_j}^{W^{\alpha}}=\id_{L_j}$. 
		\end{enumerate}
	\end{lem}
	
	\begin{proof}
		\
		\begin{enumerate}
			\item This follows from \cref{lem:relations_extension_restriction_functors} \labelcref{lem:relations_extension_restriction_functors_2,lem:relations_extension_restriction_functors_5}. 
			
			\item From \cref{lem:relations_extension_restriction_functors} \labelcref{lem:relations_extension_restriction_functors_2,lem:relations_extension_restriction_functors_5}, we have $r_{W^{\alpha}}^{W_j^{\alpha}} i_K^{W^{\alpha}}=i_K^{W_j^{\alpha}} r_K^K=i_K^{W_j^{\alpha}}$. 
			Since $\widetilde{\theta}_j|_K=\theta_j$, \cref{lem:relations_extension_restriction_functors} \labelcref{lem:relations_extension_restriction_functors_3} gives $p_{W_j^{\alpha}}^{L_j} i_K^{W_j^{\alpha}}=p_K^{L_j}$. 
			
			\item From \cref{lem:relations_extension_restriction_functors} \labelcref{lem:relations_extension_restriction_functors_2,lem:relations_extension_restriction_functors_5}, we have $r_{W^{\alpha}}^{W_j^{\alpha}} i_{L_j}^{W^{\alpha}}=i_{L_j}^{W_j^{\alpha}}r_{L_j}^{L_j}=i_{L_j}^{W_j^{\alpha}}$. 
			Since $\widetilde{\theta}_j|_{L_j}$ is the identity map on $L_j$, \cref{lem:relations_extension_restriction_functors} \labelcref{lem:relations_extension_restriction_functors_3} gives $p_{W_j^{\alpha}}^{L_j} i_{L_j}^{W_j^{\alpha}}=\id_{L_j}$. 
			\qedhere
		\end{enumerate}
	\end{proof}
	
	The following lemma computes the relevant horizontal compositions of the natural transformations appearing in the construction of the reflection functors.
	
	\begin{lem}\label{lem:composition_iota_pi_Wj_K}
		Let $\alpha \in \{\closed, \open\}$ and $j \in J$.
		\begin{enumerate}[label=\textnormal{(\arabic*)}]	
			\item \label{lem:composition_iota_pi_Wj_K_1}
			The horizontal composite 
			\[\iota_K^{W_j^{\open}} i_K^{W^{\open}} r_{W^{\alpha}}^K \colon i_K^{W^{\closed}} r_{W^{\open}}^K i_K^{W^{\open}} r_{W^{\alpha}}^K \Rightarrow i_{L_j}^{W^{\closed}} p_{W_j^{\open}}^{L_j} r_{W^{\open}}^{W_j^{\open}} i_K^{W^{\open}} r_{W^{\alpha}}^K \]
			is equal to $\id_K \colon i_K^{W^{\closed}} r_{W^{\alpha}}^K \Rightarrow i_{L_j}^{W^{\closed}} p_K^{L_j} r_{W^{\alpha}}^K$ in $[\Cstar(W^{\alpha}), \Cstar(W^{\closed})]$.
			
			\item \label{lem:composition_iota_pi_Wj_K_2}
			The horizontal composite 
			\[\pi_{W_j^{\closed}}^K i_K^{W^{\closed}} r_{W^{\alpha}}^K \colon i_{L_j}^{W^{\open}} p_{W_j^{\closed}}^{L_j} r_{W^{\closed}}^{W_j^{\closed}} i_K^{W^{\closed}} r_{W^{\alpha}}^K \Rightarrow i_K^{W^{\open}} r_{W^{\closed}}^K i_K^{W^{\closed}} r_{W^{\alpha}}^K \]
			is equal to $\id_K \colon i_{L_j}^{W^{\open}} p_K^{L_j} r_{W^{\alpha}}^K \Rightarrow i_K^{W^{\open}} r_{W^{\alpha}}^K$ in $[\Cstar(W^{\alpha}), \Cstar(W^{\open})]$.
			
			\item\label{lem:composition_iota_pi_Wj_K_3}
			The horizontal composite 
			\[i_{L_j}^{W^{\alpha}} p_{W_j^{\closed}}^{L_j} r_{W^{\closed}}^{W_j^{\closed}} \iota_K^{W_j^{\open}} \colon i_{L_j}^{W^{\alpha}} p_{W_j^{\closed}}^{L_j} r_{W^{\closed}}^{W_j^{\closed}} i_K^{W^{\closed}} r_{W^{\open}}^K \Rightarrow i_{L_j}^{W^{\alpha}} p_{W_j^{\closed}}^{L_j} r_{W^{\closed}}^{W_j^{\closed}} i_{L_j}^{W^{\closed}} p_{W_j^{\open}}^{L_j} r_{W^{\open}}^{W_j^{\open}}\]
			is equal to $\iota_K^{W_j^{\open}} \colon i_{L_j}^{W^{\alpha}} p_K^{L_j} r_{W^{\open}}^K \Rightarrow i_{L_j}^{W^{\alpha}} p_{W_j^{\open}}^{L_j} r_{W^{\open}}^{W_j^{\open}}$ in $[\Cstar(W^{\open}), \Cstar(W^{\alpha})]$.
			
			\item\label{lem:composition_iota_pi_Wj_K_4}
			The horizontal composite 
			\[i_{L_j}^{W^{\alpha}} p_{W_j^{\open}}^{L_j} r_{W^{\open}}^{W_j^{\open}} \pi_{W_j^{\closed}}^K \colon i_{L_j}^{W^{\alpha}} p_{W_j^{\open}}^{L_j} r_{W^{\open}}^{W_j^{\open}} i_{L_j}^{W^{\open}} p_{W_j^{\closed}}^{L_j} r_{W^{\closed}}^{W_j^{\closed}} \Rightarrow i_{L_j}^{W^{\alpha}} p_{W_j^{\open}}^{L_j} r_{W^{\open}}^{W_j^{\open}} i_K^{W^{\open}} r_{W^{\closed}}^K\]
			is equal to $\pi_{W_j^{\closed}}^K \colon i_{L_j}^{W^{\alpha}} p_{W_j^{\closed}}^{L_j} r_{W^{\closed}}^{W_j^{\closed}} \Rightarrow i_{L_j}^{W^{\alpha}} p_K^{L_j} r_{W^{\closed}}^K$ in $[\Cstar(W^{\closed}), \Cstar(W^{\alpha})]$.

		\end{enumerate}
	\end{lem}
	
	\begin{proof}
		This follows from \cref{lem:composition_theta_r_iota_pi,lem:relations_pri}.
	\end{proof}
	
	Recall from \cref{nota:2-dimensional_mapping_cone} the notation for 2-dimensional mapping cones. 
	
	\begin{prop}\label{prop:composition_reflection_functors}
		\
		\begin{enumerate}[label=\textnormal{(\arabic*)}]
			\item \label{prop:composition_reflection_functors_1}
			In $[\Cstar(W^{\closed}), \Cstar(W^{\closed})]$, we have
			\[
			S_{\open}^{\closed}S_{\closed}^{\open}=M^2 \! \left(
			\begin{tikzcd}[ampersand replacement=\&, column sep=large, every label/.append style={font=\small},
				cells={nodes={font=\small}}]
				0 \ar[Rightarrow]{r} \ar[Rightarrow]{d} \&
				i_K^{W^{\closed}} r_{W^{\closed}}^K D \ar[Rightarrow]{d}{(\id_K D)_j} \\
				\bigoplus_j i_{L_j}^{W^{\closed}} p_{W_j^{\closed}}^{L_j} r_{W^{\closed}}^{W_j^{\closed}} D_j  \ar[Rightarrow, swap]{r}{\bigoplus_j \pi_{W_j^{\closed}}^K \rho_j}
				\& \bigoplus_j i_{L_j}^{W^{\closed}} p_{K}^{L_j} r_{W^{\closed}}^{K} D
			\end{tikzcd}\right)\mathrlap{.}
			\]
			
			\item \label{prop:composition_reflection_functors_2}
			In $[\Cstar(W^{\open}), \Cstar(W^{\open})]$, we have
			\[
			S_{\closed}^{\open}S_{\open}^{\closed}=M^2 \! \left(
			\begin{tikzcd}[ampersand replacement=\&, column sep=large, every label/.append style={font=\small}, cells={nodes={font=\small}}]
				\bigoplus_j i_{L_j}^{W^{\open}} p_K^{L_j} r_{W^{\open}}^K D_j \ar[Rightarrow]{r}{\bigoplus_j \iota_K^{W_j^{\open}} D_j} \ar[Rightarrow, swap]{d}{\sum_j \id_K \rho_j}
				\& \bigoplus_j i_{L_j}^{W^{\open}} p_{W_j^{\open}}^{L_j} r_{W^{\open}}^{W_j^{\open}} D_j \ar[Rightarrow]{d} \\
				i_K^{W^{\open}} r_{W^{\open}}^K D \ar[Rightarrow]{r} \& 0
			\end{tikzcd}\right)\mathrlap{.}
			\]
		\end{enumerate}
	\end{prop}

	\begin{proof}\
		\begin{enumerate}[label=\textnormal{(\arabic*)}]
			\item 
			By \cref{lem:rearrangement_mapping_cones} (see also \cref{eg:composition_1+1=2}), $S_{\open}^{\closed}S_{\closed}^{\open}$ is the $2$-dimensional mapping cone of the following commutative square
			in $[\Cstar(W^{\closed}), \Cstar(W^{\closed})]$:
			\[
			\begin{tikzcd}[column sep=huge]
				i_K^{W^{\closed}} r_{W^{\open}}^K \bigoplus_{j' \in J} i_{L_{j'}}^{W^{\open}} p_{W_{j'}^{\closed}}^{L_{j'}} r_{W^{\closed}}^{W_{j'}^{\closed}} D_{j'} \ar[Rightarrow]{r}{i_K^{W^{\closed}} r_{W^{\open}}^K \sum_{j'} \pi_{W_{j'}^{\closed}}^K \rho_{j'}} \ar[Rightarrow]{d}{(\iota_K^{W_j^{\open}})_j \bigoplus_{j'} i_{L_{j'}}^{W^{\open}} p_{W_{j'}^{\closed}}^{L_{j'}} r_{W^{\closed}}^{W_{j'}^{\closed}} D_{j'}} & i_K^{W^{\closed}} r_{W^{\open}}^K i_K^{W^{\open}} r_{W^{\closed}}^K D \ar[Rightarrow]{d}{(\iota_K^{W_j^{\open}})_j i_K^{W^{\open}} r_{W^{\closed}}^K D} \\
				\bigoplus_{j \in J} i_{L_j}^{W^{\closed}} p_{W_j^{\open}}^{L_j} r_{W^{\open}}^{W_j^{\open}} \bigoplus_{j' \in J} i_{L_{j'}}^{W^{\open}} p_{W_{j'}^{\closed}}^{L_{j'}} r_{W^{\closed}}^{W_{j'}^{\closed}} D_{j'} \ar[Rightarrow, swap]{r}[xshift=-2.5em, yshift=-0.5em]{\bigoplus_j i_{L_j}^{W^{\closed}} p_{W_j^{\open}}^{L_j} r_{W^{\open}}^{W_j^{\open}} \sum_{j'} \pi_{W_{j'}^{\closed}}^K \rho_{j'}} & \bigoplus_{j \in J} i_{L_j}^{W^{\closed}} p_{W_j^{\open}}^{L_j} r_{W^{\open}}^{W_j^{\open}} i_K^{W^{\open}} r_{W^{\closed}}^K D\mathrlap{.}
			\end{tikzcd}
			\]
			By \cref{lem:relations_ri_vanish}, this square is identical to the following one:
			\[
			\begin{tikzcd}[column sep=huge]
				0 \ar[Rightarrow]{r} \ar[Rightarrow]{d} & i_K^{W^{\closed}} r_{W^{\open}}^K i_K^{W^{\open}} r_{W^{\closed}}^K D \ar[Rightarrow]{d}{(\iota_K^{W_j^{\open}} i_K^{W^{\open}} r_{W^{\closed}}^K D)_j} \\
				\bigoplus_{j \in J} i_{L_j}^{W^{\closed}} p_{W_j^{\open}}^{L_j} r_{W^{\open}}^{W_j^{\open}} i_{L_j}^{W^{\open}} p_{W_j^{\closed}}^{L_j} r_{W^{\closed}}^{W_j^{\closed}} D_j \ar[Rightarrow, swap]{r}[xshift=-2em, yshift=-0.5em]{\bigoplus_j i_{L_j}^{W^{\closed}} p_{W_j^{\open}}^{L_j} r_{W^{\open}}^{W_j^{\open}}  \pi_{W_j^{\closed}}^K \rho_j} & \bigoplus_{j \in J} i_{L_j}^{W^{\closed}} p_{W_j^{\open}}^{L_j} r_{W^{\open}}^{W_j^{\open}} i_K^{W^{\open}} r_{W^{\closed}}^K D\mathrlap{.}
			\end{tikzcd}
			\]
			By \cref{lem:composition_iota_pi_Wj_K} \labelcref{lem:composition_iota_pi_Wj_K_1,lem:composition_iota_pi_Wj_K_4}, this square is identical to the desired one. 
			
			\item 
			As in \labelcref{prop:composition_reflection_functors_1}, $S_{\closed}^{\open}S_{\open}^{\closed}$ is the $2$-dimensional mapping cone of the following commutative square in $[\Cstar(W^{\open}), \Cstar(W^{\open})]$:
			\[
			\begin{tikzcd}[column sep=large]
				\bigoplus_{j \in J} i_{L_j}^{W^{\open}} p_{W_j^{\closed}}^{L_j} r_{W^{\closed}}^{W_j^{\closed}} D_j i_K^{W^{\closed}} r_{W^{\open}}^K \ar[Rightarrow]{r}[xshift=-1em, yshift=0.5em]{\bigoplus_j i_{L_j}^{W^{\open}} p_{W_j^{\closed}}^{L_j} r_{W^{\closed}}^{W_j^{\closed}} D_j (\iota_K^{W_{j'}^{\open}})_{j'}} \ar[Rightarrow, swap]{d}{\sum_j \pi_{W_j^{\closed}}^K \rho_j i_K^{W^{\closed}} r_{W^{\open}}^K} & \bigoplus_{j \in J} i_{L_j}^{W^{\open}} p_{W_j^{\closed}}^{L_j} r_{W^{\closed}}^{W_j^{\closed}} D_j \bigoplus_{j' \in J} i_{L_{j'}}^{W^{\closed}} p_{W_{j'}^{\open}}^{L_{j'}} r_{W^{\open}}^{W_{j'}^{\open}} \ar[Rightarrow, swap]{d}{\sum_j \pi_{W_j^{\closed}}^K \rho_j \bigoplus_{j'} i_{L_{j'}}^{W^{\closed}} p_{W_{j'}^{\open}}^{L_{j'}} r_{W^{\open}}^{W_{j'}^{\open}}} \\
				i_K^{W^{\open}} r_{W^{\closed}}^K D i_K^{W^{\closed}} r_{W^{\open}}^K \ar[Rightarrow, swap]{r}{i_K^{W^{\open}} r_{W^{\closed}}^K D (\iota_K^{W_{j'}^{\open}})_{j'}} & i_K^{W^{\open}} r_{W^{\closed}}^K D \bigoplus_{j' \in J} i_{L_{j'}}^{W^{\closed}} p_{W_{j'}^{\open}}^{L_{j'}} r_{W^{\open}}^{W_{j'}^{\open}}\mathrlap{.}
			\end{tikzcd}
			\]
			By \cref{lem:relations_ri_vanish}, this square is identical to the following one:
			\[
			\begin{tikzcd}[column sep=large]
				\bigoplus_{j \in J} i_{L_j}^{W^{\open}} p_{W_j^{\closed}}^{L_j} r_{W^{\closed}}^{W_j^{\closed}} i_K^{W^{\closed}} r_{W^{\open}}^K D_j \ar[Rightarrow]{r}[xshift=-1em, yshift=0.4em]{\bigoplus_j i_{L_j}^{W^{\open}} p_{W_j^{\closed}}^{L_j} r_{W^{\closed}}^{W_j^{\closed}} \iota_K^{W_j^{\open}} D_j} \ar[Rightarrow,swap]{d}{\sum_j \pi_{W_j^{\closed}}^K i_K^{W^{\closed}} r_{W^{\open}}^K \rho_j} & 
				\bigoplus_{j \in J} i_{L_j}^{W^{\open}} p_{W_j^{\closed}}^{L_j} r_{W^{\closed}}^{W_j^{\closed}}  i_{L_j}^{W^{\closed}} p_{W_j^{\open}}^{L_j} r_{W^{\open}}^{W_j^{\open}} D_j \ar[Rightarrow]{d} \\
				i_K^{W^{\open}} r_{W^{\closed}}^K  i_K^{W^{\closed}} r_{W^{\open}}^K D \ar[Rightarrow]{r} & 0\mathrlap{.}
			\end{tikzcd}
			\]
			By \cref{lem:composition_iota_pi_Wj_K} \labelcref{lem:composition_iota_pi_Wj_K_2,lem:composition_iota_pi_Wj_K_3}, this square is identical to the desired one.
			This completes the proof. 
			\qedhere
		\end{enumerate}
	\end{proof}
	
	We next introduce two functors 
	\[T_{\closed} \colon \Cstar(W^{\closed}) \to \Cstar(W^{\closed}), \quad T_{\open} \colon \Cstar(W^{\open}) \to \Cstar(W^{\open}), \] 
	and compare $S_{\open}^{\closed}S_{\closed}^{\open}$, $T_{\closed}$, and $SD$ in $[\Cstar(W^{\closed}), \Cstar(W^{\closed})]$, and $S_{\closed}^{\open}S_{\open}^{\closed}$, $T_{\open}$, and $DS$ in $[\Cstar(W^{\open}), \Cstar(W^{\open})]$. In view of \cref{prop:composition_reflection_functors}, we define $T_{\closed}$ and $T_{\open}$ as $2$-dimensional mapping cones. Both depend on $D$, $(D_j)_{j \in J}$, and $(\rho_j)_{j \in J}$. Our goal is to show \cref{thm:key_diagram}.

	\begin{defn}\label{def:iota_pi_Wj_W}
		Fix $j \in J$.
		\begin{enumerate}[label=\textnormal{(\arabic*)}]
			\item
			By \cref{lem:preimage_thetaj} \labelcref{lem:preimage_thetaj_1}, 
			for every $U \in \Open(W^{\closed})$, we have
			\[
			U \cap W_j=(U \cap K) \cup (U \cap L_j) \subset \theta_j^{-1}(U \cap L_j) \cup (U \cap L_j)=\widetilde{\theta}_j^{-1}(U \cap L_j). 
			\]
			Since $W_j^{\closed} \in \Closed(W^{\closed})$, \cref{def:iota_pi} \labelcref{def:iota_pi_2} (see also \cref{rem:assumptions_O_C}
			\labelcref{rem:assumptions_O_C_2}) gives a morphism
			\[
			\pi_{W^{\closed}}^{W_j^{\closed}}
			\colon
			\id_{W^{\closed}}
			\Rightarrow
			i_{L_j}^{W^{\closed}}
			p_{W_j^{\closed}}^{L_j}
			r_{W^{\closed}}^{W_j^{\closed}}
			\]
			in $[\Cstar(W^{\closed}), \Cstar(W^{\closed})]$. 
			
			\item
			By \cref{lem:preimage_thetaj} \labelcref{lem:preimage_thetaj_2}, for every $F \in \Closed(W^{\open})$, we have
			\[
			F \cap W_j=(F \cap K) \cup (F \cap L_j) \subset \theta_j^{-1}(F \cap L_j) \cup (F \cap L_j)=\widetilde{\theta}_j^{-1}(F \cap L_j). 
			\]
			Since $W_j^{\open} \in \Open(W^{\open})$, \cref{def:iota_pi} \labelcref{def:iota_pi_1} (see also \cref{rem:assumptions_O_C}
			\labelcref{rem:assumptions_O_C_1}) gives a morphism
			\[
			\iota_{W_j^{\open}}^{W^{\open}}
			\colon
			i_{L_j}^{W^{\open}}
			p_{W_j^{\open}}^{L_j}
			r_{W^{\open}}^{W_j^{\open}}
			\Rightarrow
			\id_{W^{\open}}
			\]
			in $[\Cstar(W^{\open}), \Cstar(W^{\open})]$. 
		\end{enumerate}
	\end{defn}
	
	\begin{defn}\label{def:Tc_To}
		\
		\begin{enumerate}[label=\textnormal{(\arabic*)}]
			\item  
			We define a functor $T_{\closed} \colon \Cstar(W^{\closed}) \to \Cstar(W^{\closed})$ by
			\[
			T_{\closed}:=M^2 \! \left(
			\begin{tikzcd}[ampersand replacement=\&, column sep=large, every label/.append style={font=\small}, cells={nodes={font=\small}}]
				0 \ar[Rightarrow]{r} \ar[Rightarrow]{d} \& D \ar[Rightarrow]{d}{(\pi_{W^{\closed}}^{W_j^{\closed}} D)_j} \\
				\bigoplus_j i_{L_j}^{W^{\closed}} p_{W_j^{\closed}}^{L_j} r_{W^{\closed}}^{W_j^{\closed}} D_j \ar[Rightarrow, swap]{r}{\bigoplus_j \id_{W_j^{\closed}} \rho_j}
				\& \bigoplus_j i_{L_j}^{W^{\closed}} p_{W_j^{\closed}}^{L_j} r_{W^{\closed}}^{W_j^{\closed}} D
			\end{tikzcd}\right). 
			\] 
			
			\item 
			We define a functor $T_{\open} \colon \Cstar(W^{\open}) \to \Cstar(W^{\open})$ by
			\[
			T_{\open}:=M^2 \! \left(
			\begin{tikzcd}[ampersand replacement=\&,
				every label/.append style={font=\small}, cells={nodes={font=\small}}]
				\bigoplus_j i_{L_j}^{W^{\open}} p_{W_j^{\open}}^{L_j} r_{W^{\open}}^{W_j^{\open}} D_j \ar[Rightarrow]{r}{\id} \ar[Rightarrow, swap]{d}{\sum_j \iota_{W_j^{\open}}^{W^{\open}} \rho_j}
				\& \bigoplus_j i_{L_j}^{W^{\open}} p_{W_j^{\open}}^{L_j} r_{W^{\open}}^{W_j^{\open}} D_j \ar[Rightarrow]{d} \\
				D \ar[Rightarrow]{r} \& 0
			\end{tikzcd}\right). 
			\]
		\end{enumerate}
	\end{defn}
	
	\begin{prop}\label{prop:Tc_To_exact_continuous_module}
		The functors
		\[T_{\closed} \colon \Cstar(W^{\closed}) \to \Cstar(W^{\closed}), \quad T_{\open} \colon \Cstar(W^{\open}) \to \Cstar(W^{\open})\]
		are exact, continuous, $\Cstar$-module functors. 
	\end{prop}
	
	\begin{proof}
		The proof is analogous to that of \cref{prop:reflection_functor_exact_continuous_module}, with \cref{lem:preservation_exact_continuous_module} \labelcref{lem:preservation_exact_continuous_module_1} applied to the $2$-dimensional mapping cone functor
		\[M^2 \colon [\{0, 1\}^2, [\Cstar(W^{\alpha}), \Cstar(W^{\alpha})]] \to [\Cstar(W^{\alpha}), \Cstar(W^{\alpha})] \]
		for $\alpha \in \{\closed, \open\}$. 
	\end{proof}

	\begin{defn}\label{def:eta_epsilon}
		\ 
		\begin{enumerate}[label=\textnormal{(\arabic*)}]
			\item \label{lem:eta_epsilon_1}
			The functor $SD \colon \Cstar(W^{\closed}) \to \Cstar(W^{\closed})$ is the following $2$-dimensional mapping cone:
			\[
			SD=M^2 \! \left(
			\begin{tikzcd}[ampersand replacement=\&, every label/.append style={font=\small},
				cells={nodes={font=\small}}]
				0 \ar[Rightarrow]{r} \ar[Rightarrow]{d} \& \id_{W^{\closed}} \ar[Rightarrow]{d} \\
				0 \ar[Rightarrow]{r} \& 0
			\end{tikzcd}\right) D
			=M^2 \! \left(
			\begin{tikzcd}[ampersand replacement=\&, every label/.append style={font=\small},
				cells={nodes={font=\small}}]
				0 \ar[Rightarrow]{r} \ar[Rightarrow]{d} \& D \ar[Rightarrow]{d} \\
				0 \ar[Rightarrow]{r} \& 0
			\end{tikzcd}\right). 
			\]
			We define a natural transformation $\dot{\eta} \colon T_{\closed} \Rightarrow SD$ by
			\[\dot{\eta}:=M^2 
			\begin{pmatrix}
				0 & \id \\
				0 & 0
			\end{pmatrix}. \]
			Using the description of $S_{\open}^{\closed}S_{\closed}^{\open}$ in \cref{prop:composition_reflection_functors} \labelcref{prop:composition_reflection_functors_1}, we define a natural transformation $\ddot{\eta} \colon T_{\closed} \Rightarrow S_{\open}^{\closed}S_{\closed}^{\open}$ by
			\[\ddot{\eta}:=M^2
			\begin{pmatrix}
				0 & \pi_{W^{\closed}}^K D \\
				\id & \bigoplus_j \pi_{W_j^{\closed}}^K D
			\end{pmatrix}. \]
			
			\item \label{lem:eta_epsilon_2}
			The functor $D S \colon \Cstar(W^{\open}) \to \Cstar(W^{\open})$ is the following $2$-dimensional mapping cone:
			\[
			D S=D M^2 \! \left(
			\begin{tikzcd}[ampersand replacement=\&, every label/.append style={font=\small},
				cells={nodes={font=\small}}]
				0 \ar[Rightarrow]{r} \ar[Rightarrow]{d} \& 0 \ar[Rightarrow]{d} \\
				\id_{W^{\open}} \ar[Rightarrow]{r} \& 0
			\end{tikzcd}\right)
			=M^2 \! \left(
			\begin{tikzcd}[ampersand replacement=\&, every label/.append style={font=\small},
				cells={nodes={font=\small}}]
				0 \ar[Rightarrow]{r} \ar[Rightarrow]{d} \& 0 \ar[Rightarrow]{d} \\
				D \ar[Rightarrow]{r} \& 0
			\end{tikzcd}\right). 
			\]
			We define a natural transformation $\dot{\varepsilon} \colon D S \Rightarrow T_{\open}$ by
			\[\dot{\varepsilon}:=M^2
			\begin{pmatrix}
				0 & 0 \\
				\id & 0
			\end{pmatrix}. \]
			Using the description of $S_{\closed}^{\open}S_{\open}^{\closed}$ in \cref{prop:composition_reflection_functors} \labelcref{prop:composition_reflection_functors_2}, we define a natural transformation $\ddot{\varepsilon} \colon S_{\closed}^{\open}S_{\open}^{\closed} \Rightarrow T_{\open}$ by
			\[
			\ddot{\varepsilon}:=M^2
			\begin{pmatrix}
				\bigoplus_j \iota_K^{W_j^{\open}} D_j & \id \\
				\iota_K^{W^{\open}} D &  0
			\end{pmatrix}. 
			\]
		\end{enumerate}
	\end{defn}

	\begin{prop}\label{prop:eta_epsilon_module}
		The natural transformations
		\[\dot{\eta} \colon T_{\closed} \Rightarrow SD, \quad \ddot{\eta} \colon T_{\closed} \Rightarrow S_{\open}^{\closed}S_{\closed}^{\open}, \quad \dot{\varepsilon} \colon D S \Rightarrow T_{\open}, \quad \ddot{\varepsilon} \colon S_{\closed}^{\open}S_{\open}^{\closed} \Rightarrow T_{\open}\] 
		are $\Cstar$-module natural transformations.
	\end{prop}
	
	\begin{proof}
		We prove the assertion only for
		$\ddot{\varepsilon}$,
		since the proof for $\dot{\eta}$, $\ddot{\eta}$, and $\dot{\varepsilon}$ is analogous. Observe that the four natural transformations $0$, $\id$, $\iota_K^{W^{\open}} D$, and $\bigoplus_j \iota_K^{W_j^{\open}} D_j$ are $\Cstar$-module natural transformations. Hence, by applying \cref{lem:preservation_exact_continuous_module} \labelcref{lem:preservation_exact_continuous_module_2} to the $2$-dimensional mapping cone functor
		\[M^2 \colon [\{0, 1\}^2, [\Cstar(W^{\open}), \Cstar(W^{\open})]] \to [\Cstar(W^{\open}), \Cstar(W^{\open})], \]
		we see that $\ddot{\varepsilon}$ is a $\Cstar$-module natural transformation. 
	\end{proof}
	
	The short exact sequences in the following lemma will be used to prove \cref{thm:key_diagram} and appear in \cref{lem:key_ses_admissible}. 
	
	\begin{lem}\label{lem:key_ses}
		Fix $j \in J$. 
		\begin{enumerate}[label=\textnormal{(\arabic*)}]
			\item \label{lem:key_ses_1}
			We have a short exact sequence
			\[
			\begin{tikzcd}
				0 \ar[Rightarrow]{r} & i_{L_j}^{W^{\closed}} r_{W^{\closed}}^{L_j} \ar[Rightarrow]{r}{\iota_{L_j}^{W_j^{\closed}}} & i_{L_j}^{W^{\closed}} p_{W_j^{\closed}}^{L_j} r_{W^{\closed}}^{W_j^{\closed}} \ar[Rightarrow]{r}{\pi_{W_j^{\closed}}^K} & i_{L_j}^{W^{\closed}} p_K^{L_j} r_{W^{\closed}}^K \ar[Rightarrow]{r} & 0
			\end{tikzcd}
			\] 
			in $[\Cstar(W^{\closed}), \Cstar(W^{\closed})]$. 
			
			\item \label{lem:key_ses_2}
			We have a short exact sequence
			\[
			\begin{tikzcd}
				0 \ar[Rightarrow]{r} & i_{L_j}^{W^{\open}} p_K^{L_j} r_{W^{\open}}^K \ar[Rightarrow]{r}{\iota_K^{W_j^{\open}}} & i_{L_j}^{W^{\open}} p_{W_j^{\open}}^{L_j} r_{W^{\open}}^{W_j^{\open}} \ar[Rightarrow]{r}{\pi_{W_j^{\open}}^{L_j}} & i_{L_j}^{W^{\open}} r_{W^{\open}}^{L_j} \ar[Rightarrow]{r} & 0
			\end{tikzcd}
			\]
			in $[\Cstar(W^{\open}), \Cstar(W^{\open})]$. 
		\end{enumerate}
	\end{lem}
	
	\begin{proof}
		\
		\begin{enumerate}[label=\textnormal{(\arabic*)}]
			\item
			By \cref{lem:ses_ir}, we have the short exact sequence
			\[
			\begin{tikzcd}
				0 \ar[Rightarrow]{r} & i_{L_j}^{W^{\closed}} r_{W^{\closed}}^{L_j} \ar[Rightarrow]{r}{\iota_{L_j}^{W_j^{\closed}}} & i_{W_j^{\closed}}^{W^{\closed}} r_{W^{\closed}}^{W_j^{\closed}} \ar[Rightarrow]{r}{\pi_{W_j^{\closed}}^K} & i_K^{W^{\closed}} r_{W^{\closed}}^K \ar[Rightarrow]{r} & 0
			\end{tikzcd}
			\]
			in $[\Cstar(W^{\closed}), \Cstar(W^{\closed})]$. 
			By horizontally composing the exact functor
			\[
			i_{L_j}^{W^{\closed}}
			p_{W_j^{\closed}}^{L_j}
			r_{W^{\closed}}^{W_j^{\closed}}
			\colon
			\Cstar(W^{\closed}) \to \Cstar(W^{\closed})
			\]
			from the left with this short exact sequence, we obtain a short exact sequence.
			With the help of \cref{lem:relations_extension_restriction_functors}
			\labelcref{lem:relations_extension_restriction_functors_2,lem:relations_extension_restriction_functors_5}, \cref{lem:relations_pri} \labelcref{lem:relations_pri_2,lem:relations_pri_3}, and  \cref{lem:composition_iota_pi_Wj_K}
			\labelcref{lem:composition_iota_pi_Wj_K_4},
			the resulting short exact sequence is precisely the desired one. 
			
			\item 
			As in the proof of \labelcref{lem:key_ses_1}, the desired short exact sequence is obtained by horizontally composing the exact functor
			\[
			i_{L_j}^{W^{\open}} p_{W_j^{\open}}^{L_j} r_{W^{\open}}^{W_j^{\open}}
			\colon
			\Cstar(W^{\open}) \to \Cstar(W^{\open})
			\]
			from the left with the short exact sequence
			\[
			\begin{tikzcd}
				0 \ar[Rightarrow]{r}
				&
				i_K^{W^{\open}} r_{W^{\open}}^K
				\ar[Rightarrow]{r}{\iota_K^{W_j^{\open}}}
				&
				i_{W_j^{\open}}^{W^{\open}} r_{W^{\open}}^{W_j^{\open}}
				\ar[Rightarrow]{r}{\pi_{W_j^{\open}}^{L_j}}
				&
				i_{L_j}^{W^{\open}} r_{W^{\open}}^{L_j}
				\ar[Rightarrow]{r}
				&
				0
			\end{tikzcd}
			\]
			in $[\Cstar(W^{\open}), \Cstar(W^{\open})]$.
			\qedhere
		\end{enumerate}
	\end{proof}
	
	The main result of this subsection is the following theorem. 
	The middle row and column of each diagram will play a key role in \cref{subsection:application_to_Bott_functors_1,subsection:application_to_bivariant_K-theory_1}. 
	Following \cref{nota:tensor_product}, we regard $C_{\rho_j}$ as a functor for each $j \in J$. 
	
	\begin{thm}\label{thm:key_diagram}
		\
		\begin{enumerate}[label=\textnormal{(\arabic*)}]
			\item \label{thm:key_diagram_1}
			We have the following commutative diagram with exact rows and columns in $[\Cstar(W^{\closed}), \Cstar(W^{\closed})]$:
			{\everymath{\displaystyle}
				\[
				\begin{tikzcd}[column sep=small, every label/.append style={font=\small}, cells={nodes={font=\small}}
					]
					& 0 \ar[Rightarrow]{d} & 0 \ar[Rightarrow]{d} & 0 \ar[Rightarrow]{d} & \\
					0 \ar[Rightarrow]{r} & \bigoplus_{j \in J} i_{L_j}^{W^{\closed}} r_{W^{\closed}}^{L_j}S^2D \ar[Rightarrow]{r} \ar[Rightarrow]{d} & \bigoplus_{j \in J} i_{L_j}^{W^{\closed}} p_{W_j^{\closed}}^{L_j} r_{W^{\closed}}^{W_j^{\closed}} S C_{\rho_j} \ar[Rightarrow]{r} \ar[Rightarrow]{d} & \bigoplus_{j \in J} i_{L_j}^{W^{\closed}} p_{W_j^{\open}}^{L_j} r_{W^{\open}}^{W_j^{\open}} SS_{\closed}^{\open} \ar[Rightarrow]{r}  \ar[Rightarrow]{d} & 0 \\
					0 \ar[Rightarrow]{r} & \bigoplus_{j \in J} i_{L_j}^{W^{\closed}} r_{W^{\closed}}^{L_j}CSD \ar[Rightarrow]{r} \ar[Rightarrow]{d} & T_{\closed}  \ar[Rightarrow]{r}{\ddot{\eta}} \ar[Rightarrow]{d}{\dot{\eta}} &  S_{\open}^{\closed}S_{\closed}^{\open} \ar[Rightarrow]{r} \ar[Rightarrow]{d} & 0 \\
					0\ar[Rightarrow]{r} & \bigoplus_{j \in J} i_{L_j}^{W^{\closed}} r_{W^{\closed}}^{L_j}SD \ar[Rightarrow]{r} \ar[Rightarrow]{d} & SD \ar[Rightarrow]{r} \ar[Rightarrow]{d} & i_K^{W^{\closed}} r_{W^{\closed}}^K SD \ar[Rightarrow]{r} \ar[Rightarrow]{d} & 0 \\
					& 0 & 0 & 0\mathrlap{.} & 
				\end{tikzcd}
				\]
			}
			
			\item \label{thm:key_diagram_2}
			We have the following commutative diagram with exact rows and columns in $[\Cstar(W^{\open}), \Cstar(W^{\open})]$:
			{\everymath{\displaystyle}
				\[
				\begin{tikzcd}[column sep=small, every label/.append style={font=\small}, cells={nodes={font=\small}}
					]
					& 0 \ar[Rightarrow]{d} & 0 \ar[Rightarrow]{d} & 0 \ar[Rightarrow]{d} & \\
					0 \ar[Rightarrow]{r} & i_K^{W^{\open}} r_{W^{\open}}^K D S \ar[Rightarrow]{r} \ar[Rightarrow]{d} & D S \ar[Rightarrow]{r} \ar[Rightarrow]{d}{\dot{\varepsilon}} & \bigoplus_{j \in J}  i_{L_j}^{W^{\open}} r_{W^{\open}}^{L_j} D S \ar[Rightarrow]{r}  \ar[Rightarrow]{d} & 0 \\
					0 \ar[Rightarrow]{r} & S_{\closed}^{\open}S_{\open}^{\closed} \ar[Rightarrow]{r}{\ddot{\varepsilon}} \ar[Rightarrow]{d} & T_{\open} \ar[Rightarrow]{r} \ar[Rightarrow]{d} & \bigoplus_{j \in J}  i_{L_j}^{W^{\open}} r_{W^{\open}}^{L_j} C_{\rho_j}  \ar[Rightarrow]{r} \ar[Rightarrow]{d} & 0 \\
					0\ar[Rightarrow]{r} & \bigoplus_{j \in J} i_{L_j}^{W^{\open}} p_{W_j^{\closed}}^{L_j} r_{W^{\closed}}^{W_j^{\closed}} D_j S_{\open}^{\closed} \ar[Rightarrow]{r} \ar[Rightarrow]{d} & \bigoplus_{j \in J} i_{L_j}^{W^{\open}} p_{W_j^{\open}}^{L_j} r_{W^{\open}}^{W_j^{\open}} D_j C \ar[Rightarrow]{r} \ar[Rightarrow]{d} & \bigoplus_{j \in J}  i_{L_j}^{W^{\open}} r_{W^{\open}}^{L_j} D_j \ar[Rightarrow]{r} \ar[Rightarrow]{d} & 0 \\
					& 0 & 0 & 0\mathrlap{.} &
				\end{tikzcd}
				\]
			}
		\end{enumerate}
	\end{thm}
	
	\begin{proof}
		Throughout the proof, we use \cref{lem:rearrangement_mapping_cones} (see also \cref{eg:composition_1+1=2}). 
		\begin{enumerate}
			\item 
			From \cref{lem:composition_iota_pi_Wj_K} \labelcref{lem:composition_iota_pi_Wj_K_4}, we have
			\[\bigoplus_{j \in J} i_{L_j}^{W^{\closed}} p_{W_j^{\open}}^{L_j} r_{W^{\open}}^{W_j^{\open}} SS_{\closed}^{\open}=M^2 \! \left(
			\begin{tikzcd}[ampersand replacement=\&, column sep=large, every label/.append style={font=\small},
				cells={nodes={font=\small}}]
				0 \ar[Rightarrow]{r} \ar[Rightarrow]{d} \& 0 \ar[Rightarrow]{d} \\
				\bigoplus_j i_{L_j}^{W^{\closed}} p_{W_j^{\closed}}^{L_j} r_{W^{\closed}}^{W_j^{\closed}} D_j \ar[Rightarrow, swap]{r}{\bigoplus_j \pi_{W_j^{\closed}}^K \rho_j}
				\& \bigoplus_j i_{L_j}^{W^{\closed}} p_{K}^{L_j} r_{W^{\closed}}^{K} D
			\end{tikzcd}\right). \]
			We also have
			\begin{align*}
				\bigoplus_{j \in J} i_{L_j}^{W^{\closed}} r_{W^{\closed}}^{L_j} S^2D &=M^2 \! \left(
				\begin{tikzcd}[ampersand replacement=\&, every label/.append style={font=\small},
					cells={nodes={font=\small}}]
					0 \ar[Rightarrow]{r} \ar[Rightarrow]{d} \& 0 \ar[Rightarrow]{d} \\
					0 \ar[Rightarrow]{r} \& \bigoplus_j i_{L_j}^{W^{\closed}} r_{W^{\closed}}^{L_j} D
				\end{tikzcd}\right),
				\\
				\bigoplus_{j \in J} i_{L_j}^{W^{\closed}} p_{W_j^{\closed}}^{L_j} r_{W^{\closed}}^{W_j^{\closed}} S C_{\rho_j} &=M^2 \! \left(
				\begin{tikzcd}[ampersand replacement=\&, column sep=large, every label/.append style={font=\small},
					cells={nodes={font=\small}}]
					0 \ar[Rightarrow]{r} \ar[Rightarrow]{d} \& 0 \ar[Rightarrow]{d} \\
					\bigoplus_j i_{L_j}^{W^{\closed}} p_{W_j^{\closed}}^{L_j} r_{W^{\closed}}^{W_j^{\closed}} D_j \ar[Rightarrow, swap]{r}{\bigoplus_j \id_{W_j^{\closed}} \rho_j}
					\& \bigoplus_j i_{L_j}^{W^{\closed}} p_{W_j^{\closed}}^{L_j} r_{W^{\closed}}^{W_j^{\closed}} D
				\end{tikzcd}\right),
				\\
				\bigoplus_{j \in J} i_{L_j}^{W^{\closed}} r_{W^{\closed}}^{L_j}CSD &=M^2 \! \left(
				\begin{tikzcd}[ampersand replacement=\&, every label/.append style={font=\small},
					cells={nodes={font=\small}}]
					0 \ar[Rightarrow]{r} \ar[Rightarrow]{d} \&
					\bigoplus_j i_{L_j}^{W^{\closed}} r_{W^{\closed}}^{L_j} D \ar[Rightarrow]{d}{\id} \\
					0 \ar[Rightarrow]{r} \&
					\bigoplus_j i_{L_j}^{W^{\closed}} r_{W^{\closed}}^{L_j} D
				\end{tikzcd}\right),
				\\
				\bigoplus_{j \in J} i_{L_j}^{W^{\closed}} r_{W^{\closed}}^{L_j} SD &=M^2 \! \left(
				\begin{tikzcd}[ampersand replacement=\&, every label/.append style={font=\small},
					cells={nodes={font=\small}}]
					0 \ar[Rightarrow]{r} \ar[Rightarrow]{d} \&
					\bigoplus_j i_{L_j}^{W^{\closed}} r_{W^{\closed}}^{L_j} D \ar[Rightarrow]{d} \\
					0 \ar[Rightarrow]{r} \& 0
				\end{tikzcd}\right),
				\\
				i_K^{W^{\closed}} r_{W^{\closed}}^K SD &=M^2 \! \left(
				\begin{tikzcd}[ampersand replacement=\&, every label/.append style={font=\small},
					cells={nodes={font=\small}}]
					0 \ar[Rightarrow]{r} \ar[Rightarrow]{d} \&
					i_K^{W^{\closed}} r_{W^{\closed}}^K D \ar[Rightarrow]{d} \\
					0 \ar[Rightarrow]{r} \& 0
				\end{tikzcd}\right).
			\end{align*}
			The vertical morphisms in the diagram in \labelcref{thm:key_diagram_1} are canonically defined. 
			For example, we define a morphism $\bigoplus_{j} i_{L_j}^{W^{\closed}} p_{W_j^{\closed}}^{L_j} r_{W^{\closed}}^{W_j^{\closed}} SC_{\rho_j} \Rightarrow T_{\closed}$ in $[\Cstar(W^{\closed}), \Cstar(W^{\closed})]$ as
			\[M^2 
			\!\begin{pmatrix}
				0 & 0 \\
				\id & \id
			\end{pmatrix}. \]
			By \cref{lem:ses_ir,lem:disjoint_ir}, we have the short exact sequence
			\[
			\begin{tikzcd}[ampersand replacement=\&]
				0 \ar[Rightarrow]{r} \& \bigoplus_{j \in J} i_{L_j}^{W^{\closed}} r_{W^{\closed}}^{L_j} D \ar[Rightarrow]{r}{\sum_j \iota_{L_j}^{W^{\closed}} D} \&
				D \ar[Rightarrow]{r}{\pi_{W^{\closed}}^K D} \& i_K^{W^{\closed}} r_{W^{\closed}}^K D \ar[Rightarrow]{r} \& 0
			\end{tikzcd}
			\]
			in $[\Cstar(W^{\closed}), \Cstar(W^{\closed})]$. 
			By \cref{lem:key_ses} \labelcref{lem:key_ses_1}, we have the short exact sequence
			\[
			\begin{tikzcd}[ampersand replacement=\&]
				0 \ar[Rightarrow]{r} \&
				\bigoplus_{j \in J} i_{L_j}^{W^{\closed}} r_{W^{\closed}}^{L_j} D \ar[Rightarrow]{r}{\bigoplus_j \iota_{L_j}^{W_j^{\closed}} D} \&
				\bigoplus_{j \in J} i_{L_j}^{W^{\closed}} p_{W_j^{\closed}}^{L_j} r_{W^{\closed}}^{W_j^{\closed}} D \ar[Rightarrow]{r}{\bigoplus_j \pi_{W_j^{\closed}}^K D} \&
				\bigoplus_{j \in J} i_{L_j}^{W^{\closed}} p_K^{L_j} r_{W^{\closed}}^K D \ar[Rightarrow]{r} \& 0
			\end{tikzcd}
			\]
			in $[\Cstar(W^{\closed}), \Cstar(W^{\closed})]$. 
			The horizontal morphisms in the diagram in \labelcref{thm:key_diagram_1} are defined by using these short exact sequences. 
			For example, we define a morphism $\bigoplus_j i_{L_j}^{W^{\closed}} r_{W^{\closed}}^{L_j}CSD \Rightarrow T_{\closed}$ in $[\Cstar(W^{\closed}), \Cstar(W^{\closed})]$ as
			\[M^2
			\!\begin{pmatrix}
				0 & \sum_j \iota_{L_j}^{W^{\closed}} D \\
				0 & \bigoplus_j \iota_{L_j}^{W_j^{\closed}} D
			\end{pmatrix}. \]
			By the functoriality and exactness of 
			\[M^2 \colon [\{0, 1\}^2, [\Cstar(W^{\closed}), \Cstar(W^{\closed})]] \to [\Cstar(W^{\closed}), \Cstar(W^{\closed})], \] 
			the commutativity and exactness of the corresponding diagram in 
			\[[\{0, 1\}^2, [\Cstar(W^{\closed}), \Cstar(W^{\closed})]]\] imply the desired commutativity and exactness in $[\Cstar(W^{\closed}), \Cstar(W^{\closed})]$. 
			
			\item
			From \cref{lem:composition_iota_pi_Wj_K} \labelcref{lem:composition_iota_pi_Wj_K_3}, we have
			\[\bigoplus_{j \in J} i_{L_j}^{W^{\open}} p_{W_j^{\closed}}^{L_j} r_{W^{\closed}}^{W_j^{\closed}} D_j S_{\open}^{\closed}=M^2 \! \left(
			\begin{tikzcd}[ampersand replacement=\&, column sep=large, 
				every label/.append style={font=\small}, cells={nodes={font=\small}}]
				\bigoplus_j i_{L_j}^{W^{\open}} p_K^{L_j} r_{W^{\open}}^K D_j \ar[Rightarrow]{r}{\bigoplus_j \iota_K^{W_j^{\open}} D_j} \ar[Rightarrow]{d}
				\& \bigoplus_j i_{L_j}^{W^{\open}} p_{W_j^{\open}}^{L_j} r_{W^{\open}}^{W_j^{\open}} D_j \ar[Rightarrow]{d} \\
				0 \ar[Rightarrow]{r} \& 0
			\end{tikzcd}\right).\]
			We also have
			\begin{align*}
				i_K^{W^{\open}} r_{W^{\open}}^K D S &=M^2 \! \left(
				\begin{tikzcd}[ampersand replacement=\&,
					every label/.append style={font=\small}, cells={nodes={font=\small}}]
					0 \ar[Rightarrow]{r} \ar[Rightarrow]{d} \& 0 \ar[Rightarrow]{d} \\
					i_K^{W^{\open}} r_{W^{\open}}^K D \ar[Rightarrow]{r} \& 0
				\end{tikzcd}\right),
				\\
				\bigoplus_{j \in J} i_{L_j}^{W^{\open}} r_{W^{\open}}^{L_j} DS &=M^2 \! \left(
				\begin{tikzcd}[ampersand replacement=\&, 
					every label/.append style={font=\small}, cells={nodes={font=\small}}]
					0 \ar[Rightarrow]{r} \ar[Rightarrow]{d} \& 0 \ar[Rightarrow]{d} \\
					\bigoplus_j i_{L_j}^{W^{\open}} r_{W^{\open}}^{L_j} D \ar[Rightarrow]{r} \& 0
				\end{tikzcd}\right), 
				\\
				\bigoplus_{j \in J}  i_{L_j}^{W^{\open}} r_{W^{\open}}^{L_j} C_{\rho_j} &=M^2 \! \left(
				\begin{tikzcd}[ampersand replacement=\&,
					every label/.append style={font=\small}, cells={nodes={font=\small}}]
					\bigoplus_j i_{L_j}^{W^{\open}} r_{W^{\open}}^{L_j} D_j \ar[Rightarrow]{r} \ar[Rightarrow, swap]{d}{\bigoplus_j \id_{L_j} \rho_j}
					\& 0 \ar[Rightarrow]{d} \\
					\bigoplus_j i_{L_j}^{W^{\open}} r_{W^{\open}}^{L_j} D \ar[Rightarrow]{r} \& 0
				\end{tikzcd}\right),
				\\
				\bigoplus_{j \in J} i_{L_j}^{W^{\open}} p_{W_j^{\open}}^{L_j} r_{W^{\open}}^{W_j^{\open}} D_j C &=M^2 \! \left(
				\begin{tikzcd}[ampersand replacement=\&,
					every label/.append style={font=\small}, cells={nodes={font=\small}}]
					\bigoplus_j i_{L_j}^{W^{\open}} p_{W_j^{\open}}^{L_j} r_{W^{\open}}^{W_j^{\open}} D_j \ar[Rightarrow]{r}{\id} \ar[Rightarrow]{d}
					\& \bigoplus_j i_{L_j}^{W^{\open}} p_{W_j^{\open}}^{L_j} r_{W^{\open}}^{W_j^{\open}} D_j \ar[Rightarrow]{d} \\
					0 \ar[Rightarrow]{r} \& 0
				\end{tikzcd}\right), \\
				\bigoplus_{j \in J}  i_{L_j}^{W^{\open}} r_{W^{\open}}^{L_j} D_j &=M^2 \! \left(
				\begin{tikzcd}[ampersand replacement=\&,
					every label/.append style={font=\small}, cells={nodes={font=\small}}]
					\bigoplus_j i_{L_j}^{W^{\open}} r_{W^{\open}}^{L_j} D_j \ar[Rightarrow]{r} \ar[Rightarrow]{d}
					\& 0 \ar[Rightarrow]{d} \\
					0 \ar[Rightarrow]{r} \& 0
				\end{tikzcd}\right).
			\end{align*}
			The vertical morphisms in the diagram in \labelcref{thm:key_diagram_2} are canonically defined. 
			For example, we define a morphism $T_{\open} \Rightarrow \bigoplus_j i_{L_j}^{W^{\open}} p_{W_j^{\open}}^{L_j} r_{W^{\open}}^{W_j^{\open}} D_j C$ in $[\Cstar(W^{\open}), \Cstar(W^{\open})]$ as
			\[M^2 
			\!\begin{pmatrix}
				\id & \id \\
				0 & 0
			\end{pmatrix}. \]
			By \cref{lem:ses_ir,lem:disjoint_ir}, we have the short exact sequence
			\[
			\begin{tikzcd}[ampersand replacement=\&]
				0 \ar[Rightarrow]{r} \& i_K^{W^{\open}} r_{W^{\open}}^K D \ar[Rightarrow]{r}{\iota_K^{W^{\open}} D} \&
				D \ar[Rightarrow]{r}{(\pi_{W^{\open}}^{L_j} D)_j} \&  \bigoplus_{j \in J} i_{L_j}^{W^{\open}} r_{W^{\open}}^{L_j} D \ar[Rightarrow]{r} \& 0
			\end{tikzcd}
			\]
			in $[\Cstar(W^{\open}), \Cstar(W^{\open})]$. 	
			By \cref{lem:key_ses} \labelcref{lem:key_ses_2}, we have the short exact sequence
			\[
			\begin{tikzcd}[ampersand replacement=\&]
				0 \ar[Rightarrow]{r} \&
				\bigoplus_{j \in J} i_{L_j}^{W^{\open}} p_K^{L_j} r_{W^{\open}}^K D_j \ar[Rightarrow]{r}{\bigoplus_j \iota_K^{W_j^{\open}} D_j}  \&
				\bigoplus_{j \in J} i_{L_j}^{W^{\open}} p_{W_j^{\open}}^{L_j} r_{W^{\open}}^{W_j^{\open}} D_j \ar[Rightarrow]{r}{\bigoplus_j \pi_{W_j^{\open}}^{L_j} D_j} \&
				\bigoplus_{j \in J} i_{L_j}^{W^{\open}} r_{W^{\open}}^{L_j} D_j \ar[Rightarrow]{r} \& 0
			\end{tikzcd}
			\]
			in $[\Cstar(W^{\open}), \Cstar(W^{\open})]$. 
			The horizontal morphisms in the diagram in \labelcref{thm:key_diagram_2} are defined by using these short exact sequences. 
			For example, we define a morphism $T_{\open} \Rightarrow \bigoplus_j i_{L_j}^{W^{\open}} r_{W^{\open}}^{L_j} C_{\rho_j}$ in $[\Cstar(W^{\open}), \Cstar(W^{\open})]$ by
			\[M^2 
			\!\begin{pmatrix}
				\bigoplus_j \pi_{W_j^{\open}}^{L_j} D_j & 0 \\
				(\pi_{W^{\open}}^{L_j} D)_j &  0
			\end{pmatrix}. \]
			The remainder of the proof is the same as for \labelcref{thm:key_diagram_1}. This completes the proof. 
			\qedhere
		\end{enumerate}
	\end{proof}

	\subsection{Zigzag identities}\label{subsection:zigzag_identity}
	In this subsection, we prove the following theorem, which will be used to verify the zigzag identities in \cref{thm:equivalence_KKloc_Wc_Wo,thm:equivalence_E_Wc_Wo}. Recall from \cref{subsection:notation_for_Cstar-algebras_and_functor_categories,rem:full_subcategory_functor_category} the definition of homotopy between morphisms in a functor category.
	
	\begin{thm}\label{thm:zigzag_identity}
		The diagram
		\[
		\begin{tikzcd}
			T_{\closed} S_{\open}^{\closed} \ar[Rightarrow, swap]{d}{\dot{\eta}S_{\open}^{\closed}} \ar[Rightarrow]{rr}{\ddot{\eta}S_{\open}^{\closed}} &  &  S_{\open}^{\closed}S_{\closed}^{\open}S_{\open}^{\closed} \ar[Rightarrow]{d}{S_{\open}^{\closed}\ddot{\varepsilon}} \\
			SDS_{\open}^{\closed} \ar[equal]{r} & S_{\open}^{\closed}DS \ar[Rightarrow,swap]{r}{S_{\open}^{\closed}\dot{\varepsilon}} & S_{\open}^{\closed}T_{\open}
		\end{tikzcd}
		\]
		commutes up to homotopy in $[\Cstar(W^{\open}), \Cstar(W^{\closed})]$. 
	\end{thm}
	
	To this end, we prove the following lemma, which is based on an idea of Katsura. The author would like to thank him for this insight.
	
	\begin{lem}\label{lem:homotopy_2-dimensional_mapping_cone}
		Let $\frakA$ be a category. Suppose we are given a commutative diagram in $[\frakA, \Cstar]$ of the following form:
		\[
		\begin{tikzcd}
			A_1 \ar[Rightarrow]{r}{\varphi_1^2} \ar[Rightarrow, swap]{d}{\varphi_1^3} & A_2 \ar[Rightarrow]{r} \ar[Rightarrow, swap]{d}{\varphi_2^4} & 0 \ar[Rightarrow]{d} \\
			A_3 \ar[Rightarrow, swap]{r}{\varphi_3^4} & A_4 \ar[Rightarrow, swap]{r}{\varphi_4^5} & A_5\mathrlap{.}
		\end{tikzcd}
		\]
		Set $\varphi_1^4:=\varphi_2^4 \circ \varphi_1^2=\varphi_3^4 \circ \varphi_1^3 \colon A_1 \Rightarrow A_4$. 
		Then the following diagram commutes up to homotopy in $[\frakA, \Cstar]$:
		\[
		\begin{tikzcd}[ampersand replacement=\&, row sep=large, every label/.append style={font=\small}, cells={nodes={font=\small}}]
			M^2\mathopen{} \left(\Msquare{A_1}{A_2}{A_3}{A_4}
			{\varphi_1^2}{\varphi_1^3}{\varphi_2^4}{\varphi_3^4}\right)
			\ar[Rightarrow, swap]{d}{
				M^2\!\begin{pmatrix}\id & \id \\ 0 & 0\end{pmatrix}
			}
			\ar[Rightarrow]{rr}{
				M^2\!\begin{pmatrix}\id & 0 \\ \varphi_3^4 & \varphi_4^5\end{pmatrix}
			}
			\& \&
			M^2\mathopen{} \left(\Msquare{A_1}{0}{A_4}{A_5}
			{}{\varphi_1^4}{}{\varphi_4^5}\right)
			\\
			M^2\mathopen{} \left(\Msquare{A_1}{A_2}{0}{0}
			{\varphi_1^2}{}{}{}\right)
			\ar[equal]{r}
			\& C_{\varphi_1^2} \ar[equal]{r} \&
			M^2\mathopen{} \left(\Msquare{A_1}{0}{A_2}{0}
			{}{\varphi_1^2}{}{}\right)
			\ar[Rightarrow, swap]{u}{
				M^2\!\begin{pmatrix}\id & 0 \\ \varphi_2^4 & 0\end{pmatrix}
			}\mathrlap{.}
		\end{tikzcd}
		\]
	\end{lem}

	\begin{proof}
		Let $A$ and $B$ denote the $2$-dimensional mapping cones appearing in the upper-left and upper-right corners of the above diagram, respectively. Let $\varphi \colon A \Rightarrow B$ be the top horizontal morphism, and let $\psi \colon A \Rightarrow B$ be the morphism obtained by composing the two vertical morphisms and the two identity morphisms along the bottom row. 
		We shall prove that $\varphi$ and $\psi$ are homotopic in $[\frakA, \Cstar]$. 
		
		Throughout the proof, for any C*-algebra $E$, we extend functions in $\conti_0([0, 1), E)$ and $\conti_0([0, 1) \times [0, 1), E)$ by zero to $[0, \infty)$ and $[0, 1) \times [0, \infty)$, respectively, and use the same symbols for the resulting functions.

		Fix $Y \in \frakA$. To simplify notation, for any morphism in $[\frakA, \Cstar]$, we denote its component at $Y$ by the same symbol.
		The C*-algebra $A(Y)$ consists of all elements
		\[
		\begin{aligned}
			(a_1, a_2, a_3, a_4) \in {}&A_1(Y) \oplus \conti_0([0, 1), A_2(Y))
			\oplus \conti_0([0, 1), A_3(Y))\\
			&\oplus \conti_0([0, 1) \times [0, 1), A_4(Y))
		\end{aligned}
		\]
		such that $\varphi_1^2(a_1)=a_2(0)$, $\varphi_1^3(a_1)=a_3(0)$, $\varphi_2^4 \circ a_2(t)=a_4(0, t)$, and $\varphi_3^4 \circ a_3(t)=a_4(t, 0)$ for all $t \in [0, 1)$. 
		The C*-algebra $B(Y)$ consists of all elements
		\[(b_1, b_4, b_5) \in A_1(Y) \oplus \conti_0([0, 1), A_4(Y)) \oplus \conti_0([0, 1) \times [0, 1), A_5(Y)) \]
		satisfying $\varphi_1^4(b_1)=b_4(0)$, $b_5(0, t)=0$, and $\varphi_4^5 \circ b_4(t)=b_5(t, 0)$ for all $t \in [0, 1)$.
		The $\ast$-homomorphisms $\varphi, \psi \colon A(Y) \to B(Y)$ are given by
		\[\varphi(a):=(a_1, \varphi_3^4 \circ a_3, \varphi_4^5 \circ a_4), \quad \psi(a):=(a_1, \varphi_2^4 \circ a_2, 0)\]
		for $a=(a_1, a_2, a_3, a_4) \in A(Y)$. 
		
		We shall define two $\ast$-homomorphisms 
		\[\Phi, \Psi \colon A(Y) \to \conti([0, 1], B(Y))\] 
		as follows. The C*-algebra $\conti([0, 1], B(Y))$ consists of all elements
		\[
		\begin{aligned}
			(f_1, f_4, f_5) \in {}&\conti([0, 1], A_1(Y))
			\oplus \conti_0([0, 1] \times [0, 1), A_4(Y))\\
			&\oplus \conti_0([0, 1] \times [0, 1) \times [0, 1), A_5(Y))
		\end{aligned}
		\]
		satisfying $\varphi_1^4 \circ f_1(u)=f_4(u, 0)$, $f_5(u, 0, t)=0$, and $\varphi_4^5 \circ f_4(u, t)=f_5(u, t, 0)$ for all $u \in [0, 1]$ and $t \in [0, 1)$.
		We first define $\ast$-homomorphisms $\Phi$ and $\Psi$ from $A(Y)$ to 
		\[\conti([0, 1], A_1(Y)) \oplus \conti_0([0, 1] \times [0, 1), A_4(Y)) \oplus \conti_0([0, 1] \times [0, 1) \times [0, 1), A_5(Y))\]
		by 
		\[\Phi(a):=(f^a_1, f^a_4, f^a_5), \quad \Psi(a):=(f^a_1, g^a_4, g^a_5) \]
		for $a=(a_1, a_2, a_3, a_4) \in A(Y)$, where $f^a_1$, $f^a_4$, $f^a_5$, $g^a_4$, and $g^a_5$ are given by
		\[
		\begin{alignedat}{2}
			f^a_1(u):=a_1, \quad &f^a_4(u, t):=a_4(t, ut), & \quad &f^a_5(u, t, s):=\varphi_4^5 \circ a_4(t, s+ut), \\
			&g^a_4(u, t):=a_4(ut, t), & \quad &g^a_5(u, t, s):=\varphi_4^5 \circ a_4(ut, s+t)
		\end{alignedat},
		\]
		for $u \in [0, 1]$ and $t, s \in [0, 1)$. 
		
		We now show that $\Phi$ and $\Psi$ take values in
		$\conti([0,1],B(Y))$.
		Let $a=(a_1, a_2, a_3, a_4) \in A(Y)$, $u \in [0, 1]$, and $t \in [0, 1)$. From
		\begin{gather*}
			\varphi_1^4 \circ f^a_1(u)=\varphi_1^4(a_1)=\varphi_2^4 \circ  \varphi_1^2(a_1)=\varphi_2^4 \circ a_2(0)=a_4(0, 0)=f^a_4(u, 0), \\
			f^a_5(u, 0, t)=\varphi_4^5 \circ a_4(0, t)=\varphi_4^5 \circ \varphi_2^4 \circ a_2(t)=0, \\
			\varphi_4^5 \circ f^a_4(u, t)=\varphi_4^5 \circ a_4(t, ut)=f^a_5(u, t, 0), 
		\end{gather*}
		we see that $\Phi(a) \in \conti([0, 1], B(Y))$. Similarly, from
		\begin{gather*}
			\varphi_1^4 \circ f^a_1(u)=a_4(0, 0)=g^a_4(u, 0), \\
			g^a_5(u, 0, t)=\varphi_4^5 \circ a_4(0, t)=0, \\
			\varphi_4^5 \circ g^a_4(u, t)=\varphi_4^5 \circ a_4(ut, t)=g^a_5(u, t, 0), 
		\end{gather*}
		we see that $\Psi(a) \in \conti([0, 1], B(Y))$. 
		Thus, we obtain $\ast$-homomorphisms
		\[\Phi, \Psi \colon A(Y) \to \conti([0, 1], B(Y)).\] 
		
		We claim that
		\[\ev_0 \circ \Phi=\varphi, \quad \ev_0 \circ \Psi=\psi, \quad \ev_1 \circ \Phi=\ev_1 \circ \Psi. \]
		To prove the claim, let $a=(a_1, a_2, a_3, a_4) \in A(Y)$ and $t, s \in [0, 1)$. From
		\[f^a_1(0)=a_1, \quad f^a_4(0, t)=a_4(t, 0)=\varphi_3^4 \circ a_3(t), \quad f^a_5(0, t, s)=\varphi_4^5 \circ a_4(t, s), \]
		we get $\ev_0 \circ \Phi=\varphi$. From
		\begin{gather*}
			f^a_1(0)=a_1, \quad g^a_4(0, t)=a_4(0, t)=\varphi_2^4 \circ a_2(t), \\
			g^a_5(0, t, s)=\varphi_4^5 \circ a_4(0, s+t)=\varphi_4^5 \circ \varphi_2^4 \circ a_2(s+t)=0, 
		\end{gather*}
		we get $\ev_0 \circ \Psi=\psi$. Finally, from
		\[f^a_4(1, t)=a_4(t, t)=g^a_4(1, t), \quad f^a_5(1, t, s)=\varphi_4^5 \circ a_4(t, s+t)=g^a_5(1, t, s), \]
		we get $\ev_1 \circ \Phi=\ev_1 \circ \Psi$. This proves the claim. 
		Since $\varphi_4^5$ is natural in $Y \in \frakA$, so are $\Phi$ and $\Psi$. 
		Consequently, $\varphi$ and $\psi$ are homotopic in $[\frakA, \Cstar]$. 
	\end{proof}

	\begin{proof}[Proof of \textup{\cref{thm:zigzag_identity}}]
		Recall that $[\Cstar(W^{\open}), \Cstar(W^{\closed})]$ is a full subcategory of $[\frakA, \Cstar]$ with $\frakA:=\Cstar(W^{\open}) \times \LCcat(W^{\closed})$. It suffices to show that the diagram in the statement commutes up to homotopy in $[\frakA, \Cstar]$. 
		
		With the help of \cref{lem:composition_theta_r_iota_pi,lem:rearrangement_mapping_cones,lem:composition_iota_pi_Wj_K,eg:composition_1+2=3}, we can describe the objects $T_{\closed}S_{\open}^{\closed}$, $S_{\open}^{\closed}S_{\closed}^{\open}S_{\open}^{\closed}$, and $S_{\open}^{\closed}T_{\open}$ as the $3$-dimensional mapping cones of the following three commutative diagrams in $[\frakA, \Cstar]$, respectively:
		\[
		\begin{tikzcd}[
			column sep=small,
			ampersand replacement=\&,
			every label/.append style={font=\small},
			cells={nodes={font=\small}}
			]
			\& {} \&
			0
			\ar[Rightarrow]{dl}
			\ar[Rightarrow]{dd}
			\ar[Rightarrow]{rr}
			\& \& 
			0
			\ar[Rightarrow]{dl}
			\ar[Rightarrow]{dd}
			\\
			\& 
			i_K^{W^{\closed}} r_{W^{\open}}^K D
			\ar[Rightarrow, swap]{dd}{(\id_K D)_j}
			\& \& 
			\bigoplus_j i_{L_j}^{W^{\closed}} p_{W_j^{\open}}^{L_j} r_{W^{\open}}^{W_j^{\open}} D
			\ar[Rightarrow, from=ll, crossing over, near start, swap, "(\iota_K^{W_j^{\open}}D)_j"]
			\\
			\& \&
			\bigoplus_j i_{L_j}^{W^{\closed}} p_K^{L_j} r_{W^{\open}}^K D_j
			\ar[Rightarrow, near start]{dl}{\bigoplus_j \id_K \rho_j}
			\ar[Rightarrow, near start, swap]{rr}{\bigoplus_j \iota_K^{W_j^{\open}} D_j}
			\& \& 
			\bigoplus_j i_{L_j}^{W^{\closed}} p_{W_j^{\open}}^{L_j} r_{W^{\open}}^{W_j^{\open}} D_j
			\ar[Rightarrow]{dl}{\bigoplus_j \id_{W_j^{\open}} \rho_j}
			\\
			\&
			\bigoplus_j i_{L_j}^{W^{\closed}} p_K^{L_j} r_{W^{\open}}^K D
			\ar[Rightarrow, swap]{rr}{\bigoplus_j \iota_K^{W_j^{\open}} D}
			\& \& 
			\bigoplus_j i_{L_j}^{W^{\closed}} p_{W_j^{\open}}^{L_j} r_{W^{\open}}^{W_j^{\open}} D\mathrlap{,}
			\ar[Rightarrow, from=uu, crossing over, near start, "\id"]
		\end{tikzcd}
		\]
		\[\begin{tikzcd}[ampersand replacement=\&, every label/.append style={font=\small}, cells={nodes={font=\small}}]
			\& {} \&
			[-2em] 0
			\ar[Rightarrow]{dl}
			\ar[Rightarrow]{dd}
			\ar[Rightarrow]{rr}
			\& \& 
			0
			\ar[Rightarrow]{dl}
			\ar[Rightarrow]{dd}
			\\
			\& 
			i_K^{W^{\closed}} r_{W^{\open}}^K D
			\ar[Rightarrow, swap]{dd}{(\id_K D)_j}
			\& \& 
			0
			\ar[Rightarrow, from=ll, crossing over]
			\\
			\& \&
			\bigoplus_j i_{L_j}^{W^{\closed}} p_K^{L_j} r_{W^{\open}}^K D_j
			\ar[Rightarrow, near start]{dl}{\bigoplus_j \id_K \rho_j}
			\ar[Rightarrow, very near start, swap]{rr}{\bigoplus_j \iota_K^{W_j^{\open}} D_j}
			\& \& 
			\bigoplus_j i_{L_j}^{W^{\closed}} p_{W_j^{\open}}^{L_j} r_{W^{\open}}^{W_j^{\open}} D_j
			\ar[Rightarrow]{dl}
			\\
			\&
			\bigoplus_j i_{L_j}^{W^{\closed}} p_K^{L_j} r_{W^{\open}}^K D
			\ar[Rightarrow, swap]{rr}
			\& \& 
			0\mathrlap{,}
			\ar[Rightarrow, from=uu, crossing over]
		\end{tikzcd}
		\]
		and
		\[
		\begin{tikzcd}[
			ampersand replacement=\&,
			every label/.append style={font=\small},
			cells={nodes={font=\small}}
			]
			\& {} \&
			[-2em] 0
			\ar[Rightarrow]{dl}
			\ar[Rightarrow]{dd}
			\ar[Rightarrow]{rr}
			\& \& 
			0
			\ar[Rightarrow]{dl}
			\ar[Rightarrow]{dd}
			\\
			\& 
			i_K^{W^{\closed}} r_{W^{\open}}^K D
			\ar[Rightarrow, swap]{dd}{(\iota_K^{W_j^{\open}} D)_j}
			\& \& 
			0
			\ar[Rightarrow, from=ll, crossing over]
			\\
			\& \&
			\bigoplus_j i_{L_j}^{W^{\closed}} p_{W_j^{\open}}^{L_j} r_{W^{\open}}^{W_j^{\open}} D_j
			\ar[Rightarrow, near start]{dl}{\bigoplus_j \id_{W_j^{\open}} \rho_j}
			\ar[Rightarrow, near start, swap]{rr}{\id}
			\& \& 
			\bigoplus_j i_{L_j}^{W^{\closed}} p_{W_j^{\open}}^{L_j} r_{W^{\open}}^{W_j^{\open}} D_j
			\ar[Rightarrow]{dl}
			\\
			\&
			\bigoplus_j i_{L_j}^{W^{\closed}} p_{W_j^{\open}}^{L_j} r_{W^{\open}}^{W_j^{\open}} D
			\ar[Rightarrow, swap]{rr}
			\& \& 
			0\mathrlap{.}
			\ar[Rightarrow, from=uu, crossing over]
		\end{tikzcd}
		\]
		
		Using the description of $SD$ and $DS$ in \cref{def:eta_epsilon}, we can describe the objects $SDS_{\open}^{\closed}$ and $S_{\open}^{\closed}DS$ as the $3$-dimensional mapping cones of the following two commutative diagrams in $[\frakA, \Cstar]$, respectively:
		\[
		\begin{tikzcd}[sep=tiny, ampersand replacement=\&, every label/.append style={font=\small}, cells={nodes={font=\small}}]
			\& {} \&
			[2em] 0
			\ar[Rightarrow]{dl}
			\ar[Rightarrow]{dd}
			\ar[Rightarrow]{rr}
			\& \& 
			0
			\ar[Rightarrow]{dl}
			\ar[Rightarrow]{dd}
			\\
			\& 
			i_K^{W^{\closed}} r_{W^{\open}}^K D
			\ar[Rightarrow]{dd}
			\& \& 
			\bigoplus_j i_{L_j}^{W^{\closed}} p_{W_j^{\open}}^{L_j} r_{W^{\open}}^{W_j^{\open}} D
			\ar[Rightarrow, from=ll, crossing over, near start, swap, "(\iota_K^{W_j^{\open}}D)_j"]
			\\
			\& \&
			[-2em] 0
			\ar[Rightarrow]{dl}
			\ar[Rightarrow, near start]{rr}
			\& \& 
			0
			\ar[Rightarrow]{dl}
			\\
			\&
			0
			\ar[Rightarrow, swap]{rr}
			\& \& 
			0
			\ar[Rightarrow, from=uu, crossing over]
		\end{tikzcd}
		\]
		and
		\[
		\begin{tikzcd}[row sep=small, ampersand replacement=\&, every label/.append style={font=\small}, cells={nodes={font=\small}}]
			\& {} \&
			[-2em] 0
			\ar[Rightarrow]{dl}
			\ar[Rightarrow]{dd}
			\ar[Rightarrow]{rr}
			\& \& 
			[1em] 0
			\ar[Rightarrow]{dl}
			\ar[Rightarrow]{dd}
			\\
			\& 
			i_K^{W^{\closed}} r_{W^{\open}}^K D
			\ar[Rightarrow, swap]{dd}{(\iota_K^{W_j^{\open}}D)_j}
			\& \& 
			0
			\ar[Rightarrow, from=ll, crossing over]
			\\
			\& \&
			0
			\ar[Rightarrow]{dl}
			\ar[Rightarrow, near start]{rr}
			\& \& 
			0
			\ar[Rightarrow]{dl}
			\\
			\&
			\bigoplus_j i_{L_j}^{W^{\closed}} p_{W_j^{\open}}^{L_j} r_{W^{\open}}^{W_j^{\open}} D
			\ar[Rightarrow, swap]{rr}
			\& \& 
			0\mathrlap{.}
			\ar[Rightarrow, from=uu, crossing over]
		\end{tikzcd}
		\]
		
		Following \cref{nota:3-dimensional_mapping_cone}, we can describe the four morphisms $\dot{\eta}S_{\open}^{\closed}$, $S_{\open}^{\closed}\dot{\varepsilon}$, $\ddot{\eta}S_{\open}^{\closed}$, and $S_{\open}^{\closed}\ddot{\varepsilon}$ as follows:
		\[\dot{\eta}S_{\open}^{\closed}=
		M^3\!\begin{pmatrix}
			& 0 & & 0 \\
			\id & & \id & \\
			& 0 & & 0 \\
			0 & & 0 & \\
		\end{pmatrix}, \quad 
		S_{\open}^{\closed}\dot{\varepsilon}=
		M^3\!\begin{pmatrix}
			& 0 & & 0 \\
			\id & & 0 & \\
			& 0 & & 0 \\
			\id & & 0 & \\
		\end{pmatrix}, 
		\]
		\[
		\ddot{\eta}S_{\open}^{\closed}=
		M^3\!\begin{pmatrix}
			& 0 & & 0 \\
			\id & & 0 & \\
			& \id & & \id \\
			\id & & 0 & \\
		\end{pmatrix}, \quad 
		S_{\open}^{\closed}\ddot{\varepsilon}=
		M^3\!\begin{pmatrix}
			& 0 & & 0 \\
			\id & & 0 & \\
			& \bigoplus_j \iota_K^{W_j^{\open}} D_j & & \id \\
			\bigoplus_j \iota_K^{W_j^{\open}} D & & 0 & \\
		\end{pmatrix}.
		\]
		
		We now consider the following commutative diagram in $[\frakA, \Cstar]$:
		\[
		\resizebox{\textwidth}{!}{%
			\begin{tikzcd}[ampersand replacement=\&, column sep=tiny, every label/.append style={font=\small}, cells={nodes={font=\small}}]
				\& {} \&
				[-2em] 0
				\ar[Rightarrow]{dl}
				\ar[Rightarrow]{dd}
				\ar[Rightarrow]{rr}
				\& \& 
				[-2em] 0
				\ar[Rightarrow]{dl}
				\ar[Rightarrow]{dd}
				\ar[Rightarrow]{rr}
				\& \& 
				[1em] 0
				\ar[Rightarrow]{dl}
				\ar[Rightarrow]{dd}
				\\
				\& 
				i_K^{W^{\closed}} r_{W^{\open}}^K D
				\ar[Rightarrow, swap]{dd}{(\id_K D)_j}
				\& \& 
				\bigoplus_j i_{L_j}^{W^{\closed}} p_{W_j^{\open}}^{L_j} r_{W^{\open}}^{W_j^{\open}} D
				\ar[Rightarrow, from=ll, crossing over, near start, swap, "(\iota_K^{W_j^{\open}} D)_j"]
				\& \& 
				0
				\ar[Rightarrow, from=ll, crossing over]
				\\
				\& \&
				\bigoplus_j i_{L_j}^{W^{\closed}} p_K^{L_j} r_{W^{\open}}^K D_j
				\ar[Rightarrow, near start]{dl}{\bigoplus_j \id_K \rho_j}
				\ar[Rightarrow, near start, swap]{rr}{\bigoplus_j \iota_K^{W_j^{\open}} D_j}
				\& \& 
				\bigoplus_j i_{L_j}^{W^{\closed}} p_{W_j^{\open}}^{L_j} r_{W^{\open}}^{W_j^{\open}} D_j
				\ar[Rightarrow, near start]{dl}{\bigoplus_j \id_{W_j^{\open}} \rho_j}
				\ar[Rightarrow, near end, swap]{rr}{\id}
				\& \& 
				\bigoplus_j i_{L_j}^{W^{\closed}} p_{W_j^{\open}}^{L_j} r_{W^{\open}}^{W_j^{\open}} D_j
				\ar[Rightarrow]{dl}
				\\
				\&
				\bigoplus_j i_{L_j}^{W^{\closed}} p_K^{L_j} r_{W^{\open}}^K D
				\ar[Rightarrow, swap]{rr}{\bigoplus_j \iota_K^{W_j^{\open}} D}
				\& \& 
				\bigoplus_j i_{L_j}^{W^{\closed}} p_{W_j^{\open}}^{L_j} r_{W^{\open}}^{W_j^{\open}} D
				\ar[Rightarrow, from=uu, crossing over, near start, "\id"]
				\ar[Rightarrow]{rr}
				\& \& 
				0\mathrlap{.}
				\ar[Rightarrow, from=uu, crossing over]
				\&
			\end{tikzcd}
		}
		\]
		By taking the mapping cones of the downward-left morphisms in the preceding diagram, we obtain a commutative diagram in $[\frakA, \Cstar]$ of the form appearing in the hypotheses of \cref{lem:homotopy_2-dimensional_mapping_cone}. By the above descriptions of the five objects and four morphisms in terms of $3$-dimensional mapping cones, the corresponding diagram of $2$-dimensional mapping cones in \cref{lem:homotopy_2-dimensional_mapping_cone} is identical to the following one:
		\[
		\begin{tikzcd}[column sep=large]
			T_{\closed} S_{\open}^{\closed} \ar[Rightarrow, swap]{d}{\dot{\eta}S_{\open}^{\closed}} \ar[Rightarrow]{r}{S_{\open}^{\closed}\ddot{\varepsilon} \circ \ddot{\eta}S_{\open}^{\closed}} &  S_{\open}^{\closed}T_{\open} \\
			SDS_{\open}^{\closed} \ar[equal]{r} & S_{\open}^{\closed}DS\mathrlap{.} \ar[Rightarrow, swap]{u}{S_{\open}^{\closed}\dot{\varepsilon}}
		\end{tikzcd}
		\]
		\cref{lem:homotopy_2-dimensional_mapping_cone} shows that this diagram commutes up to homotopy in $[\frakA, \Cstar]$. This completes the proof. 
	\end{proof}

	It is worth emphasizing that \cref{thm:zigzag_identity} holds for arbitrary $D$, $(D_j)_{j \in J}$, and $(\rho_j)_{j \in J}$. Although $(\rho_j)_{j \in J}$ appears in the downward-left arrows in the $3$-dimensional mapping cones, the key idea of the proof is that, after taking the mapping cone of this part, the information carried by $(\rho_j)_{j \in J}$ disappears, which allows us to apply \cref{lem:homotopy_2-dimensional_mapping_cone} to the resulting $2$-dimensional mapping cones.
	
	On the other hand, the author expects that, when $J \neq \emptyset$, the diagram
	\[
	\begin{tikzcd}
		S_{\closed}^{\open}T_{\closed} \ar[Rightarrow]{r}{S_{\closed}^{\open}\dot{\eta}} \ar[Rightarrow, swap]{d}{S_{\closed}^{\open}\ddot{\eta}} & S_{\closed}^{\open}SD \ar[equal]{r} & DSS_{\closed}^{\open} \ar[Rightarrow]{d}{\dot{\varepsilon}S_{\closed}^{\open}} \\
		S_{\closed}^{\open}S_{\open}^{\closed}S_{\closed}^{\open} \ar[Rightarrow, swap]{rr}{\ddot{\varepsilon}S_{\closed}^{\open}} &  & T_{\open} S_{\closed}^{\open}
	\end{tikzcd}
	\]
	does not commute up to homotopy in $[\Cstar(W^{\closed}), \Cstar(W^{\open})]$, even when $D$, $(D_j)_{j \in J}$, and $(\rho_j)_{j \in J}$ are as in \cref{eg:reflection_suspension} or \cref{eg:reflection_matrix} below.
	We can describe the five objects and four morphisms in this diagram in terms of $3$-dimensional mapping cones, as in the proof of \cref{thm:zigzag_identity}. In fact, there is a dual version of \cref{lem:homotopy_2-dimensional_mapping_cone}, which we include in \cref{appendix:a_dual_version}. One might hope that \cref{lem:antihomotopy_2-dimensional_mapping_cone} would imply that the above diagram commutes up to homotopy. However, although $(\rho_j)_{j \in J}$ does not appear in the constructions of $\dot{\eta}$, $\ddot{\eta}$, $\dot{\varepsilon}$, and $\ddot{\varepsilon}$, it appears in both the horizontal and vertical arrows in the corresponding $3$-dimensional mapping cones, and therefore \cref{lem:antihomotopy_2-dimensional_mapping_cone} cannot be applied.
	
	Nevertheless, we can still establish \cref{thm:equivalence_KKloc_Wc_Wo,thm:equivalence_E_Wc_Wo} using the following elementary lemma, which is easy to verify; see also the proof of \cite[Proposition~4.4.5]{Riehl_2016}.

	\begin{lem}\label{lem:zigzag_identities_equivalence}
		Let $\frakA$ and $\frakB$ be categories. Suppose we are given functors
		\[
		\begin{tikzcd}
			\frakA \ar[bend left=20]{r}{F} &
			\ar[bend left=20]{l}{G} \frakB
		\end{tikzcd}
		\]
		and natural isomorphisms $\eta \colon \id_{\frakA} \Rightarrow GF$ and $\varepsilon \colon FG \Rightarrow \id_{\frakB}$. Then $\varepsilon F \circ F \eta=\id_F$ if and only if $G \varepsilon \circ \eta G=\id_G$. 
	\end{lem}

	\subsection{Application to homotopy-invariant, stable, half-exact functors}\label{subsection:application_to_Bott_functors_1}
	
	In this subsection, we apply \cref{thm:key_diagram} to homotopy-invariant, stable, half-exact functors. 
	\begin{lem}\label{lem:reflection_cone}
		Let $\frakA$ be an additive category. 
		\begin{enumerate}[label=\textnormal{(\arabic*)}]
			\item
			For a homotopy-invariant, stable, half-exact functor $F \colon \Cstar(W^{\closed}) \to \frakA$, the natural transformation $F\ddot{\eta} \colon FT_{\closed} \Rightarrow FS_{\open}^{\closed}S_{\closed}^{\open}$ is an isomorphism in $[\Cstar(W^{\closed}), \frakA]$. 
			
			\item
			For a homotopy-invariant, stable, half-exact functor $F \colon \Cstar(W^{\open}) \to \frakA$, the natural transformation $F\dot{\varepsilon} \colon FD S \Rightarrow FT_{\open}$ is an isomorphism in $[\Cstar(W^{\open}), \frakA]$. 
		\end{enumerate}
	\end{lem}
	
	\begin{proof}\
		\begin{enumerate}
			\item
			Since $\bigoplus_j i_{L_j}^{W^{\closed}} r_{W^{\closed}}^{L_j} CSD$ is contractible in $[\Cstar(W^{\closed}), \Cstar(W^{\closed})]$, \cref{lem:Bott} \labelcref{prop:Bott_1} and \cref{thm:key_diagram} \labelcref{thm:key_diagram_1} show that $F\ddot{\eta} \colon FT_{\closed} \Rightarrow FS_{\open}^{\closed}S_{\closed}^{\open}$ is an isomorphism in $[\Cstar(W^{\closed}), \frakA]$. 
			
			\item
			Since $\bigoplus_j i_{L_j}^{W^{\open}} p_{W_j^{\open}}^{L_j} r_{W^{\open}}^{W_j^{\open}} D_jC$ is contractible in $[\Cstar(W^{\open}), \Cstar(W^{\open})]$, \cref{lem:Bott} \labelcref{prop:Bott_2} and \cref{thm:key_diagram} \labelcref{thm:key_diagram_2} show that $F\dot{\varepsilon} \colon FDS \Rightarrow FT_{\open}$ is an isomorphism in $[\Cstar(W^{\open}), \frakA]$.
			\qedhere
		\end{enumerate}
	\end{proof}
	
	Note that, in the following theorem, $C_{\rho_j}$ is regarded as a functor for each $j \in J$ as in \cref{nota:tensor_product}. 
	
	\begin{thm}\label{thm:reflection_Bott}
		Let $\frakA$ be an additive category. 
		\begin{enumerate}[label=\textnormal{(\arabic*)}]
			\item \label{thm:reflection_Bott_1}
			For a homotopy-invariant, stable, half-exact functor $F \colon \Cstar(W^{\closed}) \to \frakA$ satisfying $FC_{\rho_j}=0$ in $[\Cstar(W^{\closed}), \frakA]$ for all $j \in J$, the natural transformation $F\dot{\eta} \colon FT_{\closed} \Rightarrow FSD$ is an isomorphism in $[\Cstar(W^{\closed}), \frakA]$. Therefore, we obtain
			\[FS_{\open}^{\closed} S_{\closed}^{\open} \simeq FSD\]
			in $[\Cstar(W^{\closed}), \frakA]$ via the isomorphisms $F\dot{\eta}$ and $F\ddot{\eta}$. 
			
			\item \label{thm:reflection_Bott_2}
			For a homotopy-invariant, stable, half-exact functor $F \colon \Cstar(W^{\open}) \to \frakA$ satisfying $FC_{\rho_j}=0$ in $[\Cstar(W^{\open}), \frakA]$ for all $j \in J$, the natural transformation $F\ddot{\varepsilon} \colon FS_{\closed}^{\open}S_{\open}^{\closed} \Rightarrow FT_{\open}$ is an isomorphism in $[\Cstar(W^{\open}), \frakA]$. Therefore, we obtain
			\[FS_{\closed}^{\open} S_{\open}^{\closed} \simeq FDS\]
			in $[\Cstar(W^{\open}), \frakA]$ via the isomorphisms $F\dot{\varepsilon}$ and $F\ddot{\varepsilon}$. 
		\end{enumerate}
	\end{thm}
	
	\begin{proof}
		\
		\begin{enumerate}
			\item
			For $k=0, 1$, we have
			\[FS^k\bigoplus_{j \in J} i_{L_j}^{W^{\closed}} p_{W_j^{\closed}}^{L_j} r_{W^{\closed}}^{W_j^{\closed}}SC_{\rho_j}=F\bigoplus_{j \in J}C_{\rho_j} i_{L_j}^{W^{\closed}} p_{W_j^{\closed}}^{L_j} r_{W^{\closed}}^{W_j^{\closed}}S^{k+1}=0. \]
			The final equality follows because $F$ commutes with finite direct sums and $FC_{\rho_j}=0$ for every $j \in J$. 
			Hence, by \cref{lem:Bott} \labelcref{prop:Bott_1} and \cref{thm:key_diagram} \labelcref{thm:key_diagram_1}, $F\dot{\eta} \colon FT_{\closed} \Rightarrow FSD$ is an isomorphism in $[\Cstar(W^{\closed}), \frakA]$. 
			
			\item
			As in \labelcref{thm:reflection_Bott_1}, for $k=0, 1$, we have
			\[FS^k\bigoplus_{j \in J} i_{L_j}^{W^{\open}} r_{W^{\open}}^{L_j}C_{\rho_j}=F \bigoplus_{j \in J} C_{\rho_j} i_{L_j}^{W^{\open}} r_{W^{\open}}^{L_j}S^k=0. \]
			Therefore, by \cref{lem:Bott} \labelcref{prop:Bott_2} and \cref{thm:key_diagram} \labelcref{thm:key_diagram_2}, $F\ddot{\varepsilon} \colon FS_{\closed}^{\open}S_{\open}^{\closed} \Rightarrow FT_{\open}$ is an isomorphism in $[\Cstar(W^{\open}), \frakA]$.
			\qedhere
		\end{enumerate}
	\end{proof}
	
	We next give two important examples of $D$, $(D_j)_{j \in J}$, and $(\rho_j)_{j \in J}$ to which \cref{thm:reflection_Bott} can be applied.
	
	The first example is given by suspension, which will appear in \cref{subsection:application_to_bivariant_K-theory_1}. For a homeomorphism $h \colon (0, 1) \to (0, 1)$ onto its image, let $h_* \colon S \to S$ denote the $\ast$-homomorphism defined by
	\[
	h_*(f)(t)=
	\begin{dcases}
		f(s) & \text{if $t=h(s)$ for some $s \in (0, 1)$},\\
		0 & \text{otherwise}
	\end{dcases}
	\]
	for $f \in S$ and $t \in (0, 1)$. The key point is that, since $h \colon (0, 1) \to (0, 1)$ is a homeomorphism onto its image, the mapping cone $C_{h_*}$ is isomorphic to $CS$ and is therefore contractible. 
	
	\begin{eg}\label{eg:reflection_suspension}
		Let $(h_j)_{j \in J}$ be a family of homeomorphisms $h_j \colon (0, 1) \to (0, 1)$ onto their images such that the images of $h_j$ for $j \in J$ are pairwise disjoint. Let $D:=S$. For each $j \in J$, let $D_j:=S$ and $\rho_j:=h_{j*} \colon D_j \to D$. Then the $\ast$-homomorphisms $\rho_j$ for $j \in J$ are mutually orthogonal. 
		Let $\alpha, \beta \in \{\closed, \open\}$ with $\alpha \neq \beta$, let $\frakA$ be an additive category, and let $F \colon \Cstar(W^{\alpha}) \to \frakA$ be a homotopy-invariant, stable, half-exact functor. Since $FC_{\rho_j}=0$ in $[\Cstar(W^{\alpha}), \frakA]$ for every $j \in J$, \cref{thm:reflection_Bott}, together with Bott periodicity in \cref{lem:Bott}, shows that
		\[FS_{\beta}^{\alpha} S_{\alpha}^{\beta} \simeq FSD \simeq F\]
		in $[\Cstar(W^{\alpha}), \frakA]$.
	\end{eg}
	
	The second example is given by matrix algebras. For a non-empty finite set $F$, let $\M_F$ denote the C*-algebra of all complex $F \times F$ matrices. For each $j \in F$, let $e_j \in \M_F$ be the matrix whose $(j, j)$-th entry is $1$ and all other entries are $0$, and let $\iota_{e_j} \colon \C \hookrightarrow \M_F$ be the $\ast$-homomorphism given by $\iota_{e_j}(\lambda):=\lambda e_j$ for $\lambda \in \C$.

	\begin{eg}\label{eg:reflection_matrix}
		Assume that $J \neq \emptyset$.
		Let $D:=\M_J$. For each $j \in J$, let $D_j:=\C$ and $\rho_j:=\iota_{e_j} \colon D_j \to D$ for $j \in J$. Then the $\ast$-homomorphisms $\rho_j$ for $j \in J$ are mutually orthogonal. 
		Let $\alpha, \beta \in \{\closed, \open\}$ with $\alpha \neq \beta$, let $\frakA$ be an additive category, and let $F \colon \Cstar(W^{\alpha}) \to \frakA$ be a homotopy-invariant, stable, half-exact functor. Since $FC_{\rho_j}=0$ in $[\Cstar(W^{\alpha}), \frakA]$ for every $j \in J$, \cref{thm:reflection_Bott}, together with the stability of $F$, shows that
		\[FS_{\beta}^{\alpha} S_{\alpha}^{\beta} \simeq FSD \simeq FS\]
		in $[\Cstar(W^{\alpha}), \frakA]$. 
	\end{eg}

	\begin{rem}\label{rem:restriction_reflection_functor}
		Suppose that $D$, $(D_j)_{j \in J}$, and $(\rho_j)_{j \in J}$ are as in either \cref{eg:reflection_suspension} or \cref{eg:reflection_matrix}.
		Since $J$ is finite and $D$ and all the $D_j$ are separable, the four functors
		\[
		\begin{tikzcd}
			\ar[loop left, looseness=4, "T_{\closed}"]\Cstar(W^{\closed})\ar[bend left=10]{r}{S_{\closed}^{\open}}&
			\ar[bend left=10]{l}{S_{\open}^{\closed}}\Cstar(W^{\open})\ar[loop right, looseness=4, "T_{\open}"]
		\end{tikzcd}
		\]
		restrict to functors
		\[
		\begin{tikzcd}
			\ar[loop left, looseness=4, "T_{\closed}"]\SCstar(W^{\closed})\ar[bend left=10]{r}{S_{\closed}^{\open}}&
			\ar[bend left=10]{l}{S_{\open}^{\closed}}\SCstar(W^{\open})\ar[loop right, looseness=4, "T_{\open}"]
		\end{tikzcd}\mathrlap{.}
		\]
		Each of these restricted functors is an exact $\SCstar$-module functor and commutes with countable inductive limits.
		The four natural transformations defined in \cref{def:eta_epsilon} restrict to the natural transformations between the corresponding restricted functors. Each of these restricted natural transformations is a $\SCstar$-module natural transformation.
		Consequently, all results in \cref{subsection:composition_reflection_functor,subsection:zigzag_identity,subsection:application_to_Bott_functors_1} involving these functors and natural transformations remain valid when we replace $\Cstar(\blank)$ by $\SCstar(\blank)$.
		Since $D$ and all the $D_j$ are also nuclear in both examples, the same statements hold for the corresponding categories of nuclear C*-algebras or separable nuclear C*-algebras.
	\end{rem}
	
	\subsection{Application to bivariant K-theory}\label{subsection:application_to_bivariant_K-theory_1}
	
	In this subsection, we show that the reflection functors yield equivalences between the corresponding categories of bivariant K-theory. 
	Suppose that $K$ and all the $L_j$ for $j \in J$ are finite preordered sets.
	Then $W^{\closed}$ and $W^{\open}$ are finite topological spaces. Suppose also that $D$, $(D_j)_{j \in J}$, and $(\rho_j)_{j \in J}$ are as in \cref{eg:reflection_suspension}. Then, for each $j \in J$, the mapping cone $C_{\rho_j}$ is contractible. 
	
	We first investigate equivalences between categories of ideal-related KK-theory.  
	By \cref{lem:induced_functor_KK_E}, the restrictions of the four
	functors to separable C*-algebras, as described in
	\cref{rem:restriction_reflection_functor}, induce triangulated functors
	\[
	\begin{tikzcd}
		\KKcat(W^{\closed})
		\ar[loop left, looseness=5, "T_{\closed}"]
		\ar[bend left=10]{r}{S_{\closed}^{\open}}
		&
		\KKcat(W^{\open})
		\ar[bend left=10]{l}{S_{\open}^{\closed}}
		\ar[loop right, looseness=5, "T_{\open}"]
	\end{tikzcd}
	\]
	that commute with countable coproducts.
	Moreover, by \cref{lem:induced_natural_transformation_KK_E}, the restrictions of the four natural transformations to separable C*-algebras, as described in
	\cref{rem:restriction_reflection_functor}, induce morphisms of triangulated functors
	\[
	\dot{\eta} \colon T_{\closed} \Rightarrow SD,
	\quad
	\ddot{\eta} \colon T_{\closed} \Rightarrow S_{\open}^{\closed}S_{\closed}^{\open},
	\quad
	\dot{\varepsilon} \colon DS \Rightarrow T_{\open},
	\quad
	\ddot{\varepsilon} \colon S_{\closed}^{\open}S_{\open}^{\closed} \Rightarrow T_{\open}. 
	\]
	Here, we use the same notation for the induced functors and natural transformations as for the original ones.
	
	\begin{lem}\label{lem:key_diagram_semi-split}
		\
		\begin{enumerate}[label=\textnormal{(\arabic*)}]
			\item \label{lem:key_diagram_semi-split_1}
			For a separable C*-algebra $A$ over $W^{\closed}$, the short exact sequence
			\[\begin{tikzcd}[ampersand replacement=\&]
				0 \ar{r} \& \bigoplus_{j \in J} i_{L_j}^{W^{\closed}} p_{W_j^{\closed}}^{L_j} r_{W^{\closed}}^{W_j^{\closed}} SC_{\rho_j}A \ar{r} \&
				T_{\closed}A \ar{r}{\dot{\eta}_A} \&
				SDA \ar{r} \& 0
			\end{tikzcd}\]
			in $\SCstar(W^{\closed})$ obtained by evaluating the middle column of the diagram in \textup{\cref{thm:key_diagram} \labelcref{thm:key_diagram_1}} at $A$ is semi-split. 
			
			\item \label{lem:key_diagram_semi-split_2}
			For a separable C*-algebra $B$ over $W^{\open}$, the short exact sequence 
			\[\begin{tikzcd}[ampersand replacement=\&]
				0 \ar{r} \& DSB \ar{r}{\dot{\varepsilon}_B} \&
				T_{\open}B \ar{r} \&
				\bigoplus_{j \in J} i_{L_j}^{W^{\open}} p_{W_j^{\open}}^{L_j} r_{W^{\open}}^{W_j^{\open}} D_jCB  \ar{r} \& 0
			\end{tikzcd}\]
			in $\SCstar(W^{\open})$ obtained by evaluating the middle column of the diagram in \textup{\cref{thm:key_diagram} \labelcref{thm:key_diagram_2}} at $B$ is semi-split. 
		\end{enumerate}
		
	\end{lem}

	\begin{proof}
		By construction, these short exact sequences are mapping cone short exact sequences and hence are semi-split; see \cref{lem:semi-split_2-dimensional_mapping_cone}. 
	\end{proof}
	
	\begin{rem}
		\
		\begin{enumerate}[label=\textnormal{(\arabic*)}]
			\item
			For a separable C*-algebra $A$ over $W^{\closed}$, evaluating the other two columns of the diagram in \cref{thm:key_diagram} \labelcref{thm:key_diagram_1} at $A$ also gives semi-split short exact sequences in $\SCstar(W^{\closed})$. 
			
			\item 
			For a separable C*-algebra $B$ over $W^{\open}$, evaluating the other two columns of the diagram in \textup{\cref{thm:key_diagram} \labelcref{thm:key_diagram_2}} at $B$ also gives semi-split short exact sequences in $\SCstar(W^{\open})$. 
		\end{enumerate}
	\end{rem}
	
	For separable C*-algebras $A$ over $W^{\closed}$ and $B$ over $W^{\open}$, the objects
	\[
	\bigoplus_{j \in J} i_{L_j}^{W^{\closed}} p_{W_j^{\closed}}^{L_j}
	r_{W^{\closed}}^{W_j^{\closed}} SC_{\rho_j}A, \quad
	\bigoplus_{j \in J} i_{L_j}^{W^{\open}} p_{W_j^{\open}}^{L_j}
	r_{W^{\open}}^{W_j^{\open}} D_jCB
	\]
	are contractible in $\SCstar(W^{\closed})$ and $\SCstar(W^{\open})$, respectively. Hence, by \cref{lem:contractible_isomorphism_KK_E,lem:key_diagram_semi-split}, the morphisms
	\[
	\dot{\eta}_A \colon T_{\closed}A \to SDA, \quad
	\dot{\varepsilon}_B \colon DSB \to T_{\open}B
	\]
	are isomorphisms in $\KKcat(W^{\closed})$ and $\KKcat(W^{\open})$, respectively. Therefore, the morphisms
	\[
	\dot{\eta} \colon T_{\closed} \Rightarrow SD, \quad
	\dot{\varepsilon} \colon DS \Rightarrow T_{\open}
	\]
	are isomorphisms in $[\KKcat(W^{\closed}), \KKcat(W^{\closed})]$ and $[\KKcat(W^{\open}), \KKcat(W^{\open})]$, respectively.

	However, the short exact sequences involving $\ddot{\eta}$ and $\ddot{\varepsilon}$ are not semi-split in general. We therefore need to restrict the functors to $\KKcat(W^{\closed})_{\loc}$ and $\KKcat(W^{\open})_{\loc}$. 
	By restriction, we obtain triangulated functors
	\[
	\begin{tikzcd}
		\ar[loop left, looseness=4, "T_{\closed}"] \KKcat(W^{\closed})_{\loc} \ar[bend left=10]{r}{S_{\closed}^{\open}} &
		\ar[bend left=10]{l}{S_{\open}^{\closed}} \KKcat(W^{\open})_{\loc} \ar[loop right, looseness=4, "T_{\open}"]
	\end{tikzcd}\mathrlap{,} \quad 
	\begin{tikzcd}
		\ar[loop left, looseness=4, "T_{\closed}"] \Boot(W^{\closed}) \ar[bend left=10]{r}{S_{\closed}^{\open}} &
		\ar[bend left=10]{l}{S_{\open}^{\closed}} \Boot(W^{\open}) \ar[loop right, looseness=4, "T_{\open}"]
	\end{tikzcd}
	\]
	that commute with countable coproducts; see \cref{prop:extension_functor_KKloc_bootstrap,prop:restriction_functor_KKloc_bootstrap}. 
	
	\begin{lem}\label{lem:key_ses_admissible}
		Fix $j \in J$. 
		\begin{enumerate}[label=\textnormal{(\arabic*)}]
			\item \label{lem:key_ses_admissible_1}
			For a separable C*-algebra $A$ over $W^{\closed}$ belonging to $\KKcat(W^{\closed})_{\loc}$, the short exact sequence
			\[
			\begin{tikzcd}
				0 \ar{r} & i_{L_j}^{W^{\closed}} r_{W^{\closed}}^{L_j} A \ar{r} & i_{L_j}^{W^{\closed}} p_{W_j^{\closed}}^{L_j} r_{W^{\closed}}^{W_j^{\closed}} A \ar{r} & i_{L_j}^{W^{\closed}} p_K^{L_j} r_{W^{\closed}}^K A \ar{r} & 0
			\end{tikzcd}
			\] 
			in $\SCstar(W^{\closed})$ obtained by evaluating the short exact sequence in \textup{\cref{lem:key_ses} \labelcref{lem:key_ses_1}} at $A$ is admissible. 
			
			\item \label{lem:key_ses_admissible_2}
			For a separable C*-algebra $B$ over $W^{\open}$ belonging to $\KKcat(W^{\open})_{\loc}$, the short exact sequence
			\[
			\begin{tikzcd}
				0 \ar{r} & i_{L_j}^{W^{\open}} p_K^{L_j} r_{W^{\open}}^K B \ar{r} & i_{L_j}^{W^{\open}} p_{W_j^{\open}}^{L_j} r_{W^{\open}}^{W_j^{\open}} B \ar{r} & i_{L_j}^{W^{\open}} r_{W^{\open}}^{L_j} B \ar{r} & 0
			\end{tikzcd}
			\]
			in $\SCstar(W^{\open})$ obtained by evaluating the short exact sequence in \textup{\cref{lem:key_ses} \labelcref{lem:key_ses_2}} at $B$ is admissible. 
		\end{enumerate}
	\end{lem}
	
	\begin{proof}
		Recall from the proof of \cref{lem:key_ses} \labelcref{lem:key_ses_1} that the short exact sequence in \labelcref{lem:key_ses_admissible_1} is obtained by applying the functor $i_{L_j}^{W^{\closed}} p_{W_j^{\closed}}^{L_j} r_{W^{\closed}}^{W_j^{\closed}} \colon \Cstar(W^{\closed}) \to \Cstar(W^{\closed})$ to the short exact sequence 
		\[
		\begin{tikzcd}
			0 \ar{r} & i_{L_j}^{W^{\closed}} r_{W^{\closed}}^{L_j} A \ar{r} & i_{W_j^{\closed}}^{W^{\closed}} r_{W^{\closed}}^{W_j^{\closed}} A \ar{r} & i_K^{W^{\closed}} r_{W^{\closed}}^K A \ar{r} & 0
		\end{tikzcd}
		\] 
		in $\SCstar(W^{\closed})$. The assumption $A \in \KKcat(W^{\closed})_{\loc}$ implies that this short exact sequence is admissible. Since $i_{L_j}^{W^{\closed}} p_{W_j^{\closed}}^{L_j} r_{W^{\closed}}^{W_j^{\closed}} \colon \SCstar(W^{\closed}) \to \SCstar(W^{\closed})$ is an exact $\SCstar$-module functor, \cref{lem:admissible_exact_module} shows that the short exact sequence in \labelcref{lem:key_ses_admissible_1} is also admissible. The same argument proves \labelcref{lem:key_ses_admissible_2}. 
	\end{proof}

	\begin{lem}\label{lem:key_diagram_admissible}
		\
		\begin{enumerate}[label=\textnormal{(\arabic*)}]
			\item \label{lem:key_diagram_admissible_1}
			For a separable C*-algebra $A$ over $W^{\closed}$ belonging to $\KKcat(W^{\closed})_{\loc}$, the short exact sequence
			\[\begin{tikzcd}[ampersand replacement=\&]
				0 \ar{r} \& \bigoplus_{j \in J} i_{L_j}^{W^{\closed}} r_{W^{\closed}}^{L_j} CSDA \ar{r} \&
				T_{\closed}A \ar{r}{\ddot{\eta}_A} \&
				S_{\open}^{\closed}S_{\closed}^{\open}A \ar{r} \& 0
			\end{tikzcd}\] 
			in $\SCstar(W^{\closed})$ obtained by evaluating the middle row of the diagram in \textup{\cref{thm:key_diagram} \labelcref{thm:key_diagram_1}} at $A$ is admissible. 
			
			\item \label{lem:key_diagram_admissible_2}
			For a separable C*-algebra $B$ over $W^{\open}$ belonging to $\KKcat(W^{\open})_{\loc}$, the short exact sequence
			\[\begin{tikzcd}[ampersand replacement=\&]
				0 \ar{r} \& S_{\closed}^{\open}S_{\open}^{\closed}B \ar{r}{\ddot{\varepsilon}_B} \&
				T_{\open}B \ar{r} \&
				\bigoplus_{j \in J} i_{L_j}^{W^{\open}} r_{W^{\open}}^{L_j} C_{\rho_j} B \ar{r} \& 0
			\end{tikzcd}\] 
			in $\SCstar(W^{\open})$ obtained by evaluating the middle row of the diagram in \textup{\cref{thm:key_diagram} \labelcref{thm:key_diagram_2}} at $B$ is admissible. 
		\end{enumerate}
	\end{lem}

	\begin{proof}\
		\begin{enumerate}
			\item 
			Consider the following four short exact sequences in  $\SCstar(W^{\closed})$:
			\begin{gather*}
				\begin{tikzcd}[ampersand replacement=\&]
					0 \ar{r} \& 0 \ar{r} \& 0 \ar{r} \& 0 \ar{r} \& 0\mathrlap{,}
				\end{tikzcd} \\
				\begin{tikzcd}[ampersand replacement=\&]
					0 \ar{r} \& 0 \ar{r} \&
					\bigoplus_{j \in J} i_{L_j}^{W^{\closed}} p_{W_j^{\closed}}^{L_j} r_{W^{\closed}}^{W_j^{\closed}} D_jA \ar{r} \&
					\bigoplus_{j \in J} i_{L_j}^{W^{\closed}} p_{W_j^{\closed}}^{L_j} r_{W^{\closed}}^{W_j^{\closed}} D_jA \ar{r} \& 0\mathrlap{,}
				\end{tikzcd} \\
				\begin{tikzcd}[column sep=large, ampersand replacement=\&]
					0 \ar{r} \& \bigoplus_{j \in J} i_{L_j}^{W^{\closed}} r_{W^{\closed}}^{L_j} DA \ar{r} \&
					DA \ar{r} \& i_K^{W^{\closed}} r_{W^{\closed}}^K DA \ar{r} \& 0\mathrlap{,}
				\end{tikzcd} \\
				\begin{tikzcd}[ampersand replacement=\&, column sep=small]
					0 \ar{r} \&
					\bigoplus_{j \in J} i_{L_j}^{W^{\closed}} r_{W^{\closed}}^{L_j} DA \ar{r} \&
					\bigoplus_{j \in J} i_{L_j}^{W^{\closed}} p_{W_j^{\closed}}^{L_j} r_{W^{\closed}}^{W_j^{\closed}} DA \ar{r} \&
					\bigoplus_{j \in J} i_{L_j}^{W^{\closed}} p_K^{L_j} r_{W^{\closed}}^K DA \ar{r} \& 0\mathrlap{.}
				\end{tikzcd}
			\end{gather*}
			The first and second short exact sequences are clearly split, and hence admissible. 
			The third short exact sequence is admissible because $A \in \KKcat(W^{\closed})_{\loc}$. The fourth short exact sequence is admissible by \cref{lem:key_ses_admissible} \labelcref{lem:key_ses_admissible_1}. 
			Recall that the short exact sequence in \labelcref{lem:key_diagram_admissible_1} is obtained by taking the $2$-dimensional mapping cones of these four admissible short exact sequences. Therefore, it is also admissible by \cref{lem:admissible_mapping_cones}. 
			
			\item
			Consider the following four short exact sequences in $\SCstar(W^{\open})$:
			\begin{gather*}
				\begin{tikzcd}[ampersand replacement=\&, column sep=small]
					0 \ar{r} \&
					\bigoplus_{j \in J} i_{L_j}^{W^{\open}} p_K^{L_j} r_{W^{\open}}^K D_jB \ar{r} \&
					\bigoplus_{j \in J} i_{L_j}^{W^{\open}} p_{W_j^{\open}}^{L_j} r_{W^{\open}}^{W_j^{\open}} D_jB \ar{r} \&
					\bigoplus_{j \in J} i_{L_j}^{W^{\open}} r_{W^{\open}}^{L_j} D_jB \ar{r} \& 0\mathrlap{,}
				\end{tikzcd} \\
				\begin{tikzcd}[column sep=large, ampersand replacement=\&]
					0 \ar{r} \& i_K^{W^{\open}} r_{W^{\open}}^K DB \ar{r} \&
					DB \ar{r} \& \bigoplus_{j \in J} i_{L_j}^{W^{\open}} r_{W^{\open}}^{L_j} DB \ar{r} \& 0\mathrlap{,}
				\end{tikzcd} \\
				\begin{tikzcd}[ampersand replacement=\&]
					0 \ar{r} \& \bigoplus_{j \in J} i_{L_j}^{W^{\open}} p_{W_j^{\open}}^{L_j} r_{W^{\open}}^{W_j^{\open}} D_j B \ar{r} \& \bigoplus_{j \in J} i_{L_j}^{W^{\open}} p_{W_j^{\open}}^{L_j} r_{W^{\open}}^{W_j^{\open}} D_j B \ar{r} \& 0 \ar{r} \& 0\mathrlap{,}
				\end{tikzcd}\\
				\begin{tikzcd}[ampersand replacement=\&]
					0 \ar{r} \& 0 \ar{r} \& 0 \ar{r} \& 0 \ar{r} \& 0\mathrlap{.}
				\end{tikzcd}
			\end{gather*}
			The first short exact sequence is admissible by \cref{lem:key_ses_admissible} \labelcref{lem:key_ses_admissible_2}. 
			The remainder of the proof is the same as for \labelcref{lem:key_diagram_admissible_1}. 
			\qedhere
		\end{enumerate}
	\end{proof}

	\begin{rem}
		\
		\begin{enumerate}[label=\textnormal{(\arabic*)}]
			\item 
			For a separable C*-algebra $A$ over $W^{\closed}$ belonging to $\KKcat(W^{\closed})_{\loc}$, evaluating the other two rows of the diagram in \textup{\cref{thm:key_diagram} \labelcref{thm:key_diagram_1}} at $A$ also gives admissible short exact sequences in $\SCstar(W^{\closed})$. 
			
			\item 
			For a separable C*-algebra $B$ over $W^{\open}$ belonging to $\KKcat(W^{\open})_{\loc}$, evaluating the other two rows of the diagram in \textup{\cref{thm:key_diagram} \labelcref{thm:key_diagram_2}} at $B$ also gives admissible short exact sequences in $\SCstar(W^{\open})$. 
		\end{enumerate}
	\end{rem}
	
	For separable C*-algebras $A$ over $W^{\closed}$ and $B$ over $W^{\open}$, the objects
	\[
	\bigoplus_{j \in J} i_{L_j}^{W^{\closed}} r_{W^{\closed}}^{L_j} CSDA, \quad
	\bigoplus_{j \in J} i_{L_j}^{W^{\open}} r_{W^{\open}}^{L_j} C_{\rho_j}B
	\]
	are contractible in $\SCstar(W^{\closed})$ and $\SCstar(W^{\open})$, respectively. Hence, by \cref{lem:contractible_isomorphism_KK_E,lem:key_diagram_admissible}, for $A \in \KKcat(W^{\closed})_{\loc}$ and $B \in \KKcat(W^{\open})_{\loc}$, the morphisms
	\[
	\ddot{\eta}_A \colon T_{\closed}A \to S_{\open}^{\closed}S_{\closed}^{\open}A, \quad
	\ddot{\varepsilon}_B \colon S_{\closed}^{\open}S_{\open}^{\closed}B \to T_{\open}B
	\]
	are isomorphisms in $\KKcat(W^{\closed})_{\loc}$ and $\KKcat(W^{\open})_{\loc}$, respectively. Consequently, the morphisms
	\[
	\ddot{\eta} \colon T_{\closed} \Rightarrow S_{\open}^{\closed}S_{\closed}^{\open}, \quad
	\ddot{\varepsilon} \colon S_{\closed}^{\open}S_{\open}^{\closed} \Rightarrow T_{\open}
	\]
	are isomorphisms in $[\KKcat(W^{\closed})_{\loc}, \KKcat(W^{\closed})_{\loc}]$ and $[\KKcat(W^{\open})_{\loc}, \KKcat(W^{\open})_{\loc}]$, respectively.
	
	Let $\sfb \colon SD \Rightarrow \id$ denote the Bott periodicity isomorphism in $[\KKcat, \KKcat]$, which induces an isomorphism $\sfb \colon SD \Rightarrow \id_X$ in $[\KKcat(X), \KKcat(X)]$ for every finite topological space $X$. 
	We define an isomorphism $\eta \colon \id_{W^{\closed}} \Rightarrow S_{\open}^{\closed}S_{\closed}^{\open}$ in $[\KKcat(W^{\closed})_{\loc}, \KKcat(W^{\closed})_{\loc}]$ to be the vertical composite
	\[\id_{W^{\closed}} \xRightarrow{\sfb^{-1}} SD \xRightarrow{\dot{\eta}^{-1}} T_{\closed} \xRightarrow{\ddot{\eta}} S_{\open}^{\closed}S_{\closed}^{\open}. \]
	We define an isomorphism $\varepsilon \colon S_{\closed}^{\open}S_{\open}^{\closed} \Rightarrow \id_{W^{\open}}$ in $[\KKcat(W^{\open})_{\loc}, \KKcat(W^{\open})_{\loc}]$ to be the vertical composite
	\[S_{\closed}^{\open}S_{\open}^{\closed} \xRightarrow{\ddot{\varepsilon}} T_{\open} \xRightarrow{\dot{\varepsilon}^{-1}} DS=SD \xRightarrow{\sfb} \id_{W^{\open}}. \]

	By \cref{thm:zigzag_identity}, we obtain
	\[S_{\open}^{\closed}\varepsilon \circ
	\eta S_{\open}^{\closed}=\id_{S_{\open}^{\closed}}\]
	in $[\KKcat(W^{\open})_{\loc}, \KKcat(W^{\closed})_{\loc}]$. Since $\eta \colon \id_{W^{\closed}} \Rightarrow S_{\open}^{\closed}S_{\closed}^{\open}$ and $\varepsilon \colon S_{\closed}^{\open}S_{\open}^{\closed} \Rightarrow \id_{W^{\open}}$ are isomorphisms, \cref{lem:zigzag_identities_equivalence} and the above equality imply that
	\[\varepsilon S_{\closed}^{\open} \circ S_{\closed}^{\open}\eta =\id_{S_{\closed}^{\open}} \]
	in $[\KKcat(W^{\closed})_{\loc}, \KKcat(W^{\open})_{\loc}]$. 
	
	The preceding discussion proves the following theorem.
	
	\begin{thm}\label{thm:equivalence_KKloc_Wc_Wo}
		The reflection functors 
		\[
		\begin{tikzcd}
			\KKcat(W^{\closed})_{\loc} \ar[bend left=10]{r}{S_{\closed}^{\open}} &
			\ar[bend left=10]{l}{S_{\open}^{\closed}} \KKcat(W^{\open})_{\loc}
		\end{tikzcd}
		\] 
		together with the natural transformations 
		\[\eta \colon \id_{W^{\closed}} \Rightarrow S_{\open}^{\closed}S_{\closed}^{\open}, \quad \varepsilon \colon S_{\closed}^{\open}S_{\open}^{\closed} \Rightarrow \id_{W^{\open}}\]
		form an adjoint equivalence of triangulated categories with countable coproducts, which restricts to an adjoint equivalence between $\Boot(W^{\closed})$ and $\Boot(W^{\open})$. 
	\end{thm}

	\begin{cor}\label{cor:KKloc_tree}
		Let $X$ and $Y$ be finite $T_0$-spaces whose Hasse diagrams are orientations of the same tree. Then there exist equivalences
		\[\KKcat(X)_{\loc} \simeq \KKcat(Y)_{\loc}, \quad \Boot(X) \simeq \Boot(Y)\]
		as triangulated categories with countable coproducts. 
	\end{cor}
	
	\begin{proof}
		This follows from \cref{lem:tree_orientation,thm:equivalence_KKloc_Wc_Wo}; see also \cref{rem:BGP-reflection}. 
	\end{proof}	
	
	We next turn to ideal-related E-theory. In ideal-related E-theory, the conclusion of \cref{lem:contractible_isomorphism_KK_E} holds for every short exact sequence, without an admissibility assumption. Thus, there is no need to
	verify admissibility as in
	\cref{lem:key_diagram_semi-split,lem:key_diagram_admissible}.
	The same argument as for ideal-related KK-theory yields the following theorem.

	\begin{thm}\label{thm:equivalence_E_Wc_Wo}
		The reflection functors 
		\[
		\begin{tikzcd}
			\Ecat(W^{\closed}) \ar[bend left=10]{r}{S_{\closed}^{\open}} &
			\ar[bend left=10]{l}{S_{\open}^{\closed}} \Ecat(W^{\open})
		\end{tikzcd}
		\] 
		together with the natural transformations 
		\[\eta \colon \id_{W^{\closed}} \Rightarrow S_{\open}^{\closed}S_{\closed}^{\open}, \quad \varepsilon \colon S_{\closed}^{\open}S_{\open}^{\closed} \Rightarrow \id_{W^{\open}}\] 
		defined in the same way as for ideal-related KK-theory form an adjoint equivalence of triangulated categories with countable coproducts, which restricts to an adjoint equivalence between $\Boot_{\E}(W^{\closed})$ and $\Boot_{\E}(W^{\open})$. 
	\end{thm}
	
	\begin{rem}\label{rem:equivalence_E_Wc_Wo_infinite}
		The categories $\Ecat(X)$ and $\Boot_{\E}(X)$ can be defined for a second countable space $X$; see \cite{DM_2012}. \cref{thm:equivalence_E_Wc_Wo} remains valid when $K$ and all the $L_j$ for $j \in J$ are countable partially ordered sets. In this case, $W^{\closed}$ and $W^{\open}$ are countable Alexandrov $T_0$-spaces. Note here that an Alexandrov $T_0$-space is second countable if and only if it is countable.
	\end{rem}
	
	\begin{cor}\label{cor:E_tree}
		Let $X$ and $Y$ be finite $T_0$-spaces whose Hasse diagrams are orientations of the same tree. 
		Then there exist equivalences
		\[\Ecat(X) \simeq \Ecat(Y), \quad \Boot_{\E}(X) \simeq \Boot_{\E}(Y)\]
		as triangulated categories with countable coproducts. 
	\end{cor}
	
	\begin{proof}
		This follows from \cref{lem:tree_orientation,thm:equivalence_E_Wc_Wo}; see also \cref{rem:BGP-reflection}.  
	\end{proof}

	\begin{rem}
		The Hasse diagrams of the connected non-accordion four-point spaces are shown in \cref{figure:Hasse_diagrams_non-accordion_four-point spaces}. In \cite{MN_2012,Bentmann_2010,BK_2011}, it is shown that filtrated K-theory, enlarged by one additional invariant, satisfies the UCT for C*-algebras over the spaces $X_1,\dots,X_5$. The space $X_6$ was studied by Bentmann \cite{Bentmann_2010}, where no finite refinement of filtrated K-theory satisfying the UCT was obtained, and it was suggested that such a finite refinement is unlikely to exist. The results of this subsection, together with \cref{eg:diamond}, establish equivalences between the corresponding categories for $X_1,\dots,X_5$. No such equivalence involving $X_6$ is established here. Thus, for these four-point spaces, the equivalences obtained here closely parallel the known behavior of finite refinements of filtrated K-theory satisfying the UCT. A similar phenomenon appears in Ladkani's theory of universal derived equivalences of partially ordered sets. The spaces $X_1, \dots, X_5$ are universally derived equivalent, while $X_6$ is distinguished from them; see \cite{Ladkani_2007}.
		
		\begin{figure}[htbp]
			\centering
			\renewcommand{\arraystretch}{1.4}
			\begin{tabular}{ccc}
				\begin{minipage}[c]{0.28\textwidth}
					\centering
					\begin{tikzcd}[row sep=2em, column sep=2.3em]
						\bullet & \bullet & \bullet \\
						& \ar[ul] \bullet \ar[u] \ar[ur] &
					\end{tikzcd}
					
					\vspace{0.3em}
					$X_1$
				\end{minipage}
				&
				\begin{minipage}[c]{0.28\textwidth}
					\centering
					\begin{tikzcd}[row sep=2em, column sep=2.3em]
						& \bullet & \\
						\bullet \ar[ur] & \bullet \ar[u] & \ar[ul] \bullet
					\end{tikzcd}
					
					\vspace{0.3em}
					$X_2$
				\end{minipage}
				&
				\begin{minipage}[c]{0.28\textwidth}
					\centering
					\begin{tikzcd}[row sep=1.5em, column sep=1.8em]
						\bullet & & \bullet \\
						& \ar[ul] \bullet \ar[ur] & \\
						& \bullet \ar[u] &
					\end{tikzcd}
					
					\vspace{0.3em}
					$X_3$
				\end{minipage}
				\\[1.2em]
				\begin{minipage}[c]{0.28\textwidth}
					\centering
					\begin{tikzcd}[row sep=1.5em, column sep=1.8em]
						& \bullet & \\
						& \bullet \ar[u] & \\
						\bullet \ar[ur] & & \ar[ul] \bullet
					\end{tikzcd}
					
					\vspace{0.3em}
					$X_4$
				\end{minipage}
				&
				\begin{minipage}[c]{0.28\textwidth}
					\centering
					\begin{tikzcd}[row sep=1.5em, column sep=1.8em]
						& \bullet & \\
						\bullet \ar[ur] & & \ar[ul] \bullet \\
						& \ar[ul] \bullet \ar[ur] &
					\end{tikzcd}
					
					\vspace{0.3em}
					$X_5$
				\end{minipage}
				&
				\begin{minipage}[c]{0.28\textwidth}
					\centering
					\begin{tikzcd}[row sep=3em, column sep=3.5em]
						\bullet & \bullet \\
						\bullet \ar[u] \ar[ur] &
						\ar[ul] \ar[u] \bullet
					\end{tikzcd}
					
					\vspace{0.3em}
					$X_6$
				\end{minipage}
			\end{tabular}
			\caption{Hasse diagrams of the connected non-accordion four-point spaces}
			\label{figure:Hasse_diagrams_non-accordion_four-point spaces}
		\end{figure}
	\end{rem}
	
	\begin{rem}
		Bentmann \cite[Proposition~4.5]{Bentmann_2014} proved that, for a finite $T_0$-space $X$, the category $\Boot_{\E}(X)$ embeds into a homotopy category of diagrams of module spectra. Extending Ladkani's results on universal derived equivalences of partially ordered sets in \cite{Ladkani_2007}, Dyckerhoff, Jasso, and Walde \cite{DJW_2021} constructed generalised reflection functors giving equivalences between categories of diagrams in stable $\infty$-categories. It might be possible to use these results to obtain equivalences between the E-theoretic bootstrap categories for a broader class of pairs of finite $T_0$-spaces than the one treated here. We do not investigate this approach in the present paper. The reflection functors introduced in this paper are defined directly between categories of C*-algebras over topological spaces and induce equivalences not only between the E-theoretic bootstrap categories but also between the entire categories of ideal-related E-theory.
	\end{rem}
	
	\section{Higher-dimensional reflection functors}\label{section:higher-dimensional_reflection_functors}
	
	The reflection functors introduced in the preceding section change whether a single subset is open or closed. In this section, we introduce higher-dimensional reflection functors that allow these choices to be changed at several subsets simultaneously. 
	
	Throughout \cref{subsection:setting_of_topological_spaces_Walpha,subsection:preliminaries_on_connected_subsets,subsection:definition_of_higher-dimensional_reflection_functors,subsection:rearrangements_of_higher-dimensional_reflection_functors,subsection:application_to_Bott_functors_2,subsection:application_to_bivariant_K-theory_2}, we fix 
	\begin{itemize}
		\item finite sets $I$ and $J$;
		\item subsets $J_i \subset J$ for $i \in I$ such that the undirected graph 
		\[G:=(I \amalg J, \{\{i, j\} \mid i \in I, j \in J_i\})\] 
		is a tree; 
		\item preordered sets $K_i=(K_i, R_{K_i})$ for $i \in I$; 
		\item preordered sets $L_j=(L_j, R_{L_j})$ for $j \in J$; 
		\item order-preserving maps $\theta_{i, j} \colon K_i \to L_j$ for $i \in I$ and $j \in J_i$; 
		\item C*-algebras $D_i$ for $i \in I$; 
		\item C*-algebras $D_{i, j}$ for $i \in I$ and $j \in J_i$; 
		\item $\ast$-homomorphisms $\rho_{i, j} \colon D_{i, j} \to D_i$ for $i \in I$ and $j \in J_i$ such that, for each $i \in I$, the $\ast$-homomorphisms $\rho_{i, j}$ for $j \in J_i$ are mutually orthogonal. 
	\end{itemize}
	We define sets
	\[V:=I \amalg J, \quad W:=\coprod_{i \in I} K_i \amalg \coprod_{j \in J} L_j. \]
	We define a map $\tau \colon W \to V$ by
	\[
	\tau(x):=
	\begin{dcases}
		i & \text{if $x \in K_i$},\\
		j & \text{if $x \in L_j$}.
	\end{dcases}
	\]
	
	The assumption that $G$ is a tree implies that $|J_i \cap J_{i'}|=0$ or $1$ for $i, i' \in I$ with $i \neq i'$. If $I=\{i_0\}$, then $J=J_{i_0}$, and the above setting is the same as that of \cref{section:reflection_functors}. 
	
	In \cref{subsection:setting_of_topological_spaces_Walpha}, for each $\alpha \in \{\closed, \open\}^I$, we first define a partial order $R_{V^{\alpha}}$ on the set $V$; see \cref{def:RValpha}. We next define a preorder $R_{W^{\alpha}}$ on the set $W$ using $(R_{K_i})_{i \in I}$, $(R_{L_j})_{j \in J}$, and $(\theta_{i, j})_{i \in I, j \in J_i}$; see \cref{def:RWalpha}. We denote the associated Alexandrov spaces by $V^{\alpha}$ and $W^{\alpha}$, respectively. The Hasse diagram of the finite $T_0$-space $V^{\alpha}$ is an orientation of the tree $G$, and the map $\tau \colon W \to V$ is a continuous map $W^{\alpha} \to V^{\alpha}$; see \cref{prop:tau_continuous}. This continuous map allows us to use the simpler space $V^{\alpha}$ to study $W^{\alpha}$, in particular to determine when the subsets $K_i$ are open or closed (see \cref{prop:Ki_closed_open}) and when $W^{\alpha}$ is $T_0$ (see \cref{prop:Walpha_T0}).
	Note that, if $I=\{i_0\}$, then for $\alpha \in \{\closed, \open\}^I$, the Alexandrov space $W^{\alpha}$ coincides with the Alexandrov space obtained from the data $K_{i_0}$, $(L_j)_{j \in J}$, $(\theta_{i_0, j})_{j \in J}$, and $\alpha(i_0) \in \{\closed, \open\}$ as in \textup{\cref{def:Wc_Wo}}. 
	
	In \cref{subsection:preliminaries_on_connected_subsets}, we reconstruct $W^{\alpha}$ relative to an arbitrary subset $F \subset I$ by contracting the tree $G$. More precisely, we contract each connected component $\bfj$ of the subgraph of $G$ induced by $(I \setminus F) \amalg J$ to a single vertex, obtaining a new tree $G^{F}$. To each such component $\bfj$, we associate the subspace $L_{\bfj}^{\alpha} \subset W^{\alpha}$ with underlying set $L_{\bfj}:=\tau^{-1}(\bfj)$, while the subspaces $K_i$ for $i \in F$ remain separate. These subspaces, together with the corresponding maps $\theta_{i,\bfj} \colon K_i \to L_{\bfj}$, and the restriction $\alpha|_{F}$, reconstruct $W^{\alpha}$; see \cref{lem:reconstruction_Walpha}.
	
	In \cref{subsection:definition_of_higher-dimensional_reflection_functors}, for $\alpha, \beta \in \{\closed, \open\}^I$, we apply the preceding reconstruction to 
	\[I^{\alpha, \beta}:=\{i \in I \mid \alpha(i) \neq \beta(i)\}. \] 
	For each connected component $\bfj$ of the subgraph induced by $(I \setminus I^{\alpha, \beta}) \amalg J$, the corresponding subspaces $L_{\bfj}^{\alpha}$ and $L_{\bfj}^{\beta}$ coincide; see \cref{prop:Ljalpha_Ljbeta_equal}. Consequently, the data used to reconstruct $W^{\alpha}$ and $W^{\beta}$ have the same subspaces $K_i$ and $L_{\bfj}^{\alpha}$ and the same maps $\theta_{i, \bfj} \colon K_i \to L_{\bfj}$, and differ only in the restrictions of $\alpha$ and $\beta$ to $I^{\alpha, \beta}$. Using these data together with the C*-algebraic data fixed above, we define a higher-dimensional reflection functor
	\[
	S_{\alpha}^{\beta} \colon \Cstar(W^{\alpha}) \to \Cstar(W^{\beta})
	\]
	as an $I^{\alpha, \beta}$-dimensional mapping cone. 
	When $|I^{\alpha, \beta}|=1$, this construction agrees with the corresponding reflection functor introduced in \cref{subsection:definition_of_reflection_functors}; see \cref{prop:higher-dimensional_reflection_functor_singleton}.

	The main result of this section is the rearrangement theorem for higher-dimensional reflection functors, \cref{thm:rearrangement_reflection_functors}, proved in \cref{subsection:rearrangements_of_higher-dimensional_reflection_functors}. We develop its consequences in \cref{subsection:application_to_Bott_functors_2,subsection:application_to_bivariant_K-theory_2,subsection:Coxeter_functors}. 
	In particular, \cref{cor:Salphabeta_composition_reflection_functors} shows that, for any ordering of the indices in $I^{\alpha, \beta}$, the composite of the corresponding reflection functors equals the higher-dimensional reflection functor $S_{\alpha}^{\beta}$. In \cref{cor:higher-dimensional_reflection_Bott}, we apply higher-dimensional reflection functors to homotopy-invariant, stable, half-exact functors. \cref{cor:equivalence_KKloc_Walpha_Wbeta,cor:equivalence_E_Walpha_Wbeta} show that higher-dimensional reflection functors induce equivalences between the corresponding categories of bivariant K-theory. Finally, \cref{cor:Coxeter_functor_independence} allows us to define an analogue of Coxeter functors. 
	
	It is worth emphasizing that \cref{thm:rearrangement_reflection_functors} holds for arbitrary $(D_i)_{i \in I}$, $(D_{i, j})_{i \in I, j \in J_i}$, and $(\rho_{i, j})_{i \in I, j \in J_i}$. 
	When we apply higher-dimensional reflection functors to ideal-related KK-theory and E-theory in \cref{subsection:application_to_bivariant_K-theory_2}, however, we use the specific choice of $(D_i)_{i \in I}$, $(D_{i, j})_{i \in I, j \in J_i}$, and $(\rho_{i, j})_{i \in I, j \in J_i}$ specified in \cref{eg:higher-dimensional_reflection_suspension} (see also \cref{eg:higher-dimensional_reflection_matrix}).

	\subsection{Setting of topological spaces}\label{subsection:setting_of_topological_spaces_Walpha}
	
	\begin{defn}\label{def:RValpha}
		For $\alpha \in \{\closed, \open\}^I$, we define a relation $R_{V^{\alpha}}$ on $V$ by
		\begin{align*}
			R_{V^{\alpha}}:=\ &\{(i, i) \mid i \in I\} \amalg \{(j, j) \mid j \in J\} \\
			&\amalg \{(i, j) \mid i \in I, j \in J_i, \alpha(i)=\closed\} \amalg \{(j, i) \mid i \in I, j \in J_i, \alpha(i)=\open\} \\
			&\amalg \{(i, i') \in I \times I \mid |J_i \cap J_{i'}|=1, \alpha(i)=\closed, \alpha(i')=\open\}. 
		\end{align*}
	\end{defn}
	
	Note that, in the last set appearing in the definition of $R_{V^{\alpha}}$, the condition $i \neq i'$ is automatic. Indeed, if $\alpha(i)=\closed$ and $\alpha(i')=\open$, then $i \neq i'$, since $\closed$ and $\open$ are distinct.
	
	\begin{lem}\label{lem:RValpha_partial_order}
		For each $\alpha \in \{\closed, \open\}^I$, the relation $R_{V^{\alpha}}$ is a partial order on $V$.  
	\end{lem}
	
	\begin{proof}
		It is straightforward to verify that $R_{V^{\alpha}}$ is reflexive and antisymmetric. We show that $R_{V^{\alpha}}$ is transitive. Take $(v_1, v_2), (v_2, v_3) \in R_{V^{\alpha}}$. If $v_1=v_2$ or $v_2=v_3$, then it is clear that $(v_1, v_3) \in R_{V^{\alpha}}$. Otherwise, 
		\[v_1 \in I, \quad v_2 \in J_{v_1}, \quad \alpha(v_1)=\closed, \]
		and 
		\[v_3 \in I, \quad v_2 \in J_{v_3}, \quad \alpha(v_3)=\open. \]
		Thus, $(v_1, v_3) \in R_{V^{\alpha}}$. This shows that $R_{V^{\alpha}}$ is transitive. 
	\end{proof}

	\begin{defn}
		In view of \cref{lem:RValpha_partial_order}, for each $\alpha \in \{\closed, \open\}^I$, we define $V^{\alpha}$ to be the finite $T_0$-space associated with the partially ordered set $(V, R_{V^{\alpha}})$.
	\end{defn}
	
	The following proposition follows immediately from the definitions. 
	
	\begin{prop}\label{prop:Valpha_G}
		For every $\alpha \in \{\closed, \open\}^I$, the Hasse diagram of $V^{\alpha}$ is an orientation of the tree $G$.
	\end{prop}

	\begin{eg}\label{eg:Valpha_I_singleton}
		Consider the case where $|I|=1$. Let $i_0$ be the unique element of $I$. Since $G$ is a tree, we have $J=J_{i_0}$. Let $j_1, j_2, \dots, j_n$ be the elements of $J$. Then the tree $G$ is depicted below. 
		
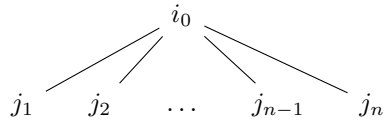
\[
		\begin{tikzcd}[column sep=small]
			& & i_0 & & \\
			j_1 \ar[-]{urr} & j_2 \ar[-]{ur} & \dots & j_{n-1} \ar[-]{ul} & j_n \ar[-]{ull}
		\end{tikzcd}
		\]
		For $\alpha \in \{\closed, \open\}^I$, we have
		\[
		R_{V^{\alpha}}=
		\begin{dcases}
			\{(i_0, i_0)\}\amalg \{(j, j) \mid j \in J\}\amalg \{(i_0, j)\mid j \in J\}
			& \text{if $\alpha(i_0)=\closed$},\\
			\{(i_0, i_0)\}\amalg \{(j, j) \mid j \in J\}\amalg \{(j, i_0)\mid j \in J\}
			& \text{if $\alpha(i_0)=\open$}.
		\end{dcases}
		\]
		The Hasse diagrams of $V^{\alpha}$ for $\alpha \in \{\closed, \open\}^I$ are depicted below. 
		\[
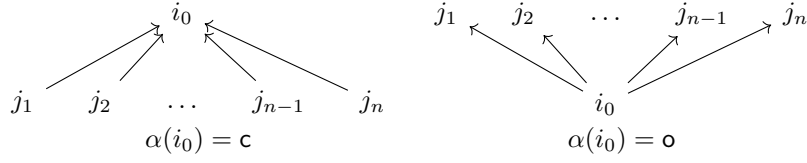
		\begin{array}{cc}
			\begin{tikzcd}[column sep=small]
				& & i_0 & & \\
				j_1 \ar{urr} & j_2 \ar{ur} & \dots & j_{n-1} \ar{ul} & j_n \ar{ull}
			\end{tikzcd}
			&
			\begin{tikzcd}[column sep=small]
				j_1 & j_2 & \dots & j_{n-1} & j_n \\
				& & i_0 \ar{ull} \ar{ul} \ar{ur} \ar{urr} & &
			\end{tikzcd}
			\\
			\alpha(i_0)=\closed & \alpha(i_0)=\open
		\end{array}
		\]
	\end{eg}

	\begin{eg}\label{eg:Valpha_I_two_points}
		Consider the case where $|I|=2$. Let $i_1$ and $i_2$ be the elements of $I$. Since $G$ is a tree, we have $J=J_{i_1}\cup J_{i_2}$, and the intersection $J_{i_1}\cap J_{i_2}$ has exactly one element.
		For example, let $J=\{j_1,j_2,j_3,j_4\}$ and
		\[
		J_{i_1}:=\{j_1,j_2\}, \quad J_{i_2}:=\{j_2,j_3,j_4\}.
		\]
		The corresponding undirected graph $G$ is a tree, depicted below. 
		
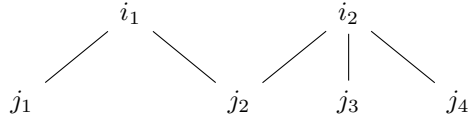
\[
		\begin{tikzcd}
			& i_1 & & i_2 & \\
			j_1 \ar[-]{ur} & & j_2 \ar[-]{ul} \ar[-]{ur} & j_3 \ar[-]{u} & j_4 \ar[-]{ul}
		\end{tikzcd}
		\]
		The Hasse diagrams of $V^{\alpha}$ for $\alpha \in \{\closed, \open\}^I$ are shown in \cref{fig:Hasse_diagram_Valpha}.
		\begin{figure}[htbp]
			\centering
			\begin{tabular}{cc}
				\begin{minipage}{6cm}
					\begin{tikzcd}[column sep=small]
						& i_1 & & i_2 & \\
						j_1 \ar{ur} & & j_2 \ar{ul} \ar{ur} & j_3 \ar{u} & j_4 \ar{ul}
					\end{tikzcd}
					\caption*{$\alpha(i_1)=\closed$, $\alpha(i_2)=\closed$}
				\end{minipage}
				&
				\begin{minipage}{6cm}
					\begin{tikzcd}[column sep=small]
						& i_1 & & & \\
						j_1 \ar{ur} & & j_2 \ar{ul} & j_3 & j_4 \\
						& & & i_2 \ar{ul} \ar{u} \ar{ur} &
					\end{tikzcd}
					\caption*{$\alpha(i_1)=\closed$, $\alpha(i_2)=\open$}
				\end{minipage}
				\\
				\begin{minipage}{6cm}
					\begin{tikzcd}[column sep=small]
						& & & i_2 & \\
						j_1 & & j_2 \ar{ur} & j_3 \ar{u} & j_4 \ar{ul} \\
						& i_1 \ar{ul} \ar{ur} & & &
					\end{tikzcd}
					\caption*{$\alpha(i_1)=\open$, $\alpha(i_2)=\closed$}
				\end{minipage}
				&
				\begin{minipage}{6cm}
					\begin{tikzcd}[column sep=small]
						j_1 & & j_2 & j_3 & j_4 \\
						& i_1 \ar{ul} \ar{ur} & & i_2 \ar{ul} \ar{u} \ar{ur} &
					\end{tikzcd}
					\caption*{$\alpha(i_1)=\open$, $\alpha(i_2)=\open$}
				\end{minipage}
			\end{tabular}
			\caption{Hasse diagrams of $V^{\alpha}$ in \cref{eg:Valpha_I_two_points}}
			\label{fig:Hasse_diagram_Valpha}
		\end{figure}
	\end{eg}

	\begin{defn}
		For $i \in I$ and $j \in J_i$, we define sets
		\begin{align*}
			&R_{i, j}:=\{(x, y) \in K_i \times L_j \mid (\theta_{i, j}(x), y) \in R_{L_j}\}, \\
			&R_{j, i}:=\{(y, x) \in L_j \times K_i \mid (y, \theta_{i, j}(x)) \in R_{L_j}\}. 
		\end{align*}
		For $i, i' \in I$ with $i \neq i'$ and $|J_i \cap J_{i'}|=1$, we define a set
		\[R_{i, i'}:=\{(x, x') \in K_i \times K_{i'} \mid (\theta_{i, j}(x), \theta_{i', j}(x')) \in R_{L_{j}}\}, \]
		where $j$ is the unique element of $J_i \cap J_{i'}$. 
	\end{defn}

	\begin{defn}\label{def:RWalpha}
		For $\alpha \in \{\closed, \open\}^I$, we define a relation $R_{W^{\alpha}}$ on $W$ by
		\[R_{W^{\alpha}}:=\coprod_{i \in I} R_{K_i} \amalg \coprod_{j \in J} R_{L_j} \amalg \coprod_{\substack{i \in I, \\ j \in J_i; \\ \alpha(i)=\closed}} R_{i, j} \amalg \coprod_{\substack{i \in I, \\ j \in J_i; \\ \alpha(i)=\open}} R_{j, i} 
		\amalg \coprod_{\substack{i, i' \in I; \\ |J_i \cap J_{i'}|=1 \\ \alpha(i)=\closed,  \alpha(i')=\open}} R_{i, i'}. \]
	\end{defn}

	\begin{lem}\label{lem:RWalpha_preorder}
		For each $\alpha \in \{\closed, \open\}^I$, the relation $R_{W^{\alpha}}$ is a preorder on $W$. 
	\end{lem}

	\begin{proof}
		The reflexivity of $R_{W^{\alpha}}$ follows from the reflexivity of $R_{K_i}$ for $i \in I$ and $R_{L_j}$ for $j \in J$. We show the transitivity of $R_{W^{\alpha}}$ using the transitivity of $R_{K_i}$ for $i \in I$ and of $R_{L_j}$ for $j \in J$. Take $(x, y), (y, z) \in R_{W^{\alpha}}$. To show that $(x, z) \in R_{W^{\alpha}}$, it suffices to check the following nine cases.
		\begin{description}[font=\normalfont\scshape]
			\item[Case 1] $(x, y), (y, z) \in R_{K_i}$ for some $i \in I$. Then $(x, z) \in R_{K_i} \subset R_{W^{\alpha}}$. 
			
			\item[Case 2] $(x, y), (y, z) \in R_{L_j}$ for some $j \in J$. Then $(x, z) \in R_{L_j} \subset R_{W^{\alpha}}$. 
			
			\item[Case 3] $(x, y) \in R_{K_i}$ and $(y, z) \in R_{i, j}$ for some $i \in I$ and $j \in J_i$ with $\alpha(i)=\closed$. Then $(\theta_{i, j}(y), z) \in R_{L_j}$. Since $\theta_{i, j} \colon K_i \to L_j$ is order-preserving, $(\theta_{i, j}(x), \theta_{i, j}(y)) \in R_{L_j}$. Hence, $(\theta_{i, j}(x), z) \in R_{L_j}$. This shows that $(x, z) \in R_{i, j} \subset R_{W^{\alpha}}$. 
			
			\item[Case 4] $(x, y) \in R_{j, i}$ and $(y, z) \in R_{K_i}$ for some $i \in I$ and $j \in J_i$ with $\alpha(i)=\open$. As in Case 3, we get $(x, z) \in R_{j, i} \subset R_{W^{\alpha}}$. 
			
			\item[Case 5] $(x, y) \in R_{i, j}$ and $(y, z) \in R_{L_j}$ for some $i \in I$ and $j \in J_i$ with $\alpha(i)=\closed$. Then $(\theta_{i, j}(x), y) \in R_{L_j}$. Hence, $(\theta_{i, j}(x), z) \in R_{L_j}$. This shows that $(x, z) \in R_{i, j} \subset R_{W^{\alpha}}$. 
			
			\item[Case 6] $(x, y) \in R_{L_j}$ and $(y, z) \in R_{j, i}$ for some $i \in I$ and $j \in J_i$ with $\alpha(i)=\open$. As in Case 5, we get $(x, z) \in R_{j, i} \subset R_{W^{\alpha}}$. 
			
			\item[Case 7] $(x, y) \in R_{K_i}$ and $(y, z) \in R_{i, i'}$ for some $i, i' \in I$ with $|J_i \cap J_{i'}|=1$, $\alpha(i)=\closed$, and $\alpha(i')=\open$. Let $j \in J$ be the unique element of $J_i \cap J_{i'}$. Then $(\theta_{i, j}(y), \theta_{i', j}(z)) \in R_{L_{j}}$. Since $\theta_{i, j} \colon K_i \to L_{j}$ is order-preserving, $(\theta_{i, j}(x), \theta_{i, j}(y)) \in R_{L_{j}}$. Thus, $(\theta_{i, j}(x), \theta_{i', j}(z)) \in R_{L_{j}}$. This shows that $(x, z) \in R_{i, i'} \subset R_{W^{\alpha}}$. 
			
			\item[Case 8] $(x, y) \in R_{i, i'}$ and $(y, z) \in R_{K_{i'}}$ for some $i, i' \in I$ with $|J_i \cap J_{i'}|=1$, $\alpha(i)=\closed$, and $\alpha(i')=\open$. As in Case 7, we get $(x, z) \in R_{i, i'} \subset R_{W^{\alpha}}$. 
			
			\item[Case 9] $(x, y) \in R_{i, j}$ and $(y, z) \in R_{j, i'}$ for some $i, i' \in I$ and $j \in J_i \cap J_{i'}$ with $\alpha(i)=\closed$ and $\alpha(i')=\open$. Since $(\theta_{i, j}(x), y), (y, \theta_{i', j}(z)) \in R_{L_j}$, we get $(\theta_{i, j}(x), \theta_{i', j}(z)) \in R_{L_j}$. It follows that $(x, z) \in R_{i, i'} \subset R_{W^{\alpha}}$. 
		\end{description}
		This completes the proof. 
	\end{proof}
	
	\begin{defn}\label{def:Walpha}
		In view of \cref{lem:RWalpha_preorder}, for each $\alpha \in \{\closed, \open\}^I$, we define $W^{\alpha}$ to be the Alexandrov space associated with the preordered set $(W, R_{W^{\alpha}})$.
	\end{defn}

	\begin{rem}\label{rem:Walpha_I_singleton}
		Consider the case where $|I|=1$. Let $i_0$ be the unique element of $I$. Recall that $J=J_{i_0}$. The corresponding tree $G$ and the Hasse diagrams of $V^{\alpha}$ for $\alpha \in \{\closed, \open\}^I$ are shown in \cref{eg:Valpha_I_singleton}.
		For $\alpha \in \{\closed, \open\}^I$, the Alexandrov space $W^{\alpha}$ defined in \cref{def:Walpha} coincides with the Alexandrov space obtained from the data $K_{i_0}$, $(L_j)_{j \in J}$, $(\theta_{i_0, j})_{j \in J}$, and $\alpha(i_0) \in \{\closed, \open\}$ as in \textup{\cref{def:Wc_Wo}}.
	\end{rem}

	\begin{prop}\label{prop:tau_continuous}
		The map $\tau \colon W \to V$ is a continuous map $W^{\alpha} \to V^{\alpha}$ for any $\alpha \in \{\closed, \open\}^I$. 
	\end{prop}
	
	\begin{proof}
		This follows because $\tau \colon W \to V$ is an order-preserving map between the preordered sets $(W, R_{W^{\alpha}})$ and $(V, R_{V^{\alpha}})$. 
	\end{proof}
	
	\begin{prop}
		For $i \in I$ and $\alpha \in \{\closed, \open\}^I$, the relative topology on $K_i$ induced from $W^{\alpha}$ coincides with the Alexandrov topology on $K_i$ given by the preorder $R_{K_i}$. 
	\end{prop}
	
	\begin{proof}
		This follows from $R_{W^{\alpha}} \cap (K_i \times K_i)=R_{K_i}$. 
	\end{proof} 
	
	\begin{prop}
		For $\alpha \in \{\closed, \open\}^I$, the relative topology on $\coprod_{j \in J} L_j$ induced from $W^{\alpha}$ coincides with the Alexandrov topology on $\coprod_{j \in J} L_j$ given by the preorder $\coprod_{j \in J} R_{L_j}$. In particular, for each $j \in J$, the relative topology on $L_j$ induced from $W^{\alpha}$ coincides with the Alexandrov topology on $L_j$ given by the preorder $R_{L_j}$. 
	\end{prop}
	
	\begin{proof}
		The first assertion follows from 
		\[R_{W^{\alpha}} \cap \Bigl(\coprod_{j \in J} L_j \times \coprod_{j \in J} L_j\Bigr)=\coprod_{j \in J} R_{L_j}.\]
		The second assertion follows from the first. 
	\end{proof}
	
	\begin{lem}\label{lem:i_closed_open}
		Let $\alpha \in \{\closed, \open\}^I$. 
		\begin{enumerate}[label=\textnormal{(\arabic*)}]
			\item  
			For $i \in I$ with $\alpha(i)=\closed$, we have $\{i\} \in \Closed(V^{\alpha})$. 
			
			\item
			For $i \in I$ with $\alpha(i)=\open$, we have $\{i\} \in \Open(V^{\alpha})$. 
		\end{enumerate}
	\end{lem}
	
	\begin{proof}
		If $\alpha(i)=\closed$, then $R_{V^{\alpha}} \cap ((V \setminus \{i\}) \times \{i\})=\emptyset$. 
		If $\alpha(i)=\open$, then $R_{V^{\alpha}} \cap (\{i\} \times (V \setminus \{i\}))=\emptyset$. This shows the assertion. 
	\end{proof}

	\begin{prop}\label{prop:Ki_closed_open}
		Let $\alpha \in \{\closed, \open\}^I$. 
		\begin{enumerate}[label=\textnormal{(\arabic*)}]
			\item  
			For $i \in I$ with $\alpha(i)=\closed$, we have $K_i \in \Closed(W^{\alpha})$. 
			
			\item
			For $i \in I$ with $\alpha(i)=\open$, we have $K_i \in \Open(W^{\alpha})$. 
		\end{enumerate}
	\end{prop}
	
	\begin{proof}
		Since $\tau^{-1}(\{i\})=K_i$ for $i \in I$, the assertion follows from \cref{prop:tau_continuous} and \cref{lem:i_closed_open}.
	\end{proof}

	\begin{prop}\label{prop:Walpha_T0}
		The following conditions are equivalent: 
		\begin{enumerate}[label=\textnormal{(\roman*)}]
			\item $W^{\alpha}$ is $T_0$ for every $\alpha \in \{\closed, \open\}^I$;
			\item $W^{\alpha}$ is $T_0$ for some $\alpha \in \{\closed, \open\}^I$; 
			\item $K_i$ is $T_0$ for every $i \in I$, and $L_j$ is $T_0$ for every $j \in J$. 
		\end{enumerate}
	\end{prop}
	
	\begin{proof}
		Recall that $V^{\alpha}$ is a finite $T_0$-space, that $W^{\alpha}$ is an Alexandrov space, and that $\tau \colon W^{\alpha} \to V^{\alpha}$ is continuous by \cref{prop:tau_continuous}. 
		Since $\tau^{-1}(\{i\})=K_i$ for $i \in I$ and $\tau^{-1}(\{j\})=L_j$ for $j \in J$, \cref{lem:T0condition} shows that $W^{\alpha}$ is $T_0$ if and only if all the $K_i$ and  $L_j$ are $T_0$. Since the latter condition is independent of $\alpha$, the conditions (i), (ii), and (iii) are equivalent.
	\end{proof}
	
	\begin{eg}\label{eg:Walpha_I_two_point_set}
		Consider the case where $|I|=2$. Let $i_1$ and $i_2$ be the elements of $I$.
		Recall that $J=J_{i_1}\cup J_{i_2}$ and that $J_{i_1}\cap J_{i_2}$ has exactly one element.
		As in \cref{eg:Valpha_I_two_points}, let $J=\{j_1,j_2,j_3,j_4\}$ and
		\[
		J_{i_1}:=\{j_1,j_2\}, \quad J_{i_2}:=\{j_2,j_3,j_4\}.
		\]
		The corresponding tree $G$ and the Hasse diagrams of $V^{\alpha}$ for $\alpha \in \{\closed, \open\}^I$ are shown in \cref{eg:Valpha_I_two_points}.
		Let $K_{i_1}=\{1,2\}$, $K_{i_2}=\{3,4,5,6\}$, $L_{j_1}=\{7\}$, $L_{j_2}=\{8,9,10\}$, $L_{j_3}=\{11\}$, and $L_{j_4}=\{12,13\}$ be the partially ordered sets whose Hasse diagrams are depicted below. 
		\[\setlength{\arraycolsep}{1.2em}
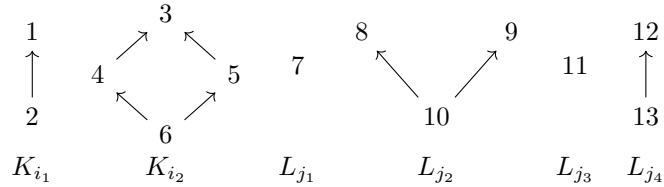
		\begin{array}{cccccc}
			\begin{tikzcd}
				1 \\
				2 \ar{u}
			\end{tikzcd}
			&
			\begin{tikzcd}[sep=small]
				& 3 &\\
				4 \ar{ur} & & 5 \ar{ul} \\
				& 6 \ar{ul} \ar{ur} &
			\end{tikzcd}
			&
			7
			&
			\begin{tikzcd}[column sep=small]
				8 & & 9 \\
				& 10 \ar{ul} \ar{ur} &
			\end{tikzcd}
			&
			11
			&
			\begin{tikzcd}
				12 \\
				13 \ar{u}
			\end{tikzcd}
			\\[1em]
			K_{i_1} & K_{i_2} & L_{j_1} & L_{j_2} & L_{j_3} & L_{j_4}
		\end{array}
		\]
		Let $\theta_{i,j} \colon K_i \to L_j$ for $i \in I$ and $j \in J_i$ be the order-preserving maps defined by
		\begin{align*}
			&\theta_{i_1,j_1}(1)=\theta_{i_1,j_1}(2):=7, \quad \theta_{i_1,j_2}(1)=\theta_{i_1,j_2}(2):=8, \\
			&\theta_{i_2,j_2}(3)=\theta_{i_2,j_2}(4):=9, \quad \theta_{i_2,j_2}(5)=\theta_{i_2,j_2}(6):=10, \\
			&\theta_{i_2,j_3}(3)=\theta_{i_2,j_3}(4)=\theta_{i_2,j_3}(5)=\theta_{i_2,j_3}(6):=11, \\
			&\theta_{i_2,j_4}(3):=12, \quad \theta_{i_2,j_4}(4)=\theta_{i_2,j_4}(5)=\theta_{i_2,j_4}(6):=13.
		\end{align*}
		The Hasse diagrams of $W^{\alpha}$ for $\alpha \in \{\closed, \open\}^I$ are shown in \cref{fig:Hasse_diagram_Walpha}.
		\begin{figure}[htbp]
			\centering
			\begin{tabular}{cc}
				\begin{minipage}{6cm}
					\begin{tikzcd}[row sep=small, column sep=tiny]
						& & & & & 3 & & \\
						& 1 & & & 4 \ar{ur} & & 5 \ar{ul} & \\
						& 2 \ar{u} & & & & 6 \ar{ul} \ar{ur} & & \\
						7 \ar{ur} & 8 \ar{u} & & 9 \ar{uur} & & & & 12 \ar[bend left=20]{uuull} \\
						& & 10 \ar{ul} \ar{ur} \ar[bend right=20]{uurrr} & & & 11 \ar{uu} & & 13 \ar{uull} \ar{u}
					\end{tikzcd}
					\caption*{$\alpha(i_1)=\closed$, $\alpha(i_2)=\closed$}
				\end{minipage} & \quad
				\begin{minipage}{6cm}
					\begin{tikzcd}[row sep=small, column sep=tiny]
						& 1 & & & & & & \\
						& 2 \ar{u} & & & & & & \\
						7 \ar{ur} & 8 \ar{u} & & 9 & & & & 12 \\
						& & 10 \ar{ul} \ar{ur} & & & 11 & & 13 \ar{u} \\
						& & & & & 3 \ar{uull} \ar{u} \ar{uurr} & & \\
						& & & & 4 \ar{ur} \ar[bend right=20]{uurrr} & & 5 \ar[bend left=10]{uullll} \ar{ul} \ar[bend right=20]{uur} & \\
						& & & & & 6 \ar{ul} \ar{ur} & &
					\end{tikzcd}
					\caption*{$\alpha(i_1)=\closed$, $\alpha(i_2)=\open$}
				\end{minipage} \\
				\begin{minipage}{6cm}
					\begin{tikzcd}[row sep=small, column sep=tiny]
						& & & & & 3 & & \\
						& & & & 4 \ar{ur} & & 5 \ar{ul} & \\
						& & & & & 6 \ar{ul} \ar{ur} & & \\
						7 & 8 & & 9 \ar{uur} & & & & 12 \ar[bend left=20]{uuull} \\
						& & 10 \ar{ul} \ar{ur} \ar[bend right=20]{uurrr} & & & 11 \ar{uu} & & 13 \ar{uull} \ar{u} \\
						& 1 \ar{uul} \ar{uu} & & & & & & \\
						& 2 \ar{u} & & & & & &
					\end{tikzcd}
					\caption*{$\alpha(i_1)=\open$, $\alpha(i_2)=\closed$}
				\end{minipage} & \quad
				\begin{minipage}{6cm}
					\begin{tikzcd}[row sep=small, column sep=tiny]
						& & & & & & & \\
						& & & & & & & \\
						7 & 8 & & 9 & & & & 12 \\
						& & 10 \ar{ul} \ar{ur} & & & 11 & & 13 \ar{u} \\
						& 1 \ar{uul} \ar{uu} & & & & 3 \ar{uull} \ar{u} \ar{uurr} & & \\
						& 2 \ar{u} & & & 4 \ar{ur} \ar[bend right=20]{uurrr} & & 5 \ar[bend left=10]{uullll} \ar{ul} \ar[bend right=20]{uur} & \\
						& & & & & 6 \ar{ul} \ar{ur} & &
					\end{tikzcd}
					\caption*{$\alpha(i_1)=\open$, $\alpha(i_2)=\open$}
				\end{minipage}
			\end{tabular}
			\caption{Hasse diagrams of $W^{\alpha}$ in \cref{eg:Walpha_I_two_point_set}}
			\label{fig:Hasse_diagram_Walpha}
		\end{figure}
	\end{eg}
	
	\begin{prop}\label{prop:connectedness_Walpha}
		The following conditions are equivalent: \begin{enumerate}[label=\textnormal{(\roman*)}]
			\item $W^{\alpha}$ is connected for every $\alpha \in \{\closed, \open\}^I$; 
			\item $W^{\alpha}$ is connected for some $\alpha \in \{\closed, \open\}^I$; 
			\item $W \neq \emptyset$ and the equivalence relation on $W$ generated by
			\[\coprod_{i \in I} R_{K_i} \amalg \coprod_{j \in J} R_{L_j} \amalg \coprod_{i \in I, j \in J_i} \{(x, \theta_{i, j}(x)) \mid x \in K_i\} \]
			is equal to $W \times W$. 
		\end{enumerate}
		Furthermore, if $J \neq \emptyset$, $K_i \neq \emptyset$ for every $i \in I$, and if $L_j$ is connected for every $j \in J$, then these equivalent conditions hold. 
	\end{prop}

	\begin{proof}
		Let $R$ denote the equivalence relation on $W$ appearing in (iii).
		For each $\alpha \in \{\closed, \open\}^I$, it is straightforward to verify from the definitions that the equivalence relation on $W$ generated by $R_{W^{\alpha}}$ is equal to $R$. 
		Hence, by \cref{lem:connected_Alexandrov_space}, $W^{\alpha}$ is connected if and only if (iii) holds.
		Since the latter condition is independent of $\alpha$, the conditions (i), (ii), and (iii) are equivalent.
		For the final assertion, if $I=\emptyset$, then $J$ has exactly one element $j_0$, and $W^{\alpha}=L_{j_0}$ is connected for every $\alpha \in  \{\closed, \open\}^I$.
		If $I\neq\emptyset$, then the assertion follows by the same argument as the final assertion of \cref{prop:connectedness_Wc_Wo}, by proceeding along the unique simple paths in $G$.
	\end{proof}
	
	\subsection{Reconstruction of topological spaces from connected subsets}\label{subsection:preliminaries_on_connected_subsets}
	
	\begin{defn}
		For a connected subset $\bfj \subset V$ in $G$, we define a subset
		\[L_{\bfj}:=\tau^{-1}(\bfj) \subset W. \] 
	\end{defn}
	
	For connected subsets $\bfj, \bfj' \subset V$ in $G$ satisfying $\bfj \cap \bfj' \neq \emptyset$, the intersection $\bfj \cap \bfj'$ is also connected in $G$ because $G$ is a tree. In this case, we have 
	\[L_{\bfj \cap \bfj'}=L_{\bfj} \cap L_{\bfj'}. \]
	
	\begin{prop}\label{prop:Lj_locally_closed}
		For a connected subset $\bfj \subset V$ in $G$ and $\alpha \in \{\closed, \open\}^I$, we have $L_{\bfj} \in \LC(W^{\alpha})$.
	\end{prop}
	
	\begin{proof}
		By \cref{lem:connected_locally_closed_tree,prop:Valpha_G}, we have $\bfj \in \LC(V^{\alpha})$. Since $\tau \colon W^{\alpha} \to V^{\alpha}$ is continuous by \cref{prop:tau_continuous}, we obtain $L_{\bfj}=\tau^{-1}(\bfj) \in \LC(W^{\alpha})$. 
	\end{proof}
	
	\begin{defn}
		For a connected subset $\bfj \subset V$ in $G$ and $\alpha \in \{\closed, \open\}^I$, let  $L_{\bfj}^{\alpha}$ denote the subspace of $W^{\alpha}$ with underlying set $L_{\bfj}$. 
	\end{defn}
	
	For a connected subset $\bfj \subset V$ in $G$, the undirected graph 
	\[\bigl((\bfj \cap I) \amalg (\bfj \cap J), \{\{i, j\} \mid i \in \bfj \cap I, j \in \bfj \cap J_i\}\bigr)\]
	is the subgraph of the tree $G$ induced by $\bfj$, and hence is itself a tree. We then have the following proposition. 
	
	\begin{prop}\label{prop:Ljalpha}
		For a connected subset $\bfj \subset V$ in $G$ and $\alpha \in \{\closed, \open\}^I$, the subspace $L_{\bfj}^{\alpha}$ of $W^{\alpha}$ coincides with the Alexandrov space obtained from the data $(K_i)_{i \in \bfj \cap I}$, $(L_j)_{j \in \bfj \cap J}$, $(\theta_{i, j})_{i \in \bfj \cap I, j \in \bfj \cap J_i}$, and $\alpha|_{\bfj \cap I}$ as in \textup{\cref{def:Walpha}}. 
	\end{prop}
	
	\begin{proof}
		The underlying set of $L_{\bfj}^{\alpha}$ is
		\[L_{\bfj}=\coprod_{i \in \bfj \cap I} K_i \amalg \coprod_{j \in \bfj \cap J} L_j. \]
		The specialization preorder on $L_{\bfj}$ is
		\[R_{W^{\alpha}} \cap (L_{\bfj} \times L_{\bfj})=\coprod_{i \in \bfj \cap I} R_{K_i} \amalg \coprod_{j \in \bfj \cap J} R_{L_j} \amalg \coprod_{\substack{i \in \bfj \cap I, \\ j \in \bfj \cap J_i; \\ \alpha(i)=\closed}} R_{i, j} \amalg \coprod_{\substack{i \in \bfj \cap I, \\ j \in \bfj \cap J_i; \\ \alpha(i)=\open}} R_{j, i} 
		\amalg \coprod_{\substack{i, i' \in \bfj \cap I; \\ |J_i \cap J_{i'}|=1 \\ \alpha(i)=\closed,  \alpha(i')=\open}} R_{i, i'}. \]
		In the indexing set of the last disjoint union, the element of $J_i \cap J_{i'}$ belongs to $\bfj$ because $\bfj$ is connected in the tree $G$. This shows the assertion. 
	\end{proof}
	
	The following notation will be used in this subsection and, more importantly, in the next subsection to specify the tensor factors in the C*-algebraic data; see \cref{def:tensor_Di_Dij}.
	
	\begin{defn}
		Since $G$ is a tree, for a connected subset $\bfj \subset V$ in $G$ and $i \in I \setminus \bfj$, there exists a unique element $j \in J_i$ such that there exists a path between $j$ and some element of $\bfj$ in the subgraph of $G$ induced by $(I \setminus \{i\}) \cup J$. 
		This unique element $j$ is denoted by $j_{i, \bfj}$.
	\end{defn}

	\begin{eg}\label{eg:tree_jij}
		Consider the case in \cref{eg:Valpha_I_two_points}. 
		We have $j_{i_1, \{j_1\}}=j_1$. For any connected subset $\bfj \subset \{i_2, j_2, j_3, j_4\}$ in $G$, we have $j_{i_1, \bfj}=j_2$. 
		\end{eg}

		The following elementary lemma records the compatibility of the notation $j_{i,\bfj}$ with connected subsets and will be used repeatedly in what follows.
		
		\begin{lem}\label{lem:j_i_bfj}
			For connected subsets $\bfj, \bfj' \subset V$ in $G$ and $i \in I \setminus (\bfj \cup \bfj')$, we have 
			\[j_{i, \bfj}=j_{i, \bfj'}\] 
			if, for some $v \in \bfj$ and $v' \in \bfj'$, there exists a path between $v$ and $v'$ in the subgraph of $G$ induced by $(I \setminus \{i\}) \cup J$. 
		\end{lem}

		\begin{defn}
			For a connected subset $\bfj \subset V$ in $G$, we define a set
			\[
			I_{\bfj}:=\{i \in I \setminus \bfj \mid j_{i, \bfj} \in \bfj \}. 
			\]
		\end{defn}

		\begin{defn}
			For a connected subset $\bfj \subset V$ in $G$ and $i \in I_{\bfj}$, we define a map 
			\[\theta_{i, \bfj} \colon K_i \to L_{\bfj}\] 
			to be the composite 
			\[
			K_i \xrightarrow{\theta_{i, j_{i, \bfj}}} L_{j_{i, \bfj}} \hookrightarrow L_{\bfj}.
			\]
		\end{defn}

		\begin{prop}
			Let $\bfj \subset V$ be a connected subset in $G$ and $i \in I_{\bfj}$. The map $\theta_{i, \bfj} \colon K_i \to L_{\bfj}$ is a continuous map $K_i \to L_{\bfj}^{\alpha}$ for any $\alpha \in \{\closed, \open\}^I$. 
		\end{prop}
		
		\begin{proof}
			We know that $\theta_{i, j_{i, \bfj}} \colon K_i \to L_{j_{i, \bfj}}$ is continuous.
			The inclusion map $L_{j_{i, \bfj}} \hookrightarrow L_{\bfj}^{\alpha}$ is continuous because $R_{L_{j_{i, \bfj}}} \subset R_{W^{\alpha}} \cap (L_{\bfj} \times L_{\bfj})$. This shows the assertion. 
		\end{proof}

		As in the construction of the reflection functors in \cref{subsection:definition_of_reflection_functors}, we need to enlarge $L_{\bfj}$ by adjoining the neighboring subsets $K_i$. This leads to the following subsets, which will be used in the construction of higher-dimensional reflection functors in the next subsection. 
		
		\begin{defn}
			For a connected subset $\bfj \subset V$ in $G$, the subset $\bfj \cup I_{\bfj} \subset V$ is connected in $G$. Hence, we can define a subset
			\[W_{\bfj}:=L_{\bfj \cup I_{\bfj}} \subset W. \]
		\end{defn}
		
		As in the construction of the reflection functors, the following maps will induce the functors entering the definition of higher-dimensional reflection functors; see \cref{def:functor_Wbfj_Lbfj}. 
		
		\begin{defn}
			For a connected subset $\bfj \subset V$ in $G$, we define a map \[\widetilde{\theta}_{\bfj} \colon W_{\bfj} \to L_{\bfj}\]
			by
			\[
			\widetilde{\theta}_{\bfj}(x):=
			\begin{dcases}
				x & \text{if $\tau(x) \in \bfj$},\\
				\theta_{\tau(x),\bfj}(x) & \text{if $\tau(x) \in I_{\bfj}$}.
			\end{dcases}
			\]
		\end{defn}
		
		\begin{defn}
			For a connected subset $\bfj \subset V$ in $G$ and $\alpha \in \{\closed, \open\}^I$, let $W_{\bfj}^{\alpha}$ denote the subspace of $W^{\alpha}$ with underlying set $W_{\bfj}$. 
		\end{defn}

		\begin{prop}\label{prop:thetatilde_bfj_continuous}
			Let $\bfj \subset V$  be a connected subset in $G$. The map $\widetilde{\theta}_{\bfj} \colon W_{\bfj} \to L_{\bfj}$ is a continuous map $W_{\bfj}^{\alpha} \to L_{\bfj}^{\alpha}$ for any $\alpha \in \{\closed, \open\}^I$. 
		\end{prop}
		
		\begin{proof}
			Take $x, y \in W_{\bfj}^{\alpha}$ with $(x, y) \in R_{W^{\alpha}}$. We show $(\widetilde{\theta}_{\bfj}(x), \widetilde{\theta}_{\bfj}(y)) \in R_{W^{\alpha}}$ by cases. 
			\begin{description}[font=\normalfont\scshape]
				\item[Case 1] $(x, y) \in R_{L_j}$ for some $j \in J$. Then $(\widetilde{\theta}_{\bfj}(x), \widetilde{\theta}_{\bfj}(y))=(x, y) \in R_{L_j} \subset R_{W^{\alpha}}$. 
				
				\item[Case 2] $(x, y) \in R_{K_i}$ for some $i \in I$. If $i \in \bfj$, then 
				\[(\widetilde{\theta}_{\bfj}(x), \widetilde{\theta}_{\bfj}(y))=(x, y) \in R_{K_i} \subset R_{W^{\alpha}}. \] 
				If $i \in I_{\bfj}$, then 
				\[(\widetilde{\theta}_{\bfj}(x), \widetilde{\theta}_{\bfj}(y))=(\theta_{i, j}(x), \theta_{i, j}(y)) \in R_{L_j} \subset R_{W^{\alpha}} \]
				because $\theta_{i, j} \colon K_i \to L_j$ is order-preserving, where $j:=j_{i, \bfj}$. 
				
				\item[Case 3] $(x, y) \in R_{i, j}$ for some $i \in I$ and $j \in J_i$ with $\alpha(i)=\closed$. If $i \in \bfj$, then 
				\[(\widetilde{\theta}_{\bfj}(x), \widetilde{\theta}_{\bfj}(y))=(x, y) \in R_{i, j} \subset R_{W^{\alpha}}. \]
				If $i \in I_{\bfj}$, then $j_{i, \bfj}=j$ and  \[(\widetilde{\theta}_{\bfj}(x), \widetilde{\theta}_{\bfj}(y))=(\theta_{i, j}(x), y) \in R_{L_j} \subset R_{W^{\alpha}}. \]
				
				\item[Case 4] $(x, y) \in R_{j, i}$ for some $i \in I$ and $j \in J_i$ with $\alpha(i)=\open$. As in Case 3, we have $(\widetilde{\theta}_{\bfj}(x), \widetilde{\theta}_{\bfj}(y)) \in R_{W^{\alpha}}$. 
				
				\item[Case 5] $(x, y) \in R_{i, i'}$ for some $i, i' \in I$ with $|J_i \cap J_{i'}|=1$, $\alpha(i)=\closed$, and $\alpha(i')=\open$. Let $j \in J$ be the unique element of $J_i \cap J_{i'}$. If $i, i' \in \bfj$, then 
				\[(\widetilde{\theta}_{\bfj}(x), \widetilde{\theta}_{\bfj}(y))=(x, y) \in R_{i, i'} \subset R_{W^{\alpha}}. \]
				If $i \in I_{\bfj}$ and $i' \in \bfj$, then $j_{i, \bfj}=j$ and  
				\[(\widetilde{\theta}_{\bfj}(x), \widetilde{\theta}_{\bfj}(y))=(\theta_{i, j}(x), y) \in R_{j, i'} \subset R_{W^{\alpha}}. \]
				By symmetry, if $i \in \bfj$ and $i' \in I_{\bfj}$, then $(\widetilde{\theta}_{\bfj}(x), \widetilde{\theta}_{\bfj}(y)) \in R_{i, j} \subset  R_{W^{\alpha}}$. 
				If $i, i' \in I_{\bfj}$, then $j_{i, \bfj}=j_{i', \bfj}=j$ and \[(\widetilde{\theta}_{\bfj}(x), \widetilde{\theta}_{\bfj}(y))=(\theta_{i, j}(x), \theta_{i', j}(y)) \in R_{L_j} \subset R_{W^{\alpha}}. \]
			\end{description}
			This completes the proof. 
		\end{proof}
		
		The following elementary lemma will be used to show the next lemma. 
		
		\begin{lem}\label{lem:preimage_tau_theta}
			For a connected subset $\bfj \subset V$ in $G$ and a subset $F \subset \bfj$, we have 
			\[(\tau \circ \widetilde{\theta}_{\bfj})^{-1}(F)=\tau^{-1}(F) \cup \tau^{-1}(\{i \in I_{\bfj} \mid j_{i, \bfj} \in F\}). \]
			In particular, if $F \subset \bfj \cap I$, then $(\tau \circ \widetilde{\theta}_{\bfj})^{-1}(F)=\tau^{-1}(F)$. 
		\end{lem}
		
		The following lemma will be used to show \cref{lem:composition_thetaj,lem:ipripr}. 
		
		\begin{lem}\label{lem:W_intersection}
			Let $\bfj, \bfj' \subset V$ be connected subsets in $G$. Suppose that $J_i \subset \bfj$ for all $i \in \bfj \cap I$. 
			\begin{enumerate}[label=\textnormal{(\arabic*)}]
				\item \label{lem:W_intersection_1}
				If $\bfj \cap \bfj'=\emptyset$, then 
				\[L_{\bfj} \cap W_{\bfj'}=\emptyset. \]
				
				\item \label{lem:W_intersection_2}
				If $\bfj \cap \bfj' \neq \emptyset$, then 
				\[\widetilde{\theta}_{\bfj}^{-1}(L_{\bfj} \cap W_{\bfj'})=W_{\bfj \cap \bfj'}. \]
			\end{enumerate}
		\end{lem}
		
		\begin{proof}
			\
			\begin{enumerate}[label=\textnormal{(\arabic*)}]
				\item 
				Suppose that $\bfj \cap \bfj'=\emptyset$. Since
				\[
				L_{\bfj} \cap W_{\bfj'}=\tau^{-1}(\bfj) \cap \tau^{-1}(\bfj' \cup I_{\bfj'})=\tau^{-1}(\bfj \cap \bfj') \cup \tau^{-1}(\bfj \cap I_{\bfj'})=\tau^{-1}(\bfj \cap I_{\bfj'}),
				\]
				it suffices to show that $\bfj \cap I_{\bfj'}=\emptyset$.
				Suppose that $i \in \bfj \cap I_{\bfj'}$. Since $i \in I_{\bfj'}$, we have $j_{i,\bfj'} \in \bfj'$. Since $i \in \bfj \cap I$, we also have $j_{i,\bfj'} \in J_i \subset \bfj$ by assumption. Thus, $\bfj \cap \bfj' \neq \emptyset$, a contradiction.
				
				\item
				Suppose that $\bfj \cap \bfj' \neq \emptyset$. We claim that
				\[
				I_{\bfj \cap \bfj'}=\{i \in I_{\bfj} \mid j_{i, \bfj} \in \bfj'\} \cup (\bfj \cap I_{\bfj'}).
				\]
				We repeatedly use \cref{lem:j_i_bfj}. First, take $i \in I_{\bfj \cap \bfj'}$. If $i \notin \bfj$, then
				$j_{i, \bfj}=j_{i, \bfj \cap \bfj'} \in \bfj \cap \bfj'$, and hence $i \in I_{\bfj}$ with $j_{i, \bfj} \in \bfj'$. If $i \in \bfj$, then $i \notin \bfj'$ and
				$j_{i, \bfj'}=j_{i, \bfj \cap \bfj'} \in \bfj'$. Hence, $i \in \bfj \cap I_{\bfj'}$. This proves the inclusion $\subset$.
				
				Conversely, let $i \in I_{\bfj}$ with $j_{i, \bfj} \in \bfj'$. Then
				$j_{i, \bfj \cap \bfj'}=j_{i, \bfj} \in \bfj \cap \bfj'$, and hence $i \in I_{\bfj \cap \bfj'}$. Now let $i \in \bfj \cap I_{\bfj'}$. Then
				$j_{i, \bfj \cap \bfj'}=j_{i, \bfj'} \in \bfj'$. Since $i \in \bfj \cap I$, we also have
				$j_{i, \bfj \cap \bfj'} \in J_i \subset \bfj$ by assumption. Thus, $i \in I_{\bfj \cap \bfj'}$. This proves the reverse inclusion $\supset$ and hence the claim.
				
				We now have
				\begin{align*}
					\widetilde{\theta}_{\bfj}^{-1}(L_{\bfj} \cap W_{\bfj'})
					&=(\tau \circ \widetilde{\theta}_{\bfj})^{-1}(\bfj \cap \bfj') \cup (\tau \circ \widetilde{\theta}_{\bfj})^{-1}(\bfj \cap I_{\bfj'}) \\
					&=\tau^{-1}(\bfj \cap \bfj') \cup \tau^{-1}(\{i \in I_{\bfj} \mid j_{i, \bfj} \in \bfj'\}) \cup \tau^{-1}(\bfj \cap I_{\bfj'}) \\
					&=\tau^{-1}(\bfj \cap \bfj') \cup \tau^{-1}(I_{\bfj \cap \bfj'})
					=W_{\bfj \cap \bfj'}.
				\end{align*}
				The second equality follows from \cref{lem:preimage_tau_theta}, and the third follows from the claim. This completes the proof. 
				\qedhere
			\end{enumerate}
		\end{proof}
		
		We now turn to the reconstruction of the entire space $W^{\alpha}$. 
		
		\begin{defn}\label{def:reconstruction_J}
			For a subset $F \subset I$, let $J^{F}$ be the set of all connected components of the subgraph of $G$ induced by $(I \setminus F) \amalg J$.
			For each $j \in J$, let $[j]^{F} \in J^{F}$ denote the connected component containing $j$.
			For each $i \in F$, define a subset
			\[
			J_i^{F}:=\{[j]^{F} \mid j \in J_i\} \subset J^{F}.
			\]
		\end{defn}
		
		The following lemma is easily verified. 
		
		\begin{lem}\label{lem:JOmega_i}
			Let $F \subset I$. For each $i \in F$, the map
			\[
			J_i \to J^{F}_i, \quad j \mapsto [j]^{F}
			\]
			is a bijection and its inverse is the map
			\[
			J_i^{F} \to J_i, \quad \bfj \mapsto j_{i,\bfj}.
			\]
			Moreover, for $i \in F$ and $\bfj \in J^{F}$, we have $\bfj \in J^{F}_i$ if and only if $i \in I_{\bfj}$. 
		\end{lem}

		\begin{defn}\label{def:reconstruction_G}
			For a subset $F \subset I$, define an undirected graph 
			\[
			G^{F}:=
			\left(F \amalg J^{F}, \{\{i,\bfj\} \mid i \in F,\ \bfj \in J_i^{F}\}\right). 
			\]
		\end{defn}
		
		The undirected graph $G^{F}$ is obtained from the tree $G$ by contracting each connected component of the subgraph induced by $(I\setminus F)\amalg J$.
		By a basic fact from graph theory, contracting connected subgraphs of a tree again yields a tree; see \cite[Chapter~I, \S2, Corollary~2]{Serre_1980}.
		Thus, $G^{F}$ is a tree. 
		
		\begin{eg}
			Consider the case in \cref{eg:Valpha_I_two_points}. For example, let $F:=\{i_2\} \subset I$. 
			Then
			\[
			[j_2]^{F}=\{i_1, j_1, j_2\}, \quad
			[j_3]^{F}=\{j_3\}, \quad
			[j_4]^{F}=\{j_4\}. 
			\]
			We have 
			\[J^{F}=J^{F}_{i_2}=\{[j_2]^{F}, [j_3]^{F}, [j_4]^{F}\}. \] 
			Note that $I_{\bfj}=\{i_2\}$ for every $\bfj \in J^{F}$. 
			The corresponding tree $G^{F}$ is depicted below. 
			
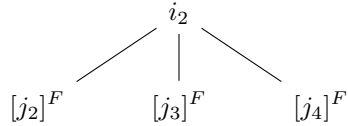
\[
			\begin{tikzcd}
				& i_2 & \\
				{[j_2]^{F}} \ar[-]{ur} & {[j_3]^{F}} \ar[-]{u} & {[j_4]^{F}} \ar[-]{ul} 
			\end{tikzcd}
			\]
		\end{eg}
		
		For a subset $F \subset I$ and $\bfj, \bfj' \in J^{F}$ with $\bfj \neq \bfj'$, we have $\bfj \cap \bfj'=\emptyset$ and hence $L_{\bfj} \cap L_{\bfj'}=\emptyset$. 
		
		\begin{lem}\label{lem:reconstruction_Walpha}
			Let $\alpha \in \{\closed, \open\}^I$. For every $F \subset I$, the Alexandrov space $W^{\alpha}$ is the Alexandrov space obtained from the data $(K_i)_{i \in F}$, $(L_{\bfj}^{\alpha})_{\bfj \in J^{F}}$, $(\theta_{i,\bfj})_{i \in F,\ \bfj \in J_i^{F}}$, and $\alpha|_{F} \in \{\closed, \open\}^{F}$ as in \textup{\cref{def:Walpha}}.
		\end{lem}
		
		\begin{proof}
			Let $W^{\alpha|_{F}}$ denote the Alexandrov space obtained from the data $(K_i)_{i \in F}$, $(L_{\bfj}^{\alpha})_{\bfj \in J^{F}}$, $(\theta_{i,\bfj})_{i \in F,\ \bfj \in J_i^{F}}$, and $\alpha|_{F} \in \{\closed, \open\}^{F}$ as in \textup{\cref{def:Walpha}}. 
			The underlying set of $W^{\alpha|_{F}}$ is equal to 
			\begin{align*}
				\coprod_{i \in F} K_i \amalg \coprod_{\bfj \in J^{F}} L_{\bfj}
				&=\coprod_{i \in F} K_i \amalg \coprod_{\bfj \in J^{F}} \Biggl(\coprod_{i \in \bfj \cap I} K_i \amalg \coprod_{j \in \bfj \cap J} L_j \Biggr) \\
				&=\coprod_{i \in F} K_i \amalg \coprod_{i \in I \setminus F} K_i \amalg \coprod_{j \in J} L_j=W. 
			\end{align*}
			Thus, $W^{\alpha|_{F}}$ and $W^{\alpha}$ have the same underlying set $W$.
			
			It remains to compare the specialization preorders on $W^{\alpha|_{F}}$ and $W^{\alpha}$. Recall that the specialization preorder $R_{W^{\alpha}}$ on $W^{\alpha}$ is specified in \cref{def:RWalpha}. 
			Let $R_{W^{\alpha|_{F}}}$ denote the specialization preorder on $W^{\alpha|_{F}}$. 
			For each $\bfj \in J^{F}$, set
			\[R_{L_{\bfj}^{\alpha}}:=R_{W^{\alpha}} \cap (L_{\bfj} \times L_{\bfj}). \]
			For $i \in F$ and $\bfj \in J^{F}_i$, set
			\begin{align*}
				&R_{i, \bfj}^{\alpha}:=\{(x, y) \in K_i \times L_{\bfj} \mid (\theta_{i, \bfj}(x), y) \in R_{L_{\bfj}^{\alpha}}\}, \\
				&R_{\bfj, i}^{\alpha}:=\{(y, x) \in L_{\bfj} \times K_i \mid (y, \theta_{i, \bfj}(x)) \in R_{L_{\bfj}^{\alpha}}\}. 
			\end{align*}
			For $i, i' \in F$ with $i \neq i'$ and $|J^{F}_i \cap J^{F}_{i'}|=1$, set
			\[R_{i, i'}^{\alpha}:=\{(x, x') \in K_i \times K_{i'} \mid (\theta_{i, \bfj}(x), \theta_{i', \bfj}(x')) \in R_{L_{\bfj}^{\alpha}}\}, \]
			where $\bfj$ is the unique element of $J^{F}_i \cap J^{F}_{i'}$. 
			By the definition of $R_{W^{\alpha|_{F}}}$, we have
			\[R_{W^{\alpha|_{F}}}=\coprod_{i \in F} R_{K_i} \amalg \coprod_{\bfj \in J^{F}} R_{L_{\bfj}^{\alpha}} \amalg \coprod_{\substack{i \in F, \\ \bfj \in J^{F}_i; \\ \alpha(i)=\closed}} R_{i, \bfj}^{\alpha} \amalg \coprod_{\substack{i' \in F, \\ \bfj \in J^{F}_{i'}; \\ \alpha(i')=\open}} R_{\bfj, i'}^{\alpha} \amalg \coprod_{\substack{i, i' \in F; \\ |J^{F}_i \cap J^{F}_{i'}|=1 \\ \alpha(i)=\closed, \alpha(i')=\open}} R_{i, i'}^{\alpha}. \]
			We have
			\[\coprod_{\bfj \in J^{F}} R_{L_{\bfj}^{\alpha}}=\coprod_{i \in I \setminus F} R_{K_i} \amalg \coprod_{j \in J} R_{L_j} \amalg \coprod_{\substack{i \in I \setminus F, \\ j \in J_i; \\ \alpha(i)=\closed}} R_{i, j} \amalg \coprod_{\substack{i' \in I \setminus F, \\ j \in J_{i'}; \\ \alpha(i')=\open}} R_{j, i'} 
			\amalg \coprod_{\substack{i, i' \in I \setminus F; \\ |J_i \cap J_{i'}|=1 \\ \alpha(i)=\closed, \alpha(i')=\open}} R_{i, i'}. \]
			With the help of \cref{lem:JOmega_i}, we have
			\begin{align*}
				\coprod_{\substack{i \in F, \\ \bfj \in J^{F}_i; \\ \alpha(i)=\closed}} R_{i, \bfj}^{\alpha}
				&=\coprod_{\substack{i \in F, \\ j \in J_i; \\ \alpha(i)=\closed}} R_{i, j} \amalg \coprod_{\substack{i \in F, i' \in I \setminus F; \\ |J_i \cap J_{i'}|=1 \\ \alpha(i)=\closed,  \alpha(i')=\open}} R_{i, i'}, \\
				\coprod_{\substack{i' \in F, \\ \bfj \in J^{F}_{i'}; \\ \alpha(i')=\open}} R_{\bfj, i'}^{\alpha}
				&=\coprod_{\substack{i' \in F, \\ j \in J_{i'}; \\ \alpha(i')=\open}} R_{j, i'} \amalg \coprod_{\substack{i \in I \setminus F, i' \in F; \\ |J_i \cap J_{i'}|=1 \\ \alpha(i)=\closed,  \alpha(i')=\open}} R_{i, i'}. 
			\end{align*}
			For $i, i' \in F$ with $i \neq i'$ and $|J_i \cap J_{i'}|=1$, we have $|J^{F}_i \cap J^{F}_{i'}|=1$ and $R_{i, i'}^{\alpha}=R_{i, i'}$. 
			For $i, i' \in F$ with $i \neq i'$, $J_i \cap J_{i'}=\emptyset$, and $|J^{F}_i \cap J^{F}_{i'}|=1$, we have $R_{i, i'}^{\alpha}=\emptyset$. 
			Therefore,
			\[\coprod_{\substack{i, i' \in F; \\ |J^{F}_i \cap J^{F}_{i'}|=1 \\ \alpha(i)=\closed, \alpha(i')=\open}} R_{i, i'}^{\alpha}=\coprod_{\substack{i, i' \in F; \\ |J_i \cap J_{i'}|=1 \\ \alpha(i)=\closed, \alpha(i')=\open}} R_{i, i'}.\]
			Combining the preceding identities, we obtain
			\[R_{W^{\alpha|_{F}}}=R_{W^{\alpha}}. \]
			Consequently, $W^{\alpha|_{F}}=W^{\alpha}$ as Alexandrov spaces.
		\end{proof}
		
		\begin{cor}\label{cor:Walpha_Wc_Wo}
			Let $\alpha \in \{\closed, \open\}^I$. For every $i \in I$, the Alexandrov space $W^{\alpha}$ coincides with the Alexandrov space obtained from the data $K_i$, $(L_{\bfj}^{\alpha})_{\bfj \in J^{\{i\}}_i}$, $(\theta_{i, \bfj})_{\bfj \in J^{\{i\}}_i}$, and $\alpha(i) \in \{\closed, \open\}$ as in \textup{\cref{def:Wc_Wo}}.
		\end{cor}
		
		\begin{proof}
			This follows from \cref{lem:reconstruction_Walpha,rem:Walpha_I_singleton}. 
		\end{proof}
		
		\subsection{Definition of higher-dimensional reflection functors}\label{subsection:definition_of_higher-dimensional_reflection_functors}
		
		Throughout this subsection, we fix $\alpha, \beta \in \{\closed, \open\}^I$.
		
		\begin{defn}
			We define subsets of $I$ by
			\begin{align*}
				I^{\alpha, \beta}_{\open, \closed}
				&:=\{i \in I \mid \alpha(i)=\open,\ \beta(i)=\closed\}, \\
				I^{\alpha, \beta}_{\closed, \open}
				&:=\{i \in I \mid \alpha(i)=\closed,\ \beta(i)=\open\}, \\
				I^{\alpha, \beta}
				&:=I^{\alpha, \beta}_{\closed, \open} \cup I^{\alpha, \beta}_{\open, \closed}
				=\{i \in I \mid \alpha(i) \neq \beta(i)\}.
			\end{align*}
		\end{defn}

		Applying \cref{def:reconstruction_J,def:reconstruction_G} to the subset $I^{\alpha, \beta} \subset I$, we introduce the following simplified notation:
		\[J^{\alpha, \beta}:=J^{I^{\alpha, \beta}}, \quad G^{\alpha, \beta}:=G^{I^{\alpha, \beta}}. \] 
		For each $i \in I^{\alpha, \beta}$, we write
		\[
		J_i^{\alpha, \beta}:=J_i^{I^{\alpha, \beta}}.
		\]
		For each $j \in J$, we simply write $[j]^{\alpha, \beta}:=[j]^{I^{\alpha, \beta}} \in J^{\alpha, \beta}$.

		\begin{eg}\label{eg:Galphabeta}
			Let $I=\{i_1, i_2, i_3, i_4, i_5\}$, $J=\{j_1, j_2, j_3, j_4, j_5, j_6, j_7\}$, and
			\[J_{i_1}=\{j_1, j_2, j_3\}, \ J_{i_2}=\{j_3, j_4, j_5\}, \ J_{i_3}=\{j_5, j_6\}, \ J_{i_4}=\{j_5, j_7\}, \ J_{i_5}=\{j_7\}. \]
			Then the corresponding undirected graph $G$ is a tree, depicted below. 
			
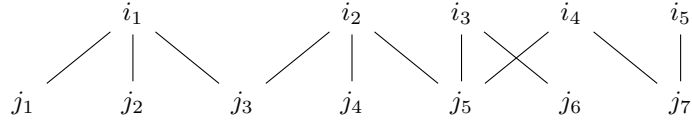
\[
			\begin{tikzcd}
				& i_1 & & i_2 & i_3 & i_4 & i_5 \\
				j_1 \ar[-]{ur} & j_2 \ar[-]{u} & j_3 \ar[-]{ul} \ar[-]{ur} & j_4 \ar[-]{u} & j_5 \ar[-]{ul} \ar[-]{u} \ar[-]{ur} & j_6 \ar[-]{ul} & j_7 \ar[-]{ul} \ar[-]{u}
			\end{tikzcd}
			\]
			Let $\alpha, \beta \in \{\closed, \open\}^I$ be defined by
			\[
			\begin{alignedat}{2}
				&\alpha(i_1)=\alpha(i_4)=\alpha(i_5):=\closed, &\quad  &\alpha(i_2)=\alpha(i_3):=\open, \\
				&\beta(i_3)=\beta(i_5):=\closed, &\quad &\beta(i_1)=\beta(i_2)=\beta(i_4):=\open. 
			\end{alignedat}
			\]
			Then
			\[I^{\alpha, \beta}_{\open, \closed}=\{i_3\}, \ I^{\alpha, \beta}_{\closed, \open}=\{i_1, i_4\}, \ I^{\alpha, \beta}=\{i_1, i_3, i_4\}. \]
			Set
			\[
			\bfj_1:=\{j_1\}, \quad
			\bfj_2:=\{j_2\}, \quad
			\bfj_3:=\{i_2, j_3, j_4, j_5\}, \quad
			\bfj_4:=\{j_6\}, \quad
			\bfj_5:=\{i_5, j_7\}. 
			\]
			Then $J^{\alpha, \beta}=\{\bfj_1, \bfj_2, \bfj_3, \bfj_4, \bfj_5\}$ and
			\[J^{\alpha, \beta}_{i_1}=\{\bfj_1, \bfj_2, \bfj_3\}, \quad J^{\alpha, \beta}_{i_3}=\{\bfj_3, \bfj_4\}, \quad J^{\alpha, \beta}_{i_4}=\{\bfj_3, \bfj_5\}. \]
			Note that 
			\[I_{\bfj_1}=I_{\bfj_2}=\{i_1\}, \quad I_{\bfj_3}=\{i_1, i_3, i_4\}, \quad I_{\bfj_4}=\{i_3\}, \quad I_{\bfj_5}=\{i_4\}. \]
			The corresponding tree $G^{\alpha, \beta}$ is depicted below. 
			
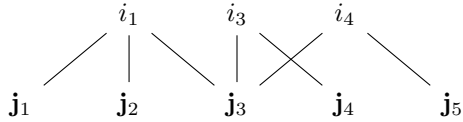
\[
			\begin{tikzcd}
				& i_1 & i_3 & i_4 & \\
				\bfj_1 \ar[-]{ur} & \bfj_2 \ar[-]{u} & \bfj_3 \ar[-]{ul} \ar[-]{u} \ar[-]{ur} & \bfj_4 \ar[-]{ul} & \bfj_5 \ar[-]{ul}
			\end{tikzcd}
			\]
		\end{eg}
		
		\begin{prop}\label{prop:Ljalpha_Ljbeta_equal}
			For each $\bfj \in J^{\alpha, \beta}$, we have 
			\[L_{\bfj}^{\alpha}=L_{\bfj}^{\beta} \]
			as Alexandrov spaces. 
		\end{prop}
		
		\begin{proof}
			Since $\alpha|_{\bfj \cap I}=\beta|_{\bfj \cap I}$, \cref{prop:Ljalpha} applies. 
		\end{proof}

		\begin{defn}\label{def:functor_Wbfj_Lbfj}
			For each $\bfj \in J^{\alpha, \beta}$, \cref{prop:thetatilde_bfj_continuous,prop:Ljalpha_Ljbeta_equal} give a continuous map $\widetilde{\theta}_{\bfj} \colon W_{\bfj}^{\alpha} \to L_{\bfj}^{\alpha}=L_{\bfj}^{\beta}$. 
			Let
			\[p_{W_{\bfj}^{\alpha}}^{L_{\bfj}^{\beta}} \colon \Cstar(W_{\bfj}^{\alpha}) \to \Cstar(L_{\bfj}^{\beta})\]
			denote the functor induced by this continuous map as in \cref{def:functor_induced_by_continuous_map}. 
		\end{defn}
		
		The following lemma is analogous to \cref{lem:preimage_thetaj}. 
		
		\begin{lem}\label{lem:preimage_thetaij}
			Let $i \in I^{\alpha, \beta}$ and $\bfj \in J^{\alpha, \beta}_i$. 
			\begin{enumerate}[label=\textnormal{(\arabic*)}]
				\item \label{lem:preimage_thetaij_1}
				If $i \in I^{\alpha, \beta}_{\open, \closed}$, then $U \cap K_i \subset \theta_{i, \bfj}^{-1}(U \cap L_{\bfj})$ for every $U \in \Open(W^{\beta})$.
				\item \label{lem:preimage_thetaij_2}
				If $i \in I^{\alpha, \beta}_{\closed, \open}$, then $F \cap K_i \subset \theta_{i, \bfj}^{-1}(F \cap L_{\bfj})$ for every $F \in \Closed(W^{\beta})$.
			\end{enumerate}
		\end{lem}
		
		\begin{proof}
			By symmetry, it suffices to show \labelcref{lem:preimage_thetaij_1}.
			Let $U \in \Open(W^\beta)$, and take $x \in U \cap K_i$.
			Set $j:=j_{i,\bfj} \in J_i$.
			We have 
			\[(x, \theta_{i, \bfj}(x))=(x, \theta_{i, j}(x)) \in R_{i, j} \subset R_{W^{\beta}}. \]
			Here, the last inclusion follows because $\beta(i)=\closed$. Since $x \in U$ and $U \in \Open(W^{\beta})$, we obtain $\theta_{i, \bfj}(x) \in U$. This shows that $U \cap K_i \subset \theta_{i, \bfj}^{-1}(U \cap L_{\bfj})$. 
		\end{proof}

		\begin{defn}\label{def:iota_pi_Wbfj_W}
			Let $i \in I^{\alpha, \beta}$ and $\bfj \in J^{\alpha, \beta}_i$. 
			\begin{enumerate}[label=\textnormal{(\arabic*)}]
				\item  
				If $i \in I^{\alpha, \beta}_{\open, \closed}$, then  \cref{lem:preimage_thetaij} \labelcref{lem:preimage_thetaij_1} shows that
				\[U \cap K_i \subset \theta_{i, \bfj}^{-1}(U \cap L_{\bfj}) \subset \widetilde{\theta}_{\bfj}^{-1}(U \cap L_{\bfj})\] 
				for every $U \in \Open(W^{\beta})$. 
				Since $K_i \in \Open(W_{\bfj}^{\alpha})$, \cref{def:iota_pi} \labelcref{def:iota_pi_1} gives a morphism 
				\[\iota_{K_i}^{W_{\bfj}^{\alpha}} \colon i_{K_i}^{W^{\beta}} r_{W^{\alpha}}^{K_i} \Rightarrow i_{L_{\bfj}^{\beta}}^{W^{\beta}} p_{W_{\bfj}^{\alpha}}^{L_{\bfj}^{\beta}} r_{W^{\alpha}}^{W_{\bfj}^{\alpha}}\]
				in $[\Cstar(W^{\alpha}), \Cstar(W^{\beta})]$.
				
				\item
				If $i \in I^{\alpha, \beta}_{\closed, \open}$, then \cref{lem:preimage_thetaij} \labelcref{lem:preimage_thetaij_2} shows that
				\[F \cap K_i \subset \theta_{i, \bfj}^{-1}(F \cap L_{\bfj}) \subset \widetilde{\theta}_{\bfj}^{-1}(F \cap L_{\bfj})\] 
				for every $F \in \Closed(W^{\beta})$. 
				Since $K_i \in \Closed(W_{\bfj}^{\alpha})$, \cref{def:iota_pi} \labelcref{def:iota_pi_2} gives a morphism 
				\[\pi_{W_{\bfj}^{\alpha}}^{K_i} \colon i_{L_{\bfj}^{\beta}}^{W^{\beta}} p_{W_{\bfj}^{\alpha}}^{L_{\bfj}^{\beta}} r_{W^{\alpha}}^{W_{\bfj}^{\alpha}} \Rightarrow i_{K_i}^{W^{\beta}} r_{W^{\alpha}}^{K_i}\]
				in $[\Cstar(W^{\alpha}), \Cstar(W^{\beta})]$.
			\end{enumerate}
		\end{defn}
		
		In the following definition, we adopt the convention that the maximal tensor product over the empty index set is $\C$.
		
		\begin{defn}\label{def:tensor_Di_Dij}
			\
			\begin{enumerate}[label=\textnormal{(\arabic*)}]
				\item  
				For $i \in I^{\alpha, \beta}_{\open, \closed}$, we define a C*-algebra $D_i^{\alpha, \beta}$ as the maximal tensor product 
				\[D_i^{\alpha, \beta}:=\bigotimes_{k \in I^{\alpha, \beta}_{\closed, \open}} D_{k, j_{k, \{i\}}}. \]
				
				\item
				For $i \in I^{\alpha, \beta}_{\closed, \open}$, we define a C*-algebra $D_i^{\alpha, \beta}$ as the maximal tensor product 
				\[D_i^{\alpha, \beta}:=D_i \otimes \bigotimes_{k \in I^{\alpha, \beta}_{\closed, \open} \setminus \{i\}} D_{k, j_{k, \{i\}}}. \]
				
				\item 
				For $\bfj \in J^{\alpha, \beta}$, we define a C*-algebra $D_{\bfj}^{\alpha, \beta}$ as the maximal tensor product 
				\[D_{\bfj}^{\alpha, \beta}:=\bigotimes_{k \in I^{\alpha, \beta}_{\closed, \open}} D_{k, j_{k, \bfj}}. \] 
			\end{enumerate}
		\end{defn}

		\begin{defn}
			Let $i \in I^{\alpha, \beta}$ and $\bfj \in J^{\alpha, \beta}_i$. 
			\begin{enumerate}[label=\textnormal{(\arabic*)}]
				\item  
				If $i \in I^{\alpha, \beta}_{\open, \closed}$, then  \cref{lem:j_i_bfj} implies that $j_{k, \{i\}}=j_{k, \bfj}$ for all $k \in I^{\alpha, \beta}_{\closed, \open}$, whence $D_i^{\alpha, \beta}=D_{\bfj}^{\alpha, \beta}$. 
				Let
				\[\id_{i, \bfj}^{\alpha, \beta} \colon D_i^{\alpha, \beta} \to D_{\bfj}^{\alpha, \beta}\] 
				denote the identity map. 
				
				\item If $i \in I^{\alpha, \beta}_{\closed, \open}$, then \cref{lem:j_i_bfj} implies that $j_{k, \bfj}=j_{k, \{i\}}$ for all $k \in I^{\alpha, \beta}_{\closed, \open} \setminus \{i\}$. 
				Define a $\ast$-homomorphism 
				\[\rho_{i, \bfj}^{\alpha, \beta} \colon D_{\bfj}^{\alpha, \beta} \to D_i^{\alpha, \beta}\] 
				as the maximal tensor product
				\[\rho_{i, j_{i, \bfj}} \otimes \id \colon D_{i, j_{i, \bfj}} \otimes \bigotimes_{k \in I^{\alpha, \beta}_{\closed, \open} \setminus \{i\}} D_{k, j_{k, \{i\}}} \to D_i \otimes \bigotimes_{k \in I^{\alpha, \beta}_{\closed, \open} \setminus \{i\}} D_{k, j_{k, \{i\}}}. \]	
			\end{enumerate}
		\end{defn}
		
		\begin{eg}\label{eg:tensor_Di_Dj}
			Consider the case in \cref{eg:Galphabeta}. 
			The C*-algebras $D_i^{\alpha, \beta}$ for $i \in I^{\alpha, \beta}$ are given by 
			\[D_{i_1}^{\alpha, \beta}=D_{i_1}\otimes D_{i_4,j_5}, \quad
			D_{i_3}^{\alpha, \beta}=D_{i_1,j_3}\otimes D_{i_4,j_5}, \quad
			D_{i_4}^{\alpha, \beta}= D_{i_1,j_3}\otimes D_{i_4}. \]
			The C*-algebras $D_{\bfj}^{\alpha, \beta}$ for $\bfj \in J^{\alpha, \beta}$ are given by 
			\[
			\begin{alignedat}{3}
				D_{\bfj_1}^{\alpha, \beta} &= D_{i_1,j_1}\otimes D_{i_4,j_5}, \quad&
				D_{\bfj_2}^{\alpha, \beta} &= D_{i_1,j_2}\otimes D_{i_4,j_5}, \quad&
				D_{\bfj_3}^{\alpha, \beta} &= D_{i_1,j_3}\otimes D_{i_4,j_5}, \\
				D_{\bfj_4}^{\alpha, \beta} &= D_{i_1,j_3}\otimes D_{i_4,j_5}, \quad&
				D_{\bfj_5}^{\alpha, \beta} &= D_{i_1,j_3}\otimes D_{i_4,j_7}. &&
			\end{alignedat}
			\]
			Observe that $D_{i_3}^{\alpha, \beta}=D_{\bfj_3}^{\alpha, \beta}=D_{\bfj_4}^{\alpha, \beta}$. 
			The $\ast$-homomorphisms $\rho_{i, \bfj}^{\alpha, \beta} \colon D_{\bfj}^{\alpha, \beta} \to D_i^{\alpha, \beta}$ for $i \in I^{\alpha, \beta}_{\closed, \open}$ and $\bfj \in J^{\alpha, \beta}_i$ are given by
			\[
			\begin{alignedat}{3}
				\rho_{i_1, \bfj_1}^{\alpha, \beta} &=\rho_{i_1, j_1} \otimes \id_{D_{i_4, j_5}}, \quad &
				\rho_{i_1, \bfj_2}^{\alpha, \beta} &=\rho_{i_1, j_2} \otimes \id_{D_{i_4, j_5}},
				\quad &
				\rho_{i_1, \bfj_3}^{\alpha, \beta} &=\rho_{i_1, j_3} \otimes \id_{D_{i_4, j_5}}, \\
				\rho_{i_4, \bfj_3}^{\alpha, \beta} &=\id_{D_{i_1, j_3}} \otimes \rho_{i_4, j_5}, \quad &
				\rho_{i_4, \bfj_5}^{\alpha, \beta} &=\id_{D_{i_1, j_3}} \otimes \rho_{i_4, j_7}. &&
			\end{alignedat}
			\]
		\end{eg}

		The following notation specifies certain vertices of the $I^{\alpha, \beta}$-dimensional cube $\{0, 1\}^{I^{\alpha, \beta}}$ indexing the diagram considered below. We distinguish one vertex $\chi^{\alpha, \beta}$ and, for each $i \in I^{\alpha, \beta}$, the adjacent vertex $\chi_i^{\alpha, \beta}$ obtained by changing only the $i$-th coordinate.
		
		\begin{defn}
			Define $\chi^{\alpha, \beta} \in \{0, 1\}^{I^{\alpha, \beta}}$ by
			\[
			\chi^{\alpha, \beta}(k):=
			\begin{dcases}
				0 & \text{if $k \in I^{\alpha, \beta}_{\closed, \open}$},\\
				1 & \text{if $k \in I^{\alpha, \beta}_{\open, \closed}$}.
			\end{dcases}
			\]
			For each $i \in I^{\alpha, \beta}$, let $\chi_i^{\alpha, \beta} \in \{0, 1\}^{I^{\alpha, \beta}}$ be the unique element such that $\chi_i^{\alpha, \beta}(k)=\chi^{\alpha, \beta}(k)$ for every $k \in I^{\alpha, \beta} \setminus \{i\}$ and $\chi_i^{\alpha, \beta}(i) \neq \chi^{\alpha, \beta}(i)$. 
		\end{defn}
		
		Observe that $(\chi_i^{\alpha, \beta}, \chi^{\alpha, \beta}) \in R_{\{0, 1\}^{I^{\alpha, \beta}}}$ for every $i \in I^{\alpha, \beta}_{\open, \closed}$, and that $(\chi^{\alpha, \beta}, \chi_i^{\alpha, \beta}) \in R_{\{0, 1\}^{I^{\alpha, \beta}}}$ for every $i \in I^{\alpha, \beta}_{\closed, \open}$. 
		
		Following \cref{nota:tensor_product}, we may regard C*-algebras and $\ast$-homomorphisms as functors and natural transformations, respectively.
		
		\begin{defn}\label{def:Falphabeta}
			For each $\lambda \in \{0, 1\}^{I^{\alpha, \beta}}$, we define a functor 
			\[(F_{\alpha}^{\beta})_{\lambda} \colon \Cstar(W^{\alpha}) \to \Cstar(W^{\beta})\]
			by
			\[
			(F_{\alpha}^{\beta})_{\lambda}:=
			\begin{dcases}
				\id
				& \text{if $\alpha=\beta$},\\
				i_{K_i}^{W^{\beta}} r_{W^{\alpha}}^{K_i} D_i^{\alpha, \beta}
				& \text{if $\alpha \neq \beta$ and $\lambda=\chi_i^{\alpha, \beta}$ for some $i \in I^{\alpha, \beta}$},\\
				\bigoplus_{\bfj \in J^{\alpha, \beta}} i_{L_{\bfj}^{\beta}}^{W^{\beta}} p_{W_{\bfj}^{\alpha}}^{L_{\bfj}^{\beta}} r_{W^{\alpha}}^{W_{\bfj}^{\alpha}} D_{\bfj}^{\alpha, \beta}
				& \text{if $\alpha \neq \beta$ and $\lambda=\chi^{\alpha, \beta}$},\\
				0
				& \text{otherwise}.
			\end{dcases}
			\]
		\end{defn}
		
		The following definition uses the operations on natural transformations associated with direct sums introduced in \cref{subsection:natural_transformations_associated_with_direct_sums}, as well as higher-dimensional mapping cones in functor categories developed in \cref{subsection:higher-dimensional_mapping_cones}. 
		
		\begin{defn}\label{def:phialphabeta}
			For $(\lambda, \mu) \in R_{\{0, 1\}^{I^{\alpha, \beta}}}$, we define a natural transformation 
			\[(\varphi_{\alpha}^{\beta})_{\lambda}^{\mu} \colon (F_{\alpha}^{\beta})_{\lambda} \Rightarrow (F_{\alpha}^{\beta})_{\mu}\] 
			by cases.
			\begin{description}[font=\normalfont\scshape]
				\item[Case 1] $\lambda=\mu$. We define $(\varphi_{\alpha}^{\beta})_{\lambda}^{\mu}$ to be the identity natural transformation. 
				
				\item[Case 2] $\lambda=\chi_i^{\alpha, \beta}$ for some $i \in I^{\alpha, \beta}_{\open, \closed}$, and $\mu=\chi^{\alpha, \beta}$. 
				For each $\bfj \in J^{\alpha, \beta}_i$, we have the horizontal composite
				\[\iota_{K_i}^{W_{\bfj}^{\alpha}} \id_{i, \bfj}^{\alpha, \beta} \colon i_{K_i}^{W^{\beta}} r_{W^{\alpha}}^{K_i} D_i^{\alpha, \beta} \Rightarrow i_{L_{\bfj}^{\beta}}^{W^{\beta}} p_{W_{\bfj}^{\alpha}}^{L_{\bfj}^{\beta}} r_{W^{\alpha}}^{W_{\bfj}^{\alpha}} D_{\bfj}^{\alpha, \beta}. \]
				Since $J^{\alpha, \beta}_i$ is finite, we can define $(\varphi_{\alpha}^{\beta})_{\lambda}^{\mu}$ to be the natural transformation
				\[\Bigl(\iota_{K_i}^{W_{\bfj}^{\alpha}} \id_{i, \bfj}^{\alpha, \beta}\Bigr)_{\bfj \in J^{\alpha, \beta}_i} \colon i_{K_i}^{W^{\beta}} r_{W^{\alpha}}^{K_i} D_i^{\alpha, \beta} \Rightarrow \bigoplus_{\bfj \in J^{\alpha, \beta}} i_{L_{\bfj}^{\beta}}^{W^{\beta}} p_{W_{\bfj}^{\alpha}}^{L_{\bfj}^{\beta}} r_{W^{\alpha}}^{W_{\bfj}^{\alpha}} D_{\bfj}^{\alpha, \beta}. \]
				
				\item[Case 3] $\lambda=\chi^{\alpha, \beta}$, and $\mu=\chi_i^{\alpha, \beta}$ for some $i \in I^{\alpha, \beta}_{\closed, \open}$. The $\ast$-homomorphisms $\rho_{i, \bfj}^{\alpha, \beta} \colon D_{\bfj}^{\alpha, \beta} \to D_i^{\alpha, \beta}$ for $\bfj \in J^{\alpha, \beta}_i$ are mutually orthogonal. 
				Hence, the horizontal composites
				\[\pi_{W_{\bfj}^{\alpha}}^{K_i} \rho_{i, \bfj}^{\alpha, \beta} \colon i_{L_{\bfj}^{\beta}}^{W^{\beta}} p_{W_{\bfj}^{\alpha}}^{L_{\bfj}^{\beta}} r_{W^{\alpha}}^{W_{\bfj}^{\alpha}} D_{\bfj}^{\alpha, \beta} \Rightarrow i_{K_i}^{W^{\beta}} r_{W^{\alpha}}^{K_i} D_i^{\alpha, \beta}, \quad \bfj \in J^{\alpha, \beta}_i\]
				are mutually orthogonal. We can therefore define $(\varphi_{\alpha}^{\beta})_{\lambda}^{\mu}$ to be the natural transformation
				\[\sum_{\bfj \in J^{\alpha, \beta}_i} \pi_{W_{\bfj}^{\alpha}}^{K_i} \rho_{i, \bfj}^{\alpha, \beta} \colon \bigoplus_{\bfj \in J^{\alpha, \beta}} i_{L_{\bfj}^{\beta}}^{W^{\beta}} p_{W_{\bfj}^{\alpha}}^{L_{\bfj}^{\beta}} r_{W^{\alpha}}^{W_{\bfj}^{\alpha}} D_{\bfj}^{\alpha, \beta} \Rightarrow i_{K_i}^{W^{\beta}} r_{W^{\alpha}}^{K_i} D_i^{\alpha, \beta}. \]
				\item[Case 4] Otherwise. We define $(\varphi_{\alpha}^{\beta})_{\lambda}^{\mu}:=0$. 
			\end{description}
		\end{defn}

		\begin{lem}\label{lem:functoriality_Falphabeta}
			For $(\lambda, \mu), (\mu, \nu) \in R_{\{0, 1\}^{I^{\alpha, \beta}}}$, we have
			\[(\varphi_{\alpha}^{\beta})_{\mu}^{\nu} \circ (\varphi_{\alpha}^{\beta})_{\lambda}^{\mu}=(\varphi_{\alpha}^{\beta})_{\lambda}^{\nu}\]
			as natural transformations $(F_{\alpha}^{\beta})_{\lambda} \Rightarrow (F_{\alpha}^{\beta})_{\nu}$.
		\end{lem}
		
		\begin{proof}
			If $\alpha=\beta$, equivalently, $I^{\alpha, \beta}=\emptyset$, then $\{0, 1\}^{I^{\alpha, \beta}}$ is the category with one object and one morphism. Hence, the assertion is clear. 
			Suppose that $\alpha \neq \beta$. 
			We may assume that $\lambda \neq \mu$ and $\mu \neq \nu$. 
			We show that
			\[(\varphi_{\alpha}^{\beta})_{\mu}^{\nu} \circ (\varphi_{\alpha}^{\beta})_{\lambda}^{\mu}=0=(\varphi_{\alpha}^{\beta})_{\lambda}^{\nu}. \]
			We proceed by cases. 
			\begin{description}[font=\normalfont\scshape]
				\item[Case 1] $(F_{\alpha}^{\beta})_{\lambda}=0$ or $(F_{\alpha}^{\beta})_{\nu}=0$. 
				The desired equality follows because a morphism is unique whenever either its source or its target is a zero object. 
				\item[Case 2] 
				$(F_{\alpha}^{\beta})_{\lambda} \neq 0$ and $(F_{\alpha}^{\beta})_{\nu} \neq 0$. 
				Then $\lambda=\chi_i^{\alpha, \beta}$ for some $i \in I^{\alpha, \beta}_{\open, \closed}$, and $\nu=\chi_{i'}^{\alpha, \beta}$ for some $i' \in I^{\alpha, \beta}_{\closed, \open}$. 
				Hence, $(\varphi_{\alpha}^{\beta})_{\lambda}^{\nu} \colon (F_{\alpha}^{\beta})_{\lambda} \Rightarrow (F_{\alpha}^{\beta})_{\nu}$ is equal to
				\[0 \colon i_{K_i}^{W^{\beta}} r_{W^{\alpha}}^{K_i} D_i^{\alpha, \beta} \Rightarrow i_{K_{i'}}^{W^{\beta}} r_{W^{\alpha}}^{K_{i'}} D_{i'}^{\alpha, \beta}. \]
				If $(F_{\alpha}^{\beta})_{\mu}=0$, then $(\varphi_{\alpha}^{\beta})_{\lambda}^{\mu}=0$ and $(\varphi_{\alpha}^{\beta})_{\mu}^{\nu}=0$. Hence, the desired equality holds. Suppose that $(F_{\alpha}^{\beta})_{\mu} \neq 0$. Then $\mu=\chi^{\alpha, \beta}$. 
				By the definitions of
				$(\varphi_{\alpha}^{\beta})_{\lambda}^{\mu}$ and
				$(\varphi_{\alpha}^{\beta})_{\mu}^{\nu}$, their vertical composite is the sum of the vertical composites 
				\[i_{K_i}^{W^{\beta}} r_{W^{\alpha}}^{K_i} D_i^{\alpha, \beta} \xRightarrow{\iota_{K_i}^{W_{\bfj}^{\alpha}} \id_{i, \bfj}^{\alpha, \beta}} i_{L_{\bfj}^{\beta}}^{W^{\beta}} p_{W_{\bfj}^{\alpha}}^{L_{\bfj}^{\beta}} r_{W^{\alpha}}^{W_{\bfj}^{\alpha}} D_{\bfj}^{\alpha, \beta} \xRightarrow{\pi_{W_{\bfj}^{\alpha}}^{K_{i'}} \rho_{i', \bfj}^{\alpha, \beta}} i_{K_{i'}}^{W^{\beta}} r_{W^{\alpha}}^{K_{i'}} D_{i'}^{\alpha, \beta} \] 
				indexed by $\bfj \in J^{\alpha, \beta}_i \cap J^{\alpha, \beta}_{i'}$; note that this sum has at most one term, since $G^{\alpha,\beta}$
				is a tree.
				Since $K_i \cap K_{i'}=\emptyset$, the vertical composite
				through each common summand is equal to $0$.
				Thus, the desired equality holds.
			\end{description}
			This completes the proof. 
		\end{proof}

		\begin{defn}
			By \cref{lem:functoriality_Falphabeta}, 
			the functors $(F_{\alpha}^{\beta})_{\lambda}$ for $\lambda \in \{0, 1\}^{I^{\alpha, \beta}}$, and the natural transformations $(\varphi_{\alpha}^{\beta})_{\lambda}^{\mu}$ for $(\lambda, \mu) \in R_{\{0, 1\}^{I^{\alpha, \beta}}}$ define a functor
			\[(F_{\alpha}^{\beta}; \varphi_{\alpha}^{\beta}) \colon \{0, 1\}^{I^{\alpha, \beta}} \to [\Cstar(W^{\alpha}), \Cstar(W^{\beta})]. \]
		\end{defn}
		
		\begin{defn}\label{def:higher-dimensional_reflection_functor}
			We define a functor
			\[S_{\alpha}^{\beta} \colon \Cstar(W^{\alpha}) \to \Cstar(W^{\beta})\] by $S_{\alpha}^{\beta}:=M^{I^{\alpha, \beta}}(F_{\alpha}^{\beta}; \varphi_{\alpha}^{\beta})$. We call $S_{\alpha}^{\beta}$ the \textit{higher-dimensional reflection functor}. 
		\end{defn}
		
		\begin{rem}
			If $\alpha=\beta$, equivalently, $I^{\alpha, \beta}=\emptyset$, then 
			\[[\{0, 1\}^{I^{\alpha, \beta}}, [\Cstar(W^{\alpha}), \Cstar(W^{\beta})]]=[\Cstar(W^{\alpha}), \Cstar(W^{\alpha})]. \] 
			Under this identification, $(F_{\alpha}^{\beta}; \varphi_{\alpha}^{\beta})$ is the identity functor on $\Cstar(W^{\alpha})$ and $M^{I^{\alpha, \beta}}$ is the identity functor on $[\Cstar(W^{\alpha}), \Cstar(W^{\alpha})]$; see \cref{eg:0-dimensional_mapping_cone}. Hence, $S_{\alpha}^{\beta}$ is the identity functor on $\Cstar(W^{\alpha})$. 
		\end{rem}

		\begin{prop}\label{prop:higher-dimensional_reflection_functor_singleton}
			Suppose that $I^{\alpha, \beta}=\{i\}$ for some $i \in I$. Recall from \textup{\cref{cor:Walpha_Wc_Wo}} that the Alexandrov space $W^{\alpha}$ (respectively, $W^{\beta}$) coincides with the Alexandrov space obtained from the data $K_i$, $(L_{\bfj}^{\alpha})_{\bfj \in J^{\alpha, \beta}_i}$, $(\theta_{i, \bfj})_{\bfj \in J^{\alpha, \beta}_i}$, and $\alpha(i) \in \{\closed, \open\}$ (respectively, $\beta(i) \in \{\closed, \open\}$) as in \textup{\cref{def:Wc_Wo}}; see also \textup{\cref{prop:Ljalpha_Ljbeta_equal}}. 
			
			\begin{enumerate}[label=\textnormal{(\arabic*)}]
				\item  
				If $i \in I^{\alpha, \beta}_{\open, \closed}$, then the higher-dimensional reflection functor 
				\[S_{\alpha}^{\beta} \colon \Cstar(W^{\alpha}) \to \Cstar(W^{\beta})\] 
				coincides with the reflection functor defined in \textup{\cref{def:reflection_functor} \labelcref{def:reflection_functor_1}}. 
				
				\item 
				If $i \in I^{\alpha, \beta}_{\closed, \open}$, then the higher-dimensional reflection functor 
				\[S_{\alpha}^{\beta} \colon \Cstar(W^{\alpha}) \to \Cstar(W^{\beta})\] 
				coincides with the reflection functor obtained from the data $D_i^{\alpha, \beta}$, $(D_{\bfj}^{\alpha, \beta})_{\bfj \in J^{\alpha, \beta}_i}$, and $(\rho_{i, \bfj}^{\alpha, \beta})_{\bfj \in J^{\alpha, \beta}_i}$ defined in \textup{\cref{def:reflection_functor} \labelcref{def:reflection_functor_2}}. 
			\end{enumerate}
		\end{prop}
		
		\begin{proof}
			Since $|I^{\alpha, \beta}|=1$, each object of $[\{0, 1\}^{I^{\alpha, \beta}}, [\Cstar(W^{\alpha}), \Cstar(W^{\beta})]]$ may be regarded as a morphism in $[\Cstar(W^{\alpha}), \Cstar(W^{\beta})]$. 
			
			If $i \in I^{\alpha, \beta}_{\open, \closed}$, then the functors $D_i^{\alpha, \beta}$ and $D_{\bfj}^{\alpha, \beta}$ for $\bfj \in J^{\alpha, \beta}_i$ coincide with the identity functor on $\Cstar(W^{\alpha})$. Thus, the functor $(F_{\alpha}^{\beta}; \varphi_{\alpha}^{\beta})$
			is regarded as the natural transformation
			\[\Bigl(\iota_{K_i}^{W_{\bfj}^{\alpha}} \id_{i, \bfj}^{\alpha, \beta}\Bigr)_{\bfj \in J^{\alpha, \beta}_i} \colon i_{K_i}^{W^{\beta}} r_{W^{\alpha}}^{K_i} \Rightarrow \bigoplus_{\bfj \in J^{\alpha, \beta}_i} i_{L_{\bfj}^{\beta}}^{W^{\beta}} p_{W_{\bfj}^{\alpha}}^{L_{\bfj}^{\beta}} r_{W^{\alpha}}^{W_{\bfj}^{\alpha}}. \]
			This is precisely the natural transformation appearing in \textup{\cref{def:reflection_functor} \labelcref{def:reflection_functor_1}}.
			
			If $i \in I^{\alpha, \beta}_{\closed, \open}$, then the functor $(F_{\alpha}^{\beta}; \varphi_{\alpha}^{\beta})$
			is regarded as the natural transformation
			\[\sum_{\bfj \in J^{\alpha, \beta}_i} \pi_{W_{\bfj}^{\alpha}}^{K_i} \rho_{i, \bfj}^{\alpha, \beta} \colon \bigoplus_{\bfj \in J^{\alpha, \beta}_i} i_{L_{\bfj}^{\beta}}^{W^{\beta}} p_{W_{\bfj}^{\alpha}}^{L_{\bfj}^{\beta}} r_{W^{\alpha}}^{W_{\bfj}^{\alpha}} D_{\bfj}^{\alpha, \beta} \Rightarrow i_{K_i}^{W^{\beta}} r_{W^{\alpha}}^{K_i} D_i^{\alpha, \beta}. \]
			This is precisely the natural transformation appearing in \textup{\cref{def:reflection_functor} \labelcref{def:reflection_functor_2}}.
			The assertions now follow from the definitions of the reflection functors and $S_{\alpha}^{\beta}$.
		\end{proof}
		
		\begin{eg}\label{eg:3-dimensional_reflection_functor}
			Consider the case where $|I^{\alpha, \beta}|=3$ as in \cref{eg:Galphabeta}; see also \cref{eg:tensor_Di_Dj}. The higher-dimensional reflection functor 
			\[S_{\alpha}^{\beta} \colon \Cstar(W^{\alpha}) \to \Cstar(W^{\beta})\] 
			is the 3-dimensional mapping cone of the following commutative diagram in $[\Cstar(W^{\alpha}), \Cstar(W^{\beta})]$:
			\[
			\begin{tikzcd}[column sep=3.5em]
				& {} &
				[-4em] i_{K_{i_3}}^{W^{\beta}} r_{W^{\alpha}}^{K_{i_3}}
				D_{i_3}^{\alpha, \beta}
				\ar[Rightarrow, swap]{dl}{
					(\iota_{K_{i_3}}^{W_{\bfj_3}^{\alpha}}
					\id_{i_3, \bfj_3}^{\alpha, \beta},
					\iota_{K_{i_3}}^{W_{\bfj_4}^{\alpha}}
					\id_{i_3, \bfj_4}^{\alpha, \beta})}
				\ar[Rightarrow]{dd}
				\ar[Rightarrow]{rr}
				& & 0 \ar[Rightarrow]{dl} \ar[Rightarrow]{dd}
				\\
				&
				\bigoplus_{\bfj \in J^{\alpha, \beta}}
				i_{L_{\bfj}^{\beta}}^{W^{\beta}}
				p_{W_{\bfj}^{\alpha}}^{L_{\bfj}^{\beta}}
				r_{W^{\alpha}}^{W_{\bfj}^{\alpha}} D_{\bfj}^{\alpha, \beta}
				\ar[Rightarrow, swap]{dd}{
					\pi_{W_{\bfj_3}^{\alpha}}^{K_{i_4}}
					\rho^{\alpha, \beta}_{i_4, \bfj_3}
					+\pi_{W_{\bfj_5}^{\alpha}}^{K_{i_4}}
					\rho^{\alpha, \beta}_{i_4, \bfj_5}}
				& &
				i_{K_{i_1}}^{W^{\beta}} r_{W^{\alpha}}^{K_{i_1}}
				D_{i_1}^{\alpha, \beta}
				\ar[Rightarrow, from=ll, crossing over, swap,
				"{\substack{
						\pi_{W_{\bfj_1}^{\alpha}}^{K_{i_1}}
						\rho^{\alpha, \beta}_{i_1, \bfj_1}\\
						{}+\pi_{W_{\bfj_2}^{\alpha}}^{K_{i_1}}
						\rho^{\alpha, \beta}_{i_1, \bfj_2}\\
						{}+\pi_{W_{\bfj_3}^{\alpha}}^{K_{i_1}}
						\rho^{\alpha, \beta}_{i_1, \bfj_3}
				}}"{pos=.76}]
				\\[0.8em]
				& & 0 \ar[Rightarrow]{dl} \ar[Rightarrow]{rr}
				& & 0 \ar[Rightarrow]{dl}
				\\
				&
				i_{K_{i_4}}^{W^{\beta}} r_{W^{\alpha}}^{K_{i_4}}
				D_{i_4}^{\alpha, \beta} \ar[Rightarrow, swap]{rr}
				& & 0\mathrlap{.} \ar[Rightarrow, crossing over, from=uu]
			\end{tikzcd}
			\]
		\end{eg}
		
		Under the identification 
		\[[\{0, 1\}^{I^{\alpha, \beta}}, [\Cstar(W^{\alpha}), \Cstar(W^{\beta})]]=[\Cstar(W^{\alpha}), [\{0, 1\}^{I^{\alpha, \beta}}, \Cstar(W^{\beta})]], \]
		we may regard $(F_{\alpha}^{\beta}; \varphi_{\alpha}^{\beta})$ as a functor $\Cstar(W^{\alpha}) \to [\{0, 1\}^{I^{\alpha, \beta}}, \Cstar(W^{\beta})]$. 
		
		\begin{lem}\label{lem:Falphabeta_exact_continuous_module}
			The functor 
			\[(F_{\alpha}^{\beta}; \varphi_{\alpha}^{\beta}) \colon \Cstar(W^{\alpha}) \to [\{0, 1\}^{I^{\alpha, \beta}}, \Cstar(W^{\beta})]\]
			is an exact, continuous, $\Cstar$-module functor. 
		\end{lem}
		
		\begin{proof}
			It suffices to show that $(F_{\alpha}^{\beta})_{\lambda} \colon \Cstar(W^{\alpha}) \to \Cstar(W^{\beta})$ is an exact, continuous, $\Cstar$-module functor for every $\lambda \in \{0, 1\}^{I^{\alpha, \beta}}$ and that $(\varphi_{\alpha}^{\beta})_{\lambda}^{\mu} \colon (F_{\alpha}^{\beta})_{\lambda} \Rightarrow (F_{\alpha}^{\beta})_{\mu}$ is a $\Cstar$-module natural transformation for every $(\lambda, \mu) \in R_{\{0, 1\}^{I^{\alpha, \beta}}}$; see \cref{rem:exact_continuous_module_iff}. 
			This follows by the same argument as in the proof of
			\cref{prop:reflection_functor_exact_continuous_module}.
		\end{proof}

		\begin{prop}\label{prop:Salphabeta_exact_continuous_module}
			The higher-dimensional reflection functor
			\[S_{\alpha}^{\beta} \colon \Cstar(W^{\alpha}) \to \Cstar(W^{\beta})\]
			is an exact, continuous, $\Cstar$-module functor. 
		\end{prop}
		
		\begin{proof}
			This follows from \cref{lem:preservation_exact_continuous_module} \labelcref{lem:preservation_exact_continuous_module_1} applied to the $I^{\alpha, \beta}$-dimensional mapping cone functor
			\[M^{I^{\alpha, \beta}} \colon [\{0, 1\}^{I^{\alpha, \beta}}, [\Cstar(W^{\alpha}), \Cstar(W^{\beta})]] \to [\Cstar(W^{\alpha}), \Cstar(W^{\beta})]\] 
			and from \cref{prop:mapping_cone_exact_continuous_module,lem:Falphabeta_exact_continuous_module}. 
		\end{proof}

		\subsection{Rearrangement of higher-dimensional reflection functors}\label{subsection:rearrangements_of_higher-dimensional_reflection_functors}
		
		In this subsection, we prove the following main theorem of this section.
		
		\begin{thm}\label{thm:rearrangement_reflection_functors}
			For $\alpha, \beta, \gamma \in \{\closed, \open\}^I$ with $I^{\alpha, \beta} \cap I^{\beta, \gamma}=\emptyset$, we have
			\[S_{\beta}^{\gamma} S_{\alpha}^{\beta}=S_{\alpha}^{\gamma}\]
			in $[\Cstar(W^{\alpha}), \Cstar(W^{\gamma})]$. 
		\end{thm}
		
		\cref{thm:rearrangement_reflection_functors} follows from \cref{lem:rearrangement_mapping_cones,lem:Falphabeta_exact_continuous_module,prop:Falphagamma_varphialphagamma} below. 
		Note that the assumption in \cref{thm:rearrangement_reflection_functors} is equivalent to $I^{\alpha, \beta} \cup I^{\beta, \gamma}=I^{\alpha, \gamma}$. Under this assumption, we have
		\[\{0, 1\}^{I^{\alpha, \beta}} \times \{0, 1\}^{I^{\beta, \gamma}}=\{0, 1\}^{I^{\alpha, \gamma}} \] 
		as categories; see \cref{lem:product_index_sets}. 
		
		\begin{prop}\label{prop:Falphagamma_varphialphagamma}
			For $\alpha, \beta, \gamma \in \{\closed, \open\}^I$ with $I^{\alpha, \beta} \cap I^{\beta, \gamma}=\emptyset$, we have 
			\[(F_{\beta}^{\gamma} F_{\alpha}^{\beta}; \varphi_{\beta}^{\gamma} \varphi_{\alpha}^{\beta})=(F_{\alpha}^{\gamma}; \varphi_{\alpha}^{\gamma})\]
			in $[\{0, 1\}^{I^{\alpha, \gamma}}, [\Cstar(W^{\alpha}), \Cstar(W^{\gamma})]]$. 
		\end{prop}
		
		In what follows, we shall prove \cref{prop:Falphagamma_varphialphagamma}. 
		The following elementary lemma reduces the amount of computation needed in the proof of \cref{prop:Falphagamma_varphialphagamma}.
		We will apply it with $F=I^{\alpha, \gamma}$ and $\frakA=[\Cstar(W^{\alpha}), \Cstar(W^{\gamma})]$. 
		
		\begin{lem}\label{lem:equality_diagram}
			Let $F$ be a finite set. Let $\frakA$ be a category with a zero object $0$. Let 
			\[(A; \varphi), (B; \psi) \in [\{0, 1\}^F, \frakA].\]
			Suppose that 
			\begin{itemize}
				\item $A_{\lambda}=B_{\lambda}$ for every $\lambda \in \{0, 1\}^F$; 
				\item $\varphi_{\lambda}^{\mu}=\psi_{\lambda}^{\mu}$ for every $(\lambda, \mu) \in R_{\{0, 1\}^F}$ satisfying $A_{\lambda} \neq 0$, $A_{\mu} \neq 0$, and $|\{i \in F \mid \lambda_i \neq \mu_i\}|=1$.
			\end{itemize}
			Then $(A; \varphi)=(B; \psi)$. 
		\end{lem}
		
		\begin{proof}
			It suffices to show that $\varphi_{\lambda}^{\mu}=\psi_{\lambda}^{\mu}$ for every $(\lambda, \mu) \in R_{\{0, 1\}^F}$. 
			If $\lambda=\mu$, then $\varphi_{\lambda}^{\mu}$ and $\psi_{\lambda}^{\mu}$ are the identity morphism on $A_{\lambda}$, and hence $\varphi_{\lambda}^{\mu}=\psi_{\lambda}^{\mu}$. Suppose that $\lambda \neq \mu$. 
			Choose a sequence
			\[
			\lambda=\lambda_0, \lambda_1, \dots, \lambda_n=\mu,
			\]
			in $\{0, 1\}^F$ such that
			\[
			(\lambda_{k-1}, \lambda_k) \in R_{\{0, 1\}^F}, \quad \bigl|\{i \in F \mid \lambda_{k-1}(i) \neq \lambda_k(i)\}\bigr|=1,
			\]
			for every $k=1, 2, \dots, n$.
			For each $k=1, 2, \dots, n$, if both $A_{\lambda_{k-1}} \neq 0$ and $A_{\lambda_k} \neq 0$, then $\varphi_{\lambda_{k-1}}^{\lambda_k}=\psi_{\lambda_{k-1}}^{\lambda_k}$ by assumption.
			If either $A_{\lambda_{k-1}} = 0$ or $A_{\lambda_k} = 0$, the same equality holds because a morphism is unique whenever either its source or its target is a zero object. 
			Therefore,
			\[
			\varphi_{\lambda_{k-1}}^{\lambda_k}
			=
			\psi_{\lambda_{k-1}}^{\lambda_k}
			\]
			for every $k=1, 2, \dots, n$.
			This, together with the functoriality of $(A; \varphi)$ and $(B; \psi)$, shows that
			\[\varphi_{\lambda}^{\mu}
			=
			\varphi_{\lambda_{n-1}}^{\lambda_n}
			\circ \cdots \circ
			\varphi_{\lambda_0}^{\lambda_1}
			=
			\psi_{\lambda_{n-1}}^{\lambda_n}
			\circ \cdots \circ
			\psi_{\lambda_0}^{\lambda_1}
			=
			\psi_{\lambda}^{\mu}.
			\]
			This completes the proof.
		\end{proof}
		
		Fix $\alpha, \beta, \gamma \in \{\closed, \open\}^I$ satisfying $I^{\alpha, \beta} \cap I^{\beta, \gamma}=\emptyset$. We may assume that $\alpha \neq \beta$ and $\beta \neq \gamma$. Then $\alpha \neq \gamma$.  
		
		\begin{lem}\label{lem:ipri}
			For $i \in I^{\alpha, \beta}$ and $\bfj \in J^{\beta, \gamma}$, we have
			\[
			i_{L_{\bfj}^{\gamma}}^{W^{\gamma}}
			p_{W_{\bfj}^{\beta}}^{L_{\bfj}^{\gamma}}
			r_{W^{\beta}}^{W_{\bfj}^{\beta}}
			i_{K_i}^{W^{\beta}}
			=
			\begin{dcases}
				i_{K_i}^{W^{\gamma}} & \text{if $i \in \bfj$},\\
				0 & \text{if $i \notin \bfj$}.
			\end{dcases}
			\]
		\end{lem}
		
		\begin{proof}
			By \cref{lem:relations_extension_restriction_functors}
			\labelcref{lem:relations_extension_restriction_functors_1,lem:relations_extension_restriction_functors_2,lem:relations_extension_restriction_functors_5}, we have
			\[
			r_{W^{\beta}}^{W_{\bfj}^{\beta}}
			i_{K_i}^{W^{\beta}}
			=
			i_{K_i \cap W_{\bfj}^{\beta}}^{W_{\bfj}^{\beta}}
			r_{K_i}^{K_i \cap W_{\bfj}^{\beta}}
			=
			\begin{dcases}
				i_{K_i}^{W_{\bfj}^{\beta}} & \text{if $i \in \bfj$},\\
				0 & \text{if $i \notin \bfj$}.
			\end{dcases}
			\]
			Thus, if $i \notin \bfj$, then
			\[
			i_{L_{\bfj}^{\gamma}}^{W^{\gamma}}
			p_{W_{\bfj}^{\beta}}^{L_{\bfj}^{\gamma}}
			r_{W^{\beta}}^{W_{\bfj}^{\beta}}
			i_{K_i}^{W^{\beta}}
			=0. 
			\]
			If $i \in \bfj$, then the composite map
			\[
			K_i \hookrightarrow W_{\bfj}
			\xrightarrow{\widetilde{\theta}_{\bfj}}
			L_{\bfj}
			\hookrightarrow W
			\]
			is the inclusion map $K_i \hookrightarrow W$.
			Therefore, by \cref{lem:relations_extension_restriction_functors}
			\labelcref{lem:relations_extension_restriction_functors_3}, we obtain
			\[
			i_{L_{\bfj}^{\gamma}}^{W^{\gamma}}
			p_{W_{\bfj}^{\beta}}^{L_{\bfj}^{\gamma}}
			i_{K_i}^{W_{\bfj}^{\beta}}
			=
			i_{K_i}^{W^{\gamma}}.
			\]
			This proves the case $i \in \bfj$ and completes the proof.
		\end{proof}
		
		\begin{lem}\label{lem:ripr}
			For $\bfj \in J^{\alpha, \beta}$ and $i \in I^{\beta, \gamma}$, we have
			\[
			r_{W^{\beta}}^{K_i}
			i_{L_{\bfj}^{\beta}}^{W^{\beta}}
			p_{W_{\bfj}^{\alpha}}^{L_{\bfj}^{\beta}}
			r_{W^{\alpha}}^{W_{\bfj}^{\alpha}}
			=
			\begin{dcases}
				r_{W^{\alpha}}^{K_i} & \text{if $i \in \bfj$},\\
				0 & \text{if $i \notin \bfj$}.
			\end{dcases}
			\]
		\end{lem}
		
		\begin{proof}
			By \cref{lem:relations_extension_restriction_functors}
			\labelcref{lem:relations_extension_restriction_functors_1,lem:relations_extension_restriction_functors_2,lem:relations_extension_restriction_functors_5}, we have
			\[
			r_{W^{\beta}}^{K_i}
			i_{L_{\bfj}^{\beta}}^{W^{\beta}}
			=
			i_{K_i \cap L_{\bfj}^{\beta}}^{K_i}
			r_{L_{\bfj}^{\beta}}^{K_i \cap L_{\bfj}^{\beta}}
			=
			\begin{dcases}
				r_{L_{\bfj}^{\beta}}^{K_i} & \text{if $i \in \bfj$},\\
				0 & \text{if $i \notin \bfj$}.
			\end{dcases}
			\]
			Thus, if $i \notin \bfj$, then
			\[
			r_{W^{\beta}}^{K_i}
			i_{L_{\bfj}^{\beta}}^{W^{\beta}}
			p_{W_{\bfj}^{\alpha}}^{L_{\bfj}^{\beta}}
			r_{W^{\alpha}}^{W_{\bfj}^{\alpha}}
			=0.
			\]
			
			Suppose that $i \in \bfj$.
			Then $\widetilde{\theta}_{\bfj} \colon W_{\bfj} \to L_{\bfj}$ restricts to the identity map on $K_i$.
			Hence, by \cref{lem:relations_extension_restriction_functors}
			\labelcref{lem:relations_extension_restriction_functors_5}, we have
			\[
			r_{L_{\bfj}^{\beta}}^{K_i}
			p_{W_{\bfj}^{\alpha}}^{L_{\bfj}^{\beta}}
			=
			r_{W_{\bfj}^{\alpha}}^{K_i}.
			\]
			By \cref{lem:relations_extension_restriction_functors}
			\labelcref{lem:relations_extension_restriction_functors_4}, we also have
			\[
			r_{W_{\bfj}^{\alpha}}^{K_i}
			r_{W^{\alpha}}^{W_{\bfj}^{\alpha}}
			=
			r_{W^{\alpha}}^{K_i}.
			\]
			Combining these identities gives
			\[
			r_{W^{\beta}}^{K_i}
			i_{L_{\bfj}^{\beta}}^{W^{\beta}}
			p_{W_{\bfj}^{\alpha}}^{L_{\bfj}^{\beta}}
			r_{W^{\alpha}}^{W_{\bfj}^{\alpha}}
			=
			r_{W^{\alpha}}^{K_i}.
			\]
			This proves the case $i \in \bfj$ and completes the proof.
		\end{proof}

		\begin{lem}\label{lem:Jalphagamma}
			For each $j \in J$, we have
			\[[j]^{\alpha, \beta} \cap [j]^{\beta, \gamma}=[j]^{\alpha, \gamma}. \]
			Moreover, we have
			\[J^{\alpha, \gamma}=\{\bfj \cap \bfj' \mid (\bfj, \bfj') \in J^{\alpha, \beta} \times J^{\beta, \gamma}, \ \bfj \cap \bfj' \neq \emptyset\}. \]
		\end{lem}
		
		\begin{proof}
			This follows because $G$ is a tree and $I \setminus I^{\alpha, \gamma}=(I \setminus I^{\alpha, \beta}) \cap (I \setminus I^{\beta, \gamma})$. 
		\end{proof}

		For $\bfj \in J^{\alpha, \beta}$, we have $J_i \subset \bfj$ for all $i \in \bfj \cap I$, and hence \cref{lem:W_intersection} can be applied.  
		The following lemma is straightforward to verify. 
		
		\begin{lem}\label{lem:composition_thetaj}
			Let $\bfj \in J^{\alpha, \beta}$ and $\bfj' \in J^{\beta, \gamma}$ with $\bfj \cap \bfj' \neq \emptyset$. By \textup{\cref{lem:W_intersection} \labelcref{lem:W_intersection_2}}, the map $\widetilde{\theta}_{\bfj} \colon W_{\bfj} \to L_{\bfj}$ restricts to a map $W_{\bfj \cap \bfj'} \to L_{\bfj} \cap W_{\bfj'}$, which we also denote by $\widetilde{\theta}_{\bfj}$. The two composites
			\[W_{\bfj \cap \bfj'} \xrightarrow{\widetilde{\theta}_{\bfj}} L_{\bfj} \cap W_{\bfj'}\hookrightarrow W_{\bfj'} \xrightarrow{\widetilde{\theta}_{\bfj'}} L_{\bfj'} \hookrightarrow W\]
			and
			\[W_{\bfj \cap \bfj'} \xrightarrow{\widetilde{\theta}_{\bfj \cap \bfj'}} L_{\bfj \cap \bfj'} \hookrightarrow W \]
			are equal.
		\end{lem}
		
		\begin{lem}\label{lem:ipripr}
			For $\bfj \in J^{\alpha, \beta}$ and $\bfj' \in J^{\beta, \gamma}$, we have
			\[
			i_{L_{\bfj'}^{\gamma}}^{W^{\gamma}}
			p_{W_{\bfj'}^{\beta}}^{L_{\bfj'}^{\gamma}}
			r_{W^{\beta}}^{W_{\bfj'}^{\beta}}
			i_{L_{\bfj}^{\beta}}^{W^{\beta}}
			p_{W_{\bfj}^{\alpha}}^{L_{\bfj}^{\beta}}
			r_{W^{\alpha}}^{W_{\bfj}^{\alpha}}
			=
			\begin{dcases}
				i_{L_{\bfj \cap \bfj'}^{\gamma}}^{W^{\gamma}}
				p_{W_{\bfj \cap \bfj'}^{\alpha}}^{L_{\bfj \cap \bfj'}^{\gamma}}
				r_{W^{\alpha}}^{W_{\bfj \cap \bfj'}^{\alpha}}
				& \text{if $\bfj \cap \bfj' \neq \emptyset$},\\
				0
				& \text{if $\bfj \cap \bfj'=\emptyset$}.
			\end{dcases}
			\]
		\end{lem}
		
		\begin{proof}
			By \cref{lem:relations_extension_restriction_functors}
			\labelcref{lem:relations_extension_restriction_functors_5}, we have
			\[
			r_{W^{\beta}}^{W_{\bfj'}^{\beta}}
			i_{L_{\bfj}^{\beta}}^{W^{\beta}}
			=
			i_{L_{\bfj}^{\beta} \cap W_{\bfj'}^{\beta}}^{W_{\bfj'}^{\beta}}
			r_{L_{\bfj}^{\beta}}^{L_{\bfj}^{\beta} \cap W_{\bfj'}^{\beta}}.
			\]
			If $\bfj \cap \bfj'=\emptyset$, then $L_{\bfj} \cap W_{\bfj'}=\emptyset$ by
			\cref{lem:W_intersection} \labelcref{lem:W_intersection_1}.
			Hence,
			\[
			i_{L_{\bfj}^{\beta} \cap W_{\bfj'}^{\beta}}^{W_{\bfj'}^{\beta}}
			r_{L_{\bfj}^{\beta}}^{L_{\bfj}^{\beta} \cap W_{\bfj'}^{\beta}}
			=0.
			\]
			This proves the case $\bfj \cap \bfj'=\emptyset$.
			
			It remains to consider the case $\bfj \cap \bfj' \neq \emptyset$.
			Let
			\[
			p_{W_{\bfj \cap \bfj'}^{\alpha}}^{L_{\bfj}^{\beta} \cap W_{\bfj'}^{\beta}}
			\colon
			\Cstar(W_{\bfj \cap \bfj'}^{\alpha})
			\to
			\Cstar(L_{\bfj}^{\beta} \cap W_{\bfj'}^{\beta})
			\]
			denote the functor induced by the continuous map
			$\widetilde{\theta}_{\bfj} \colon
			W_{\bfj \cap \bfj'}^{\alpha} \to
			L_{\bfj}^{\beta} \cap W_{\bfj'}^{\beta}$.
			By \cref{lem:relations_extension_restriction_functors}
			\labelcref{lem:relations_extension_restriction_functors_5}
			and \cref{lem:W_intersection} \labelcref{lem:W_intersection_2}, we have
			\[
			r_{L_{\bfj}^{\beta}}^{L_{\bfj}^{\beta} \cap W_{\bfj'}^{\beta}}
			p_{W_{\bfj}^{\alpha}}^{L_{\bfj}^{\beta}}
			=
			p_{W_{\bfj \cap \bfj'}^{\alpha}}^{L_{\bfj}^{\beta} \cap W_{\bfj'}^{\beta}}
			r_{W_{\bfj}^{\alpha}}^{W_{\bfj \cap \bfj'}^{\alpha}}.
			\]
			By \cref{lem:relations_extension_restriction_functors}
			\labelcref{lem:relations_extension_restriction_functors_3}
			and \cref{lem:composition_thetaj}, we obtain
			\[
			i_{L_{\bfj'}^{\gamma}}^{W^{\gamma}}
			p_{W_{\bfj'}^{\beta}}^{L_{\bfj'}^{\gamma}}
			i_{L_{\bfj}^{\beta} \cap W_{\bfj'}^{\beta}}^{W_{\bfj'}^{\beta}}
			p_{W_{\bfj \cap \bfj'}^{\alpha}}^{L_{\bfj}^{\beta} \cap W_{\bfj'}^{\beta}}
			=
			i_{L_{\bfj \cap \bfj'}^{\gamma}}^{W^{\gamma}}
			p_{W_{\bfj \cap \bfj'}^{\alpha}}^{L_{\bfj \cap \bfj'}^{\gamma}}.
			\]
			Finally, by \cref{lem:relations_extension_restriction_functors}
			\labelcref{lem:relations_extension_restriction_functors_4}, we have
			\[
			r_{W_{\bfj}^{\alpha}}^{W_{\bfj \cap \bfj'}^{\alpha}}
			r_{W^{\alpha}}^{W_{\bfj}^{\alpha}}
			=
			r_{W^{\alpha}}^{W_{\bfj \cap \bfj'}^{\alpha}}.
			\]
			Combining the preceding identities gives
			\[
			i_{L_{\bfj'}^{\gamma}}^{W^{\gamma}}
			p_{W_{\bfj'}^{\beta}}^{L_{\bfj'}^{\gamma}}
			r_{W^{\beta}}^{W_{\bfj'}^{\beta}}
			i_{L_{\bfj}^{\beta}}^{W^{\beta}}
			p_{W_{\bfj}^{\alpha}}^{L_{\bfj}^{\beta}}
			r_{W^{\alpha}}^{W_{\bfj}^{\alpha}}
			=
			i_{L_{\bfj \cap \bfj'}^{\gamma}}^{W^{\gamma}}
			p_{W_{\bfj \cap \bfj'}^{\alpha}}^{L_{\bfj \cap \bfj'}^{\gamma}}
			r_{W^{\alpha}}^{W_{\bfj \cap \bfj'}^{\alpha}}.
			\]
			This proves the case $\bfj \cap \bfj' \neq \emptyset$ and completes the proof.
		\end{proof}

		\begin{lem}\label{lem:tensor_Di_Dj}
			\
			\begin{enumerate}[label=\textnormal{(\arabic*)}]
				\item \label{lem:tensor_Di_Dj_1}
				For $i \in I^{\alpha, \beta}$ and $\bfj \in J^{\beta, \gamma}$ with $i \in \bfj$, we have
				\[
				D_i^{\alpha, \beta} \otimes D_{\bfj}^{\beta, \gamma}
				=
				D_i^{\alpha, \gamma}.
				\]
				
				\item \label{lem:tensor_Di_Dj_2}
				For $\bfj \in J^{\alpha, \beta}$ and $i \in I^{\beta, \gamma}$ with $i \in \bfj$, we have
				\[
				D_{\bfj}^{\alpha, \beta} \otimes D_i^{\beta, \gamma}
				=
				D_i^{\alpha, \gamma}.
				\]
				
				\item \label{lem:tensor_Di_Dj_3}
				For $\bfj \in J^{\alpha, \beta}$ and $\bfj' \in J^{\beta, \gamma}$ with $\bfj \cap \bfj' \neq \emptyset$, we have
				\[
				D_{\bfj}^{\alpha, \beta} \otimes D_{\bfj'}^{\beta, \gamma}
				=
				D_{\bfj \cap \bfj'}^{\alpha, \gamma}.
				\]
			\end{enumerate}
		\end{lem}
		
		\begin{proof}
			Since $I^{\alpha, \beta} \cap I^{\beta, \gamma}=\emptyset$, we have
			\[
			I^{\alpha, \beta}_{\closed, \open}
			\cap
			I^{\beta, \gamma}_{\closed, \open}
			=
			\emptyset,
			\qquad
			I^{\alpha, \beta}_{\closed, \open}
			\cup
			I^{\beta, \gamma}_{\closed, \open}
			=
			I^{\alpha, \gamma}_{\closed, \open}.
			\]
			We use these identities for the index sets appearing in the tensor products below.
			
			\begin{enumerate}[label=\textnormal{(\arabic*)}]
				\item
				By \cref{lem:j_i_bfj}, we have $j_{k,\bfj}=j_{k,\{i\}}$ for all $k \in I^{\beta, \gamma}_{\closed, \open}$ because $i \in \bfj$.
				If $i \in I^{\alpha, \beta}_{\open, \closed}$, then
				\begin{align*}
					D_i^{\alpha, \beta} \otimes D_{\bfj}^{\beta, \gamma}
					&=
					\bigotimes_{k \in I^{\alpha, \beta}_{\closed, \open}}
					D_{k,j_{k,\{i\}}}
					\otimes
					\bigotimes_{k \in I^{\beta, \gamma}_{\closed, \open}}
					D_{k,j_{k,\{i\}}} \\
					&=
					\bigotimes_{k \in I^{\alpha, \gamma}_{\closed, \open}}
					D_{k,j_{k,\{i\}}}
					=
					D_i^{\alpha, \gamma}.
				\end{align*}
				If $i \in I^{\alpha, \beta}_{\closed, \open}$, then
				\begin{align*}
					D_i^{\alpha, \beta} \otimes D_{\bfj}^{\beta, \gamma}
					&=
					D_i
					\otimes
					\bigotimes_{k \in I^{\alpha, \beta}_{\closed, \open} \setminus \{i\}}
					D_{k,j_{k,\{i\}}}
					\otimes
					\bigotimes_{k \in I^{\beta, \gamma}_{\closed, \open}}
					D_{k,j_{k,\{i\}}} \\
					&=
					D_i
					\otimes
					\bigotimes_{k \in I^{\alpha, \gamma}_{\closed, \open} \setminus \{i\}}
					D_{k,j_{k,\{i\}}}
					=
					D_i^{\alpha, \gamma}.
				\end{align*}
				
				\item
				By symmetry, \labelcref{lem:tensor_Di_Dj_2} follows from
				\labelcref{lem:tensor_Di_Dj_1}.
				
				\item
				By \cref{lem:j_i_bfj}, we have $j_{k,\bfj}=j_{k,\bfj \cap \bfj'}$ for all $k \in I^{\alpha, \beta}_{\closed, \open}$, and $j_{k,\bfj'}=j_{k,\bfj \cap \bfj'}$ for all $k \in I^{\beta, \gamma}_{\closed, \open}$.
				Thus,
				\begin{align*}
					D_{\bfj}^{\alpha, \beta} \otimes D_{\bfj'}^{\beta, \gamma}
					&=
					\bigotimes_{k \in I^{\alpha, \beta}_{\closed, \open}}
					D_{k,j_{k,\bfj \cap \bfj'}}
					\otimes
					\bigotimes_{k \in I^{\beta, \gamma}_{\closed, \open}}
					D_{k,j_{k,\bfj \cap \bfj'}} \\
					&=
					\bigotimes_{k \in I^{\alpha, \gamma}_{\closed, \open}}
					D_{k,j_{k,\bfj \cap \bfj'}}
					=
					D_{\bfj \cap \bfj'}^{\alpha, \gamma}.
				\end{align*}
			\end{enumerate}
			This completes the proof.
		\end{proof}

		\begin{proof}[Proof of \textup{\cref{prop:Falphagamma_varphialphagamma}}]
			We first prove that
			\[(F_{\beta}^{\gamma})_{\lambda'} (F_{\alpha}^{\beta})_{\lambda}=(F_{\alpha}^{\gamma})_{(\lambda, \lambda')}\]
			for $\lambda \in \{0, 1\}^{I^{\alpha, \beta}}$ and $\lambda' \in \{0, 1\}^{I^{\beta, \gamma}}$. 
			\begin{description}[font=\normalfont\scshape]
				\item[Case 1] $\lambda=\chi_i^{\alpha, \beta}$ for some $i \in I^{\alpha, \beta}$, and $\lambda'=\chi_{i'}^{\beta, \gamma}$ for some $i' \in I^{\beta, \gamma}$. Since $K_i \cap K_{i'}=\emptyset$, \cref{lem:relations_extension_restriction_functors} \labelcref{lem:relations_extension_restriction_functors_1,lem:relations_extension_restriction_functors_5} show that
				\[r_{W^{\beta}}^{K_{i'}} i_{K_i}^{W^{\beta}}=0. \] 
				It follows that
				\begin{align*}
					(F_{\beta}^{\gamma})_{\lambda'} (F_{\alpha}^{\beta})_{\lambda}
					&=i_{K_{i'}}^{W^{\gamma}} r_{W^{\beta}}^{K_{i'}} D_{i'}^{\beta, \gamma} i_{K_i}^{W^{\beta}} r_{W^{\alpha}}^{K_i} D_i^{\alpha, \beta} \\
					&=i_{K_{i'}}^{W^{\gamma}} r_{W^{\beta}}^{K_{i'}} i_{K_i}^{W^{\beta}} r_{W^{\alpha}}^{K_i} D_{i'}^{\beta, \gamma} D_i^{\alpha, \beta} \\
					&=0=(F_{\alpha}^{\gamma})_{(\lambda, \lambda')}. 
				\end{align*}
				
				\item[Case 2] $\lambda=\chi_i^{\alpha, \beta}$ for some $i \in I^{\alpha, \beta}$, and $\lambda'=\chi^{\beta, \gamma}$. Then $(\lambda, \lambda')=\chi_i^{\alpha, \gamma}$ and 
				\begin{align*}
					(F_{\beta}^{\gamma})_{\lambda'} (F_{\alpha}^{\beta})_{\lambda}
					&=\bigoplus_{\bfj \in J^{\beta, \gamma}} i_{L_{\bfj}^{\gamma}}^{W^{\gamma}} p_{W_{\bfj}^{\beta}}^{L_{\bfj}^{\gamma}} r_{W^{\beta}}^{W_{\bfj}^{\beta}} D_{\bfj}^{\beta, \gamma} i_{K_i}^{W^{\beta}} r_{W^{\alpha}}^{K_i} D_i^{\alpha, \beta} \\
					&=\bigoplus_{\bfj \in J^{\beta, \gamma}} i_{L_{\bfj}^{\gamma}}^{W^{\gamma}} p_{W_{\bfj}^{\beta}}^{L_{\bfj}^{\gamma}} r_{W^{\beta}}^{W_{\bfj}^{\beta}} i_{K_i}^{W^{\beta}} r_{W^{\alpha}}^{K_i} D_{\bfj}^{\beta, \gamma} D_i^{\alpha, \beta} \\
					&=i_{K_i}^{W^{\gamma}} r_{W^{\alpha}}^{K_i} D_i^{\alpha, \gamma}=(F_{\alpha}^{\gamma})_{(\lambda, \lambda')}. 
				\end{align*}
				The third equality follows from \cref{lem:ipri,lem:tensor_Di_Dj} \labelcref{lem:tensor_Di_Dj_1} and the fact that there exists a unique element $\bfj \in J^{\beta, \gamma}$ such that $i \in \bfj$. 
				
				\item[Case 3] $\lambda=\chi^{\alpha, \beta}$ and $\lambda'=\chi_i^{\beta, \gamma}$ for some $i \in I^{\beta, \gamma}$. 
				Then $(\lambda, \lambda')=\chi_i^{\alpha, \gamma}$ and 
				\begin{align*}
					(F_{\beta}^{\gamma})_{\lambda'} (F_{\alpha}^{\beta})_{\lambda}
					&=i_{K_i}^{W^{\gamma}} r_{W^{\beta}}^{K_i} D_i^{\beta, \gamma} \bigoplus_{\bfj \in J^{\alpha, \beta}} i_{L_{\bfj}^{\beta}}^{W^{\beta}} p_{W_{\bfj}^{\alpha}}^{L_{\bfj}^{\beta}} r_{W^{\alpha}}^{W_{\bfj}^{\alpha}} D_{\bfj}^{\alpha, \beta} \\
					&=\bigoplus_{\bfj \in J^{\alpha, \beta}} i_{K_i}^{W^{\gamma}} r_{W^{\beta}}^{K_i} i_{L_{\bfj}^{\beta}}^{W^{\beta}} p_{W_{\bfj}^{\alpha}}^{L_{\bfj}^{\beta}} r_{W^{\alpha}}^{W_{\bfj}^{\alpha}} D_i^{\beta, \gamma} D_{\bfj}^{\alpha, \beta} \\
					&=i_{K_i}^{W^{\gamma}} r_{W^{\alpha}}^{K_i} D_i^{\alpha, \gamma}=(F_{\alpha}^{\gamma})_{(\lambda, \lambda')}. 
				\end{align*}
				The third equality follows from \cref{lem:ripr,lem:tensor_Di_Dj} \labelcref{lem:tensor_Di_Dj_2} and the fact that there exists a unique element $\bfj \in J^{\alpha, \beta}$ such that $i \in \bfj$. 
				
				\item[Case 4] $\lambda=\chi^{\alpha, \beta}$ and $\lambda'=\chi^{\beta, \gamma}$. Then $(\lambda, \lambda')=\chi^{\alpha, \gamma}$ and 
				\begin{align*}
					(F_{\beta}^{\gamma})_{\lambda'} (F_{\alpha}^{\beta})_{\lambda}
					&=\bigoplus_{\bfj' \in J^{\beta, \gamma}} i_{L_{\bfj'}^{\gamma}}^{W^{\gamma}} p_{W_{\bfj'}^{\beta}}^{L_{\bfj'}^{\gamma}} r_{W^{\beta}}^{W_{\bfj'}^{\beta}} D_{\bfj'}^{\beta, \gamma} \bigoplus_{\bfj \in J^{\alpha, \beta}} i_{L_{\bfj}^{\beta}}^{W^{\beta}} p_{W_{\bfj}^{\alpha}}^{L_{\bfj}^{\beta}} r_{W^{\alpha}}^{W_{\bfj}^{\alpha}} D_{\bfj}^{\alpha, \beta} \\
					&=\bigoplus_{(\bfj, \bfj') \in J^{\alpha, \beta} \times J^{\beta, \gamma}} i_{L_{\bfj'}^{\gamma}}^{W^{\gamma}} p_{W_{\bfj'}^{\beta}}^{L_{\bfj'}^{\gamma}} r_{W^{\beta}}^{W_{\bfj'}^{\beta}} i_{L_{\bfj}^{\beta}}^{W^{\beta}} p_{W_{\bfj}^{\alpha}}^{L_{\bfj}^{\beta}} r_{W^{\alpha}}^{W_{\bfj}^{\alpha}} D_{\bfj'}^{\beta, \gamma} D_{\bfj}^{\alpha, \beta} \\
					&=\bigoplus_{\substack{(\bfj, \bfj') \in J^{\alpha, \beta} \times J^{\beta, \gamma} \\ \bfj \cap \bfj'\neq \emptyset}} i_{L_{\bfj \cap \bfj'}^{\gamma}}^{W^{\gamma}} p_{W_{\bfj \cap \bfj'}^{\alpha}}^{L_{\bfj \cap \bfj'}^{\gamma}} r_{W^{\alpha}}^{W_{\bfj \cap \bfj'}^{\alpha}} D_{\bfj \cap \bfj'}^{\alpha, \gamma} \\
					&=\bigoplus_{\bfj'' \in J^{\alpha, \gamma}} i_{L_{\bfj''}^{\gamma}}^{W^{\gamma}} p_{W_{\bfj''}^{\alpha}}^{L_{\bfj''}^{\gamma}} r_{W^{\alpha}}^{W_{\bfj''}^{\alpha}} D_{\bfj''}^{\alpha, \gamma}=(F_{\alpha}^{\gamma})_{(\lambda, \lambda')}.
				\end{align*}
				The third equality follows from \cref{lem:ipripr,lem:tensor_Di_Dj} \labelcref{lem:tensor_Di_Dj_3}. The fourth equality follows from \cref{lem:Jalphagamma}. 
				
				\item[Case 5] Otherwise. We have \[(F_{\beta}^{\gamma})_{\lambda'} (F_{\alpha}^{\beta})_{\lambda}=0=(F_{\alpha}^{\gamma})_{(\lambda, \lambda')}. \]
			\end{description}
			We have proved that $(F_{\beta}^{\gamma})_{\lambda'} (F_{\alpha}^{\beta})_{\lambda}=(F_{\alpha}^{\gamma})_{(\lambda, \lambda')}$ for all $\lambda \in \{0, 1\}^{I^{\alpha, \beta}}$ and $\lambda' \in \{0, 1\}^{I^{\beta, \gamma}}$. 
			
			By \cref{lem:equality_diagram}, it remains only to prove that
			\[
			(\varphi_{\beta}^{\gamma})_{\lambda'}^{\mu'} (\varphi_{\alpha}^{\beta})_{\lambda}^{\mu}
			=
			(\varphi_{\alpha}^{\gamma})_{(\lambda, \lambda')}^{(\mu, \mu')}
			\]
			for the following four cases of $(\lambda, \mu) \in R_{\{0, 1\}^{I^{\alpha, \beta}}}$ and $(\lambda', \mu') \in R_{\{0, 1\}^{I^{\beta, \gamma}}}$.
			\begin{description}[font=\normalfont\scshape]
				\item[Case 1] $\lambda=\chi_i^{\alpha, \beta}$ for some $i \in I^{\alpha, \beta}_{\open, \closed}$, $\mu=\chi^{\alpha, \beta}$, and $\lambda'=\mu'=\chi^{\beta, \gamma}$.
				
				\item[Case 2] $\lambda=\chi^{\alpha, \beta}$, $\mu=\chi_i^{\alpha, \beta}$ for some $i \in I^{\alpha, \beta}_{\closed, \open}$, and $\lambda'=\mu'=\chi^{\beta, \gamma}$. 
				
				\item[Case 3] $\lambda=\mu=\chi^{\alpha, \beta}$, $\lambda'=\chi_i^{\beta, \gamma}$ for some $i \in I^{\beta, \gamma}_{\open, \closed}$, and $\mu'=\chi^{\beta, \gamma}$. 
				
				\item[Case 4] $\lambda=\mu=\chi^{\alpha, \beta}$, $\lambda'=\chi^{\beta, \gamma}$, and $\mu'=\chi_i^{\beta, \gamma}$ for some $i \in I^{\beta, \gamma}_{\closed, \open}$. 
			\end{description}
			Since the arguments for these cases are analogous, we only treat Case 1. Using the bijections $J_i \ni j \mapsto [j]^{\alpha,\beta} \in J_i^{\alpha,\beta}$  and $J_i \ni j \mapsto [j]^{\alpha,\gamma} \in J_i^{\alpha,\gamma}$ from \cref{lem:JOmega_i}, we reindex the families defining $(\varphi_{\alpha}^{\beta})_{\lambda}^{\mu}$ and $(\varphi_{\alpha}^{\gamma})_{(\lambda,\lambda')}^{(\mu,\mu')}$ by $j \in J_i$, replacing $\bfj$ by $[j]^{\alpha,\beta}$ and $[j]^{\alpha,\gamma}$, respectively.
			In the following computation of the horizontal composite, we use the identifications established in Cases 2 and 4 of the preceding computation of the composite functors. 
			We have
			\begin{align*}
				(\varphi_{\beta}^{\gamma})_{\lambda'}^{\mu'} (\varphi_{\alpha}^{\beta})_{\lambda}^{\mu} 
				&=\bigoplus_{\bfj \in J^{\beta, \gamma}} i_{L_{\bfj}^{\gamma}}^{W^{\gamma}} p_{W_{\bfj}^{\beta}}^{L_{\bfj}^{\gamma}} r_{W^{\beta}}^{W_{\bfj}^{\beta}} D_{\bfj}^{\beta, \gamma} \Bigl(\iota_{K_i}^{W_{[j]^{\alpha, \beta}}^{\alpha}} \id_{i, [j]^{\alpha, \beta}}^{\alpha, \beta} \Bigr)_{j \in J_i} \\
				&=\Bigl(i_{L_{[j]^{\beta, \gamma}}^{\gamma}}^{W^{\gamma}} p_{W_{[j]^{\beta, \gamma}}^{\beta}}^{L_{[j]^{\beta, \gamma}}^{\gamma}} r_{W^{\beta}}^{W_{[j]^{\beta, \gamma}}^{\beta}} \iota_{K_i}^{W_{[j]^{\alpha, \beta}}^{\alpha}} D_{[j]^{\beta, \gamma}}^{\beta, \gamma} \id_{i, [j]^{\alpha, \beta}}^{\alpha, \beta} \Bigr)_{j \in J_i} \\
				&=\Bigl(\iota_{K_i}^{W_{[j]^{\alpha, \gamma}}^{\alpha}} \id_{i, [j]^{\alpha, \gamma}}^{\alpha, \gamma} \Bigr)_{j \in J_i}=(\varphi_{\alpha}^{\gamma})_{(\lambda, \lambda')}^{(\mu, \mu')}. 
			\end{align*}
			The second equality follows because for $j \in J_i$ and $\bfj \in J^{\beta, \gamma}$, we have $i \in \bfj$ if and only if $\bfj=[j]^{\beta, \gamma}$. 
			The third equality follows from \cref{lem:composition_theta_r_iota_pi} \labelcref{lem:composition_theta_r_iota_pi_2} \labelcref{lem:composition_theta_r_iota_pi_2_1} and \cref{lem:Jalphagamma}. 
			This completes the proof of \cref{prop:Falphagamma_varphialphagamma}.
		\end{proof}
		
		We have completed the proof of \cref{thm:rearrangement_reflection_functors}.

		\subsection{Application to homotopy-invariant, stable, half-exact functors}\label{subsection:application_to_Bott_functors_2}

		\begin{cor}\label{cor:Salphabeta_composition_reflection_functors}
			Let $\alpha, \beta \in \{\closed, \open\}^I$ with $n:=|I^{\alpha, \beta}| \geq 1$. 
			Let $f \colon \{1, 2, \dots, n\} \to I^{\alpha, \beta}$ be any bijection. Let 
			\[\alpha=\alpha_0, \alpha_1, \dots, \alpha_n=\beta\]
			be the unique sequence in $\{\closed, \open\}^I$ such that $I^{\alpha_{k-1}, \alpha_k}=\{f(k)\}$ for each $k=1, 2, \dots, n$. 
			Then 
			\[S_{\alpha_{n-1}}^{\alpha_n}S_{\alpha_{n-2}}^{\alpha_{n-1}} \cdots S_{\alpha_0}^{\alpha_1}=S_{\alpha}^{\beta} \]
			in $[\Cstar(W^{\alpha}), \Cstar(W^{\beta})]$. 
		\end{cor}
		
		\begin{proof}
			The proof proceeds by induction on $n$. The case $n=1$ is immediate. Assume that $n \geq 2$ and that the assertion holds for $n-1$.
			Applying the induction hypothesis to the pair $(\alpha_1, \beta)$, we obtain
			\[
			S_{\alpha_{n-1}}^{\alpha_n}
			S_{\alpha_{n-2}}^{\alpha_{n-1}}
			\cdots
			S_{\alpha_1}^{\alpha_2}
			=
			S_{\alpha_1}^{\beta}.
			\]
			Since $I^{\alpha,\alpha_1} \cap I^{\alpha_1,\beta}=\emptyset$, \cref{thm:rearrangement_reflection_functors} gives
			\[
			S_{\alpha_1}^{\beta}S_{\alpha}^{\alpha_1}
			=
			S_{\alpha}^{\beta}.
			\]
			This completes the induction.
		\end{proof}
		
		In the setting of \cref{cor:Salphabeta_composition_reflection_functors}, by \cref{prop:higher-dimensional_reflection_functor_singleton}, each higher-dimensional reflection functor
		\[S_{\alpha_{k-1}}^{\alpha_k} \colon \Cstar(W^{\alpha_{k-1}}) \to \Cstar(W^{\alpha_k})\]
		coincides with the reflection functor associated with the index $f(k)$. Since the bijection $f$ is arbitrary, the corollary shows that, for any ordering of the indices in $I^{\alpha, \beta}$, the composite of the corresponding reflection functors equals the higher-dimensional reflection functor $S_{\alpha}^{\beta}$.
		
		For $\alpha, \beta \in \{\closed, \open\}^I$, define a C*-algebra
		\[E^{\alpha, \beta}:=\bigotimes_{i \in I^{\alpha, \beta}} SD_i. \]
		Following \cref{nota:tensor_product}, we regard $E^{\alpha, \beta}$ as a functor.

		\begin{cor}\label{cor:higher-dimensional_reflection_Bott}
			Let $\alpha, \beta \in \{\closed, \open\}^I$. 
			Let $\frakA$ be an additive category and $F \colon \Cstar(W^{\alpha}) \to \frakA$ be a homotopy-invariant, stable, half-exact functor. If $FC_{\rho_{i, j}}=0$ in $[\Cstar(W^{\alpha}), \frakA]$ for every $i \in I$ and $j \in J_i$, then
			\[FS_{\beta}^{\alpha} S_{\alpha}^{\beta} \simeq FE^{\alpha, \beta} \]
			in $[\Cstar(W^{\alpha}), \frakA]$. 
		\end{cor}
		
		\begin{proof}
			The proof proceeds by induction on $n:=|I^{\alpha, \beta}|$. If $n=0$, equivalently, $\alpha=\beta$, then the functors $S_{\beta}^{\alpha} S_{\alpha}^{\beta}$ and $E^{\alpha, \beta}$ are the identity functor on $\Cstar(W^{\alpha})$. This shows the case $n=0$. 
			
			Assume that $n \geq 1$ and that the assertion holds for $n-1$. Choose $i_0 \in I^{\alpha, \beta}$. Let $\beta_0 \in \{\closed, \open\}^I$ be the unique element such that $I^{\beta_0, \beta}=\{i_0\}$. Then $I^{\alpha, \beta_0}=I^{\alpha, \beta}\setminus\{i_0\}$. 
			By \cref{thm:rearrangement_reflection_functors}, we have
			\[FS_{\beta}^{\alpha} S_{\alpha}^{\beta}=FS_{\beta_0}^{\alpha} S_{\beta}^{\beta_0} S_{\beta_0}^{\beta} S_{\alpha}^{\beta_0}. \]
			Since $I^{\beta_0, \beta}=\{i_0\}$, \cref{prop:higher-dimensional_reflection_functor_singleton} shows that the higher-dimensional reflection functors
			\[
			\begin{tikzcd}
				\Cstar(W^{\beta_0}) \ar[bend left=10]{r}{S_{\beta_0}^{\beta}}&
				\Cstar(W^{\beta}) \ar[bend left=10]{l}{S_{\beta}^{\beta_0}}
			\end{tikzcd}
			\]
			coincide with the corresponding reflection functors. 
			By \cref{lem:homotopy-invariant_stable_module,prop:Salphabeta_exact_continuous_module}, $FS_{\beta_0}^{\alpha} \colon \Cstar(W^{\beta_0}) \to \frakA$ is a homotopy-invariant, stable, half-exact functor. Moreover, $FS_{\beta_0}^{\alpha} C_{\rho_{i_0, j}}=FC_{\rho_{i_0, j}}S_{\beta_0}^{\alpha}=0$ in $[\Cstar(W^{\beta_0}), \frakA]$ for all $j \in J_{i_0}$. Therefore, \cref{thm:reflection_Bott} shows that
			\[FS_{\beta_0}^{\alpha} S_{\beta}^{\beta_0} S_{\beta_0}^{\beta} \simeq FS_{\beta_0}^{\alpha} SD_{i_0}\]
			in $[\Cstar(W^{\beta_0}), \frakA]$. 
			It follows that
			\[FS_{\beta}^{\alpha} S_{\alpha}^{\beta} \simeq FS_{\beta_0}^{\alpha} SD_{i_0}S_{\alpha}^{\beta_0}=FS_{\beta_0}^{\alpha}S_{\alpha}^{\beta_0} SD_{i_0}. \]
			Since $|I^{\alpha, \beta_0}|=n-1$, the induction hypothesis applied to the pair $(\alpha, \beta_0)$ gives
			\[FS_{\beta_0}^{\alpha}S_{\alpha}^{\beta_0} \simeq FE^{\alpha, \beta_0} \]
			in $[\Cstar(W^{\alpha}), \frakA]$. 
			Since $i_0 \notin I^{\alpha, \beta_0}$ and $I^{\alpha, \beta_0} \cup \{i_0\}=I^{\alpha, \beta}$, we have
			\[E^{\alpha, \beta_0} SD_{i_0}=E^{\alpha, \beta}. \]
			This completes the proof. 
		\end{proof}
		
		As in \cref{eg:reflection_suspension,eg:reflection_matrix}, we give two important examples of $(D_i)_{i \in I}$, $(D_{i, j})_{i \in I, j \in J_i}$, and $(\rho_{i, j})_{i \in I, j \in J_i}$ to which \cref{cor:higher-dimensional_reflection_Bott} applies.
		
		\begin{eg}\label{eg:higher-dimensional_reflection_suspension}
			Let $(h_{i, j})_{i \in I, j \in J_i}$ be a family of homeomorphisms $h_{i, j} \colon (0, 1) \to (0, 1)$ onto their images such that, for each $i \in I$, the images of $h_{i, j}$ for $j \in J_i$ are pairwise disjoint. For each $i \in I$, let $D_i:=S$, and for each $j \in J_i$, let $D_{i, j}:=S$ and $\rho_{i, j}:=h_{i, j*} \colon D_{i, j} \to D_i$. Then, for each $i \in I$, the $\ast$-homomorphisms $\rho_{i, j}$ for $j \in J_i$ are mutually orthogonal. Let $\alpha, \beta \in \{\closed, \open\}^I$, let $\frakA$ be an additive category, and let $F \colon \Cstar(W^{\alpha}) \to \frakA$ be a homotopy-invariant, stable, half-exact functor. Since $FC_{\rho_{i, j}}=0$ in $[\Cstar(W^{\alpha}), \frakA]$ for every $i \in I$ and $j \in J_i$, \cref{cor:higher-dimensional_reflection_Bott}, together with Bott periodicity, shows that
			\[FS_{\beta}^{\alpha} S_{\alpha}^{\beta} \simeq F\]
			in $[\Cstar(W^{\alpha}), \frakA]$. 
		\end{eg}

		\begin{eg}\label{eg:higher-dimensional_reflection_matrix}
			Assume that $J \neq \emptyset$. Since $G$ is a tree, $J_i \neq \emptyset$ for every $i \in I$. 
			For each $i \in I$, let $D_i:=\M_{J_i}$, and for each $j \in J_i$, let $D_{i, j}:=\C$ and $\rho_{i, j}:=\iota_{e_j} \colon D_{i, j} \to D_i$. Then, for each $i \in I$, the $\ast$-homomorphisms $\rho_{i, j}$ for $j \in J_i$ are mutually orthogonal. Let $\alpha, \beta \in \{\closed, \open\}^I$, let $\frakA$ be an additive category, and let $F \colon \Cstar(W^{\alpha}) \to \frakA$ be a homotopy-invariant, stable, half-exact functor. Since $FC_{\rho_{i, j}}=0$ in $[\Cstar(W^{\alpha}), \frakA]$ for every $i \in I$ and $j \in J_i$, \cref{cor:higher-dimensional_reflection_Bott}, together with the stability of $F$ and Bott periodicity, shows that
			\[FS_{\beta}^{\alpha} S_{\alpha}^{\beta} \simeq \left\{
			\begin{alignedat}{2}
				&F &  & \text{if $|I^{\alpha, \beta}|$ is even}, \\
				&FS & \qquad & \text{if $|I^{\alpha, \beta}|$ is odd}
			\end{alignedat}
			\right. \]
			in $[\Cstar(W^{\alpha}), \frakA]$. 
		\end{eg}

		\begin{rem}\label{rem:restriction_higher-dimensional_reflection_functor}
			Suppose that $(D_i)_{i \in I}$, $(D_{i, j})_{i \in I, j \in J_i}$, and $(\rho_{i, j})_{i \in I, j \in J_i}$ are as in either \cref{eg:higher-dimensional_reflection_suspension} or \cref{eg:higher-dimensional_reflection_matrix}. 
			Since $J^{\alpha, \beta}$ is finite and all the $D_i$ and $D_{i, j}$ are separable, the higher-dimensional reflection functor
			\[S_{\alpha}^{\beta} \colon \Cstar(W^{\alpha}) \to \Cstar(W^{\beta})\]
			restricts to a functor
			\[S_{\alpha}^{\beta} \colon \SCstar(W^{\alpha}) \to \SCstar(W^{\beta}).\]
			The resulting functor is an exact $\SCstar$-module functor that commutes with countable inductive limits. All results in \cref{subsection:rearrangements_of_higher-dimensional_reflection_functors,subsection:application_to_Bott_functors_2} involving this functor remain valid when the categories are replaced by their full subcategories consisting of separable C*-algebras.
			Since all the $D_i$ and $D_{i, j}$ are also nuclear in both examples, the same statements hold for the corresponding categories of nuclear C*-algebras or separable nuclear C*-algebras. 
		\end{rem}
		
		\subsection{Application to bivariant K-theory}\label{subsection:application_to_bivariant_K-theory_2}
		
		In this subsection, we apply higher-dimensional reflection functors to bivariant K-theory. 
		Suppose that all the $K_i$ for $i \in I$ and all the $L_j$ for $j \in J$ are finite preordered sets. Then $W^{\alpha}$ is a finite topological space for every $\alpha \in \{\closed, \open\}^I$. Suppose also that $(D_i)_{i \in I}$, $(D_{i, j})_{i \in I, j \in J_i}$, and $(\rho_{i, j})_{i \in I, j \in J_i}$ are as in \cref{eg:higher-dimensional_reflection_suspension}. 
		
		For $\alpha, \beta \in \{\closed, \open\}^I$, the restricted functor in \cref{rem:restriction_higher-dimensional_reflection_functor} induces, by \cref{lem:induced_functor_KK_E}, a triangulated functor
		\[S_{\alpha}^{\beta} \colon \KKcat(W^{\alpha}) \to \KKcat(W^{\beta}) \]
		that commutes with countable coproducts. This induced functor restricts to functors
		\[S_{\alpha}^{\beta} \colon \KKcat(W^{\alpha})_{\loc} \to \KKcat(W^{\beta})_{\loc}, \quad S_{\alpha}^{\beta} \colon \Boot(W^{\alpha}) \to \Boot(W^{\beta}). \]
		
		\begin{cor}\label{cor:equivalence_KKloc_Walpha_Wbeta}
			For $\alpha, \beta \in \{\closed, \open\}^I$, the higher-dimensional reflection functors
			\[
			\begin{tikzcd}
				\KKcat(W^{\alpha})_{\loc} \ar[bend left=10]{r}{S_{\alpha}^{\beta}} &
				\ar[bend left=10]{l}{S_{\beta}^{\alpha}} \KKcat(W^{\beta})_{\loc}
			\end{tikzcd}
			\]
			form an equivalence of triangulated categories with countable coproducts, which restricts to an equivalence between $\Boot(W^{\alpha})$ and $\Boot(W^{\beta})$. 
		\end{cor}
		
		\begin{proof}
			This follows from \cref{thm:equivalence_KKloc_Wc_Wo,cor:Salphabeta_composition_reflection_functors}.
		\end{proof}
		
		The same argument as for ideal-related KK-theory applies to E-theory. 
		
		\begin{cor}\label{cor:equivalence_E_Walpha_Wbeta}
			For $\alpha, \beta \in \{\closed, \open\}^I$, the higher-dimensional reflection functors
			\[
			\begin{tikzcd}
				\Ecat(W^{\alpha}) \ar[bend left=10]{r}{S_{\alpha}^{\beta}} &
				\ar[bend left=10]{l}{S_{\beta}^{\alpha}} \Ecat(W^{\beta})
			\end{tikzcd}
			\]
			defined in the same way as for ideal-related KK-theory form an equivalence of triangulated categories with countable coproducts, which restricts to an equivalence between $\Boot_{\E}(W^{\alpha})$ and $\Boot_{\E}(W^{\beta})$. 
		\end{cor}

		\begin{proof}
			This follows from \cref{thm:equivalence_E_Wc_Wo,cor:Salphabeta_composition_reflection_functors}. 
		\end{proof}
		
		\begin{rem}\label{rem:equivalence_E_Walpha_Wbeta_infinite}
			\cref{cor:equivalence_E_Walpha_Wbeta} remains valid when all the $K_i$ for $i \in I$ and all the $L_j$ for $j \in J$ are countable partially ordered sets; see \cref{rem:equivalence_E_Wc_Wo_infinite}. In this case, $W^{\alpha}$ is a countable Alexandrov $T_0$-space for every $\alpha \in \{\closed, \open\}^I$. 
		\end{rem}

		\subsection{Coxeter functors}\label{subsection:Coxeter_functors}
		
		The classical Coxeter functors associated with an oriented tree are obtained by composing the reflection functors along an admissible numbering of its vertices; see \cite{BGP_1973}. In this subsection, we consider the analogous construction for the reflection functors introduced in \cref{section:reflection_functors}. As a consequence of \cref{thm:rearrangement_reflection_functors}, we show in \cref{cor:Coxeter_functor_independence} that the resulting composite is independent of the choice of admissible numbering, paralleling the corresponding property of classical Coxeter functors; see \cite[Lemma~1.2(3)]{BGP_1973}. 
		
		Throughout this subsection, we fix a tree $G=(V,E)$. Recall from \cref{rem:BGP-reflection} that the construction in \cref{def:Wc_Wo} recovers the reflection of orientations in Hasse diagrams introduced in \cref{subsection:reflection_of_orientations_in_quivers}. 
		
		For an orientation $X$ of $G$ and an open point $x$ in $X$, we write
		\[
		S_x^X \colon \Cstar(X) \to \Cstar(\sigma_xX)
		\]
		for the reflection functor defined as in \cref{def:reflection_functor}
		\labelcref{def:reflection_functor_1}. 
		
		We next consider the case for closed points. For each $x \in V$, set 
		\[J_x:=\{y \in V \mid \{x, y\} \in E\}. \]
		In what follows, we fix
		\begin{itemize}
			\item C*-algebras $D_x$ for $x \in V$;
			\item C*-algebras $D_{x, y}$ for $x \in V$ and $y \in J_x$;
			\item $\ast$-homomorphisms $\rho_{x, y} \colon D_{x, y} \to D_x$
			for $x \in V$ and $y \in J_x$ such that, for every $x \in V$,
			the $\ast$-homomorphisms $\rho_{x, y}$ for $y \in J_x$ are mutually orthogonal.
		\end{itemize}
		For an orientation $X$ of $G$ and a closed point $x$ in $X$, we write 
		\[
		S_X^x \colon \Cstar(X) \to \Cstar(\sigma_xX)
		\]
		for the reflection functor defined as in \cref{def:reflection_functor}
		\labelcref{def:reflection_functor_2} using $D_x$, $(D_{x, y})_{y \in J_x}$, and
		$(\rho_{x, y})_{y \in J_x}$. 
		
		Let $X$ be an orientation of $G$. Let $x=(x_1,x_2,\dots,x_n)$ be a finite sequence in $V$.
		For $i=0, 1, \dots, n$, define $X_i$ recursively by
		\[
		X_i:=
		\begin{cases}
			X & \text{if $i=0$},\\
			\sigma_{x_i}X_{i-1} & \text{if $i \geq 1$}.
		\end{cases}
		\]
		If $x$ is an admissible sequence of open points in $X$, define a functor
		\[
		S_{x}^X
		\colon
		\Cstar(X) \to \Cstar(X_n)
		\]
		by
		\[
		S_{x}^X
		:=
		S_{x_n}^{X_{n-1}}
		S_{x_{n-1}}^{X_{n-2}}
		\cdots
		S_{x_1}^{X_0}.
		\]
		If $x$ is an admissible sequence of closed points in $X$, define a functor
		\[
		S_X^{x}
		\colon
		\Cstar(X) \to \Cstar(X_n)
		\]
		by
		\[
		S_X^{x}
		:=
		S_{X_{n-1}}^{x_n}
		S_{X_{n-2}}^{x_{n-1}}
		\cdots
		S_{X_0}^{x_1}.
		\]

		\begin{cor}\label{cor:Coxeter_functor_independence}
			Let $X$ be a finite $T_0$-space whose Hasse diagram is an orientation
			of $G$.
			Let $x=(x_1, x_2, \dots, x_n)$ be a sequence in $V$, and let \[f \colon \{1, 2, \dots, n\} \to \{1, 2, \dots, n\}\] 
			be a bijection.
			Set
			\[
			x_f:=(x_{f(1)}, x_{f(2)}, \dots, x_{f(n)}).
			\]
			Define a finite $T_0$-space
			\[Y:=\sigma_{x_n}\sigma_{x_{n-1}}\cdots\sigma_{x_1}X=\sigma_{x_{f(n)}}\sigma_{x_{f(n-1)}}\cdots\sigma_{x_{f(1)}}X. \]
			Then the following equalities hold in $[\Cstar(X), \Cstar(Y)]$. 
			\begin{enumerate}[label=\textnormal{(\arabic*)}]
				\item \label{cor:Coxeter_functor_independence_1}
				If $x$ and $x_f$ are admissible sequences of open points in $X$, then
				\[
				S_{x}^X=S_{x_f}^X. 
				\]
				
				\item \label{cor:Coxeter_functor_independence_2}
				If $x$ and $x_f$ are admissible sequences of closed points in $X$, then
				\[
				S_X^{x}=S_X^{x_f}. 
				\]
			\end{enumerate}
		\end{cor}
		
		\begin{proof}
			By \cref{cor:Salphabeta_composition_reflection_functors}, for an orientation $Z$ of $G$ and open points $y, z$ in $Z$ with $y \neq z$, we have
			\[
			S_z^{\sigma_yZ}S_y^Z
			=
			S_y^{\sigma_zZ}S_z^Z.
			\]
			Similarly, for closed points $y, z$ in $Z$ with $y \neq z$, we have
			\[
			S_{\sigma_yZ}^zS_Z^y
			=
			S_{\sigma_zZ}^yS_Z^z.
			\]
			The assertions now follow from the combinatorial argument in the proof of \cite[Lemma~1.2(3)]{BGP_1973}.
		\end{proof}

		\section{Application of reflection functors to filtrated K-theory}\label{section:application_of_reflection_functors_to_filtrated_K-theory}
		
		In this section, we apply reflection functors to filtrated K-theory and obtain the UCT for C*-algebras over accordion spaces. 
		
		In \cref{subsection:general_machinery}, we define filtrated K-theory in terms of functor categories. When one studies whether a given homological invariant satisfies the UCT, the essential point is to verify that the values of the invariant admit projective resolutions of length $1$. We introduce general machinery that is useful for this verification; see \cref{lem:isomorphism_full_subcategories}.
		
		In \cref{subsection:reflection_of_connected_locally_closed_subsets}, we study the composites obtained by applying filtrated K-theory functors after reflection functors in a particular setting. The main results of this subsection are \cref{thm:isomorphism_NT,cor:equivalence_UCT}. 
		
		In \cref{subsection:classification_of_Cstar-algebras_over_accordion_spaces}, combining \cref{cor:equivalence_UCT} with the classification result of Meyer and Nest in \cite{MN_2012}, we prove that filtrated K-theory satisfies the UCT for C*-algebras over accordion spaces; see \cref{thm:UCT_accordion}. 
		
		As observed in \cref{subsection:application_to_bivariant_K-theory_1}, ideal-related E-theory is better suited to reflection functors than ideal-related KK-theory. 
		By \cite[Theorem~7.1]{Gabe_2022}, Kirchberg's classification theorem \cite{Kirchberg_2000} (see \cite[Theorem~G]{Gabe_2024}) remains valid with ideal-related KK-theory replaced by ideal-related E-theory. 
		Moreover, the classification results in \cite{MN_2012} also remain valid with ideal-related KK-theory replaced by ideal-related E-theory due to the general machinery in \cite{MN_2010}. 
		For these reasons, we use a formulation in terms of ideal-related E-theory in this section. 
		Readers who wish to interpret the results in terms of ideal-related KK-theory may replace $\Ecat(\blank)$ by $\KKcat(\blank)$ and replace $\Boot_{\E}(\blank)$ by $\Boot(\blank)$. However, in \cref{thm:isomorphism_NT}, $\Ecat(\blank)$ must be replaced by $\KKcat(\blank)_{\loc}$.

		\subsection{General machinery}\label{subsection:general_machinery}
		We first define filtrated K-theory in terms of functor categories. 
		
		For preadditive categories $\frakB$ and $\frakA$, let $\Add(\frakB, \frakA)$ be the full subcategory of $[\frakB, \frakA]$ consisting of additive functors. Then $\Add(\frakB, \frakA)$ is a preadditive category with respect to the pointwise structure. If $\frakA$ is an abelian category, then so is $\Add(\frakB, \frakA)$ with respect to the pointwise structure. 
		
		Let $\frakB$ and $\frakA$ be preadditive categories and let $\calN$ be a full subcategory of $\Add(\frakB, \frakA)$. Then $\calN$ is also a preadditive category. 
		Using the identification $\Add(\frakB, \Add(\calN, \frakA))=\Add(\calN, \Add(\frakB, \frakA))$, we define an additive functor 
		\[I_{\calN} \colon \frakB \to \Add(\calN, \frakA)\]
		as the inclusion functor $\calN \hookrightarrow \Add(\frakB, \frakA)$. 
		
		Let $\Ab$ denote the abelian category whose objects are abelian groups and whose morphisms are group homomorphisms. Let $\Abc$ denote the full subcategory of $\Ab$ consisting of countable abelian groups. Then $\Abc$ is also an abelian category. 
		
		Let $\Ktheory_0 \colon \Cstar \to \Ab$ denote the $\Ktheory_0$-functor. Let $\Ktheory_1:=\Ktheory_0 S \colon \Cstar \to \Ab$. Recall that $\Ktheory_1 S \simeq \Ktheory_0$ in $[\Cstar, \Ab]$ by Bott periodicity. For $i \in \{0, 1\}$, the functor $\Ktheory_i \colon \Cstar \to \Ab$ restricts to a functor $\Ktheory_i \colon \SCstar \to \Abc$, which induces an additive functor $\Ktheory_i \colon \Ecat \to \Abc$. 
		
		\begin{defn}
			Let $X$ be a finite $T_0$-space. 
			Let $\LCc(X)$ denote the set of all connected locally closed subsets of $X$. 
			For each $(i, Y) \in \{0, 1\} \times \LCc(X)$, we associate the additive functor 
			\[\FK^X_{i, Y}:=\Ktheory_i \ev_X^Y \colon \Ecat(X) \xrightarrow{\ev_X^Y} \Ecat \xrightarrow{\Ktheory_i} \Abc, \]
			where $\ev_X^Y \colon \Ecat(X) \to \Ecat$ denotes the functor induced by the evaluation functor $\ev_X^Y \colon \SCstar(X) \to \SCstar$  at $Y$ as in \cref{lem:induced_functor_KK_E}. 
			Let $\NT(X)$ denote the full subcategory of $\Add(\Ecat(X), \Abc)$ consisting of all the objects $\FK^X_{i, Y}$ for $(i, Y) \in \{0, 1\} \times \LCc(X)$. 
			Then $\NT(X)$ is a preadditive category. Following the notation in \cite{MN_2012}, we write 
			\[\Mod(\NT(X)):=\Add(\NT(X), \Abc). \] 
			Then $\Mod(\NT(X))$ is an abelian category. 
			We define an additive functor
			\[\FK^X:=I_{\NT(X)} \colon \Ecat(X) \to \Mod(\NT(X)), \] 
			which we call \textit{filtrated K-theory}. 
		\end{defn}
		
		\begin{rem}
			The definition above differs from the original definition of filtrated K-theory due to Meyer and Nest \cite{MN_2012} in several respects. First, it is formulated in terms of ideal-related E-theory rather than ideal-related KK-theory. Namely, we use $\Ecat(X)$, rather than $\KKcat(X)$, as the domain of filtrated K-theory. 
			Second, we use $\LCc(X)$ instead of $\LC(X)$. However, as explained by \cite{MN_2012}, and more explicitly by \cite[Remark~3.1]{BK_2011}, one may restrict from $\LC(X)$ to $\LCc(X)$ without changing filtrated K-theory in any essential way. The restriction to connected locally closed subsets will play an essential role in the arguments of \cref{subsection:reflection_of_connected_locally_closed_subsets}; see \cref{lem:connected_locally_closed_subsets_Wc_Wo}.
			Third, we use the set $\{0, 1\}$ instead of a $\Z/2\Z$-grading. As explained by \cite{ABK_2014}, this formulation does not change filtrated K-theory in any essential way. Since this formulation is more concise, we adopt it throughout this paper.
		\end{rem}
		
		In view of \cite[Theorem~4.5]{MN_2012}, the essential point for the UCT to hold is whether $\FK^X(A)$ admits a projective resolution of length $1$ for every $A \in \Boot_{\E}(X)$. By \cite[Theorem~4.9]{MN_2012}, this is the case when $X$ arises from a totally ordered set. 
		Bentmann and Köhler extended this result to accordion spaces in \cite{BK_2011} by reducing the problem to the totally ordered case.
		By applying the following lemma to reflection functors, we will obtain a simpler way to carry out this reduction.

		\begin{lem}\label{lem:isomorphism_full_subcategories}
			Let $\frakB_1$, $\frakB_2$, and $\frakA$ be preadditive categories. Let $\Omega_1$ and $\Omega_2$ be sets. Let $(F^{(1)}_i)_{i \in \Omega_1}$ and $(F^{(2)}_j)_{j \in \Omega_2}$ be families of objects
			of $\Add(\frakB_1, \frakA)$ and $\Add(\frakB_2, \frakA)$, respectively, such that
			$F^{(1)}_i \simeq F^{(1)}_{i'}$ implies $i=i'$ for all $i,i' \in \Omega_1$, and
			$F^{(2)}_j \simeq F^{(2)}_{j'}$ implies $j=j'$ for all $j,j' \in \Omega_2$.
			Let $\calN_1$ and $\calN_2$ be the full subcategories of $\Add(\frakB_1, \frakA)$ and $\Add(\frakB_2, \frakA)$ whose object sets are
			$\{F^{(1)}_i \mid i \in \Omega_1 \}$ and $\{F^{(2)}_j \mid j \in \Omega_2 \}$, respectively.
			Suppose we are given an equivalence $G \colon \frakB_1 \to \frakB_2$ of categories, and a bijection $f \colon \Omega_2 \to \Omega_1$. If $F^{(2)}_j G \simeq F^{(1)}_{f(j)}$ in $\Add(\frakB_1, \frakA)$ for every $j \in \Omega_2$, then there exists an isomorphism
			\[\Phi \colon \calN_2 \to \calN_1 \] 
			of preadditive categories, which makes the following diagram commute up to natural isomorphism:
			\[
			\begin{tikzcd}
				\frakB_1 \ar{r}{G}[swap]{\simeq} \ar[swap]{d}{I_{\calN_1}} & \frakB_2 \ar{d}{I_{\calN_2}}  \\
				\Add(\calN_1, \frakA) \ar{r}[swap]{\Phi^*}{\simeq} & \Add(\calN_2, \frakA)\mathrlap{.}
			\end{tikzcd}
			\]
		\end{lem}
		
		\begin{proof} 
			For each $j \in \Omega_2$, choose an isomorphism $\varphi_j \colon F^{(2)}_j G \Rightarrow F^{(1)}_{f(j)}$ in $\Add(\frakB_1, \frakA)$. 
			
			We define a functor $\Phi \colon \calN_2 \to \calN_1$ as follows. 
			For each $j \in \Omega_2$, we define $\Phi(F^{(2)}_j):=F^{(1)}_{f(j)}$. 
			For $j, j' \in \Omega_2$ and a morphism $\xi \colon F^{(2)}_j \Rightarrow F^{(2)}_{j'}$ in $\calN_2$, we define a morphism $\Phi(\xi) \colon \Phi(F^{(2)}_j) \Rightarrow \Phi(F^{(2)}_{j'})$ in $\calN_1$ to be the composite
			\[\Phi(F^{(2)}_j) \xRightarrow{\varphi_j^{-1}} F^{(2)}_jG \xRightarrow{\xi G} F^{(2)}_{j'}G \xRightarrow{\varphi_{j'}} \Phi(F^{(2)}_{j'}). \]
			Since the objects $F^{(2)}_j$ for $j \in \Omega_2$ are pairwise distinct, these assignments define $\Phi$ without ambiguity. It is straightforward to verify that $\Phi$ is an additive functor. 
			
			Since $f$ is bijective, $\Phi$ is bijective-on-objects. Since $G$ is an equivalence of categories, $\Phi$ is fully faithful. Thus, $\Phi$ is an isomorphism of preadditive categories. 
			
			Since $\Phi$ is additive, the functor $\Phi^* \colon [\calN_1, \frakA] \to [\calN_2, \frakA]$ given by composing $\Phi$ on the right restricts to a functor
			\[\Phi^* \colon \Add(\calN_1, \frakA) \to \Add(\calN_2, \frakA). \]
			Since $\Phi$ is an isomorphism of preadditive categories, so is $\Phi^*$. 
			
			For $B \in \frakB_1$ and $j \in \Omega_2$, we have
			\[I_{\calN_2} G(B)(F^{(2)}_j)=F^{(2)}_jG(B), \quad \Phi^* I_{\calN_1}(B)(F^{(2)}_j)=F^{(1)}_{f(j)}(B). \]
			We know that the isomorphisms $(\varphi_j)_B \colon F^{(2)}_jG(B) \to F^{(1)}_{f(j)}(B)$ in $\frakA$ are natural in $B \in \frakB_1$. By the definition of $\Phi$, the isomorphisms $(\varphi_j)_B$ are natural in $F^{(2)}_j \in \calN_2$. It follows that the isomorphisms $(\varphi_j)_B$ for $B \in \frakB_1$ and $j \in \Omega_2$ define an isomorphism $I_{\calN_2} G \Rightarrow \Phi^* I_{\calN_1}$ in $[\frakB_1, \Add(\calN_2, \frakA)]$. 
			This completes the proof. 
		\end{proof}
		
		In the setting of \cref{lem:isomorphism_full_subcategories}, suppose that $\frakA$ is an abelian category. Then
		\[\Phi^* \colon \Add(\calN_1, \frakA) \to \Add(\calN_2, \frakA) \]
		is an isomorphism of abelian categories and therefore preserves projective resolutions. We will use this observation in the proof of \cref{cor:equivalence_UCT}.
		
		For $i, j \in \{0, 1\}$, $Y \in \LCc(X)$, and $x \in X$, we have
		\[
		\FK^X_{i,Y} S^j i_{\{x\}}^X\C \simeq
		\begin{cases}
			\Z & \text{if $i=j$ and $x \in Y$},\\
			0 & \text{otherwise}.
		\end{cases}
		\]
		Using this calculation, we can easily check the following lemma, which allows us to apply \cref{lem:isomorphism_full_subcategories} to filtrated K-theory. 
		
		\begin{lem}\label{lem:filtrated_K-theory_essentially_injective}
			Let $X$ be a finite $T_0$-space. Let $(i, Y), (i', Y') \in \{0, 1\} \times \LCc(X)$. If $\FK^X_{i, Y} \simeq \FK^X_{i', Y'}$ in $\Add(\Ecat(X), \Abc)$, then $(i, Y)=(i', Y')$. 
		\end{lem}

		\begin{rem}		
			Bentmann and Meyer \cite{MR_2017} classify objects in the bootstrap class that have projective resolutions of length $2$ using a homological invariant together with an obstruction class in an $\Ext^2$-group. \cref{lem:isomorphism_full_subcategories} may be useful for verifying the existence of projective resolutions of length $2$. A possible direction for future work is to investigate applications of reflection functors to their results.
		\end{rem}
		
		\subsection{Reflection of connected locally closed subsets}\label{subsection:reflection_of_connected_locally_closed_subsets}
		
		Throughout this subsection, we fix a non-empty finite set $J$, a finite partially ordered set $K$, finite partially ordered sets $L_j$ for $j \in J$, and order-preserving maps $\theta_j \colon K \to L_j$ for $j \in J$. For $\alpha \in \{\closed, \open\}$, let $W^{\alpha}$ be the finite $T_0$-space defined as in \cref{def:Wc_Wo}. 
		
		Throughout this subsection, we assume that $K$ consists of a single element, denoted by $x$. For each $j \in J$, put 
		\[y_j:=\theta_j(x) \in L_j. \] 
		In this case, we can describe all connected locally closed subsets of $W^{\closed}$ and $W^{\open}$ as follows. 
		For each non-empty subset $F \subset J$, put
		\[
		\calL_{F}:=\Bigg\{\{x\} \amalg \coprod_{j \in F}Y_j \Bigm| \text{$Y_j \in \LCc(L_j)$, $y_j \in Y_j$ for all $j \in F$} \Bigg\}.
		\]
		
		\begin{lem}\label{lem:connected_locally_closed_subsets_Wc_Wo}
			For $\alpha \in \{\closed, \open\}$, we have
			\[
			\LCc(W^{\alpha})=\{\{x\}\} \amalg \coprod_{j \in J}\LCc(L_j) \amalg \coprod_{\emptyset \neq F \subset J} \calL_{F}.
			\]
		\end{lem}
		
		\begin{proof}
			We first prove the inclusion $\supset$. It suffices to prove that $\calL_F \subset \LCc(W^{\alpha})$ for every non-empty subset $F \subset J$, since the remainder is clear. Let $Y \in \calL_F$.
			Then
			\[
			Y=\{x\}\amalg\coprod_{j \in F}Y_j,
			\]
			for some $(Y_j)_{j \in F}$ such that $Y_j \in \LCc(L_j)$ and $y_j \in Y_j$ for every $j \in F$.
			Write $Y^{\alpha}$ for the subspace of $W^{\alpha}$ with underlying set $Y$. For each $j \in F$, write $Y_j^{\alpha}$ for the subspace of $W^{\alpha}$ with underlying set $Y_j$.
			The subspace $Y^{\alpha}$ of $W^{\alpha}$ is the Alexandrov space obtained from the data $\{x\}$, $(Y_j^{\alpha})_{j \in F}$, and $(\theta_j \colon \{x\} \to Y_j)_{j \in F}$ as in \textup{\cref{def:Wc_Wo}}. Since $F \neq \emptyset$ and $Y_j^{\alpha}$ is a connected space for every $j \in F$, \cref{prop:connectedness_Wc_Wo} implies that $Y^{\alpha}$ is a connected space. Hence, $Y$ is a connected subset of $W^{\alpha}$. 
			
			It remains to prove that $Y \in \LC(W^{\alpha})$. Consider the case $\alpha=\closed$. Let $a, b, c \in W$ satisfy $a, c \in Y$ and
			$(a, b), (b, c) \in R_{W^{\closed}}$. We show that $b \in Y$. If $b=x$, this is immediate. Suppose that $b \in L_j$ for some $j \in J$. Since $(b, c) \in R_{W^{\closed}}$ and $c \in Y$, we have $j \in F$ and $c \in Y_j$.
			If $a=x$, then $(y_j, b) \in R_{L_j}$. Since $y_j, c \in Y_j$, $(y_j, b), (b, c) \in R_{L_j}$, and $Y_j \in \LC(L_j)$, we obtain $b \in Y_j$. If $a \neq x$, then $a \in Y_j$ and $(a, b), (b, c) \in R_{L_j}$, and again $b \in Y_j$ because $Y_j \in \LC(L_j)$. Thus, $b \in Y$, and hence $Y \in \LC(W^{\closed})$.
			By symmetry, $Y \in \LC(W^{\open})$ as well. Consequently, $Y \in \LCc(W^{\alpha})$ and hence $\calL_F \subset \LCc(W^{\alpha})$. 
			
			We next prove the reverse inclusion $\subset$. 
			Let $Y \in \LCc(W^{\alpha})$. 
			Set $F:=\{j \in J \mid Y \cap L_j \neq \emptyset\}$. 
			If $F=\emptyset$, then $Y=\{x\}$ since $Y \neq \emptyset$. 
			Suppose that $F \neq \emptyset$. If $x \notin Y$, then the connectedness of $Y$ implies that $F=\{j\}$ for some $j \in J$, and hence $Y \in \LCc(L_j)$.
			It remains to consider the case where $x \in Y$. For each $j \in F$, set $Y_j:=Y \cap L_j$. Then $Y_j \in \LC(L_j)$ because $Y \in \LC(W^{\alpha})$. 
			We claim that $y_j \in Y_j$ and $Y_j$ is connected for every $j \in F$. 
			Since $Y_j \neq \emptyset$, choose $z \in Y_j$.
			By \cref{lem:connected_Alexandrov_space} applied to $Y$, there exists a finite sequence
			\[
			z=z_0, z_1, \dots, z_n=x
			\]
			in $Y$ such that $z_{i-1}$ and $z_i$ are comparable in $W^{\alpha}$ for every $i=1, 2, \dots, n$.
			Let $m \in \{1, 2, \dots, n\}$ be the smallest index such that $z_m \notin L_j$.
			By the definition of $R_{W^{\alpha}}$, we have $z_m=x$. Put $w:=z_{m-1}\in Y_j$.
			If $\alpha=\closed$, then
			\[
			(x,y_j),(y_j,w)\in R_{W^{\alpha}},
			\]
			whereas if $\alpha=\open$, then
			\[
			(w,y_j),(y_j,x)\in R_{W^{\alpha}}.
			\]
			Since $Y \in \LC(W^{\alpha})$ and $x, w\in Y$, it follows that $y_j\in Y$. Therefore, $y_j\in Y\cap L_j=Y_j$.
			Now let $z \in Y_j$ be arbitrary and choose a sequence as above.
			Then $z_0, \dots, z_{m-1}$ all belong to $Y_j$, and $z_{m-1}$ is
			comparable with $y_j$. Thus, $z$ belongs to the same equivalence class as $y_j$ under the equivalence relation on $Y_j$ generated by $R_{L_j} \cap (Y_j \times Y_j)$. Hence, \cref{lem:connected_Alexandrov_space} shows that $Y_j$ is connected. This proves the claim. 
			It follows that 
			\[Y=\{x\} \amalg \coprod_{j \in F} Y_j \in \calL_{F}. \]
			This completes the proof. 
		\end{proof}
		
		The following lemma describes the composites of $S_{\open}^{\closed}$ with certain evaluation functors associated with connected locally closed subsets of $W^{\closed}$; see also \cref{rem:reflection_functor_locally_closed_subset} \labelcref{rem:reflection_functor_locally_closed_subset_1}.
		
		\begin{lem}\label{lem:reflection_LCC_Cstar}
			In $[\SCstar(W^{\open}), \SCstar]$, we have the following. 
			\begin{enumerate}[label=\textnormal{(\arabic*)}]		
				\item \label{lem:reflection_LCC_Cstar_1} 
				We have
				\[\ev_{W^{\closed}}^{\{x\}} S_{\open}^{\closed}=\ev_{W^{\open}}^{\{x\}}. \] 
				
				\item \label{lem:reflection_LCC_Cstar_2}
				For $j \in J$ and $Y \in \LCc(L_j)$ with $y_j \notin Y$, we have
				\[\ev_{W^{\closed}}^Y S_{\open}^{\closed}=S \ev_{W^{\open}}^Y. \] 
				
				\item \label{lem:reflection_LCC_Cstar_3} 
				For $j \in J$ and $Y \in \LCc(L_j)$ with $y_j \in Y$, we have
				\[\ev_{W^{\closed}}^Y S_{\open}^{\closed}=S \ev_{W^{\open}}^{Y \cup \{x\}}. \] 
				
				\item \label{lem:reflection_LCC_Cstar_4} 
				For $j \in J$ and $Y \in \LCc(L_j)$ with $y_j \in Y$, we have a short exact sequence
				\[
				\begin{tikzcd}[ampersand replacement=\&]
					0 \ar[Rightarrow]{r} \& C \ev_{W^{\open}}^{\{x\}} \ar[Rightarrow]{r} \& \ev_{W^{\closed}}^{Y \cup \{x\}} S_{\open}^{\closed} \ar[Rightarrow]{r} \& S \ev_{W^{\open}}^Y \ar[Rightarrow]{r} \& 0\mathrlap{.}
				\end{tikzcd} 
				\]
				\item \label{lem:reflection_LCC_Cstar_5} 
				For $F \subset J$ with $|F|=2$ and for $Y \in \calL_{F}$, we have a short exact sequence
				\[
				\begin{tikzcd}[ampersand replacement=\&]
					0 \ar[Rightarrow]{r} \& \ev_{W^{\closed}}^Y S_{\open}^{\closed} \ar[Rightarrow]{r} \& S \ev_{W^{\open}}^Y \ar[Rightarrow]{r} \& C\ev_{W^{\open}}^{Y \setminus \{x\}} \ar[Rightarrow]{r} \& 0\mathrlap{.} 
				\end{tikzcd} 
				\]
			\end{enumerate}
		\end{lem}
		
		\begin{proof}
			For each $Y \in \LCc(W^{\closed})$, since $\ev_{W^{\closed}}^Y \colon \SCstar(W^{\closed}) \to \SCstar$ is an exact $\SCstar$-module functor, \cref{lem:exact_module_functor_commute_mapping_cone} (see also \cref{lem:rearrangement_mapping_cones}) shows that $\ev_{W^{\closed}}^Y S_{\open}^{\closed}$ is the mapping cone of the morphism
			\[\bigl(\ev_{W^{\closed}}^Y \iota_{\{x\}}^{W_j^{\open}}\bigr)_{j \in J} \colon \ev_{W^{\closed}}^Y i_{\{x\}}^{W^{\closed}} r_{W^{\open}}^{\{x\}}  \Rightarrow \bigoplus_{j \in J} \ev_{W^{\closed}}^Y i_{L_j}^{W^{\closed}} p_{W_j^{\open}}^{L_j} r_{W^{\open}}^{W_j^{\open}} \]
			in $[\SCstar(W^{\open}), \SCstar]$. 
			Here, $W_j^{\open}$ denotes the subspace of $W^{\open}$ with underlying set $W_j:=\{x\} \amalg L_j$ for each $j \in J$. 
			We compute the composite functors occurring in the source and target
			of this morphism using \cref{lem:evaluation_extension,lem:evaluation_restriction}.
			\begin{enumerate}[label=\textnormal{(\arabic*)}]		
				\item
				We have
				\[\ev_{W^{\closed}}^{\{x\}} S_{\open}^{\closed}=M \bigl(\ev_{W^{\open}}^{\{x\}} \Rightarrow 0 \bigr)=\ev_{W^{\open}}^{\{x\}}. \] 
				
				\item 
				For $j \in J$ and $Y \in \LCc(L_j)$ with $y_j \notin Y$, we have
				\[\ev_{W^{\closed}}^Y S_{\open}^{\closed}=M \bigl(0 \Rightarrow \ev_{W^{\open}}^Y \bigr)=S \ev_{W^{\open}}^Y . \] 
				
				\item
				For $j \in J$ and $Y \in \LCc(L_j)$ with $y_j \in Y$, we have			
				\[\ev_{W^{\closed}}^Y S_{\open}^{\closed}=M \Bigl(0 \Rightarrow \ev_{W^{\open}}^{Y \cup \{x\}} \Bigr)=S\ev_{W^{\open}}^{Y \cup \{x\}}. \] 
				
				\item
				For $j \in J$ and $Y \in \LCc(L_j)$ with $y_j \in Y$, \cref{lem:composition_theta_r_iota_pi} \labelcref{lem:composition_theta_r_iota_pi_2} \labelcref{lem:composition_theta_r_iota_pi_2_1} shows that $\ev_{W^{\closed}}^{Y \cup \{x\}} S_{\open}^{\closed}$ is the mapping cone of the morphism
				\[\iota_{\{x\}}^{Y \cup \{x\}} \colon \ev_{W^{\open}}^{\{x\}} \Rightarrow \ev_{W^{\open}}^{Y \cup \{x\}}\]
				in $[\SCstar(W^{\open}), \SCstar]$.
				Consider the following commutative diagram with exact rows in $[\SCstar(W^{\open}), \SCstar]$:
				\[
				\begin{tikzcd}
					0 \ar[Rightarrow]{r} & \ev_{W^{\open}}^{\{x\}} \ar[Rightarrow]{r} \ar[equal]{d} & \ev_{W^{\open}}^{\{x\}} \ar[Rightarrow]{r} \ar[Rightarrow, swap]{d}{\iota_{\{x\}}^{Y \cup \{x\}}} & 0 \ar[Rightarrow]{r} \ar[Rightarrow]{d}& 0 \\
					0 \ar[Rightarrow]{r} & \ev_{W^{\open}}^{\{x\}} \ar[Rightarrow, swap]{r}{\iota_{\{x\}}^{Y \cup \{x\}}}  & \ev_{W^{\open}}^{Y \cup \{x\}} \ar[Rightarrow, swap]{r}{\pi_{Y \cup \{x\}}^Y} & \ev_{W^{\open}}^Y \ar[Rightarrow]{r} & 0\mathrlap{.}
				\end{tikzcd}
				\] 
				By taking the mapping cones of the vertical morphisms, we obtain the desired short exact sequence. 
				
				\item
				Let $F \subset J$ with $|F|=2$, and let $j$ and $j'$ be the elements of $F$. For $Y \in \calL_F$, we have $Y=Z \cup Z' \cup \{x\}$ for some $Z \in \LCc(L_j)$ with $y_j \in Z$, and some $Z' \in \LCc(L_{j'})$ with $y_{j'} \in Z'$. By \cref{lem:composition_theta_r_iota_pi} \labelcref{lem:composition_theta_r_iota_pi_2} \labelcref{lem:composition_theta_r_iota_pi_2_1}, $\ev_{W^{\closed}}^Y S_{\open}^{\closed}$ is the mapping cone of the morphism
				\[\bigl(\iota_{\{x\}}^{Z \cup \{x\}}, \iota_{\{x\}}^{Z' \cup \{x\}}\bigr) \colon \ev_{W^{\open}}^{\{x\}} \Rightarrow \ev_{W^{\open}}^{Z \cup \{x\}} \oplus \ev_{W^{\open}}^{Z' \cup \{x\}} \]
				in $[\SCstar(W^{\open}), \SCstar]$.
				Consider the following commutative diagram with exact rows in $[\SCstar(W^{\open}), \SCstar]$:
				\[
				\begin{tikzcd}
					0 \ar[Rightarrow]{r} & \ev_{W^{\open}}^{\{x\}} \ar[Rightarrow]{r}{\iota_{\{x\}}^Y} \ar[Rightarrow, swap]{d}{(\iota_{\{x\}}^{Z \cup \{x\}}, \iota_{\{x\}}^{Z' \cup \{x\}})} & \ev_{W^{\open}}^Y \ar[Rightarrow]{r}{(\pi_Y^{Z'}, \pi_Y^Z)} \ar[Rightarrow, swap]{d}{(\id_Y, \id_Y)} & \ev_{W^{\open}}^{Z'} \oplus \ev_{W^{\open}}^Z \ar[Rightarrow]{r} \ar[equal]{d} & 0 \\
					0 \ar[Rightarrow]{r} & \ev_{W^{\open}}^{Z \cup \{x\}} \oplus \ev_{W^{\open}}^{Z' \cup \{x\}} \ar[Rightarrow, swap]{r}[yshift=-0.3em]{\iota_{Z \cup \{x\}}^Y \oplus \iota_{Z' \cup \{x\}}^Y}  & \ev_{W^{\open}}^Y \oplus \ev_{W^{\open}}^Y \ar[Rightarrow, swap]{r}{\pi_Y^{Z'} \oplus \pi_Y^Z} & \ev_{W^{\open}}^{Z'} \oplus \ev_{W^{\open}}^Z \ar[Rightarrow]{r} & 0\mathrlap{.}
				\end{tikzcd}
				\]
				The objects $\ev_{W^{\closed}}^Y S_{\open}^{\closed}$ and $C \ev_{W^{\open}}^{Y \setminus \{x\}}$ are the mapping cones of the left and right vertical morphisms, respectively. 
				The object $S \ev_{W^{\open}}^Y$ is the mapping cone of the middle vertical morphism because the C*-algebra $S$ is the mapping cone of the $\ast$-homomorphism $\C \ni \lambda \mapsto (\lambda, \lambda) \in \C \oplus \C$. Therefore, by taking the mapping cones of the vertical morphisms, we obtain the desired short exact sequence. 
				\qedhere
			\end{enumerate}
		\end{proof}
		
		As in \cref{subsection:application_to_bivariant_K-theory_1}, the reflection functor $S_{\open}^{\closed} \colon \SCstar(W^{\open}) \to \SCstar(W^{\closed})$ induces a functor 
		\[S_{\open}^{\closed} \colon \Ecat(W^{\open}) \to \Ecat(W^{\closed}). \]
		The following lemma describes the composites of $S_{\open}^{\closed}$ with filtrated K-theory functors.
		
		\begin{lem}\label{lem:reflection_LCC_Ktheory}
			In $[\Ecat(W^{\open}), \Abc]$, we have the following for $i \in \{0, 1\}$. 
			\begin{enumerate}[label=\textnormal{(\arabic*)}]		
				\item \label{lem:reflection_LCC_Ktheory_1} 
				We have
				\[\FK_{i, \{x\}}^{W^{\closed}} S_{\open}^{\closed}=\FK_{i, \{x\}}^{W^{\open}}. \] 
				
				\item \label{lem:reflection_LCC_Ktheory_2}
				For $j \in J$ and $Y \in \LCc(L_j)$ with $y_j \notin Y$, we have
				\[\FK_{i, Y}^{W^{\closed}} S_{\open}^{\closed} \simeq \FK_{1-i, Y}^{W^{\open}}. \] 
				
				\item \label{lem:reflection_LCC_Ktheory_3} 
				For $j \in J$ and $Y \in \LCc(L_j)$ with $y_j \in Y$, we have
				\[\FK_{i, Y}^{W^{\closed}} S_{\open}^{\closed} \simeq \FK_{1-i, Y \cup \{x\}}^{W^{\open}}. \]  
				
				\item \label{lem:reflection_LCC_Ktheory_4} 
				For $j \in J$ and $Y \in \LCc(L_j)$ with $y_j \in Y$, we have
				\[\FK_{i, Y \cup \{x\}}^{W^{\closed}} S_{\open}^{\closed} \simeq \FK_{1-i, Y}^{W^{\open}}. \]  
				
				\item \label{lem:reflection_LCC_Ktheory_5} 
				For $F \subset J$ with $|F|=2$ and for $Y \in \calL_{F}$, we have
				\[\FK_{i, Y}^{W^{\closed}} S_{\open}^{\closed} \simeq \FK_{1-i, Y}^{W^{\open}}. \] 
			\end{enumerate}
		\end{lem}
		
		\begin{proof}
			\
			\begin{enumerate}[label=\textnormal{(\arabic*)}]		
				\item
				By \cref{lem:reflection_LCC_Cstar} \labelcref{lem:reflection_LCC_Cstar_1}, we have
				\[\FK_{i, \{x\}}^{W^{\closed}} S_{\open}^{\closed}=\Ktheory_i\ev_{W^{\closed}}^{\{x\}} S_{\open}^{\closed}=\Ktheory_i \ev_{W^{\open}}^{\{x\}}=\FK_{i, \{x\}}^{W^{\open}} \] 
				in $[\Ecat(W^{\open}), \Abc]$. 
				
				\item 
				For $j \in J$ and $Y \in \LCc(L_j)$ with $y_j \notin Y$, by \cref{lem:reflection_LCC_Cstar} \labelcref{lem:reflection_LCC_Cstar_2} and Bott periodicity, we have
				\[\FK_{i, Y}^{W^{\closed}} S_{\open}^{\closed}=\Ktheory_i\ev_{W^{\closed}}^Y S_{\open}^{\closed}=\Ktheory_i S\ev_{W^{\open}}^Y \simeq \Ktheory_{1-i} \ev_{W^{\open}}^Y=\FK_{1-i, Y}^{W^{\open}} \]
				in $[\Ecat(W^{\open}), \Abc]$. 
				
				\item 
				The same argument as in \labelcref{lem:reflection_LCC_Ktheory_2}, using \cref{lem:reflection_LCC_Cstar} \labelcref{lem:reflection_LCC_Cstar_3} in place of \cref{lem:reflection_LCC_Cstar} \labelcref{lem:reflection_LCC_Cstar_2}, proves the assertion.
				
				\item 
				For $j \in J$ and $Y \in \LCc(L_j)$ with $y_j \in Y$, the object $C\ev_{W^{\open}}^{\{x\}}$ is contractible in $[\SCstar(W^{\open}), \SCstar]$. Therefore, by \cref{lem:contractible_isomorphism_KK_E} and \cref{lem:reflection_LCC_Cstar} \labelcref{lem:reflection_LCC_Cstar_4}, we obtain
				\[\ev_{W^{\closed}}^{Y \cup \{x\}} S_{\open}^{\closed} \simeq S\ev_{W^{\open}}^Y\]
				in $[\Ecat(W^{\open}), \Ecat]$. Applying $\Ktheory_i$ and Bott periodicity gives the desired isomorphism.
				
				\item 
				For $F \subset J$ with $|F|=2$ and for $Y \in \calL_{F}$, the object $C\ev_{W^{\open}}^{Y \setminus \{x\}}$ is contractible in $[\SCstar(W^{\open}), \SCstar]$. 
				Hence, by \cref{lem:contractible_isomorphism_KK_E} and \cref{lem:reflection_LCC_Cstar} \labelcref{lem:reflection_LCC_Cstar_5}, we obtain
				\[\ev_{W^{\closed}}^Y S_{\open}^{\closed} \simeq S\ev_{W^{\open}}^Y\]
				in $[\Ecat(W^{\open}), \Ecat]$. Applying $\Ktheory_i$ and Bott periodicity gives the desired isomorphism.
				\qedhere
			\end{enumerate}		
		\end{proof}

		\begin{lem}\label{lem:bijection_connected_locally_closed}
			Suppose that $|J| \leq 2$. 
			Then there exists a bijection
			\[f \colon \{0, 1\} \times \LCc(W^{\closed}) \to \{0, 1\} \times \LCc(W^{\open}) \]
			such that, for each $(i, Y) \in \{0, 1\} \times \LCc(W^{\closed})$, there exists an isomorphism
			\[\FK_{i, Y}^{W^{\closed}} S_{\open}^{\closed} \simeq \FK_{f(i, Y)}^{W^{\open}}\]
			in $[\Ecat(W^{\open}), \Abc]$. 
		\end{lem}
		
		\begin{proof}
			With the help of \cref{lem:connected_locally_closed_subsets_Wc_Wo}, we can define a map $g \colon \LCc(W^{\closed}) \to \LCc(W^{\open})$ by
			\[
			g(Y):=
			\begin{dcases}
				\{x\} & \text{if $Y=\{x\}$},\\
				Y & \text{if $Y \in \LCc(L_j)$ and $y_j \notin Y$ for some $j \in J$},\\
				Y\cup\{x\} & \text{if $Y \in \LCc(L_j)$ and $y_j \in Y$ for some $j \in J$},\\
				Y \setminus \{x\} & \text{if $Y \in \calL_{F}$ for some $F \subset J$ with $|F|=1$},\\
				Y & \text{if $Y \in \calL_{F}$ for some $F \subset J$ with $|F|=2$}.
			\end{dcases}
			\]
			Then $g$ is bijective and its inverse is given by the same formula.
			Define a bijection
			\[
			f \colon \{0, 1\} \times \LCc(W^{\closed}) \to \{0, 1\} \times \LCc(W^{\open})
			\]
			by
			\[
			f(i, Y):=
			\begin{dcases}
				(i, g(Y)) & \text{if $Y=\{x\}$},\\
				(1-i, g(Y)) & \text{otherwise}.
			\end{dcases}
			\]
			By \cref{lem:reflection_LCC_Ktheory}, this $f$ shows the assertion. 
		\end{proof}

		\begin{rem}
			The assumption that $|J| \leq 2$ cannot be removed in general.
			Suppose that $|J| \geq 3$ and that all the $L_j$ for $j \in J$ are connected.
			Then $W \in \LCc(W^{\closed})$. 
			Let $B:=i_{\{x\}}^{W^{\open}} \C$. 
			The C*-algebra $S_{\open}^{\closed}B(W^{\closed})$ is the mapping cone of the $\ast$-homomorphism
			\[(\id_{\C})_j \colon \C \to \bigoplus_{j \in J} \C. \]
			We have
			\[
			\FK^{W^{\closed}}_{0, W}S_{\open}^{\closed}B=0, 
			\quad 
			\FK^{W^{\closed}}_{1, W}S_{\open}^{\closed}B \simeq \Z^{|J|-1}.  \]
			Therefore, 
			\[\FK_{1, W}^{W^{\closed}} S_{\open}^{\closed}B \not\simeq \FK_{f(1, W)}^{W^{\open}}B\]
			in $\Abc$ for every bijection 
			\[
			f \colon \{0, 1\} \times \LCc(W^{\closed}) \to \{0, 1\} \times \LCc(W^{\open}). 
			\]
		\end{rem}

		\begin{thm}\label{thm:isomorphism_NT}
			Suppose that $|J| \leq 2$.
			Then there exists an isomorphism 
			\[\Phi \colon \NT(W^{\closed}) \to \NT(W^{\open})\]
			of preadditive categories, which makes the following diagram commute up to natural isomorphism:
			\[
			\begin{tikzcd}
				\Ecat(W^{\open}) \ar{r}{S_{\open}^{\closed}}[swap]{\simeq} \ar[swap]{d}{\FK^{W^{\open}}} & \Ecat(W^{\closed}) \ar{d}{\FK^{W^{\closed}}} \\
				\Mod(\NT(W^{\open})) \ar{r}[swap]{\Phi^*}{\simeq} & \Mod(\NT(W^{\closed}))\mathrlap{.}
			\end{tikzcd}
			\]
		\end{thm}
		
		\begin{proof}
			By \cref{thm:equivalence_E_Wc_Wo}, $S_{\open}^{\closed} \colon \Ecat(W^{\open}) \to \Ecat(W^{\closed})$ is an equivalence of categories.
			The assertion now follows from \cref{lem:isomorphism_full_subcategories,lem:filtrated_K-theory_essentially_injective,lem:bijection_connected_locally_closed}.
		\end{proof}

		\begin{cor}\label{cor:equivalence_UCT}
			Suppose that $|J| \leq 2$.
			Then $\FK^{W^{\open}}(B)$ has a projective resolution of length $1$ in $\Mod(\NT(W^{\open}))$ for every $B \in \Boot_{\E}(W^{\open})$ if and only if $\FK^{W^{\closed}}(A)$ has a projective resolution of length $1$ in $\Mod(\NT(W^{\closed}))$ for every $A \in \Boot_{\E}(W^{\closed})$. 
		\end{cor}
		
		\begin{proof}
			By restriction of $S_{\open}^{\closed} \colon \Ecat(W^{\open}) \to \Ecat(W^{\closed})$ to the bootstrap categories $\Boot_{\E}(W^{\open})$ and $\Boot_{\E}(W^{\closed})$, \cref{thm:isomorphism_NT} gives a natural isomorphism
			\[
			\Phi^*\FK^{W^{\open}} \simeq \FK^{W^{\closed}}S_{\open}^{\closed}
			\]
			of functors $\Boot_{\E}(W^{\open}) \to \Mod(\NT(W^{\closed}))$. 
			Since $\Phi^*$ is an isomorphism of abelian categories, it follows for each $B \in \Boot_{\E}(W^{\open})$ that $\FK^{W^{\open}}(B)$ has a projective resolution of length $1$ in $\Mod(\NT(W^{\open}))$ if and only if $\FK^{W^{\closed}}S_{\open}^{\closed}(B)$ has a projective resolution of length $1$ in $\Mod(\NT(W^{\closed}))$. 
			The assertion now follows because
			$S_{\open}^{\closed} \colon
			\Boot_{\E}(W^{\open}) \to \Boot_{\E}(W^{\closed})$
			is essentially surjective by \cref{thm:equivalence_E_Wc_Wo}.
		\end{proof}

		\subsection{Classification of C*-algebras over accordion spaces}\label{subsection:classification_of_Cstar-algebras_over_accordion_spaces}
		
		\begin{thm}\label{thm:UCT_accordion}
			Let $X$ be an accordion space.
			Then $\FK^X(A)$ has a projective resolution of length $1$ in $\Mod(\NT(X))$ for every $A \in \Boot_{\E}(X)$.
		\end{thm}
		
		\begin{proof}
			Let $Y$ be the finite $T_0$-space associated with a totally ordered set with the same cardinality as $X$. Since $X$ and $Y$ are obtained by orienting the same tree, \cref{lem:tree_orientation} gives an admissible sequence $(x_1, x_2, \dots, x_n)$ of open points in $Y$ such that
			\[
			\sigma_{x_n}\sigma_{x_{n-1}}\cdots\sigma_{x_1}Y=X.
			\]
			For $i=0, 1, \dots, n$, define $Y_i$ recursively by
			\[
			Y_i:=
			\begin{cases}
				Y & \text{if $i=0$},\\
				\sigma_{x_i} Y_{i-1} & \text{if $i \geq 1$}.
			\end{cases}
			\]
			The E-theoretic version of \cite[Theorem~4.9]{MN_2012} shows that $\FK^{Y_0}(A)$ has a projective resolution of length $1$ in $\Mod(\NT(Y_0))$ for every $A \in \Boot_{\E}(Y_0)$. Since the underlying tree is a Dynkin diagram of type $A$, each $x_i$ has at most two adjacent vertices. Therefore, \cref{cor:equivalence_UCT}, applied successively to $Y_{i-1}$ and $Y_i$ for $i=1, 2, \dots, n$, shows that the same property holds for $Y_n$. Since $Y_n=X$, the assertion follows.
		\end{proof}
		
		\begin{cor}\label{cor:filtrated_K-theory_E-equivalence_accordion}
			Let $X$ be an accordion space. If $A$ and $B$ are separable C*-algebras over $X$ belonging to $\Boot_{\E}(X)$, then every isomorphism $\FK^X(A) \simeq \FK^X(B)$ in $\Mod(\NT(X))$ lifts to an isomorphism $A \simeq B$ in $\Ecat(X)$. 
		\end{cor}
		
		\begin{proof}
			Combine \cref{thm:UCT_accordion} with the E-theoretic version of \cite[Corollary~4.6]{MN_2012}. 
		\end{proof}
		
		\begin{cor}\label{cor:filtrated_K-theory_complete_invariant_accordion}
			Let $X$ be an accordion space. If $A$ and $B$ are separable, stable, nuclear, $\calO_{\infty}$-stable, tight C*-algebras over $X$ belonging to $\Boot_{\E}(X)$, then every isomorphism $\FK^X(A) \simeq \FK^X(B)$ in $\Mod(\NT(X))$ lifts to an isomorphism $A \simeq B$ in $\SCstar(X)$. 
		\end{cor}
		
		\begin{proof}
			Combine \cref{cor:filtrated_K-theory_E-equivalence_accordion} with the E-theoretic version of Kirchberg's classification theorem specified in \cite[Theorem~7.1]{Gabe_2022}. 
		\end{proof}

		\appendix
		\crefalias{section}{appendix}

		\section{Proof of Proposition~\texorpdfstring{\ref{prop:CstarX_functor}}{\ref*{prop:CstarX_functor}}}\label{appendix:proof_CstarX_functor}
		
		In this appendix, we prove \cref{prop:CstarX_functor}. Fix a topological space $X$. Let $i \colon \Open(X) \hookrightarrow \LCcat(X)$ denote the inclusion functor. We may identify $\Cstar(X)$ with a full subcategory of $[\Open(X), \Cstar]$.
		It is easily checked that the functor 
		\[i^* \colon [\LCcat(X), \Cstar] \to [\Open(X), \Cstar]\] 
		restricts to a functor
		\[i^* \colon [\LCcat(X), \Cstar]_{\substack{\ex \\ \cont}} \to \Cstar(X). \]
		
		We first prove that $i^* \colon [\LCcat(X), \Cstar]_{\substack{\ex \\ \cont}} \to \Cstar(X)$ is fully faithful. 
		
		\begin{prop}\label{prop:faithful}
			The functor $i^* \colon [\LCcat(X), \Cstar]_{\substack{\ex \\ \cont}} \to \Cstar(X)$ is faithful. 
		\end{prop}
		
		\begin{proof}
			Let $A, B \in [\LCcat(X), \Cstar]_{\substack{\ex \\ \cont}}$. Let $\varphi, \varphi' \colon A \Rightarrow B$ be morphisms in $[\LCcat(X), \Cstar]$ such that $\varphi_U=\varphi'_U$ as $\ast$-homomorphisms $A(U) \to B(U)$ for every $U \in \Open(X)$. For $Y \in \LC(X)$, choose $U \in \Open(X)$ such that $Y \in \Closed(U)$. Then
			\[\varphi_Y \circ A(\pi_U^Y)=B(\pi_U^Y) \circ \varphi_U=B(\pi_U^Y) \circ \varphi'_U=\varphi'_Y \circ A(\pi_U^Y). \] 
			Since $A(\pi_U^Y) \colon A(U) \to A(Y)$ is surjective, we obtain $\varphi_Y=\varphi'_Y$. Therefore, $\varphi=\varphi'$. This shows that $i^*$ is faithful.  
		\end{proof}

		\begin{prop}\label{prop:full}
			The functor $i^* \colon [\LCcat(X), \Cstar]_{\substack{\ex \\ \cont}} \to \Cstar(X)$ is full. 
		\end{prop}
		
		\begin{proof}
			Fix $A, B \in [\LCcat(X), \Cstar]_{\substack{\ex \\ \cont}}$ and a morphism $\psi \colon i^* A \Rightarrow i^* B$ in $[\Open(X), \Cstar]$. For $U \in \Open(X)$ and $Y \in \Closed(U)$, consider the following diagram with exact rows, in which the dashed arrow is to be constructed:
			\[
			\begin{tikzcd}
				0 \ar{r} & A(U \setminus Y) \ar{r}{A(\iota_{U \setminus Y}^U)} \ar[swap]{d}{\psi_{U \setminus Y}} & A(U) \ar{r}{A(\pi_U^Y)} \ar[swap]{d}{\psi_U} & A(Y) \ar{r} \ar[dashed]{d}{\varphi_{Y, U}} & 0 \\
				0 \ar{r} & B(U \setminus Y) \ar[swap]{r}{B(\iota_{U \setminus Y}^U)} & B(U) \ar[swap]{r}{B(\pi_U^Y)} &  B(Y) \ar{r} & 0\mathrlap{.}
			\end{tikzcd}
			\]
			Since the left square commutes by the naturality of $\psi$, there exists a unique $\ast$-homomorphism $\varphi_{Y, U} \colon A(Y) \to B(Y)$ making the right square commute.
			
			We first prove that $\varphi_{Y, U}$ is independent of $U$. Let
			$U_1,U_2 \in \Open(X)$ with $Y \in \Closed(U_1) \cap \Closed(U_2)$, and put
			$U:=U_1 \cap U_2$. Then $U \in \Open(X)$ and $Y \in \Closed(U)$. 
			For $j=1, 2$, the following diagram commutes:
			\[
			\begin{tikzcd}[sep=large]
				A(U) \ar[twoheadrightarrow, bend left=20]{rr}{A(\pi_{U}^Y)} \ar[hookrightarrow, swap]{r}{A(\iota_{U}^{U_j})} \ar[swap]{d}{\psi_{U}} & A(U_j)  \ar[twoheadrightarrow, swap]{r}{A(\pi_{U_j}^Y)} \ar[swap]{d}{\psi_{U_j}} & A(Y) \ar{d}{\varphi_{Y, U_j}} \\
				B(U) \ar[twoheadrightarrow, bend right=20, swap]{rr}{B(\pi_{U}^Y)} \ar[hookrightarrow]{r}{B(\iota_{U}^{U_j})} & B(U_j) \ar[twoheadrightarrow]{r}{B(\pi_{U_j}^Y)} & B(Y)\mathrlap{.}
			\end{tikzcd}
			\] 
			Hence, 
			\[\varphi_{Y,U_j} \circ A(\pi_U^Y)=B(\pi_U^Y) \circ \psi_U. \]
			It follows that 
			\[\varphi_{Y, U_1}=\varphi_{Y, U}=\varphi_{Y, U_2}. \] 
			
			Thus, for $Y \in \LC(X)$, we define a $\ast$-homomorphism
			\[\varphi_Y :=\varphi_{Y,U} \colon A(Y) \to B(Y), \]
			where $U \in \Open(X)$ is arbitrary with $Y \in \Closed(U)$.
			
			We prove that $\varphi_Y$ is natural in $Y \in \LCcat(X)$. Let $Z \in \LC(X)$ and choose $U \in \Open(X)$ such that $Z \cap U=\emptyset$ and $Z \cup U \in \Open(X)$.
			For $Y \in \Open(Z)$, we have $Y \cap U=\emptyset$ and $Y \cup U \in \Open(X)$.
			Consider the diagram
			\[
			\begin{tikzcd}[sep=small]
				& {} &
				A(Y \cup U)
				\ar[swap]{dl}{\psi_{Y \cup U}}
				\ar[twoheadrightarrow, near start]{dd}{A(\pi_{Y \cup U}^Y)}
				\ar[hookrightarrow]{rr}{A(\iota_{Y \cup U}^{Z \cup U})}
				& & 
				A(Z \cup U)
				\ar[swap]{dl}{\psi_{Z \cup U}}
				\ar[twoheadrightarrow, near start]{dd}{A(\pi_{Z \cup U}^Z)}
				\\
				& 
				B(Y \cup U)
				\ar[twoheadrightarrow, swap]{dd}{B(\pi_{Y \cup U}^Y)}
				& & 
				B(Z \cup U)
				\ar[from=ll, hookrightarrow, crossing over, near start, swap, "B(\iota_{Y \cup U}^{Z \cup U})"]
				&
				\\
				& &
				A(Y)
				\ar{dl}{\varphi_Y}
				\ar[hookrightarrow, near start, swap]{rr}{A(\iota_Y^Z)}
				& & 
				A(Z)
				\ar{dl}{\varphi_Z}
				& 
				\\
				&
				B(Y)
				\ar[hookrightarrow, swap]{rr}{B(\iota_Y^Z)}
				& & 
				B(Z)\mathrlap{.}
				\ar[twoheadrightarrow, from=uu, crossing over, near start, "B(\pi_{Z \cup U}^Z)"]
				& & 
			\end{tikzcd}
			\]
			The commutativity of the squares except the bottom square and the surjectivity of $A(\pi_{Y \cup U}^Y)$ give the commutativity of the bottom square: 
			\[
			\varphi_Z \circ A(\iota_Y^Z)=B(\iota_Y^Z) \circ \varphi_Y. 
			\]
			For $Y \in \Closed(Z)$, the same argument, using the relation $\pi_Z^Y \pi_{Z \cup U}^Z=\pi_{Z \cup U}^Y$ in $\LCcat(X)$, gives
			\[\varphi_Y \circ A(\pi_Z^Y)=B(\pi_Z^Y) \circ \varphi_Z. \]
			
			Since every morphism from $Y_1$ to $Y_2$ in $\LCcat(X)$ is of the form
			$\iota_{Y_3}^{Y_2} \pi_{Y_1}^{Y_3}$ for some
			$Y_3 \in \Closed(Y_1) \cap \Open(Y_2)$, the preceding two equalities show that
			$\varphi_Y$ is natural in $Y \in \LCcat(X)$. Thus, the $\ast$-homomorphisms $\varphi_Y$ for $Y \in \LCcat(X)$ define a natural transformation $\varphi \colon A \Rightarrow B$. By construction, $i^*\varphi=\psi$. Consequently, $i^*$ is full.
		\end{proof}
		
		We next prove that $i^* \colon [\LCcat(X), \Cstar]_{\substack{\ex \\ \cont}} \to \Cstar(X)$ is essentially surjective. For this purpose, we shall use the following universal property of $\LCcat(X)$. 
		
		\begin{lem}\label{lem:universal_property_LC}
			Suppose we are given a category $\frakA$, objects $A_Y \in \frakA$ for $Y \in \LC(X)$, morphisms $\dot{\iota}_Y^Z \in \frakA(A_Y, A_Z)$ for $Z \in \LC(X)$ and $Y \in \Open(Z)$, and morphisms $\dot{\pi}_Z^Y \in \frakA(A_Z, A_Y)$ for $Z \in \LC(X)$ and $Y \in \Closed(Z)$, satisfying the following relations:
			\begin{enumerate}[label=\textnormal{(\roman*)}]
				\item $\dot{\iota}_Y^Y=\dot{\pi}_Y^Y=\id_{A_Y}$ for $Y \in \LC(X)$;
				
				\item $\dot{\iota}_{Y_2}^{Y_3}\dot{\iota}_{Y_1}^{Y_2}=\dot{\iota}_{Y_1}^{Y_3}$ for $Y_3 \in \LC(X)$, $Y_2 \in \Open(Y_3)$, and $Y_1 \in \Open(Y_2)$;
				
				\item $\dot{\pi}_{Y_2}^{Y_3}\dot{\pi}_{Y_1}^{Y_2}=\dot{\pi}_{Y_1}^{Y_3}$ for $Y_1 \in \LC(X)$, $Y_2 \in \Closed(Y_1)$, and $Y_3 \in \Closed(Y_2)$;
				
				\item $\dot{\pi}_{Y_2}^{Y_3}\dot{\iota}_{Y_1}^{Y_2}=\dot{\iota}_{Y_1 \cap Y_3}^{Y_3}\dot{\pi}_{Y_1}^{Y_1 \cap Y_3}$ for $Y_2 \in \LC(X)$, $Y_1 \in \Open(Y_2)$, and $Y_3 \in \Closed(Y_2)$.
			\end{enumerate}
			Then there exists a unique functor
			\[
			A \colon \LCcat(X) \to \frakA
			\]
			satisfying the following conditions:
			\begin{itemize}
				\item $A(Y)=A_Y$ for every $Y \in \LC(X)$; 
				\item for every $Z \in \LC(X)$, we have $A(\iota_Y^Z)=\dot{\iota}_Y^Z$ whenever $Y \in \Open(Z)$, and $A(\pi_Z^Y)=\dot{\pi}_Z^Y$ whenever
				$Y \in \Closed(Z)$.
			\end{itemize}
		\end{lem}
		
		\begin{proof}
			Uniqueness is clear.
			We define $A(Y):=A_Y$ for each $Y \in \LC(X)$, and
			\[
			A(\iota_Z^{Y_2}\pi_{Y_1}^Z)
			:=
			\dot{\iota}_Z^{Y_2}\dot{\pi}_{Y_1}^Z
			\]
			for $Y_1, Y_2 \in \LC(X)$ and $Z \in \Closed(Y_1) \cap \Open(Y_2)$.
			We show that $A$ preserves identity morphisms and composition.
			From (i), we have $A(\id_Y)=\id_{A(Y)}$ for $Y \in \LC(X)$.
			For $Y_1, Y_2, Y_3 \in \LC(X)$, $Z_1 \in \Closed(Y_1) \cap \Open(Y_2)$, and $Z_2 \in \Closed(Y_2) \cap \Open(Y_3)$, we have
			\begin{align*} 
				A(\iota_{Z_2}^{Y_3} \pi_{Y_2}^{Z_2}) A(\iota_{Z_1}^{Y_2} \pi_{Y_1}^{Z_1}) &=\dot{\iota}_{Z_2}^{Y_3} \dot{\pi}_{Y_2}^{Z_2} \dot{\iota}_{Z_1}^{Y_2} \dot{\pi}_{Y_1}^{Z_1}=\dot{\iota}_{Z_2}^{Y_3} \dot{\iota}_{Z_1 \cap Z_2}^{Z_2} \dot{\pi}_{Z_1}^{Z_1 \cap Z_2} \dot{\pi}_{Y_1}^{Z_1}=\dot{\iota}_{Z_1 \cap Z_2}^{Y_3} \dot{\pi}_{Y_1}^{Z_1 \cap Z_2} \\ 
				&=A(\iota_{Z_1 \cap Z_2}^{Y_3} \pi_{Y_1}^{Z_1 \cap Z_2})=A(\iota_{Z_2}^{Y_3} \pi_{Y_2}^{Z_2} \iota_{Z_1}^{Y_2} \pi_{Y_1}^{Z_1}). 
			\end{align*}
			The second equality follows from (iv), and the third equality follows from (ii) and (iii). The final equality follows by the same relations in $\LCcat(X)$. 
			Thus, $A$ is a functor $\LCcat(X) \to \frakA$.
			By construction, $A$ satisfies the desired conditions.
			This completes the proof.
		\end{proof}
		
		For the remainder of the proof, we fix $B \in \Cstar(X)$. For $U \in \Open(X)$ and $Y \in \Closed(U)$, set
		\[
		B_U(Y):=\frac{B(U)}{B(U \setminus Y)}.
		\]
		
		\begin{lem}\label{lem:isomorphism_between_quotients}
			Let $U_1, U_2 \in \Open(X)$ and $Y \in \Closed(U_1) \cap \Closed(U_2)$.
			For each $V \in \Open(X)$ with $Y \subset V \subset U_1 \cap U_2$, there exists a unique $\ast$-isomorphism
			\[
			\sigma_{U_1,Y}^{U_2,V} \colon B_{U_1}(Y) \to B_{U_2}(Y)
			\]
			making the following diagram commute:
			\[
			\begin{tikzcd}[row sep=tiny, column sep=small]
				& B(U_1) \ar[twoheadrightarrow]{r} & B_{U_1}(Y) \ar{dd}{\sigma_{U_1,Y}^{U_2,V}} \\
				B(V) \ar[hookrightarrow]{ur} \ar[hookrightarrow]{dr} & \\
				& B(U_2) \ar[twoheadrightarrow]{r} & B_{U_2}(Y)\mathrlap{.}
			\end{tikzcd}
			\]
			Moreover,
			\[
			\sigma_{U_1,Y}^{U_2,V_1}=\sigma_{U_1,Y}^{U_2,V_2}
			\]
			for $V_1, V_2 \in \Open(X)$ with $Y \subset V_1 \subset U_1 \cap U_2$ and $Y \subset V_2 \subset U_1 \cap U_2$.
		\end{lem}
		
		\begin{proof}
			We prove the first assertion. Let $V \in \Open(X)$ with $Y \subset V \subset U_1 \cap U_2$. Then $Y \in \Closed(V)$. 
			For $j=1, 2$, consider the following diagram with exact rows, in which the dashed arrow is to be constructed:
			\[
			\begin{tikzcd}
				0 \ar{r} & B(V \setminus Y) \ar{r} \ar[hookrightarrow]{d} & B(V) \ar{r} \ar[hookrightarrow]{d} & B_V(Y) \ar{r} \ar[dashed]{d} & 0 \\
				0 \ar{r} & B(U_j \setminus Y) \ar{r} & B(U_j) \ar{r} & B_{U_j}(Y) \ar{r} & 0\mathrlap{.}
			\end{tikzcd}
			\]
			Since the left square commutes, there exists a unique $\ast$-homomorphism $B_V(Y) \to B_{U_j}(Y)$ making the right square commute. Since
			\[
			V \cap (U_j \setminus Y)=V \setminus Y,
			\quad
			V \cup (U_j \setminus Y)=U_j,
			\]
			the defining property of $B \in \Cstar(X)$ gives
			\[
			B(V) \cap B(U_j \setminus Y)=B(V \setminus Y),
			\quad
			B(V)+B(U_j \setminus Y)=B(U_j).
			\]
			Hence, the induced $\ast$-homomorphism is a $\ast$-isomorphism. Using these $\ast$-isomorphisms, define
			\[
			\sigma_{U_1,Y}^{U_2,V} \colon
			B_{U_1}(Y)
			\xrightarrow{\simeq}
			B_V(Y)
			\xrightarrow{\simeq}
			B_{U_2}(Y).
			\]
			The diagram in the statement then commutes.
			Since $B(U_1)=B(V)+B(U_1 \setminus Y)$, the composite
			\[
			B(V) \hookrightarrow B(U_1) \twoheadrightarrow B_{U_1}(Y)
			\]
			is surjective. Thus, $\sigma_{U_1,Y}^{U_2,V}$ is the unique $\ast$-homomorphism $B_{U_1}(Y) \to B_{U_2}(Y)$ making the diagram in the statement commute.
			
			We prove the second assertion. 
			Let $V_1, V_2 \in \Open(X)$ satisfy the assumptions and put $V:=V_1 \cap V_2$. Then $V \in \Open(X)$ and $Y \subset V \subset U_1 \cap U_2$. For $j=1,2$, the commutativity of the defining diagram for $\sigma_{U_1,Y}^{U_2,V_j}$ remains valid after restricting the two composites from $B(V_j)$ to $B(V)$. Hence, $\sigma_{U_1,Y}^{U_2,V_j}$ also makes the defining diagram for $\sigma_{U_1,Y}^{U_2,V}$ commute. Therefore, we obtain $\sigma_{U_1,Y}^{U_2,V_1}=\sigma_{U_1,Y}^{U_2,V}=\sigma_{U_1,Y}^{U_2,V_2}$. 
		\end{proof}
		
		In view of \cref{lem:isomorphism_between_quotients}, for $U_1, U_2 \in \Open(X)$ and $Y \in \Closed(U_1) \cap \Closed(U_2)$, we define a $\ast$-isomorphism
		\[
		\sigma_{U_1,Y}^{U_2}:=\sigma_{U_1,Y}^{U_2,V} \colon B_{U_1}(Y) \to B_{U_2}(Y),
		\]
		where $V \in \Open(X)$ is arbitrary with $Y \subset V \subset U_1 \cap U_2$.
		
		\begin{lem}\label{lem:composition_alpha}
			For $U_1, U_2, U_3 \in \Open(X)$ and $Y \in \Closed(U_1) \cap \Closed(U_2) \cap \Closed(U_3)$, we have 
			\[\sigma_{U_2, Y}^{U_3} \circ \sigma_{U_1, Y}^{U_2}=\sigma_{U_1, Y}^{U_3}. \]
		\end{lem}

		\begin{proof}
			Put $V:=U_1 \cap U_2 \cap U_3$. Then $V \in \Open(X)$ and $Y \subset V$. 
			By the preceding convention, we may use $V$ in the definitions of all three $\ast$-isomorphisms in the statement. 
			The commutativity of the defining diagrams for $\sigma_{U_1, Y}^{U_2, V}$ and $\sigma_{U_2, Y}^{U_3, V}$ shows that
			$\sigma_{U_2, Y}^{U_3, V} \circ \sigma_{U_1, Y}^{U_2, V}$ also makes the defining diagram for $\sigma_{U_1, Y}^{U_3, V}$ commute. Thus, we obtain $\sigma_{U_2, Y}^{U_3} \circ \sigma_{U_1, Y}^{U_2}=\sigma_{U_1, Y}^{U_3, V}=\sigma_{U_1, Y}^{U_3}$.
		\end{proof}
		
		For $Y \in \LC(X)$, we fix $U_Y \in \Open(X)$ with $Y \in \Closed(U_Y)$, and define
		\[
		A_Y:=B_{U_Y}(Y).
		\]
		For $U \in \Open(X)$ and $Y \in \Closed(U)$, we define a surjective $\ast$-homomorphism
		\[
		\rho_U^Y \colon B(U) \twoheadrightarrow B_U(Y) \xrightarrow{\sigma_{U,Y}^{U_Y}} B_{U_Y}(Y)=A_Y.
		\]
		Note that $\rho_Y^Y=\sigma_{Y,Y}^{U_Y}$ for $Y \in \Open(X)$.
		
		\begin{lem}\label{lem:rho_iota}
			For $U, V \in \Open(X)$ and $Y \in \Closed(U)$ with $Y \subset V \subset U$, we have $Y \in \Closed(V)$ and the following diagram commutes:
			\[
			\begin{tikzcd}
				B(V) \ar[hookrightarrow]{d} \ar[twoheadrightarrow]{dr}{\rho_V^Y} & \\
				B(U) \ar[twoheadrightarrow, swap]{r}{\rho_U^Y} & A_Y\mathrlap{.}
			\end{tikzcd}
			\]
		\end{lem}
		
		\begin{proof}
			It is clear that $Y \in \Closed(V)$. Consider the diagram
			\[
			\begin{tikzcd}
				B(V) \ar[hookrightarrow]{d} \ar[twoheadrightarrow]{r} & B_V(Y) \ar[swap]{d}{\sigma_{V,Y}^{U}} \ar{dr}{\sigma_{V,Y}^{U_Y}} & \\
				B(U) \ar[twoheadrightarrow]{r} & B_U(Y) \ar[swap]{r}{\sigma_{U,Y}^{U_Y}} & B_{U_Y}(Y)=A_Y\mathrlap{.}
			\end{tikzcd}
			\]
			The left square commutes by the definition of $\sigma_{V,Y}^{U}$; see \cref{lem:isomorphism_between_quotients}. The right triangle commutes by \cref{lem:composition_alpha}. This shows the desired commutativity.
		\end{proof}
		
		Let $Z \in \LC(X)$ and $Y \in \Closed(Z)$.
		For $U \in \Open(X)$ with $Z \in \Closed(U)$, we have $Y \in \Closed(U)$. Consider the following diagram with exact rows, in which the dashed arrow is to be constructed:
		\[
		\begin{tikzcd}
			0 \ar{r} & B(U \setminus Z) \ar{r} \ar[hookrightarrow]{d} & B(U) \ar{r}{\rho_U^Z} \ar[equal]{d} & A_Z \ar{r} \ar[dashed]{d}{\dot{\pi}_{Z, U}^Y} & 0 \\
			0 \ar{r} & B(U \setminus Y) \ar{r} & B(U) \ar[swap]{r}{\rho_U^Y} & A_Y \ar{r} & 0\mathrlap{.}
		\end{tikzcd}
		\]
		Since the left square commutes, there exists a unique $\ast$-homomorphism $\dot{\pi}_{Z, U}^Y \colon A_Z \to A_Y$ making the right square commute.
		
		We show that $\dot{\pi}_{Z, U}^Y$ is independent of $U$.
		Let $U_1, U_2 \in \Open(X)$ with $Z \in \Closed(U_1) \cap \Closed(U_2)$, and put $U:=U_1 \cap U_2$.
		Then $U \in \Open(X)$ and $Z \in \Closed(U)$. 
		For $j=1,2$, \cref{lem:rho_iota} shows that the commutativity of the defining diagram for $\dot{\pi}_{Z,U_j}^Y$ remains valid after restricting $\rho_{U_j}^Z$ and $\rho_{U_j}^Y$ to $B(U)$. Hence, $\dot{\pi}_{Z,U_j}^Y$ also makes the defining diagram for $\dot{\pi}_{Z,U}^Y$ commute. Hence,
		\[
		\dot{\pi}_{Z,U_1}^Y=\dot{\pi}_{Z,U}^Y=\dot{\pi}_{Z,U_2}^Y.
		\]
		
		Thus, for $Z \in \LC(X)$ and $Y \in \Closed(Z)$, we define a $\ast$-homomorphism
		\[
		\dot{\pi}_Z^Y:=\dot{\pi}_{Z, U}^Y \colon A_Z \to A_Y, 
		\]
		where $U \in \Open(X)$ is arbitrary with $Z \in \Closed(U)$.
		
		Let $Z \in \LC(X)$ and $Y \in \Open(Z)$. 
		For $U, V \in \Open(X)$ satisfying $V \subset U$, $Y \cap V=Z \cap U=\emptyset$, and $Y \cup V, Z \cup U \in \Open(X)$, consider the following diagram with exact rows, in which the dashed arrow is to be constructed:
		\[
		\begin{tikzcd}
			0 \ar{r} & B(V) \ar{r} \ar[hookrightarrow]{d} & B(Y \cup V) \ar{r}{\rho_{Y \cup V}^Y} \ar[hookrightarrow]{d} & A_Y \ar{r} \ar[dashed]{d}{\dot{\iota}_{Y, V}^{Z, U}} & 0 \\
			0 \ar{r} & B(U) \ar{r} & B(Z \cup U) \ar[swap]{r}{\rho_{Z \cup U}^Z} & A_Z \ar{r} & 0\mathrlap{.}
		\end{tikzcd}
		\] 
		Since the left square commutes, there exists a unique $\ast$-homomorphism $\dot{\iota}_{Y, V}^{Z, U} \colon A_Y \to A_Z$ making the right square commute. 
		
		We show that $\dot{\iota}_{Y, V}^{Z, U}$ is independent of $U$ and $V$.
		Let $(U_j, V_j)$ for $j=1,2$ satisfy the above conditions, and put $U:=U_1 \cap U_2$ and $V:=V_1 \cap V_2$. Then $(U, V)$ also satisfies the above conditions. 
		For $j=1,2$, \cref{lem:rho_iota} shows that the commutativity of the defining diagram for $\dot{\iota}_{Y,V_j}^{Z,U_j}$ remains valid after restricting $\rho_{Y \cup V_j}^Y$ and $\rho_{Z \cup U_j}^Z$ to $B(Y \cup V)$ and $B(Z \cup U)$, respectively. Hence, $\dot{\iota}_{Y,V_j}^{Z,U_j}$ also makes the defining diagram for $\dot{\iota}_{Y, V}^{Z, U}$ commute. 
		Hence,
		\[
		\dot{\iota}_{Y, V_1}^{Z, U_1}=\dot{\iota}_{Y, V}^{Z, U}=\dot{\iota}_{Y, V_2}^{Z, U_2}.
		\]
		
		Thus, for $Z \in \LC(X)$ and $Y \in \Open(Z)$, we define a $\ast$-homomorphism
		\[
		\dot{\iota}_Y^Z:=\dot{\iota}_{Y, V}^{Z, U} \colon A_Y \to A_Z, 
		\]
		where $U, V \in \Open(X)$ are arbitrary with the above conditions.

		\begin{lem}
			The C*-algebras $A_Y$ for $Y \in \LC(X)$, the $\ast$-homomorphisms $\dot{\iota}_Y^Z \colon A_Y \to A_Z$ for $Z \in \LC(X)$ and $Y \in \Open(Z)$, and the $\ast$-homomorphisms $\dot{\pi}_Z^Y \colon A_Z \to A_Y$ for $Z \in \LC(X)$ and $Y \in \Closed(Z)$, satisfy the conditions \textup{(i)--(iv)} specified in \textup{\cref{lem:universal_property_LC}}. 
		\end{lem}
		
		\begin{proof}
			We only prove (iv), since the proofs of (i)--(iii) are similar. Let $Y_2 \in \LC(X)$, $Y_1 \in \Open(Y_2)$, and $Y_3 \in \Closed(Y_2)$. Choose $U \in \Open(X)$ with $Y_2 \cap U=\emptyset$ and $Y_2 \cup U \in \Open(X)$. Consider the diagram
			\[
			\begin{tikzcd}[column sep=small]
				B(Y_1 \cup U) \ar[hookrightarrow]{rr} \ar[twoheadrightarrow]{dr}[near start, below]{\rho_{Y_1 \cup U}^{Y_1}} \ar[twoheadrightarrow]{dd}[swap]{\rho_{Y_1 \cup U}^{Y_1 \cap Y_3}} & & B(Y_2 \cup U) \ar[twoheadrightarrow]{dr}{\rho_{Y_2 \cup U}^{Y_2}} &  \\
				& A_{Y_1} \ar[hookrightarrow, near start, swap]{rr}{\dot{\iota}_{Y_1}^{Y_2}} \ar[twoheadrightarrow]{dl}{\dot{\pi}_{Y_1}^{Y_1 \cap Y_3}} & & A_{Y_2} \ar[twoheadrightarrow]{dl}{{\dot{\pi}_{Y_2}^{Y_3}}} \\
				A_{Y_1 \cap Y_3} \ar[hookrightarrow, swap]{rr}{\dot{\iota}_{Y_1 \cap Y_3}^{Y_3}} & & A_{Y_3}\mathrlap{.} \ar[from=uu, twoheadrightarrow, crossing over, near start, swap, "\rho_{Y_2 \cup U}^{Y_3}"] &
			\end{tikzcd}
			\]
			By the defining properties of $\dot{\pi}_{Y_2}^{Y_3}$, $\dot{\pi}_{Y_1}^{Y_1 \cap Y_3}$, and $\dot{\iota}_{Y_1}^{Y_2}$, the corresponding faces commute. For $\dot{\iota}_{Y_1 \cap Y_3}^{Y_3}$, apply its defining property with
			\[
			U':=(Y_2 \setminus Y_3) \cup U, \quad V':=(Y_1 \setminus Y_3) \cup U.
			\]
			These sets satisfy the required conditions, and
			\[
			(Y_1 \cap Y_3) \cup V'=Y_1 \cup U, \quad
			Y_3 \cup U'=Y_2 \cup U.
			\]
			Hence, all faces except the bottom square commute.
			Since $\rho_{Y_1 \cup U}^{Y_1}$ is surjective, the bottom square also commutes. This proves (iv). 
		\end{proof}
		
		By \cref{lem:universal_property_LC}, there exists a unique functor $A \colon \LCcat(X) \to \Cstar$ such that $A(Y)=A_Y$ for all $Y \in \LC(X)$, $A(\iota_Y^Z)=\dot{\iota}_Y^Z$ for $Z \in \LC(X)$ and $Y \in \Open(Z)$, and $A(\pi_Z^Y)=\dot{\pi}_Z^Y$ for $Z \in \LC(X)$ and $Y \in \Closed(Z)$.
		
		\begin{lem}
			We have $A \in [\LCcat(X), \Cstar]_{\substack{\ex \\ \cont}}$. 
		\end{lem}
		
		\begin{proof}
			Let $Z \in \LC(X)$ and $Y \in \Open(Z)$. Choose $U \in \Open(X)$ with $Z \cap U=\emptyset$ and $Z \cup U \in \Open(X)$. Then $Y \cap U=\emptyset$ and $Y \cup U \in \Open(X)$. The following diagram commutes:
			\[
			\begin{tikzcd}
				& 0 \ar{d} & 0 \ar{d} & & \\
				& B(U) \ar[equal]{r} \ar{d} & B(U) \ar{d} & & \\
				0\ar{r} & B(Y \cup U) \ar{r} \ar[swap]{d}{\rho_{Y \cup U}^Y} & B(Z \cup U) \ar{r}{\rho_{Z \cup U}^{Z \setminus Y}} \ar[swap]{d}{\rho_{Z \cup U}^Z} & A(Z \setminus Y) \ar{r} \ar[equal]{d} & 0 \\
				0\ar{r} & A(Y) \ar[swap]{r}{A(\iota_Y^Z)} \ar{d} & A(Z) \ar[swap]{r}{A(\pi_Z^{Z \setminus Y})} \ar{d} & A(Z \setminus Y) \ar{r} & 0 \\
				& 0 & 0\mathrlap{.} & & 
			\end{tikzcd}
			\]
			Since the left and middle columns and the middle row are exact, so is the bottom row by the nine lemma.  
			
			Let $Y \in \LC(X)$, let $\Lambda=(\Lambda, R_{\Lambda})$ be an upward-directed set, and let $(Y_{\lambda})_{\lambda \in \Lambda} \subset \Open(Y)$ be an increasing net with $Y=\bigcup_{\lambda \in \Lambda} Y_{\lambda}$. Choose $U \in \Open(X)$ with $Y \cap U=\emptyset$ and $Y \cup U \in \Open(X)$. Then, for each $\lambda \in \Lambda$, we have $Y_{\lambda} \cap U=\emptyset$ and $Y_{\lambda} \cup U \in \Open(X)$. For $(\lambda, \mu) \in R_{\Lambda}$, we have the following commutative diagram with exact rows:
			\[
			\begin{tikzcd}[column sep=tiny]
				& 0
				\ar{rr}
				& &
				[1em] B(U)
				\ar{dd}{\id}
				\ar{rr}
				& & 
				[-1em] B(Y_{\lambda} \cup U)
				\ar[hookrightarrow]{dd}
				\ar{rr}{\rho_{Y_{\lambda} \cup U}^{Y_{\lambda}}}
				& & 
				[1em] A(Y_{\lambda})
				\ar[hookrightarrow]{dd}{A(\iota_{Y_{\lambda}}^{Y_{\mu}})}
				\ar{rr}
				& &
				0
				\\
				& & & & & & & & &
				\\
				&
				0
				\ar{rr}
				& &
				B(U)
				\ar{rr}
				\ar{dl}{\id}
				& & 
				B(Y_{\mu} \cup U)
				\ar[near start]{rr}{\rho_{Y_{\mu} \cup U}^{Y_{\mu}}}
				\ar[hookrightarrow]{dl}
				& & 
				A(Y_{\mu})
				\ar{rr}
				\ar[hookrightarrow]{dl}{A(\iota_{Y_{\mu}}^Y)}
				& &
				0 
				\\
				0
				\ar{rr}
				& &
				B(U)
				\ar{rr}
				\ar[from=uuur, crossing over, bend right=20, swap, "\id"]
				& & 
				B(Y \cup U)
				\ar[swap]{rr}{\rho_{Y \cup U}^Y}
				\ar[from=uuur, crossing over, hookrightarrow, bend right=20]
				& & 
				A(Y)
				\ar{rr}
				\ar[from=uuur, crossing over, hookrightarrow, near start, bend right=20, swap, "A(\iota_{Y_{\lambda}}^Y)"]
				& &
				0\mathrlap{.}
			\end{tikzcd}
			\]
			Since $Y \cup U \in \Open(X)$ and $(Y_{\lambda} \cup U)_{\lambda \in \Lambda}$ is an increasing net in $\Open(Y \cup U)$, the defining property of $B \in \Cstar(X)$ gives
			\[B(Y \cup U)=\overline{\bigcup_{\lambda \in \Lambda} B(Y_{\lambda} \cup U)}. \]
			By the exactness of the inductive limit functor $[\Lambda, \Cstar] \to \Cstar$, we obtain
			\[A(Y)=\overline{\bigcup_{\lambda \in \Lambda} A(\iota_{Y_{\lambda}}^Y)(A(Y_{\lambda}))}. \]
			Therefore, $A \in [\LCcat(X), \Cstar]_{\substack{\ex \\ \cont}}$. 
		\end{proof}
		
		\begin{lem}
			We have $i^*A \simeq B$ in $\Cstar(X)$.
		\end{lem}
		
		\begin{proof}
			For $U \in \Open(X)$, by the defining property of $A(\iota_U^X)=\dot{\iota}_{U, \emptyset}^{X, \emptyset}$, the following diagram commutes:
			\[
			\begin{tikzcd}
				B(U) \ar{r}{\rho_U^U} \ar[hookrightarrow]{d} & A(U) \ar[hookrightarrow]{d}{A(\iota_U^X)}\\
				B(X) \ar[swap]{r}{\rho_X^X} & A(X)\mathrlap{.}
			\end{tikzcd}
			\] 
			Hence, $\rho_X^X \colon B(X) \to A(X)$ is an $X$-equivariant $\ast$-homomorphism $B \to i^*A$. Since $\rho_U^U \colon B(U) \to  A(U)$ is a $\ast$-isomorphism for every $U \in \Open(X)$, we obtain $B \simeq i^*A$ in $\Cstar(X)$. 
		\end{proof}
		
		The preceding discussion proves the following proposition. 
		
		\begin{prop}\label{prop:essentially_surjective}
			The functor $i^* \colon [\LCcat(X), \Cstar]_{\substack{\ex \\ \cont}} \to \Cstar(X)$ is essentially surjective. 
		\end{prop}
		
		By \cref{prop:faithful,prop:full,prop:essentially_surjective}, $i^* \colon [\LCcat(X), \Cstar]_{\substack{\ex \\ \cont}} \to \Cstar(X)$ is an equivalence of categories. 
		We have completed the proof of \cref{prop:CstarX_functor}.

		\section{Proof of Lemma~\texorpdfstring{\ref{lem:mapping_cone_of_mapping_cone}}{\ref*{lem:mapping_cone_of_mapping_cone}}}\label{appendix:proof_of_mapping_cones_of_mapping_cones}

		We first establish the universal property of the higher-dimensional mapping cone functor $M^I \colon [\{0, 1\}^I, \Cstar] \to \Cstar$. We shall view $M^I$ as the limit of a certain diagram; see \cref{prop:mapping_cone_limit}. 
		
		\begin{defn}
			Let $I$ be a finite set.
			Define a partial order $R_{R_{\{0, 1\}^I}}$ on $R_{\{0, 1\}^I}$ by
			\[
			R_{R_{\{0, 1\}^I}}
			:=
			\bigl\{
			((\lambda, \mu), (\lambda', \mu'))
			\in R_{\{0, 1\}^I} \times R_{\{0, 1\}^I}
			\bigm|
			(\lambda', \lambda), (\mu, \mu')
			\in R_{\{0, 1\}^I}
			\bigr\}.
			\]
			We regard the partially ordered set $\bigl(R_{\{0, 1\}^I}, R_{R_{\{0, 1\}^I}}\bigr)$
			as a category and denote it again by $R_{\{0, 1\}^I}$.
		\end{defn}

		\begin{eg}\label{eg:relation_I_singleton}
			If $|I|=1$, then $R_{\{0, 1\}^I}=\{(0, 0), (0, 1), (1, 1)\}$. 
			Regarded as a category, $R_{\{0, 1\}^I}$ can be depicted as follows, where identity morphisms and morphisms obtained by composition are omitted. 
			\[
			\begin{tikzcd}
				& (0, 0) \ar{d} \\
				(1, 1) \ar{r} & (0, 1)
			\end{tikzcd}
			\]
		\end{eg}
		
		\begin{eg}\label{eg:relation_I_two_point_set}
			If $|I|=2$, then 
			\[R_{\{0, 1\}^I}=\{(\lambda, \lambda) \mid \lambda \in \{0, 1\}^I\} \cup \{(00, 01), (00, 10), (00, 11), (01, 11), (10, 11)\}. \]
			Regarded as a category, $R_{\{0, 1\}^I}$ can be depicted as follows, where identity morphisms and morphisms obtained by composition are omitted. 
			\[
			\begin{tikzcd}
				(00, 00) \ar{r} \ar{dr} & (00, 01) \ar{dr} & \\
				(01, 01) \ar{ur} \ar{dr} & (00, 10) \ar{r} & (00, 11) \\
				(10, 10) \ar{ur} \ar{dr} & (01, 11) \ar{ur} & \\
				(11, 11) \ar{r} \ar{ur} & (10, 11) \ar{uur} &
			\end{tikzcd}
			\]
		\end{eg}
		
		\begin{defn}
			Let $I$ be a finite set.
			We define a functor
			\[
			D^I
			\colon
			R_{\{0, 1\}^I}
			\to
			[[\{0, 1\}^I, \Cstar], \Cstar]
			\]
			as follows.
			For $(\lambda, \mu) \in R_{\{0, 1\}^I}$, define a functor $D^I_{\lambda, \mu} \colon [\{0, 1\}^I, \Cstar] \to \Cstar$
			by
			\[
			D^I_{\lambda, \mu}(A; \varphi):=\conti_0(\Omega_{\lambda}, A_{\mu})
			\]
			for $(A; \varphi) \in [\{0, 1\}^I, \Cstar]$, and
			\[
			D^I_{\lambda, \mu}(\xi):=\xi_{\mu*} \colon \conti_0(\Omega_{\lambda}, A_{\mu}) \to \conti_0(\Omega_{\lambda}, B_{\mu})
			\]
			for a morphism $\xi \colon (A; \varphi) \Rightarrow (B; \psi)$ in $[\{0, 1\}^I, \Cstar]$.
			For $((\lambda, \mu), (\lambda', \mu')) \in R_{R_{\{0, 1\}^I}}$, define a natural transformation $D^I_{\lambda, \mu} \Rightarrow D^I_{\lambda', \mu'}$ whose component at $(A; \varphi) \in [\{0, 1\}^I, \Cstar]$ is the $\ast$-homomorphism
			\[
			\conti_0(\Omega_{\lambda}, A_{\mu})
			\to
			\conti_0(\Omega_{\lambda'}, A_{\mu'}),
			\quad
			f
			\mapsto
			\varphi_{\mu}^{\mu'} \circ f \circ \omega_{\lambda'}^{\lambda}.
			\]
			These assignments define the functor $D^I$.
		\end{defn}
		
		\begin{eg}\label{eg:mapping_cone_pullback}
			Consider the case where $|I|=1$; see \cref{eg:relation_I_singleton}. 
			For $(A; \varphi) \in [\{0, 1\}^I, \Cstar]$, the object $D^I(A; \varphi) \in [R_{\{0, 1\}^I}, \Cstar]$ is the following diagram in $\Cstar$.
			\[
			\begin{tikzcd}
				& A_0 \ar{d}{\varphi} \\
				CA_1 \ar[twoheadrightarrow]{r}[swap]{\ev_0} & A_1
			\end{tikzcd}
			\]
			The mapping cone $M^I(A; \varphi)=C_{\varphi}$ is the limit of this diagram, that is, the pullback along the $\ast$-homomorphisms $\varphi \colon A_0 \to A_1$ and $\ev_0 \colon CA_1 \to A_1$. The pullback diagram fits into the following commutative diagram with exact rows in $\Cstar$:
			\[
			\begin{tikzcd}
				0 \ar{r} & SA_1 \ar{r} \ar[equal]{d} & C_{\varphi} \ar{r} \ar{d} & A_0 \ar{r} \ar{d}{\varphi} & 0 \\
				0 \ar{r} & SA_1 \ar{r} & CA_1 \ar[swap]{r}{\ev_0} & A_1 \ar{r} & 0\mathrlap{.}
			\end{tikzcd}
			\] 
			The upper row is the mapping cone short exact sequence of $\varphi$. 
		\end{eg}
		
		\begin{eg}
			Consider the case where $|I|=2$; see \cref{eg:relation_I_two_point_set}. 
			For $(A; \varphi) \in [\{0, 1\}^I, \Cstar]$, the object $D^I(A; \varphi) \in [R_{\{0, 1\}^I}, \Cstar]$ is the following commutative diagram in $\Cstar$.
			\[
			\begin{tikzcd}[column sep=huge, row sep=large]
				A_{00} \ar{r}{\varphi_{00}^{01}} \ar{dr}[near start, xshift=-0.5em]{\varphi_{00}^{10}} & A_{01} \ar{dr}{\varphi_{01}^{11}} & \\
				CA_{01} \ar[twoheadrightarrow]{ur}[swap, near start, xshift=-0.5em]{\ev_0} \ar{dr}[near start, xshift=-0.5em]{C\varphi_{01}^{11}} & A_{10} \ar{r}{\varphi_{10}^{11}} & A_{11} \\
				CA_{10} \ar[twoheadrightarrow]{ur}[swap, near start, xshift=-0.5em]{\ev_0} \ar{dr}[near start, xshift=-0.5em]{C\varphi_{10}^{11}} & CA_{11} \ar[twoheadrightarrow]{ur}{\ev_0} & \\
				C^2A_{11} \ar[twoheadrightarrow]{ur}[swap, near start, xshift=-0.5em]{\ev_0C} \ar[twoheadrightarrow]{r}[swap]{C\ev_0} & CA_{11} \ar[twoheadrightarrow]{uur}[swap]{\ev_0} &
			\end{tikzcd}
			\]
			Here, the $\ast$-homomorphisms $\ev_0C$ and $C\ev_0$ are given by
			\[\quad \ev_0C(f)(t):=f(0, t), \quad C\ev_0(f)(t):=f(t, 0)\]
			for $f \in C^2A_{11}=\conti_0([0, 1) \times [0, 1), A_{11})$ and $t \in [0, 1)$. 
			The $2$-dimensional mapping cone $M^I(A; \varphi)$ is the limit of this diagram. 
		\end{eg}
		
		The following proposition is straightforward to verify.
		
		\begin{prop}\label{prop:mapping_cone_limit}
			Let $I$ be a finite set.
			For each $(\lambda, \mu) \in R_{\{0, 1\}^I}$, consider the morphism
			\[
			M^I \Rightarrow D^I_{\lambda, \mu}
			\]
			in $[[\{0, 1\}^I, \Cstar], \Cstar]$ whose component at $(A; \varphi) \in [\{0, 1\}^I, \Cstar]$ is the $\ast$-homomorphism
			\[
			M^I(A;\varphi)
			\to
			D^I_{\lambda, \mu}(A; \varphi),
			\quad
			f
			\mapsto
			\varphi_{\lambda}^{\mu}\circ f_{\lambda}
			=
			f_{\mu} \circ \omega_{\lambda}^{\mu}.
			\]
			Together with these morphisms, the functor 
			\[M^I \colon [\{0, 1\}^I, \Cstar] \to \Cstar\]
			is the limit of the diagram
			$D^I \colon R_{\{0, 1\}^I} \to [[\{0, 1\}^I, \Cstar], \Cstar]$. 
		\end{prop}
		
		We shall prove \cref{lem:mapping_cone_of_mapping_cone}.
		
		Note that the category $\Cstar$ is complete, that is, every small
		diagram in $\Cstar$ has a limit; see
		\cite[Proposition~19]{Meyer_2008}.
		For any finite set $I$, the functor category
		$[\{0, 1\}^I, \Cstar]$ is therefore also complete,
		with limits computed pointwise.
		
		We first establish the following lemma.
		
		\begin{lem}\label{lem:diagram_finite_limit}
			For a finite set $I$ and
			$(\lambda, \mu) \in R_{\{0, 1\}^I}$, the functor
			\[
			D^I_{\lambda, \mu}
			\colon
			[\{0, 1\}^I, \Cstar]
			\to
			\Cstar
			\]
			commutes with finite limits.
		\end{lem}

		\begin{proof}
			Since the functor $\conti_0(\Omega_{\lambda}) \otimes \blank \colon \Cstar \to \Cstar$ preserves zero objects and commutes with pullbacks, it commutes with finite limits. This shows the assertion. 
		\end{proof}

		\begin{lem}\label{lem:product_relations}
			Let $I_1$ and $I_2$ be finite sets and $I:=I_1 \amalg I_2$. 
			Then
			\[
			R_{\{0, 1\}^{I_1}} \times R_{\{0, 1\}^{I_2}}=R_{\{0, 1\}^I}
			\]
			as categories. 
		\end{lem}	
		
		\begin{proof}
			Recall from \cref{lem:product_index_sets} that $R_{\{0, 1\}^{I_1}} \times R_{\{0, 1\}^{I_2}}=R_{\{0, 1\}^I}$ as sets. 
			It is routine to verify that $R_{R_{\{0, 1\}^{I_1}}} \times R_{R_{\{0, 1\}^{I_2}}}=R_{R_{\{0, 1\}^I}}$ as sets. This shows the assertion. 
		\end{proof}

		\begin{lem}\label{lem:composition_diagrams}
			Let $I_1$ and $I_2$ be finite sets and $I:=I_1 \amalg I_2$. For $(\lambda, \mu) \in R_{\{0, 1\}^{I_1}}$ and $(\lambda', \mu') \in R_{\{0, 1\}^{I_2}}$, the composite
			\[
			[\{0, 1\}^{I_2}, [\{0, 1\}^{I_1}, \Cstar]]
			\xrightarrow{(D^{I_1}_{\lambda, \mu})_*}
			[\{0, 1\}^{I_2}, \Cstar]
			\xrightarrow{D^{I_2}_{\lambda', \mu'}}
			\Cstar
			\]
			is equal to the functor
			\[
			D^I_{(\lambda, \lambda'), (\mu, \mu')}
			\colon
			[\{0, 1\}^I, \Cstar] \to \Cstar.
			\]
		\end{lem}
		
		\begin{proof}
			For $\lambda \in \{0, 1\}^{I_1}$ and $\lambda' \in \{0, 1\}^{I_2}$, we have $\Omega_{\lambda} \times \Omega_{\lambda'}=\Omega_{(\lambda, \lambda')}$ and hence
			\[
			\conti_0(\Omega_{\lambda})\otimes\conti_0(\Omega_{\lambda'})
			=\conti_0(\Omega_{\lambda} \times \Omega_{\lambda'})
			=\conti_0(\Omega_{(\lambda, \lambda')}).
			\]
			Moreover, for $(\lambda, \mu) \in R_{\{0, 1\}^{I_1}}$ and $(\lambda', \mu') \in R_{\{0, 1\}^{I_2}}$, the continuous map
			\[\omega_{\lambda}^{\mu} \times \omega_{\lambda'}^{\mu'} \colon \Omega_{\lambda} \times \Omega_{\lambda'}
			\to
			\Omega_{\mu} \times \Omega_{\mu'}
			\]
			is equal to the continuous map
			\[\omega_{(\lambda, \lambda')}^{(\mu, \mu')} \colon \Omega_{(\lambda, \lambda')} \to \Omega_{(\mu, \mu')}. \]
			This shows the assertion. 
		\end{proof}

		\begin{proof}[Proof of \textup{\cref{lem:mapping_cone_of_mapping_cone}}]
			By \cref{prop:mapping_cone_limit}, the mapping cones involved can be identified with the corresponding limits. Thus, 
			\begin{align*}
				M^{I_2}M^{I_1}_*
				&=
				\biggl(\lim_{(\lambda', \mu') \in R_{\{0, 1\}^{I_2}}}
				D^{I_2}_{\lambda', \mu'}\biggr)
				\biggl(\lim_{(\lambda, \mu) \in R_{\{0, 1\}^{I_1}}}
				(D^{I_1}_{\lambda, \mu})_*\biggr) \\
				&=
				\lim_{(\lambda', \mu') \in R_{\{0, 1\}^{I_2}}}
				\lim_{(\lambda, \mu) \in R_{\{0, 1\}^{I_1}}}
				D^{I_2}_{\lambda', \mu'}(D^{I_1}_{\lambda, \mu})_*\\
				&=
				\lim_{(\lambda', \mu') \in R_{\{0, 1\}^{I_2}}}
				\lim_{(\lambda, \mu) \in R_{\{0, 1\}^{I_1}}}
				D^I_{(\lambda, \lambda'), (\mu, \mu')}\\
				&=
				\lim_{((\lambda, \mu), (\lambda', \mu')) \in R_{\{0, 1\}^{I_1}} \times R_{\{0, 1\}^{I_2}}}
				D^I_{(\lambda, \lambda'), (\mu, \mu')}\\
				&=
				\lim_{(\lambda'', \mu'') \in R_{\{0, 1\}^I}}
				D^I_{\lambda'', \mu''}
				=
				M^I.
			\end{align*}
			The second equality follows from \cref{lem:diagram_finite_limit} because $R_{\{0, 1\}^{I_1}}$ is finite. The third equality follows from \cref{lem:composition_diagrams}. The fourth equality follows from the canonical identification of the iterated limit with the limit over the product category.
			The fifth equality follows from \cref{lem:product_relations}. This completes the proof. 
		\end{proof}

		\section{A dual version of Lemma~\texorpdfstring{\ref{lem:homotopy_2-dimensional_mapping_cone}}{\ref*{lem:homotopy_2-dimensional_mapping_cone}}}\label{appendix:a_dual_version}
		
		The following lemma is a dual version of \cref{lem:homotopy_2-dimensional_mapping_cone}, which is also based on an idea of Katsura. 
		
		\begin{lem}\label{lem:antihomotopy_2-dimensional_mapping_cone}
			Let $\frakA$ be a category. Suppose we are given a commutative diagram in $[\frakA, \Cstar]$ of the following form:
			\[
			\begin{tikzcd}
				B_1 \ar[Rightarrow]{r}{\psi_1^2} \ar[Rightarrow]{d} & B_2 \ar[Rightarrow]{r}{\psi_2^3} \ar[Rightarrow, swap]{d}{\psi_2^4} & B_3 \ar[Rightarrow]{d}{\psi_3^5} \\
				0 \ar[Rightarrow]{r} & B_4 \ar[Rightarrow, swap]{r}{\psi_4^5} & B_5\mathrlap{.}
			\end{tikzcd}
			\]
			Set $\psi_2^5:=\psi_3^5 \circ \psi_2^3=\psi_4^5 \circ \psi_2^4 \colon B_2 \Rightarrow B_5$. 
			Consider the following diagram in $[\frakA, \Cstar]$:
			\[
			\begin{tikzcd}[
				ampersand replacement=\&,
				column sep=huge,
				row sep=small,
				every label/.append style={font=\small},
				cells={nodes={font=\small}}
				]
				M^2\mathopen{} \left(\Msquare{B_1}{B_2}{0}{B_5}{\psi_1^2}{}{\psi_2^5}{}\right)
				\ar[Rightarrow]{r}{
					M^2\!\begin{pmatrix}
						0 & \psi_2^4 \\
						0 & \id
					\end{pmatrix}
				}
				\ar[Rightarrow, swap]{dd}{
					M^2\!\begin{pmatrix}
						\psi_1^2 & \psi_2^3 \\
						0 & \id
					\end{pmatrix}
				}
				\&
				M^2\mathopen{} \left(\Msquare{0}{B_4}{0}{B_5}{}{}{\psi_4^5}{}\right)
				\ar[equal]{d}
				\\
				\& SC_{\psi_4^5} \ar[equal]{d} \\
				M^2\mathopen{} \left(\Msquare{B_2}{B_3}{B_4}{B_5}{\psi_2^3}{\psi_2^4}{\psi_3^5}{\psi_4^5}\right)
				\&
				M^2\mathopen{} \left(\Msquare{0}{0}{B_4}{B_5}{}{}{}{\psi_4^5}\right)
				\ar[Rightarrow]{l}{
					M^2\!\begin{pmatrix}
						0 & 0 \\
						\id & \id
					\end{pmatrix}
				}\mathrlap{.}
			\end{tikzcd}
			\]
			Then the two morphisms from the upper-left object to the lower-left object are homotopic to mutually orthogonal morphisms whose sum is homotopic to $0$. 
		\end{lem}
		
		\begin{proof}
			Let $A$ and $B$ denote the $2$-dimensional mapping cones appearing in the upper-left and lower-left corners of the above diagram, respectively.
			Let $\varphi \colon A \Rightarrow B$ be the left vertical morphism, and let $\psi \colon A \Rightarrow B$ be the morphism obtained by composing the two horizontal morphisms with the two identity morphisms along the right column of the above diagram. We shall homotope $\varphi$ and $\psi$ to mutually orthogonal morphisms and then homotope their sum to $0$ in $[\frakA, \Cstar]$. 
			
			Throughout the proof, for any C*-algebra $E$, we extend functions in $\conti_0([0, 1), E)$ and $\conti_0([0, 1) \times (0, 1), E)$ by zero to $[0, \infty)$ and $[0, 1) \times \R$, respectively, and use the same symbols for the resulting functions.
			
			Fix $Y \in \frakA$. To simplify notation, for any morphism in $[\frakA, \Cstar]$, we denote its component at $Y$ by the same symbol.
			The C*-algebra $A(Y)$ consists of all elements
			\[(a_1, a_2, a_5) \in B_1(Y) \oplus \conti_0([0, 1), B_2(Y)) \oplus \conti_0([0, 1) \times (0, 1), B_5(Y)) \]
			satisfying $\psi_1^2(a_1)=a_2(0)$ and $\psi_2^5 \circ a_2(t)=a_5(0, t)$ for all $t \in [0, 1)$.
			The C*-algebra $B(Y)$ consists of all elements
			\[
			\begin{aligned}
				(b_2, b_3, b_4, b_5) \in {}&B_2(Y) \oplus \conti_0([0, 1), B_3(Y))
				\oplus \conti_0([0, 1), B_4(Y))\\
				&\oplus \conti_0([0, 1) \times [0, 1), B_5(Y))
			\end{aligned}
			\]
			satisfying $\psi_2^3(b_2)=b_3(0)$, $\psi_2^4(b_2)=b_4(0)$, $\psi_3^5 \circ b_3(t)=b_5(0, t)$, and $\psi_4^5 \circ b_4(t)=b_5(t, 0)$ for all $t \in [0, 1)$.
			The $\ast$-homomorphisms $\varphi, \psi \colon A(Y) \to B(Y)$ are given by
			\[\varphi(a):=(\psi_1^2(a_1), \psi_2^3 \circ a_2, 0, a_5), \quad \psi(a):=(0, 0, \psi_2^4 \circ a_2, a'_5)\]
			for $a=(a_1, a_2, a_5) \in A(Y)$, where $a'_5$ is given by
			\[a'_5(t, s):=a_5(s, t)\]
			for $t, s \in [0, 1)$. 
			
			We shall define $\ast$-homomorphisms 
			\[\Phi, \Psi, \Theta \colon A(Y) \to \conti([0, 1], B(Y))\] 
			as follows. 
			The C*-algebra $\conti([0, 1], B(Y))$ consists of all elements
			\begin{align*}
				(f_2, f_3, f_4, f_5) \in & \ \conti([0, 1], B_2(Y)) \oplus \conti_0([0, 1] \times [0, 1), B_3(Y)) \\
				&\oplus \conti_0([0, 1] \times [0, 1), B_4(Y)) \oplus \conti_0([0, 1] \times [0, 1) \times [0, 1), B_5(Y)) 
			\end{align*}
			satisfying $\psi_2^3 \circ f_2(u)=f_3(u, 0)$, $\psi_2^4 \circ f_2(u)=f_4(u, 0)$, $\psi_3^5 \circ f_3(u, t)=f_5(u, 0, t)$, and $\psi_4^5 \circ f_4(u, t)=f_5(u, t, 0)$ for all $u \in [0, 1]$ and $t \in [0, 1)$.
			We first define $\ast$-homomorphisms $\Phi$, $\Psi$, and $\Theta$ from $A(Y)$ to 
			\begin{align*}
				&\conti([0, 1], B_2(Y)) \oplus \conti_0([0, 1] \times [0, 1), B_3(Y)) \\
				&\oplus \conti_0([0, 1] \times [0, 1), B_4(Y)) \oplus \conti_0([0, 1] \times [0, 1) \times [0, 1), B_5(Y)) 
			\end{align*}
			by 
			\[
			\Phi(a):=(f^a_2, f^a_3, 0, f^a_5), \quad \Psi(a):=(0, 0, g^a_4, g^a_5), \quad \Theta(a):=(h^a_2, h^a_3, h^a_4, h^a_5)
			\]
			for $a=(a_1, a_2, a_5) \in A(Y)$, where $f^a_2$, $f^a_3$, $f^a_5$, $g^a_4$, $g^a_5$, $h^a_2$, $h^a_3$, $h^a_4$, and $h^a_5$ are given by
			\begin{gather*}
				\begin{alignedat}{3}
					&f^a_2(u):=\psi_1^2(a_1), &\quad &f^a_3(u, s):=\psi_2^3 \circ a_2(s), & \quad &f^a_5(u, t, s):=a_5 \Bigl(t, \frac{s-ut}{1-ut} \Bigr), \\
					& & &g^a_4(u, t):=\psi_2^4 \circ a_2(t), & \quad &g^a_5(u, t, s):=a_5 \Bigl(s,  \frac{t-us}{1-us}\Bigr), \\
					&h^a_2(u):=a_2(u), &\quad &h^a_3(u, s):=\psi_2^3 \circ a_2(s+u), & \quad &h^a_4(u, t):=\psi_2^4 \circ a_2(t+u),
				\end{alignedat} \\
				h^a_5(u,t,s):=
				\begin{dcases}
					a_5\Bigl(t,\frac{s-t+u}{1-t}\Bigr) & \text{if $t \leq s$},\\
					a_5\Bigl(s,\frac{t-s+u}{1-s}\Bigr) & \text{if $t \geq s$}
				\end{dcases}
			\end{gather*}
			for $u \in [0, 1]$ and $t, s \in [0, 1)$. 
			
			We now show that $\Phi$, $\Psi$, and $\Theta$ take values in $\conti([0, 1], B(Y))$. Let $a=(a_1, a_2, a_5) \in A(Y)$, $u \in [0, 1]$, and $t \in [0, 1)$. 
			From
			\begin{gather*}
				\psi_2^3 \circ f^a_2(u)=\psi_2^3 \circ \psi_1^2(a_1)=\psi_2^3 \circ a_2(0)=f^a_3(u, 0), \\
				\psi_2^4 \circ f^a_2(u)=\psi_2^4 \circ \psi_1^2(a_1)=0, \\
				\psi_3^5 \circ f^a_3(u, t)=\psi_3^5 \circ \psi_2^3 \circ a_2(t)=\psi_2^5 \circ a_2(t)=a_5(0, t)=f^a_5(u, 0, t), \\
				f^a_5(u, t, 0)=a_5\Bigl(t, \frac{-ut}{1-ut} \Bigr)=0,
			\end{gather*}
			we see that $\Phi(a) \in \conti([0, 1], B(Y))$. 
			From
			\begin{gather*}
				g^a_4(u, 0)=\psi_2^4 \circ a_2(0)=\psi_2^4 \circ \psi_1^2(a_1)=0, \\
				g^a_5(u, 0, t)=a_5 \Bigl(t, \frac{-ut}{1-ut}\Bigr)=0, \\
				\psi_4^5 \circ g^a_4(u, t)=\psi_4^5 \circ \psi_2^4 \circ a_2(t)=\psi_2^5 \circ a_2(t)=a_5(0, t)=g^a_5(u, t, 0), 
			\end{gather*}
			we see that $\Psi(a) \in \conti([0, 1], B(Y))$. 
			From
			\begin{gather*}
				\psi_2^3 \circ h^a_2(u)=\psi_2^3 \circ a_2(u)=h^a_3(u, 0), \\
				\psi_2^4 \circ h^a_2(u)=\psi_2^4 \circ a_2(u)=h^a_4(u, 0), \\
				\psi_3^5 \circ h^a_3(u, t)=\psi_3^5 \circ \psi_2^3 \circ a_2(t+u)=\psi_2^5 \circ a_2(t+u)=a_5(0, t+u)=h^a_5(u, 0, t), \\
				\psi_4^5 \circ h^a_4(u, t)=\psi_4^5 \circ \psi_2^4 \circ a_2(t+u)=\psi_2^5 \circ a_2(t+u)=a_5(0, t+u)=h^a_5(u, t, 0),
			\end{gather*}
			we see that $\Theta(a) \in \conti([0, 1], B(Y))$. 
			Thus, we obtain $\ast$-homomorphisms
			\[\Phi, \Psi, \Theta \colon A(Y) \to \conti([0, 1], B(Y)).\] 
			
			We claim that $\ev_1 \circ \Phi$ and $\ev_1 \circ \Psi$ are mutually orthogonal, and that
			\[\ev_0 \circ \Phi=\varphi, \quad \ev_0 \circ \Psi=\psi, \quad \ev_1 \circ \Phi+\ev_1 \circ \Psi=\ev_0 \circ \Theta, \quad \ev_1 \circ \Theta=0. \]
			To prove the claim, let $a=(a_1, a_2, a_5) \in A(Y)$ and $t, s \in [0, 1)$. From
			\[f^a_2(0)=\psi_1^2(a_1), \quad f^a_3(0, s)=\psi_2^3 \circ a_2(s), \quad f^a_5(0, t, s)=a_5(t, s), \]
			we get $\ev_0 \circ \Phi=\varphi$. 
			From
			\[g^a_4(0, t)=\psi_2^4 \circ a_2(t), \quad g^a_5(0, t, s)=a_5(s, t)=a'_5(t, s), \]
			we get $\ev_0 \circ \Psi=\psi$. 
			We have
			\[f^a_5(1,t,s)
			=
			\begin{dcases}
				a_5\Bigl(t,\frac{s-t}{1-t}\Bigr) & \text{if $t \leq s$},\\
				0 & \text{if $t \geq s$},
			\end{dcases} \quad
			g^a_5(1,t,s)
			=
			\begin{dcases}
				0 & \text{if $t \leq s$},\\
				a_5\Bigl(s,\frac{t-s}{1-s}\Bigr) & \text{if $t \geq s$}. 
			\end{dcases}
			\]
			These formulas show that, for any $b, c \in A(Y)$,
			\[
			f_5^b(1,t,s) g_5^c(1,t,s)=0. 
			\]
			Thus, $\ev_1 \circ \Phi$ and $\ev_1 \circ \Psi$ are mutually orthogonal. 
			From 
			\begin{gather*}
				f^a_2(1)=\psi_1^2(a_1)=a_2(0)=h^a_2(0), \\
				f^a_3(1, s)=\psi_2^3 \circ a_2(s)=h^a_3(0, s), \quad g^a_4(1, t)=\psi_2^4 \circ a_2(t)=h^a_4(0, t), \\
				f^a_5(1, t, s)+g^a_5(1, t, s)=h^a_5(0, t, s),
			\end{gather*}
			we get $\ev_1 \circ \Phi+\ev_1 \circ \Psi=\ev_0 \circ \Theta$. 
			Finally, from
			\begin{gather*}
				h^a_2(1)=a_2(1)=0, \quad h^a_3(1, s)=\psi_2^3 \circ a_2(s+1)=0, \quad h^a_4(1, t)=\psi_2^4 \circ a_2(t+1)=0, \\
				h^a_5(1,t,s)=
				\begin{dcases}
					a_5\Bigl(t,1+\frac{s}{1-t}\Bigr)=0 & \text{if $t \leq s$},\\
					a_5\Bigl(s,1+\frac{t}{1-s}\Bigr)=0 & \text{if $t \geq s$}, 
				\end{dcases}
			\end{gather*}
			we obtain $\ev_1 \circ \Theta=0$. 
			This proves the claim. 
			Since $\psi_1^2$, $\psi_2^3$, and $\psi_2^4$ are natural in $Y \in \frakA$, so are $\Phi$, $\Psi$, and $\Theta$. 		
			Consequently, $\varphi$ and $\psi$ are homotopic to mutually orthogonal morphisms whose sum is homotopic to $0$ in $[\frakA, \Cstar]$.
		\end{proof}

		\bibliographystyle{amsalpha}
		\bibliography{ref}	
		
		%%%%%%%%%%%%%%%%%%%%%%%%%%%%%%%%%%%%%%%%%%%%%%%%%%%%%%%%%%%%%%%%%%%%%%%%%%%%%%%%%%%%%%%%%%%%%%%%%%%%%%%%%%%%%%%%%%%
	\end{document}